\documentclass[draft]{article}
\usepackage{amsmath,amsfonts,latexsym, amssymb}
\usepackage{dsfont}
\usepackage{stmaryrd}
\usepackage{enumerate}

\usepackage{color}

\usepackage[notref,notcite]{showkeys}
\usepackage{showkeys}

\newcommand{\nc}{\normalcolor}

\newcommand{\xip}{ {\xi'}}

\newcommand{\fa}{{\frak a}} 
\newcommand{\fb}{{\frak b}} 
\newcommand{\fw}{{\frak w}} 
\newcommand{\fq}{{\frak q}}

\newcommand{\non}{\nonumber}

\newcommand{\wt}{\widetilde}
\newcommand{\wh}{\widehat}

\renewcommand{\epsilon}{\varepsilon}
\newcommand{\e}{\varepsilon}
\newcommand{\eps}{\varepsilon}
\newcommand{\pt}{\partial}
\newcommand{\rd}{{\rm d}}

\newcommand{\bR}{{\mathbb R}}
\newcommand{\bZ}{{\mathbb Z}}
\newcommand{\bN}{{\mathbb N}}

\newcommand{\bx}{{\bf{x}}}
\newcommand{\by}{{\bf {y} }}
\newcommand{\bu}{{\bf{u}}}
\newcommand{\bv}{{\bf{v}}}
\newcommand{\bw}{{\bf{w}}}
\newcommand{\bz}{{\bf {z}}}

\newcommand{\cB}{{\mathcal  B }}
\newcommand{\cT}{{\mathcal  T }}

\newcommand{\tby}{\widetilde\by}

\newcommand{\bla}{\mbox{\boldmath $\lambda$}}

\newcommand{\al}{\alpha}

\newcommand{\be}{\begin{equation}}
\newcommand{\ee}{\end{equation}}

\newcommand{\ga}{{\gamma}}

\newcommand{\Om}{{\Omega}}
\newcommand{\om}{{\omega}}
\newcommand{\si}{\sigma}

\newcommand{\cL}{{\mathcal L}}

\newcommand{\cG}{{\mathcal G}}

\newcommand{\cN}{{\mathcal N}}
\newcommand{\cI}{{\mathcal I}}
\newcommand{\cF}{{\mathcal F}}
\newcommand{\cK}{{\mathcal K}}
\newcommand{\cM}{{\mathcal M}}
\newcommand{\cQ}{{\mathcal Q}}
\newcommand{\cR}{{\mathcal R}}
\newcommand{\cU}{{\mathcal U}}

\newcommand{\cS}{{\mathcal S}}

\newcommand{\E}{{\mathbb E }}
\newcommand{\R}{{\mathbb R }}

\renewcommand{\P}{{\mathbb P}}
\newcommand{\bC}{{\mathbb C}}

\newtheorem{theorem}{Theorem}
\newtheorem{corollary}[theorem]{Corollary}
\newtheorem{lemma}[theorem]{Lemma}
\newtheorem{proposition}[theorem]{Proposition}

\newtheorem{remark}[theorem]{Remark}
\newtheorem{definition}[theorem]{Definition}
\newcommand{\qed}{\hfill\fbox{}\par\vspace{0.3mm}}
\newenvironment{proof}{{\bf Proof.}} {\hfill\qed}

\numberwithin{theorem}{section}

\usepackage{geometry}     
\usepackage{graphicx}

\numberwithin{equation}{section}
\def\cal{}
\def\RR{{\mathbb R}}

\def\ZZ{{\mathbb Z}}

\def\cH{{\mathcal H}}

\def\W2{W^{1,2}({\cal O}(M))}

\def\O{\mathcal{ O}}

\def\1half{\frac{1}{2}}

\newcommand{\Lnorm}[2] {\left \| #2 \right  \|_{#1}}

\def\cA{{\mathcal A}}
\def\cO{{\mathcal O}}
\def\cW{{\mathcal W}}

\title{Gap Universality of Generalized Wigner and $\beta$-Ensembles}

\date{May 18, 2014}

\author{L\'aszl\'o Erd\H os${}^1$\thanks{Partially supported
by SFB-TR 12 Grant of the German Research Council and by ERC Advanced Grant, RANMAT 338804} \quad
Horng-Tzer Yau${}^2$\thanks{Partially supported
by NSF grants DMS-1307444 and Simons Investigator Award}
 \\\\  
Institute of Science and Technology Austria \\
Am Campus 1, A-3400 Klosterneuburg, Austria \\ lerdos@ist.ac.at${}^1$
 \\ \\ 
Department of Mathematics, Harvard University\\
Cambridge MA 02138, USA \\ 
 htyau@math.harvard.edu${}^2$ 
\\}

\begin{document}

\maketitle

\vspace{1cm}

\begin{abstract}
We consider  generalized Wigner ensembles and general $\beta$-ensembles with 
analytic potentials  for any  $\beta \ge 1 $. 
 The recent universality results in particular  assert 
that the local averages of  consecutive eigenvalue gaps in the bulk of the spectrum
  are universal in the sense that they coincide with those  of
the corresponding Gaussian $\beta$-ensembles. In this article, we 
show that local averaging is not necessary for this result, i.e. we
prove that the  single gap distributions in the  bulk  are
 universal.  In fact, with an additional step,
our result can be extended to any  potential $C^4(\bR)$. 
\end{abstract}

\vskip 1.5 cm 

{\bf AMS Subject Classification (2010):} 15B52, 82B44

\medskip

\medskip

{\it Keywords:} $\beta$-ensembles, Wigner-Dyson-Gaudin-Mehta universality,  gap distribution, log-gas.

\vspace{1cm}

\newpage 

\tableofcontents

\section{Introduction}

The fundamental vision  that random matrices can be used  as  basic  models 
for large quantum systems was due to  E.  Wigner   \cite{W}. 
 He  conjectured that the eigenvalue
gap distributions of large  random matrices were  universal (``Wigner surmise") in the sense 
that large quantum systems and random matrices share  the same gap distribution functions. 
 The subsequent works 
of  Dyson, Gaudin and Mehta clarified  many related issues regarding this assertion 
and a thorough  understanding of  the  Gaussian ensembles has thus emerged
 (see the classical book of Mehta \cite{M} for a summary).
There are two main categories of random matrices: the 
invariant and the non-invariant ensembles. The universality conjecture, 
 which is also known as the Wigner-Dyson-Gaudin-Mehta (WDGM) conjecture, 
asserts that for both ensembles the eigenvalue gap distributions are universal up to symmetry classes. 
For invariant ensembles, the joint distribution function of the eigenvalues
 can be expressed explicitly  in terms of one dimensional 
particle systems with logarithmic interactions (i.e., log-gases) 
 at an inverse temperature $\beta$. The values $\beta = 1, 2, 4$,
 correspond   to the classical orthogonal, unitary and symplectic ensembles, respectively.
 Under various conditions on the external potential, the universality for the classical values 
 $\beta = 1, 2, 4$ was proved, 
via analysis on the corresponding orthogonal polynomials,     
by  Fokas-Its-Kitaev \cite{FIK},   Deift {\it et. al.}
\cite{De1,  DKMVZ1, DKMVZ2},  Bleher-Its \cite{BI},  Pastur-Shcherbina \cite{PS:97, PS} 
and in  many consecutive works, see e.g.  \cite{DG, DG1,  KS, Sch, Wid}. 
For nonclassical values of $\beta$  there is no  matrix ensemble 
behind the model,  except for the Gaussian cases \cite{DE} via tridiagonal matrices.
One may still be interested in the local correlation functions of the 
log-gas as an interacting particle system. The orthogonal 
polynomial method  is not applicable for nonclassical values of $\beta$ even for the Gaussian case.
For certain special potentials and  even integer $\beta$, however,  there are still explicit formulas for correlation functions 
\cite{For}.
Furthermore, for general $\beta$ in the Gaussian case  the  local  statistics  were 
described very precisely  with a different method by Valk\'o-Vir\'ag \cite{VV, VV2}. 
The universality for general $\beta$-ensembles was 
established only very recently  \cite{BEY, BEY2} by a new  method based on  dynamical 
  methods
using  Dirichlet form  estimates 
from  \cite{ESY4, ESYY}. This method is important for this article and
 we will discuss it in more details later on. 
 All previous results achieved by this method, however, required  in their
statement to consider a local average of  consecutive
gaps. In the current paper we will
  prove universality of {\it each single} gap in the bulk.

Turning to the non-invariant ensembles, the most important class is the $N\times N$  Wigner matrices 
characterized by the independence of their entries.  In general,  there is no longer an 
explicit expression  for the joint distribution function for the eigenvalues. 
However, there is a special class of  ensembles, the  Gaussian divisible ensembles,
that    interpolate between the general  Wigner ensembles and 
the Gaussian  ones. For these ensembles,  at least in  
the special Hermitian case,  there is  still  an explicit formula for
the joint distribution of the eigenvalues based upon the Harish-Chandra-Itzykson-Zuber integral.
This formula was first put into a mathematically 
useful form by Johansson \cite{J} (see also the later work of  Ben Arous-P\'eche \cite{BP}) to prove 
the universality of Gaussian divisible ensembles with a Gaussian component  of size order one. 
In \cite{EPRSY}, the size of the Gaussian component  needed for proving the universality 
was greatly reduced to $N^{-1/2 + \e}$. 
More importantly,  the idea of approximating 
Wigner ensembles by Gaussian divisible  ones  was first introduced and,  after
a perturbation argument,  this 
resulted in the first proof of universality 
for Hermitian ensembles with general  smooth distributions for matrix elements.  
The smoothness condition was later on removed in   \cite{TV, ERSTVY}.

 In his seminal paper \cite{Dy}, Dyson observed that the  eigenvalue distribution of 
Gaussian divisible ensembles is the same as the solution  of  a special system of stochastic
differential equations, commonly known now as the Dyson Brownian motion, at  a fixed time $t$. 
For short times, $t$ is
comparable with the variance of the Gaussian component.  He also 
 conjectured that the time to ``local equilibrium" of the Dyson Brownian motion is of order $1/N$,
which is  then  equivalent to  the universality of Gaussian divisible ensembles with
a Gaussian component of order  slightly larger than $N^{-1/2}$.
Thus the work \cite{EPRSY} can be viewed 
as proving  Dyson's  conjecture for the Hermitian case. 
 This method, however, completely tied  with  an explicit formula
that is so far  restricted to the Hermitian case.

A completely analytic   approach to estimate the time to local equilibrium of
the Dyson's Brownian motion  was initiated in \cite{ESY4} and 
further developed in \cite{ESYY, EYYrigi, EYYBern}, see \cite{EY} for
a detailed account. 
In  these papers, Dyson's conjecture in full generality   was proved 
\cite{EYYrigi} and universality was established  for  generalized Wigner ensembles for all symmetric classes. 
The idea of a dynamical approach in proving   universality turns out to be a very powerful one. 
Dyson Brownian motion can be viewed as the natural gradient flow for Gaussian $\beta$ log-gases 
 (we will often use the terminology $\beta$ log-gases for the $\beta$-ensembles 
 to emphasize  the logarithmic interaction). 
The gradient flow can be defined with respect to all $\beta$ log-gases, not just the Gaussian ones. 
Furthermore, one can consider  gradient flows of local
 log-gases with fixed  ``good boundary conditions".  Here ``local'' refers to 
 Gibbs measures on  $N^a,  0 < a < 1,$  consecutive
points of a log-gas  with  the locations of all other points fixed. 
 By   ``good boundary conditions'' we  mean that  these external points are {\it rigid}, i.e.
their locations 
 are close to their classical locations given by the limiting density of the original log-gas. 
Using this idea, we have proved the universality of general $\beta$-ensembles in \cite{BEY, BEY2}
 for analytic potential.

The main conclusion of these works   is that 
the local  gap distributions of either  the
generalized  Wigner ensembles (in all symmetry classes) 
or the general $\beta$-ensembles 
are  universal in the bulk of the spectrum (see \cite{EYBull}
for a  recent review). The dynamical approach based on Dyson's Brownian
 motion and related flows also 
provides a conceptual understanding for the origin of the universality.
 For technical reasons, however,  these proofs 
 apply to   averages of consecutive gaps, i.e. cumulative
 statistics of $N^\eps$ consecutive
 gaps were proven to be universal.  Averaging the statistics of the consecutive gaps is 
equivalent to averaging the energy parameter
 in the correlation functions. Thus, mathematically, the
results were  also  formulated in terms of universality of the  correlation
functions with  averaging in  an  energy window of size $N^{-1+\e}$.

The main goal of this paper is to remove the local averaging   in  the statistics of consecutive
 gaps   in our general approach 
using Dyson Brownian motion for both invariant and non-invariant ensembles.
We will show that  the distribution of {\it each single} gap in the bulk is universal,
 which we will refer to as the {\it single gap universality} or 
simply the {\it gap universality} whenever there is no confusion. 
The  single 
gap universality was proved for  a special class of  Hermitian Wigner matrices with the property that  
the first four moments of the matrix elements match those of the standard Gaussian 
random  variable 
 \cite{Taogap}  and no other results  have been known before.  
In particular, {\it  the  single gap universality 
 has not been proved    even for the Gaussian orthogonal ensemble (GOE). }

The  gap distributions are  closely related  to  the correlation functions which were 
often used to state the universality of random matrices. These two 
concepts   are   equivalent in a  certain  average sense. However, 
   there is no rigorous relation between 
correlation functions at a fixed energy and single gap distributions. 
Thus our results  on single gap statistics  do not
automatically imply the universality of the correlation functions
at a fixed energy  which was rigorously proved
 only for Hermitian Wigner matrices
\cite{EPRSY, TV, ERSTVY, EY}.

The removal of a local  average in the universality results proved
 in \cite{ESY4, EYYBand, EYY2}  is  a technical improvement  in itself 
and its physical meaning is not especially profound. 
 Our motivation for taking seriously this endeavor  
is due to that  the single gap distribution may be closely related to  the distribution of a single eigenvalue
in the bulk of the spectrum \cite{Gus}  or  at the edge \cite{TW, TW2}. 
 Since  our approach does  not rely on 
any  explicit formula involving  Gaussian matrices,  some  extension of this method 
may provide a way to understand  the distribution  of an individual 
eigenvalue of Wigner matrices.   In fact, partly based on the method in this paper, the edge universality
for the $\beta$-ensembles and generalized Wigner ensembles was established in \cite{BEYedge}.

The main new idea in this paper is an analysis of the Dyson Brownian motion via parabolic
 regularity  using the  De Giorgi-Nash-Moser idea. 
   Since the Hamiltonians of the local log-gases are 
convex, the 
correlation functions  can  be  re-expressed 
in terms of  a time average of certain   random walks in random environments. The connection  between 
correlation functions of  general log-concave measures and  random walks in random environments
was already pointed out in the work of Helffer and Sj\"ostrand \cite{HS} and Naddaf 
and Spencer \cite{NS}. 
This connection  was used as an effective way to estimate correlation functions 
for several models in statistical physics, 
see, e.g.  \cite{BM, DGI, GOS, He, D}, 
  as well as to remove convexity assumptions
in gradient interface models \cite{CD, CDM}.

 In this paper we observe that  the single   gap universality  is a consequence of 
the H\"older regularity of the solutions to these 
random walk problems. 
Due to the logarithmic interaction, the random walks are long ranged
 and their rates may be singular.  Furthermore, the 
random environments themselves depend 
on the gap distributions, which were exactly the problems we want to analyze!
  If we view these random walks as (discrete) 
parabolic equations with random 
coefficients, we find that they are of divergence form and are in the form of 
the equations studied in the fundamental paper 
by    Caffarelli, Chan and Vasseur  \cite{C}.  The main 
difficulty to apply \cite{C} to gain  regularity is that the jump rates  in our settings 
 are  random and  they do not satisfy the uniform upper and lower bounds required
 in \cite{C}.  In fact,  in
some space-time regime the jump rates  can be  much more singular than were allowed in \cite{C}. 
For controlling the  singularities of these coefficients, we prove an optimal level repulsion estimate for
the local log-gases. 
With these estimates, we are able to extend the method  of  \cite{C} to prove H\"older
 regularity for the solution to these  random walks problems. 
This shows that the  single gap distributions are universal for local log-gases with good boundary conditions, 
 which is  the key result of this paper.

 For   $\beta$-ensembles,   it is known that the rigidity 
of the eigenvalues ensures  that   boundary conditions are good with  high probability. 
Thus we can apply the local universality of single gap distribution to get the  single  gap universality
of the $\beta$-ensembles.
We remark, however, that the current result holds only for $\beta\ge1$
in contrast to $\beta>0$ in \cite{BEY, BEY2}, 
since the  current   proof heavily relies on the dynamics of 
the gradient flow of local log-gases.
For non-invariant ensembles,  a slightly longer argument using  the local relaxation flow 
is needed to connect the local universality result with that for the original Wigner ensemble. 
 This will be explained in Section~\ref{sec:wigner}. 

 In summary,  we have recast the question 
 of the  single   gap universality for random matrices, 
envisioned by Wigner in the sixties, 
  into 
a  problem  concerning   the 
  regularity of a  parabolic equation in divergence form studied by De Giorgi-Nash-Moser. 
Thanks to  the insight of Dyson and the important  progress by 
 Caffarelli-Chan-Vasseur \cite{C}, we are able to establish the  WDGM  universality conjecture 
for each individual gap
 via  De Giorgi-Nash-Moser's idea. 
We now introduce our models rigorously and state 
the main results.

\section{Main results}

 We will have two related results, one concerns the generalized Wigner ensembles, the
other one the general $\beta$-ensembles. 
We first define the generalized Wigner ensembles. 
  Let $H=(h_{ij})_{i,j=1}^N$  be an $N\times N$  hermitian or symmetric matrix  where the
 matrix elements $h_{ij}=\bar {h}_{ji}$, $ i \le j$, are independent 
random variables given by a probability measure $\nu_{ij}$ 
with mean zero and variance $\sigma_{ij}^2\ge 0$;
\be
  \E \, h_{ij} =0, \qquad \sigma_{ij}^2:= \E |h_{ij}|^2.
\label{aver}
\ee
The distribution $\nu_{ij}$ and its variance $\sigma_{ij}^2$ may depend on $N$,
 but we omit this fact in the notation. We also assume that the
normalized matrix elements have a uniform subexponential decay,
\be\label{subexp} 
  \P( |h_{ij}|> x \sigma_{ij})\le \theta_1 \exp{(-x^{\theta_2})}, \qquad x>0,
\ee
with some fixed  constants $\theta_1, \theta_2 >0$,  uniformly in $N, i, j$. In fact, 
with minor modifications of the proof,  an algebraic
decay 
$$
   \P( |h_{ij}| > x \sigma_{ij})\le C_Mx^{-M}
$$
with a large enough $M$ is also sufficient.

\begin{definition}[\cite{EYYBand}]
 The matrix ensemble $H$  defined  above is called generalized Wigner matrix if 
the following  assumptions hold on the variances of the matrix
elements \eqref{aver}
\begin{description}
\item[(A)] For any $j$ fixed
\be
   \sum_{i=1}^N \sigma^2_{ij} = 1 \, .
\label{sum}
\ee

\item[(B)]   There exist two positive constants, $C_{inf}$ and $C_{sup}$,
independent of $N$ such that
\be\label{1.3C}
\frac{C_{inf}}{N} \le \sigma_{ij}^2\leq \frac{C_{sup}}{N}.
\ee
\end{description} 
\end{definition}
 Let $\P$ and $\E$ denote the probability and the
expectation with respect to
this ensemble.

We will denote by  $\lambda_1\le \lambda_2 \le\ldots \le \lambda_N$
the eigenvalues of $H$.
In the special case when $\sigma^2_{ij}=1/N$ and $h_{ij}$ is  Gaussian, 
the joint probability distribution of the
eigenvalues  is given
\be\label{H}
\mu=\mu^{(N)}_{G}(\rd \bla)=
\frac{e^{-N\beta\cH(\bla)}}{Z_\beta}\rd \bla,\qquad \cH(\bla) =
\sum_{i=1}^N \frac{\lambda_{i}^{2}}{4} -  \frac{1}{N}  \sum_{i< j}
\log |\lambda_{j} - \lambda_{i}|.
\ee
The value of $\beta$ depends on the symmetry class of the matrix;
 $\beta=1$ for GOE, $\beta=2$ for GUE and $\beta=4$ for GSE.
Here $Z_\beta$ is the normalization factor so that $\mu$ is a  probability measure. 

It is well known that the density or the one point correlation function
of $\mu$ converges, as $N\to\infty$, to the Wigner semicircle law
\be\label{scdef}
   \varrho(x): =\frac{1}{2\pi}\sqrt{ (4-x^2)_+}.
\ee
 We use the notation $\gamma_{j}$
for the $j$-th quantile of this density, i.e. $\gamma_j$ is defined by
\be\label{defgamma}
   \frac{j}{N} = \int_{-2}^{\gamma_{j}}\varrho_{G}(x) \rd x.
\ee
We now define a class of test functions. Fix an integer $n$. We say that $O = O_N : \bR^n\to \R$,
a possibly $N$-dependent sequence of differentiable functions, is an {\it $n$-particle observable} if
\be\label{obs}
  O_\infty:=\sup_N \|O_N\|_\infty<\infty, \qquad \mbox{supp}\, O_N\subset [-O_{\mbox{\footnotesize supp}},
O_{\mbox{\footnotesize supp}}]^n
\ee
with some finite $O_{\mbox{\footnotesize supp}}$, independent of $N$,  but we allow $\| O'_N\|_\infty$ to grow with $N$. 
For any integers $A<B$ we also introduce the notation $\llbracket A, B\rrbracket : = \{ A, A+1, \ldots, B\}$.

Our main result on the generalized Wigner matrices 
asserts that the local gap statistics in the bulk of the spectrum
are universal for any general Wigner matrix, in particular
they coincide with those of the Gaussian case.

\begin{theorem}  [Gap universality for Wigner matrices] \label{thm:wigner} Let $H$ be 
a generalized Wigner ensemble
with subexponentially decaying matrix elements, \eqref{subexp}.
Fix  positive numbers $\al, O_\infty, O_{\mbox{\footnotesize supp}}$ and
an integer $n\in \bN$.
 There exists an $\e>0$ and $C>0$,   depending only on $\alpha$, $\O_\infty$ and
$O_{\mbox{\footnotesize supp}}$   such that for any $n$-particle observable $O=O_N$
satisfying \eqref{obs} we have 
\be\label{EEO}
   \Big| \big[\E -\E^{\mu}\big] O\big( N(\lambda_j-\lambda_{j+1}), N(\lambda_j-\lambda_{j+2}),
 \ldots , N(\lambda_j - \lambda_{j+n})\big)\Big|
    \le CN^{-\e}  \| O'\|_\infty
\ee
for any $j\in \llbracket \al N, (1-\al)N\rrbracket$ and
for any sufficiently large  $N\ge N_0$, where $N_0$ depends
on all parameters of the model, as well as on $n$, $\alpha$, $\O_\infty$ and
$O_{\mbox{\footnotesize supp}}$. 

More generally, for any $k, m \in \llbracket \al N, (1-\al)N\rrbracket$ we have
\begin{align} \label{EEO1}
   \Big| \E  O\big( & (N\varrho_k) (\lambda_k-\lambda_{k+1}),   (N\varrho_k)(\lambda_k-\lambda_{k+2}), \ldots , (N\varrho_k)(\lambda_k - \lambda_{k+n})\big) \\
& -\E^{\mu} O\big( (N\varrho_m)(\lambda_m-\lambda_{m+1}), (N\varrho_m)(\lambda_m-\lambda_{m+2}), \ldots , (N\varrho_m)(\lambda_m - \lambda_{m+n})\big)\Big|
    \le CN^{-\e} \| O'\|_\infty,  \non
\end{align}
where the local density $\varrho_k$ is defined by
$\varrho_k : =\varrho(\gamma_k)$.
\end{theorem}

It is well-known that the  
gap distribution of  Gaussian random matrices  for all symmetry classes   can be explicitly
expressed via a Fredholm determinant provided that a certain local average  is taken, see \cite{De1, DG, DG1}. 
The result for a single gap, i.e. without local averaging,  was only achieved recently  in  the special 
case of the Gaussian unitary ensemble (GUE) by Tao \cite{Taogap}  (which  then  easily  implies
 the same results for  Hermitian Wigner  matrices satisfying  the  
  four moment matching condition). It is not clear if  
a similar argument can be applied to the GOE case.

\bigskip

We now define the  $\beta$-ensembles  with a general external potential.
Let $\beta>0$ be a fixed parameter.
Let $V(x)$ be a real analytic\footnote{In fact, $V\in C^4(\bR)$ is sufficient, see Remark~\ref{rem:c4}.}
 potential on $\bR$ that grows  faster than $(2+\eps)\log |x|$ 
at infinity and satisfies
\be\label{lowerder}
    \inf_\bR V'' >-\infty.
\ee
Consider the measure
\be\label{mubeta}
   \mu=\mu_{\beta, V}^{(N)}(\rd \bla)=
\frac{e^{-N\beta\cH(\bla)}}{Z_\beta}\rd\bla,\qquad \cH(\bla) =
\frac{1}{2}\sum_{i=1}^N V(\lambda_{i}) -  \frac{1}{N}  \sum_{i< j}
\log |\lambda_{j} - \lambda_{i}|.
\ee
Since $\mu$ is symmetric in all its variables,  we will mostly
 view it as a measure restricted to the cone
\be\label{simplex}
\Xi^{(N)}:= \{\bla \; :\; \lambda_1<\lambda_2< \ldots < \lambda_N\}\subset \RR^N.
\ee
Note that the Gaussian measure \eqref{H} is a special case of \eqref{mubeta}
with $V(\lambda)=\lambda^2/2$. 
In this case we use the notation $\mu_G$ for $\mu$.

Let
$$
   \varrho^{(N)}_1(\lambda) : = \E^\mu \frac{1}{N}\sum_{j=1}^N \delta(\lambda-\lambda_j)
$$
denote  the density, or the  one-point function, of $\mu$. It is well known 
\cite{APS, BPS} that
$\varrho^{(N)}_1$  converges weakly to
the equilibrium density $\varrho=\varrho_{V}$
as $N\to\infty$. The equilibrium density can be characterized as
the unique minimizer (in the set of probability measures on $\bR$ endowed
with the weak topology) of the functional
\be\label{varprin}
   I(\nu) =  \int V(t) \rd\nu (t)  - \iint \log|t-s|\rd\nu(t)\rd\nu(s).
\ee
In the case, $V(x)=x^2/2$, the minimizer is the Wigner semicircle law $\varrho=\varrho_G$, defined in
\eqref{scdef}, where the subscript $G$ refers to the Gaussian case. 
In the general case
we assume that 
$\varrho=\varrho_{V}$ is supported on a single compact interval, $[A,B]$ and $\varrho\in C^2(A,B)$. 
Moreover, we assume that  $V$ is {\it regular} in the sense that $\varrho$ is strictly positive on $(A, B)$
and vanishes as a square root at the endpoints, see (1.4) of \cite{BEY2}.
It is known that these condition are satisfied if, for example, $V$ is strictly convex.

For any $j\le N$ define the classical location of the $j$-th particle 
$\gamma_{j,V}$ by
\be\label{defgammagen}
   \frac{j}{N} = \int_A^{\gamma_{j,V}}\varrho_V(x) \rd x,
\ee
and for the Gaussian case we have $[A,B]=[-2,2]$ and we use the notation $\gamma_{j,G} =\gamma_j$
for the corresponding classical location, defined in \eqref{defgamma}. 
We set
\be\label{rhov}
   \varrho_j^V := \varrho_V(\gamma_{j,V}),\qquad
\mbox{and}\qquad \varrho_j^{G} : = \varrho_{G} (\gamma_{j,G})
\ee
to be the limiting density at the  classical  location of the $j$-th particle.
 Our main theorem on the 
$\beta$-ensembles is the following. 

\begin{theorem}  [Gap universality for $\beta$-ensembles]  \label{thm:beta} Let   $\beta \ge 1$   and $V$ be
 a real analytic\footnote{In fact, $V\in C^4(\bR)$ is sufficient, see Remark~\ref{rem:c4}.} potential
with \eqref{lowerder} such that $\varrho_V$ is  supported on a single compact interval,
 $[A,B]$, $\varrho_V\in C^2(A,B)$, and that  $V$ is regular. 
 Fix  positive numbers $\al, O_\infty, O_{\mbox{\footnotesize supp}}$,
an integer $n\in \bN$ and an $n$-particle observable $O=O_N$ satisfying \eqref{obs}. 
  Let $\mu=\mu_V = \mu_{\beta, V}^{(N)}$ be given by \eqref{mubeta}
and let $\mu_G$ denote the same measure for the Gaussian case.
Then there exist an $\e>0$, depending only on $\alpha, \beta$ 
and  the potential $V$,
and a constant $C$ depending on  $\O_\infty$ and
$O_{\mbox{\footnotesize supp}}$ such that
\begin{align}\label{betaeq}
 \Bigg| & \E^{\mu_V} O\Big( (N\varrho_k^V) (\lambda_k-\lambda_{k+1}), (N\varrho_k^V) (\lambda_{k}-\lambda_{k+2}), \ldots , 
(N\varrho_k^V) (\lambda_k - \lambda_{k+n})\Big) \\ \nonumber
  &  - \E^{\mu_{G}} 
  O\Big( (N\varrho_m^{G}) (\lambda_m-\lambda_{m+1}), (N\varrho_m^G) (\lambda_{m}-\lambda_{m+2}), \ldots , 
(N\varrho_m^G)(\lambda_m - \lambda_{m+n})\Big)
 \Bigg| 
  \le CN^{-\e}\| O'\|_\infty
\end{align}
for any $k,m\in \llbracket \al N, (1-\al)N\rrbracket$ and
for any sufficiently large   $N\ge N_0$, where $N_0$ depends
on $V$, $\beta$, as well as on $n$, $\alpha$, $O_\infty$ and $ O_{\mbox{\footnotesize supp}}$.
In particular, the distribution of the rescaled gaps
 w.r.t. $\mu_V$ 
does not depend on the index $k$ in the bulk.
\end{theorem}

Theorem \ref{thm:beta}, in particular, asserts that the single gap distribution  in the bulk is 
independent of the index $k$.   The special GUE case of this assertion is the content of \cite{Taogap} 
 where the proof uses some  special structures of  GUE.

\medskip

The proofs of both Theorems~\ref{thm:wigner} and \ref{thm:beta}
 rely on the uniqueness of the gap distribution
for a localized version of the equilibrium measure \eqref{H} with
a certain class of boundary conditions. This main technical result will be 
formulated in  Theorem~\ref{thm:local} in the next section after 
we introduce the necessary notations.  An orientation of the content of the paper 
will be given at the end of Section \ref{sec:loc}. 

We remark that  Theorem \ref{thm:beta} is stated  only for $\beta \ge 1$; on the contrary, 
the universality with local averaging  in  \cite{BEY, BEY2}
was proved for  $\beta>0$. The main reason is that the current proof 
 relies heavily on the dynamics of the gradient flow of local log-gases.
 Hence the well-posedness of the dynamics is crucial which is available
only for $\beta \ge 1$. 
On the other hand,  in  \cite{BEY, BEY2} we use only  certain Dirichlet  form inequalities 
 (see, e.g. Lemma~5.9 in \cite{BEY}),  which  we could prove
with an effective regularization scheme for all $\beta> 0$.   For $\beta < 1$ it is not clear if such a regularization 
can also be applied  to the new inequalities we will prove here.

\section{Outline of the main ideas in the proof}

For the orientation of the reader we briefly outline the three main concepts in the proof
without any technicalities. 

\medskip

{\it 1. Local Gibbs measures and their comparison}

\medskip

The first observation is that the macroscopic structure of the Gibbs measure
$\mu_{\beta, V}^{(N)}$, see \eqref{mubeta}, heavily depends on $V$ via the density $\varrho_V$.
The microscopic structure, however, is essentially determined by the logarithmic
interaction alone, the local density plays only the role of a scaling factor.
Once the measure is localized, its dependence on $V$ is reduced to a simple linear
rescaling. This gives rise to the idea to consider the {\it local Gibbs measures},
defined on $\cK$ consecutive particles  (indexed by a set $I$)   by conditioning on all other $N-\cK$ particles.
The frozen particles, denoted collectively by $\by =\{ y_j\}_{j\not \in I}$,
play the role of the boundary conditions. 
The potential of the local Gibbs measure $\mu_\by$ is given by $\frac{1}{2}V_\by(x) =
\frac{1}{2} V(x) - \frac{1}{N}\sum_{j\not\in I} \log
|x-y_j|$. From the rigidity property of the measure $\mu$, see \cite{BEY},
 the frozen particles are typically very close to their classical
locations determined by the appropriate quantiles of the equilibrium density $\varrho_V$.
 Moreover, from the Euler-Lagrange
equation of \eqref{varprin} we have $V(x) = 2\int \log |x-y|\varrho_V(y)\rd y$.
These properties, together with the choice $\cK\ll N$  ensure that $V_\by$ is small away
from the boundary. Thus, apart from boundary effects, the local Gibbs measure
is independent of the original potential $V$. In particular, its gap statistics
can be compared with that of the Gaussian ensemble after an appropriate rescaling.
For convenience, we scale all local measures so that the typical size of their
gaps is one.

\bigskip

{\it 2. Random walk representation of the covariance}

\bigskip

The key technical difficulty is to estimate the boundary effects
which is given by 
the correlation between the external potential $\sum_i V_\by(x_i)$
and the gap observable $O(x_j-x_{j+1})$ (for simplicity we look at
one gap only). We introduce the notation $\langle X ; Y\rangle:= \E XY -
 \E X \, \E Y$ to denote the covariance of two random variables $X$ and $Y$.
Following the more customary statistical physics terminology, 
we will refer to  $\langle X ; Y\rangle$  as correlation.  Due
to the long range of the logarithmic interaction, the 
two-point correlation function $\langle \lambda_i ; \lambda_j\rangle$ \nc
 of a log-gas decays only logarithmically in $|i-j|$, i.e. very slowly.
What we really need is the correlation between a particle
 $\lambda_i$ and a gap $\lambda_j-\lambda_{j+1}$
which  decays faster, as $|i-j|^{-1}$, but we need quite precise
estimates to exploit the gap structure.  

For any  Gibbs measure $\om(\rd\bx) = e^{-\beta \cH(\bx)}\rd \bx$ with strictly
convex Hamiltonian, $\cH''\ge c>0$, the correlation of
any two observables $F$ and $G$ can be expressed as
\be\label{rwr}
  \langle F(\bx); G(\bx) \rangle_\om = \frac{1}{2}\int_0^\infty
   \rd s \int \rd\om(\bx)\E_\bx \big[\;  \nabla G(\bx(s))\cU(s, \bx(\cdot)) \nabla F(\bx) \big].
\ee
Here $\E_\bx$ is the expectation for the (random) paths $\bx(\cdot)$ starting
from $\bx(0)=\bx$ and solving the canonical SDE for the measure $\om$:
$$
    \rd \bx(s) =\rd {\bf B}(s) - \beta \nabla\cH ( \bx(s))\rd s
$$
and $\cU(s) =\cU(s, \bx(\cdot))$ is the fundamental solution to the linear
system of equations
\be\label{matrixeq}
   \partial_s \cU(s) =- \cU(s) \cA(s), \qquad \cA(s): = \beta \cH''(\bx(s))
\ee
 with $\cU(0)=I$. 
Notice that the coefficient matrix $\cA(s)$, and thus the fundamental
solution, depend on the random path $\bx(s )$.

If $G$ is a function of the gap, $G(\bx)= O(x_j-x_{j+1})$,
 then \eqref{rwr} becomes
 \be\label{rwr1}
  \langle F(\bx); O(x_j-x_{j+1}) \rangle_\om = \frac{1}{2}\int_0^\infty
   \rd s \int \rd\om(\bx)\sum_{i\in I} 
  \E_\bx \Big[\; O'(x_j-x_{j+1}) \big(
\cU_{i,j}(s) -\cU_{i,j+1}(s)\big)  \pt_i F(\bx) \Big].
\ee
We will estimate the correlation \eqref{rwr1} by showing that for a typical
path $\bx(\cdot)$ the solution $\cU(s)$ is H\"older-regular in a sense that
$\cU_{i, j}(s)-\cU_{i, j+1}(s)$ is small if $j$ is away from the boundary and
$s$ is not too small. The exceptional cases require various technical cutoff estimates.

\bigskip

{\it 3. H\"older-regularity of the solution to \eqref{matrixeq}}

\bigskip

We will apply \eqref{rwr1} with the choice $\om = \mu_\by$ and with
a function $F$ representing the effects of the boundary conditions.  
For any fixed realization of the path $\bx(\cdot)$, we will view 
the equation \eqref{matrixeq} as a finite dimensional version of
a parabolic equation. The coefficient matrix,  the Hessian of
the local Gibbs measure, is computed explicitly. It can be
 written as
$\cA = \cB +\cW$, where $\cW\ge 0$ is  diagonal, $\cB$ is symmetric
with quadratic form
$$
    \langle \bu, \cB(s) \bu\rangle = \frac{1}{2} \sum_{i,j\in I} B_{ij}(s) (u_i-u_j)^2,
  \qquad B_{ij}(s):= \frac{\beta}{(x_i(s)-x_j(s))^2}.
$$
After rescaling the problem and writing it in microscopic coordinates
where the gap size is of order one,  \nc
for a typical path and large $i-j$ we have 
\be\label{cbs}
   B_{ij}(s) \sim \frac{1}{(i-j)^2}
\ee
by rigidity. We also have a lower bound for any $i\ne j$
\be\label{lowcbs}
  B_{ij}(s) \gtrsim \frac{1}{(i-j)^2},
\ee
at least with a very high probability.
If a matching upper bound were true for any $i\ne j$, then \eqref{matrixeq}
would be the discrete  analogue of the general  equation
\be
\partial_t u(t, x) = \int K(t, x, y) [u(t, y) - u(t, x) ] \rd y, \qquad t>0, \quad x, y\in \R^d 
\ee
  considered by  Caffarelli-Chan-Vasseur  in \cite{C}.   It is assumed that the 
 kernel $K$ is symmetric and there is a constant $0< s < 2$ such 
that the short distance singularity can be bounded by  
\be\label{Kxyt1}
C_1  |x-y|^{-d-s} \le K(t, x, y) \le C_2  |x-y|^{-d-s}
\ee
for some positive constants $C_1, C_2$. 
Roughly speaking, the integral operator corresponds to the behavior 
of the operator $|p|^{s}$,  where $p=-i\nabla$.  The main result of \cite{C} asserts
that for any $t_0>0$, the solution $u(t,x)$ is $\e$-H\"older continuous,
 $u\in C^\e( (t_0, \infty), \R^d)$,
for some positive exponent $\e$  that depends only on $t_0$, $C_1$, $C_2$. 
Further generalizations and related local regularity results
such as weak Harnack inequality can be found in \cite{FK}.

Our equation \eqref{matrixeq} is of this type with $d=s=1$, but it
is discrete and in a finite interval $I$ with a potential term. 
The key difference, however, is
that the coefficient  $B_{ij}(t)$   in the elliptic part of \eqref{matrixeq} 
can be singular in the sense that $B_{ij}(t) |i-j|^2$ is not uniformly bounded
 when $i,j$ are close to each other.
 In fact, by extending the reasoning
of Ben Arous and Bourgade \cite{BB}, the minimal gap $\min_i (x_{i+1}-x_i)$ for GOE 
is typically of order $N^{-1/2}$   in the microscopic coordinates we are using now.
 Thus the analogue of the uniform upper bound  \eqref{Kxyt1} does not even hold for a fixed $t$. 
The only control we can guarantee for the singular behavior of $B_{ij}$
with a large probability is the
estimate 
\be\label{Kass1}
   \sup_{0 \le s \le \si}\sup_{0 \le M\le CK\log K} \frac{1}{ 1+ s} \int_0^s \frac{1}{M}
 \sum_{i\in I\, : \, |i-Z| \le M}
 B_{i,i+1}(s)
 \rd s \le CK^{\rho}
\ee
with some small exponent $\varrho$ 
and for any $Z\in I$ far away from the boundary of $I$.
 This estimate essentially says that 
the space-time maximal function of $B_{i, i+1}(t)$   at a fixed space-time point $(Z,0)$ 
is bounded by $K^\rho$. 
Our main generalization of the result in \cite{C} is to show
that the weak upper bound \eqref{Kass1}, together with \eqref{cbs} and \eqref{lowcbs}
(holding up to a factor $K^\xi$)
are sufficient for proving a discrete version of the H\"older continuity at the point $(Z,0)$.
More precisely,  in  Theorem \ref{holderg} 
 we essentially show that there exists a $\fq>0$ such that 
for any fixed $\si\in [K^c, K^{1-c}]$, 
 the solution to \eqref{matrixeq} satisfies
\be\label{holds}
   \sup_{|j-Z|+ |j'-Z|\le \si^{1-\al}}
  | \cU_{i, j}(\si) - \cU_{i, j'}(\si)| \le C K^\xi
\si^{-1- \frac{1}{2}\fq\al}, 
\ee
with any $\al\in[0,1/3]$  if we can 
guarantee that $\rho$ and $\xi$ are sufficiently small. 
The exponent $\fq$ is a universal positive number and it plays the
role of the H\"older regularity exponent. In fact, to obtain H\"older
regularity around one space-time point $(Z,\si)$ as in \eqref{holds},
we need to assume the bound \eqref{Kass1} around several (but not
more than $(\log K)^C$) space-time points, which in our
applications can  be guaranteed with high probability.

Notice that 
 $\cU_{i,j}(\si)$ decays as $\si^{-1}$, hence \eqref{holds} provides an
additional decay for the discrete derivative. In particular, this
guarantees that the $\rd s$ integration in \eqref{rwr1} is finite
in the most critical intermediate regime $s\in [K^c, CK\log K]$.

The proof  of Theorem \ref{holderg} is given in Section~\ref{Caff}.
In this section we also formulate a H\"older regularity  result
for initial data in $L^\infty$ (Theorem~\ref{thm:hold}), which
 is the basis of all other results. Readers interested
in the pure PDE aspect of our work are referred to Section~\ref{Caff}
which  can be read independently of the other sections of the paper.

\nc

\section{Local equilibrium measures}

\subsection{Basic properties of local equilibrium measures} 
\label{sec:loc}

Fix two small positive numbers, $\al, \delta>0$. Choose two positive integer
 parameters $L, K$ such that
\be\label{K}
L\in \llbracket \alpha N,  (1-\alpha)N\rrbracket, \qquad
N^\delta\le K\le N^{1/4}.
\ee
We consider the parameters $L$ and $K$ fixed and often we
will not indicate them in the notation.
All results will hold for any sufficiently small $\al, \delta$ 
and for any sufficiently large $N\ge N_0$, where  the threshold $N_0$ depends
on $\al, \delta$  and maybe on other parameters of the model. 
 Throughout the paper we will use $C$ and $c$ to denote
positive constants which, among others,  may depend on $\al, \delta$
and on the constants in \eqref{subexp} and \eqref{1.3C}, but we will not emphasize
this dependence.  Typically  $C$ 
 denotes a large generic constant, while $c$ denotes a small one
whose values  may change from line to line.
These constants are independent of $K$ and $N$, which
are  the limiting large parameters of the problem, but they may depend on each other.
In most cases this interdependence is harmless since it only requires
that a fresh constant $C$ be sufficiently large or $c$ be sufficiently small, depending on the size
of the previously established generic constants. In some cases, however, 
the constants are related in a more subtle manner. In this case we will use $C_0, C_1, \ldots$
and $c_0, c_1, \ldots$  etc.
to denote  specific constants in order to be able to refer to them along the proof. 
\nc

For convenience, we set
$$
  \cK: = 2K+1.
$$
Denote $ I = \nc I_{L, K}:= \llbracket L-K, L+K \rrbracket$ the set of  $\cK$ consecutive indices in the bulk.
We will distinguish the inside and outside particles
by renaming them as
\be\label{renamexy}
(\lambda_1, \lambda_2, \ldots,
\lambda_N):=(y_{1}, \ldots y_{L-K-1}, x_{L-K},  \ldots, x_{L+K}, y_{L+K+1},
  \ldots y_{N}) \in \Xi^{(N)}.
\ee
 Note that the particles  keep their original indices.
The notation $\Xi^{(N)}$ refers to the simplex  \eqref{simplex}.
In short we will write
$$
\bx=( x_{L-K},  \ldots x_{L+K} ), \qquad \mbox{and}\qquad
 \by=
 (y_{1}, \ldots y_{L-K-1}, y_{L+K+1},
  \ldots y_{N}).
$$
These points are always listed  in increasing order, i.e. $\bx\in \Xi^{(\cK)}$ and
$\by \in \Xi^{(N-\cK)}$.
We will refer to the $y$'s as the  {\it external}
points and to the $x$'s as  {\it internal} points.

We will fix the external points (often called
 boundary conditions) and study
the conditional measures on the internal points.
We first define  the
{\it local equilibrium measure}  (or {\it local measure} in short) 
 on $\bx$ with boundary condition  $\by$ by
\begin{equation}\label{eq:muyde}
 \quad
\mu_{\by} (\rd\bx) : = \mu_\by(\bx) \rd \bx, \qquad
\mu_\by(\bx):=  \mu (\by, \bx) \left [ \int \mu (\by, \bx) \rd \bx \right ]^{-1},
\end{equation}
where $\mu=\mu(\by, \bx)$ is the (global) equilibrium measure \eqref{mubeta}
(we do not distinguish between the measure $\mu$ and its density function $\mu(\by, \bx)$
in the notation).
Note that for any fixed $\by\in \Xi^{(N-\cK)}$,  all  $x_j$
lie in the {\it  open  configuration interval},  denoted by   
$$
 J=J_\by:=(y_{L-K-1}, y_{L+K+1}).
$$
Define
$$
   \bar y: = \frac{1}{2}( y_{L-K-1}+y_{L+K+1})
$$
to be the midpoint of   the configuration interval.
We also introduce
\be\label{aldef}
\alpha_j: = \bar y + \frac{j-L}{\cK+1}|J|, \qquad j\in I_{L,K},
\ee
 to denote the $\cK$ equidistant points
within the interval $J$.

For any fixed $L, K, \by$,
the equilibrium measure can also be written as a Gibbs measure,
\be\label{muyext}
   \mu_\by  = \mu_{\by, \beta, V}^{(N)}= Z_\by^{-1} e^{-N\beta \cH_\by},
\ee
with Hamiltonian
\begin{align}
   \cH_\by (\bx): = & \sum_{i\in I} \frac{1}{2}  V_\by (x_i) -\frac{1}{N}
   \sum_{i,j\in I\atop i<j} \log |x_j-x_i|, \non \\
\label{Vyext} 
   V_\by(x) := & V(x) - \frac{2}{N}\sum_{j\not\in I} \log |x-y_j|.
\end{align}
Here
 $V_\by(x)$ can be viewed as the external potential of
a $\beta$-log-gas of the points $\{ x_i \; : \; i\in I\} $
in the configuration interval $J$.

\bigskip

Our main technical result, Theorem~\ref{thm:local} below,  asserts that,
for  $K,L$ chosen according to \eqref{K}, the local gap statistics 
are essentially independent
of $V$ and $\by$ as long as  the boundary conditions $\by$ are regular.
This property is expressed by defining the following set of
  ``good'' boundary conditions 
 with some given positive parameters $\nu,  \al$: 
\begin{align}\label{yrig}  
\cR_{L, K}= \cR_{L,K}(\nu,\al):= & \{ \by:  
     |y_k-\gamma_k|\le N^{-1+\nu},
 \quad \forall k \in \llbracket\alpha N, (1-\alpha)N\rrbracket \setminus I_{L, K} \} \\ \nonumber
  & \cap 
  \{ \by:   |y_k-\gamma_k|\le N^{-4/15+\nu}, 
 \quad \forall k \in \llbracket N^{3/5+\nu},  N- N^{3/5+\nu}\rrbracket 
  \} \\
 \nonumber  &  \cap 
  \{ \by: |y_k-\gamma_k|\le 1, \;\; \forall k \in\llbracket 1, N\rrbracket\setminus I_{L,K} \} .
\end{align}
In Section~\ref{sec:beta} we will see that this definition is tailored to 
the previously proven rigidity bounds for the $\beta$-ensemble, see \eqref{muR}.
The rigidity bounds for the generalized Wigner matrices are stronger,
see \eqref{rigidity}, so this definition will suit the needs of both proofs.

\medskip

\begin{theorem} [Gap universality for local measures]  \label{thm:local} 
Fix $L, \wt L$ and $\cK= 2 K+ 1$ satisfying \eqref{K}
with an exponent $\delta>0$. Consider two
boundary conditions $\by, \wt\by$ such that the configuration intervals coincide,
\be\label{J=J}
   J = (y_{L-K-1}, y_{L+K+1}) = (\wt y_{\wt L-K-1}, \wt y_{\wt L+ K+1}).
\ee
We consider the
measures $\mu = \mu_{ \by, \beta, V}$ and $\wt\mu = \mu_{ \wt\by, \beta, \wt V}$
defined as in \eqref{muyext},
with possibly two different external potentials $V$ and $\wt V$.
 Let $\xi>0$ be a small constant. 
 Assume that  $|J|$ satisfies 
\be\label{Jlen}
  |J|  =   \frac{\cK}{N \varrho(\bar y) } + O\Big(\frac{K^\xi}{N}\Big).
\ee
Suppose that   $\by, \wt\by\in \cR_{L,K}(\xi^2\delta/2,  \al/2 )$  
and that 
\be\label{Ex}
  \max_{j\in I_{L,K}}  \Big| \E^{ \mu_\by} x_j -  \alpha_j\Big| +  
 \max_{j\in I_{\wt L,K}} 
\Big| \E^{\wt \mu_{\wt\by}} x_j -  \alpha_j\Big|\le CN^{-1}K^\xi 
\ee
holds.
Let the integer number $p$ satisfy
$|p| \le K-K^{1-\xi^*}$ for some   small $\xi^*>0$.  Then 
there exists $\xi_0 > 0$, depending on
$\delta$,   such that if $\xi,\xi^* \le \xi_0$ 
 then  for   any $n$ fixed and  any  $n$-particle observable 
$O=O_N$ satisfying \eqref{obs} with fixed control parameters $O_\infty$ and $ O_{\mbox{\footnotesize supp}}$,
we have 
\begin{align}\label{univ}
 \Bigg|  \E^{\mu_\by} 
  O\big(  N(x_{L+p}- x_{L+p+1}), & \ldots  N(x_{L+p}-x_{L+p+n} ) \big) \\ \non
& -   \E^{ \wt\mu_{\tilde \by } } 
  O\big(  N(x_{\wt L + p}- x_{\wt L + p+1}), \ldots  N(x_{\wt L + p}-x_{\wt L +p+n} ) \big)
 \Bigg|\le C K^{-\e}\| O'\|_\infty 
\end{align}
for some $\e > 0$ depending on $\delta,\al$ 
and  for some $C$ depending 
on $O_\infty$ and $ O_{\mbox{\footnotesize supp}}$. This holds for any $N\ge N_0$ sufficiently large,
where $N_0$ depends on the parameters $\xi,\xi^*, \al$,  and  $C$ in \eqref{Ex}.
\end{theorem}

In the following two theorems  we establish
rigidity and level repulsion estimates for the local log-gas $\mu_\by$ 
with  good boundary conditions $\by$. 
 While both rigidity and level repulsion  are basic questions for log gases, 
our main motivation to prove these theorems is to use them in the proof of Theorem \ref{thm:local}. 
The  current form of the  level repulsion estimate is  new, a weaker form
was proved  in (4.11) of \cite{BEY}.  
 The rigidity estimate 
 was proved  for the global equilibrium measure $\mu$  in \cite{BEY}. 
 {F}rom this estimate, 
one can  conclude that $\mu_\by$ has a good rigidity bound for 
a set of boundary conditions with high probability  w.r.t. the global measure $\mu$.  However, 
we will need  a rigidity estimate  for  $\mu_\by$ 
for a set of $\by$'s with  high probability with respect to 
some different measure,  which may be
asymptotically singular to $\mu$ for large $N$. 
 For example,  in the proof  for the gap universality of  Wigner
 matrices such a measure is given by the time evolved 
measure $f_t \mu$, see Section~\ref{sec:wigner}. 
The following result asserts that a rigidity estimate holds 
for $\mu_\by$ provided that $\by$ itself satisfies a
 rigidity bound and  an extra  condition, \eqref{Exone},  holds.
 This provides explicit criteria to describe the set of 
``good'' $\by$'s whose measure w.r.t. $f_t\mu$ can then be estimated
with different methods.

\begin{theorem}  [Rigidity estimate for local measures]  \label{thm:omrig} 
 Let $L$ and $K$ satisfy \eqref{K} with $\delta$
 the exponent appearing in \eqref{K}. Let $\xi, \al$ be any
fixed positive constants.
For   $\by \in \cR_{L, K}(\xi\delta/2,\al) $  consider
the local equilibrium measure $\mu_\by$ defined in \eqref{muyext}
  and  assume that
\be\label{Exone}
   \Big| \E^{ \mu_\by} x_j -  \alpha_j\Big| 
\le CN^{-1}K^\xi, \quad  j\in I=I_{L,K},
\ee
is satisfied.
Then there are positive constants $C, c$, depending on $\xi$,
such that  for any $k \in  I $ and $u>0$, 
\be\label{rig}
   \P^{\mu_\by}\Big( \big| x_k - \al_k\big| \ge u K^{ \xi } N^{-1}\Big)\le C e^{-c u^2 }.
\ee
\end{theorem}

Now we state the level repulsion  estimates which will be proven
in Section~\ref{sec:lr}.

\begin{theorem}  [Level repulsion estimate for local measures] 
 \label{lr2} Let $L$ and $K$ satisfy \eqref{K} 
and let $\xi, \al$ be any
fixed positive constants.
 Then for  $\by \in \cR_{L,K}=\cR_{L, K}(\xi^2\delta/2,\al) $ we have the  following estimates: 

\noindent 
i)  [Weak form of level repulsion]  For any $s>0$ 
 we have 
\be\label{k521}
\P^{ \mu_\by} [  x_{i+1} - x_{i} \le s/N   ] \le
  C \left ( N  s \right ) ^{\beta + 1},  \qquad i\in\llbracket L-K-1, L+K\rrbracket 
\ee
and 
\be\label{l21}
\P^{ \mu_\by} [   x_{i+2} - x_{i} \le s/N   ] \le  C  ( N   s  ) ^{2 \beta + 1}
 \qquad i\in\llbracket L-K-1, L+K-1\rrbracket . 
\ee
(Here we used the convention that $x_{L-K-1} := y_{L-K-1}, x_{L+K+1}:= y_{L+K+1}$.)

\noindent 
ii)   [Strong form of level repulsion] 
Suppose that there exist  positive constants   $C, c$ such that the following rigidity estimate holds
for any $k\in I$:  
\be\label{weakrig}
   \P^{\mu_\by}\Big( |x_k-\al_k|\ge CK^{\xi^{ 2}} N^{-1}\Big) \le C \exp{(-K^{c})}.
\ee
Then there exists small a constant $\theta$, depending on $C, c$ in \eqref{weakrig},
such that for any $s\ge \exp (- K^{\theta})$.
we have 
\be\label{k52}
\P^{ \mu_\by} [  x_{i+1} - x_{i} \le s/N   ] \le
  C \left ( K^{\xi } s  \log N \right ) ^{\beta + 1}, \qquad i\in\llbracket L-K-1, L+K\rrbracket 
\ee
and 
\be\label{l20}
\P^{ \mu_{\by}} [  x_{i+2} - x_{i} \le s/N   ] \le
  C \left ( K^{\xi } s   \log N \right ) ^{2 \beta + 1},  \qquad i\in\llbracket L-K-1, L+K-1\rrbracket .
\ee
\end{theorem}

 We remark that the estimates  \eqref{l20}  and \eqref{l21} on the second gap
are not needed for the main proof, we listed them only for possible further reference.
The exponents  
are not optimal; one would expect  them to be $ 3 \beta+3$. With  some extra work, it 
should  not be difficult  to get the optimal exponents. 
Moreover,  our results can be extended  to  $x_{i+k} - x_i$ for any $k$ finite.
 We also mention that the assumption \eqref{weakrig} required in part ii)
is  weaker than  what we prove in  \eqref{rig}. In fact, the weaker form  \eqref{weakrig}
of the rigidity would be enough throughout the paper except at one
place, at the end of the proof of Lemma~\ref{lm:91}.

Theorem~\ref{thm:local} is our key result. 
In  Sections \ref{sec:beta} and \ref{sec:wigner} we will show how to use 
Theorem~\ref{thm:local} to prove the main Theorems~\ref{thm:wigner}
and \ref{thm:beta}. Although  the basic structure of the proof of Theorem \ref{thm:beta} is 
 similar to the one given in \cite{BEY}
where a locally  averaged  version
 of this  theorem was proved under a  locally
averaged  version of  
Theorem~\ref{thm:local}, here we have  to verify the assumption \eqref{Ex}
which will be done in Lemma~\ref{lm:goody1}. 
 The proof of Theorem~\ref{thm:wigner}, on the other hand, 
is very different from the recent proof of universality in \cite{EYYBand, EYY2}. 
 This  will be explained  
in Section~\ref{sec:wigner}.

The proofs of the auxiliary
 Theorems~\ref{thm:omrig} and \ref{lr2} 
will be given in Section \ref{sec:profmu}. 
The proof of Theorem~\ref{thm:local} will start from Section~\ref{sec:comp}
 and will continue until the end of the paper. 
At the beginning of  Section~\ref{sec:comp} 
we will  explain the main ideas of the proof. 
For readers interested in the proof of Theorem~\ref{thm:local}, 
Sections~\ref{sec:beta} and \ref{sec:wigner} can be skipped.

\subsection{Extensions and further results}

 We formulated Theorems~\ref{thm:local}, \ref{thm:omrig} 
and \ref{lr2} with assumptions requiring that the boundary conditions $\by$
are ``good''. 
  In fact, all these results hold in a more general setting.

\begin{definition}\label{def:regU} An external potential $U$ of a $\beta$-log-gas of $K$
points in a configuration interval $J=(a,b)$ is called {\it $K^\xi$-regular}, if
the following bounds hold:
\begin{align}\label{Jlengthgen}
    |J| & =   \frac{\cK}{N \varrho(\bar y) } + O\Big(\frac{K^\xi}{N}\Big), 
\\
\label{Vby1gen}
   U'(x) & = \varrho(\bar y) \log \frac{d_+(x)}{d_-(x)}
   + O\Big(\frac{K^\xi}{N d(x)}\Big),   \qquad x\in J,
\\
\label{Vbysecgen}
   U''(x) & \ge \inf V'' +\frac{c}{ d(x)},   \qquad x\in J,
\end{align}
with some positive $c>0$ and
for some small $\xi>0$,
where
$$
   d(x) := \min\{ |x-a|, |x-b|\}
$$
is the distance to the boundary of $J$ and
$$ 
 d_-(x) := d(x) + \varrho(\bar y)N^{-1}K^\xi, 
\qquad d_+(x) := \max\{ |x-a|, |x-b|\} + \varrho(\bar y) N^{-1}K^\xi.
$$
\end{definition}

The following lemma, proven 
in Appendix~\ref{sec:goody},
asserts that ``good'' boundary conditions $\by$
give rise to regular external potential $V_\by$.

\begin{lemma}\label{lm:goody} Let $L$ and $K$ satisfy \eqref{K} 
 and $\delta$ is the exponent appearing in \eqref{K}. 
Then for any $\by \in  \cR_{L, K}(\xi \delta/2, \al/2)$
the external potential $V_\by$ \eqref{Vyext} on the configuration interval $J_\by$
is $K^\xi$-regular;
\begin{align}\label{Jlength}
    |J_\by| & =   \frac{\cK}{N \varrho(\bar y) } + O\Big(\frac{K^\xi}{N}\Big), 
\\
\label{Vby1}
   V_\by'(x) & = \varrho(\bar y) \log \frac{d_+(x)}{d_-(x)}
   + O\Big(\frac{K^\xi}{N d(x)}\Big),   \qquad x\in J_\by,
\\
\label{Vbysec}
   V_\by''(x) & \ge \inf V'' +\frac{c}{ d(x)},   \qquad x\in J_\by.
\end{align}
\end{lemma}

\noindent
 {\it Remark.} The proofs of Theorems~\ref{thm:local}, \ref{thm:omrig}, \ref{lr2} and
\ref{lr}  do not use the explicit
form of $V_\by$ and $J_\by$; they  depend only on the property that 
  $V_\by$ on $J_\by$ is regular.

\section{Gap universality for $\beta$-ensembles:  proof of Theorem~\ref{thm:beta} }\label{sec:beta}

\subsection{Rigidity bounds and its consequences}

The aim of this section is to use Theorem~\ref{thm:local} to prove Theorem~\ref{thm:beta}. 
In order to  verify   the assumptions of Theorem~\ref{thm:local}, we first  recall 
the rigidity estimate  w.r.t. $\mu$ defined in \eqref{mubeta}. 
Recall that
 $\gamma_k=\gamma_{k,V}$ denotes the classical location 
of the $k$-th point \eqref{defgammagen}.  For the case of convex potential, in 
Theorem 3.1 of \cite{BEY} it was proved that  for any fixed
$\alpha>0$ and $\nu>0$, there are constants
 $C_0,c_1,c_2>0$   such that for any $N\geq 1$ and $k\in\llbracket \alpha N,(1-\alpha) N\rrbracket$,
\begin{equation}\label{bulkrig}
\P^\mu\left(
|\lambda_k-\gamma_k|> N^{-1+\nu}\right)\leq  C_0\exp{(-c_1N^{c_2})}. 
\end{equation}
The same  estimate holds also  for 
the non-convex case,  see Theorem 1.1 of \cite{BEY2}, by using a convexification argument.

Near the spectral edges, a somewhat weaker control  was  proven  for the convex case,
 see Lemma 3.6
of \cite{BEY}  which states that   for any $\nu>0$ there are  $C_0,c_1,c_2>0$ such that
\be\label{nearedge}
    \P^\mu\left(|\lambda_k-\gamma_k|> N^{-4/15+\nu}\right)\leq   C_0\exp{(-c_1N^{c_2})}
\end{equation}
 for any $N^{3/5+\nu}\le k \le N-N^{3/5+\nu}$,
 if $N\ge N_0(\nu)$ is sufficiently
large. We can choose $C_0, c_1, c_2$ to be the same in \eqref{bulkrig} and \eqref{nearedge}. 
 Combining this result with  the convexification argument in \cite{BEY2},  one can show that 
the estimate \eqref{nearedge} holds also for the non-convex case.

Finally, we have a very weak control that holds for all points  (see (1.7) in \cite{BEY2}): 
  for any $C>0$ there are positive constants $C_0, c_1$ and $ c_2$ such that
 \be\label{global}
    \P^\mu\left(|\lambda_k-\gamma_k|>  C \right)\leq  C_0\exp{(-c_1N^{c_2})}.
\end{equation}
Given $C$, we can choose the positive constants $C_0, c_1, c_2$   to be the same in
\eqref{bulkrig},  \eqref{nearedge} and \eqref{global}. 
 
The set  $\cR_{L, K}$ in \eqref{yrig} was exactly defined
as the set of events that  these three rigidity estimates hold.
{F}rom \nc \eqref{bulkrig},  \eqref{nearedge} and \eqref{global} we have
\be
   \P^\mu ( \cR_{L, K}(\nu,\al))\ge 1- C_0\exp{(-c_1N^{c_2})}
\label{muR}
\ee
for any $\nu>0$, $\al>0$
with some positive constants $C_0, c_1, c_2$ that depend on $\nu$ and $\al$.

\begin{remark}\label{rem:c4}
The real analyticity of $V$ in this paper is used only to obtain the rigidity results \eqref{bulkrig}
\eqref{nearedge} and \eqref{global} using earlier results from \cite{BEY, BEY2}.
After the first version of the current work appeared 
in the ArXiv,  
jointly with P. Bourgade we proved the following stronger rigidity result  
(Theorem 2.4 of \cite{BEYedge})
For any $\beta>0$, $\xi>0$ and $V\in C^4(\bR)$, 
regular with equilibrium density supported on a single interval $[A,B]$,
there is a  $c>0$ and $N_0$ such that
\begin{equation}\label{rignew}
\P^\mu\left(|\lambda_k-\gamma_k|> N^{-\frac{2}{3}+ \xi }(\hat k)^{-\frac{1}{3}}
\right)\leq e^{- N^c}, \qquad \forall k\in\llbracket 1,N\rrbracket 
\end{equation}
holds for any $N\ge N_0$. This result allows us to relax the original real analyticity
condition  to $V\in C^4(\bR)$. It would also allow us to redefine
the  set  $\cR_{L, K}$ in \eqref{yrig} to the more transparent set  appearing in \eqref{rig}, 
but this generalization does not affect the rest of the proof.
\end{remark}

\begin{lemma}\label{lm:goody1} Let $L$ and $K$ satisfy \eqref{K}  and $\delta$ 
is the exponent appearing in \eqref{K}.  Then for any small $\xi$ and $\al$ 
 there exists
a set $\cR^* = \cR^*_{L,K, \mu}(\xi^{ 2 } \delta/2, \al/2)
\subset  \cR_{L, K}(\xi^{ 2 } \delta/2, \al/2)$ such that
\be\label{muR*}
   \P^\mu (\cR^*)\ge 1- C_0\exp{\Big(-\frac{1}{2}c_1N^{c_2}\Big)} 
\ee
with the constants $C_0, c_1, c_2$ from \eqref{bulkrig}. 
Moreover,
 for any $\by \in \cR^*$ we have
\be\label{Exmu}
\big|\E^{\mu_\by} x_k -\al_k\big|\le CN^{-1}K^\xi, \qquad k\in I_{L,K},
\ee
where  $\al_k$ was defined in  \eqref{aldef}.
\end{lemma}

{\it Proof.}  
 For any $\nu>0$ define
\be \label{R*}
  \cR^*_{L,K,\mu}(\nu, \al) : = 
  \Big\{ \by\in \cR(\nu,\al) \; : \; 
\P^{\mu_\by}\left(|x_k-\gamma_k|> N^{-1+\nu}
\right)\le  \exp{\big(-\frac{1}{2} c_1N^{c_2}\big)},  \;\; \forall k\in I_{L,K}
 \Big\}
\ee
with the $\nu$-dependent constants $ c_1, c_2>0$ from \eqref{muR}.
Note that  $\cR^*$, unlike $\cR$, depends on the underlying measure $\mu$
through the family of its conditional measures $\mu_\by$. 
Applying \eqref{muR} for $\nu=\xi^2\delta/2$  and setting
$\cR =  \cR_{L, K}(\xi^2 \delta/2, \al/2)$, $\cR^* = \cR^*_{L,K, \mu}(\xi^{ 2 } \delta/2, \al/2)$,
we have
$$
   \P^\mu( \cR^*) \ge 1-  C_0\exp{\big(-\frac{1}{2} c_1N^{c_2}\big)}
$$
with some $C_0, c_1, c_2$.
Now if $\by\in \cR^*$, then
\be\label{xg}
   \big| \E^{\mu_\by}x_k -\gamma_k\big|\le C_0e^{-c_1N^{c_2}/3}  + CN^{-1}K^{\xi^2},  \qquad k\in I_{L,K}.
\ee
 In order to prove \eqref{Exmu}, it remains to show that $|\alpha_k - \gamma_k| $ is 
 bounded by $CN^{-1}K^{\xi}$ for any   $k\in I_{L,K}$.
To see this,  we can use that $\varrho\in C^1$ away from the edge, thus
$$
      \varrho(x) =\varrho(\bar y) + O(x - \bar y)
$$
 (recall that $\bar y$ is the midpoint of $J$). 
By Taylor expansion we
have
$$
k-(L-K-1) = N\int_{\gamma_{L-K-1}}^{\gamma_k} \varrho  =N\int_{y_{L-K-1}}^{\gamma_k} \varrho +  O(N^{\xi\delta/2})
= N|\gamma_k-y_{L-K-1}|\varrho(\bar y)  + O(N|J|^2+N^{\xi\delta/2}),
$$
i.e.
\be\label{ggg}
     \gamma_k = y_{L-K-1} + \frac{k-L+ K+1}{N\varrho(\bar y)} + O(N^{-1}K^\xi).
\ee
 Here we used that $J=J_\by$ satisfies  \eqref{Jlength} according to Lemma~\ref{lm:goody},
since $\by\in \cR_{L,K}(\xi^2\delta/2, \al/2)\subset \cR_{L,K}(\xi\delta/2, \al) $. 
 Comparing \eqref{ggg} with the definition of $\al_k$,
\eqref{aldef}, using \eqref{Jlength} and the fact that $\bar y- y_{L-K-1}=\frac{1}{2}|J|$, we have
\be\label{alga}
   |\al_k-\gamma_k| \le CN^{-1}K^\xi.
\ee
Together with \eqref{xg} this implies \eqref{Exmu}
and this completes the proof of Lemma~\ref{lm:goody1}. \qed

\subsection{Completing the proof of Theorem~\ref{thm:beta} }

 We first notice that it is sufficient to prove Theorem~\ref{thm:beta}
for the special case $m=N/2$, i.e. when the local statistics for the Gaussian measure is 
considered   at  the central point of the spectrum. Indeed, once Theorem~\ref{thm:beta} is proved for any $V$,  $k$
and $m=N/2$, then  with the choice $V(x)= x^2/2$ we can use it to establish that
the local statistics for the Gaussian measure around  any fixed index $k$  in the bulk
coincide with the local statistics in the middle. \nc  So from now
on we assume $m=N/2$, but we carry the notation $m$ for simplicity.

Given $k$ and  $m=N/2$ as in \eqref{betaeq}, we first choose $L, \wt L, K$,
satisfying \eqref{K} (maybe with a smaller $\al$ than given in Theorem~\ref{thm:beta}),
 so that  $k=L+p$, $m=\wt L +p$ hold for some 
$|p|\le K/2$.
In particular
\be\label{wtL}
|\wt L - N/2|\le K. 
\ee
For brevity, we use $\mu=\mu_V$ and $\wt \mu=\mu^G$ in accordance
with the notation of Theorem~\ref{thm:local}.

We consider $\by \in R_{L,K,\mu}^*(\xi^2\delta/2, \al)$ 
and $\wt \by\in R_{\wt L, K,\wt\mu}^*(\xi^2\delta/2,\al)$, where
 $\delta$ is the exponent appearing in \eqref{K}.
We omit the arguments and  recall that
\be\label{except}
   \mu  (  R_{L,K,\mu}^*)  \ge 1-  C_0\exp{\Big(-\frac{1}{2}c_1N^{c_2}\Big)}, \qquad  
 \wt\mu(R_{\wt L, K,\wt\mu}^*) \ge 1-  C_0\exp{\Big(-\frac{1}{2}c_1N^{c_2}\Big)} 
\ee
with some positive constants.

\begin{proposition}\label{prop:local}  With the above choice of the parameters and 
for  any $\by \in R_{L,K,\mu}^*(\xi^2\delta/2, \al)$ and $\wt \by\in R_{\wt L, K,\wt\mu}^*(\xi^2\delta/2,\al)$,
we have
\begin{align}\label{univ1}
 \Bigg|  \E^{\mu_\by} 
  O\big(  (N\varrho^V_{L+p})& (x_{L+p}- x_{L+p+1}), \ldots   (N\varrho^V_{L+p})(x_{L+p}-x_{L+p+n} ) \big) \\ \non
& -   \E^{\wt \mu_{\wt\by}} 
  O\big(  (N\varrho^G_{\wt L +p})(x_{\wt L + p}- x_{\wt L + p+1}), \ldots  (N\varrho^G_{\wt L +p}) 
(x_{\wt L + p}-x_{\wt L +p+n} ) \big)
 \Bigg|\le C K^{-\e} \| O'\|_\infty,
\end{align}
 where $\e$ is from Theorem~\ref{thm:local}. 
\end{proposition}

Theorem~\ref{thm:beta} follows immediately from \eqref{except} and this proposition.  \qed

The rest of this section is devoted to the proof of  Proposition~\ref{prop:local}.

\bigskip

{\it Proof of Proposition~\ref{prop:local}.} We will apply Theorem~\ref{thm:local},
but first we have to bring the two measures onto the same configuration interval $J$
to satisfy \eqref{J=J}.  
This will be done in  three  steps. 
First, using the scale invariance of the Gaussian log-gas 
 we rescale it so that the local density
approximately matches with that of $\mu_V$. 
This will guarantee that the two configuration intervals have almost the
same length. In the second step
we adjust the local  Gaussian  log-gas  $\wt \mu_{\wt\by}$  so that $J_{\wt\by}$ has
exactly the correct length. Finally, we shift the two intervals so that they coincide. This allows us 
to apply Theorem~\ref{thm:local} to conclude the local statistics are identical.

The local densities $\varrho_V$ around $\gamma_{L+p,V}$ and 
$\varrho_G$ around $\gamma_{\wt L+p,G}$ may considerably differ. So in
the first step we rescale
the Gaussian log-gas  so that
\be\label{match}
   \varrho_V( \gamma_{L+p,V}) = \varrho_G(\gamma_{\wt L+p,G}).
\ee
To do that, recall
that we defined the Gaussian log-gas  with the standard $V(x)=x^2/2$
external potential, but we could choose $V_s(x) = s^2x^2/2$ with any fixed $s>0$
and consider the Gaussian  log-gas 
$$
  \mu_G^s (\bla)\sim \exp\big(-N\beta \cH_s(\bla)\big), \qquad
\cH_s(\bla): =  \frac{1}{2}\sum_{i=1}^N V_s(\lambda_i) -\frac{1}{N}\sum_{i<j} 
\log|\lambda_j-\lambda_i|.  
$$
This results in a rescaling of the semicircle density $\varrho_G$ 
to $\varrho_G^s(x):= s\varrho_G(sx)$ and $\gamma_{i, G}$ to 
$\gamma_{i, G}^s:=s^{-1}\gamma_{i, G}$ for any $i$,
so $\varrho_G(\gamma_{i,G})$ gets rescaled to $\varrho_G^s(\gamma_{i, G}^s)
=s\varrho_G(\gamma_{i,G})$. In particular, $\varrho_G(\gamma_{\wt L+p, G})$ 
is rescaled to $s\varrho_G(\gamma_{\wt L+p,G})$, and thus
choosing $s$ appropriately, we can achieve that \eqref{match} holds
 (keeping the left hand side fixed). 
Set
$$
 \cO_s(\bx) : = O\big(  (N\varrho_G^s(\gamma_{m,G}^s))(x_m-x_{m+1}), \ldots ,  (N\varrho_G^s(\gamma_{m,G}^s))
(x_m-x_{m+n})\big), \qquad m= \wt L +p,
$$
and notice that  $\cO_s(\bx) = \cO(s\bx)$. This means that
the local gap statistics $\E^{\mu_G^s} \cO_s$
is independent of the scaling parameter $s$, since
the product $(N\varrho^{G}_m) (x_m-x_{m+a})$   (notation defined in \eqref{rhov})  is unchanged under the scaling.
 So we can work
with the rescaled Gaussian measure. For notational simplicity
we will not carry the $s$ parameter further and we just assume that \eqref{match} holds
with the original Gaussian $V(x)=x^2/2$.

We  have now achieved that the two densities at  at some points 
 of the configuration intervals coincide, but
the lengths of these two intervals still slightly differ. 
In the second step we match them exactly.
Since $\by  \in R_{L,K}(\xi\delta/2, \al)$ and $\wt \by\in R_{\wt L, K}(\xi\delta/2,\al)$,
from  \eqref{Jlength} 
in Lemma~\ref{lm:goody} we see that
\begin{align}\label{JJY}
  |J_\by|=  & \;  | y_{L+K+1} -y_{L-K-1}| = \frac{\cK}{N\varrho_V(\bar y)} + O(N^{-1}K^\xi) \\
 |J_{\wt \by}|= & \;  | \wt y_{L+K+1} -\wt y_{L-K-1}|
   = \frac{\cK}{N\varrho_G(\overline {\wt y})} + O(N^{-1}K^\xi).
\end{align}
Since $\varrho_V$ is $C^1$, we have  for any $|j|\le K$,
\begin{align*}
  |\varrho_V(\bar y) -\varrho_V(\gamma_{L+j,V}) |\le & C  | \bar y - \gamma_{L+j,V}| \\
  \le  & C  | \bar y - y_{L,V}| +C|\gamma_{L+j,V}-\gamma_{L,V}|+ O(N^{-1}K^\xi) \le C K N^{-1},
\end{align*}
and similarly for $\varrho_G (\overline {\wt y})$. 
 
 Using \eqref{JJY}, \eqref{match}
and that the densities are separated away from zero, we easily obtain that
\be\label{JJ}
 s:= \frac{  |J_\by|}{ |J_{\wt \by}|} 
\qquad \mbox{satisfies} \qquad s=s_{\by,\wt\by}= 1+  O( K^{-1+\xi}).
\ee
Note that this $s$ is different from the scaling parameter in the first step
but it will play a similar role so we use the same notation.
For each fixed $\by, \wt \by$ we can now scale the conditional 
Gaussian log-gas   $\mu_{\wt \by}$
by a factor $s$, i.e. change $\wt\by$ to $s\wt\by$,
 so that after rescaling $ |J_\by|= |J_{s\wt \by}|$.

We will now show that this rescaling does not alter the gap statistics:
\begin{lemma}\label{lm:res} Suppose that $s$ satisfies \eqref{JJ}
and let $\mu=\mu_G$ be the Gaussian log-gas.  
Then we have 
\be
  \big| \big[\E^{\mu_{\wt\by}} - \E^{\mu_{s\wt\by}}\big]
  \cO(\bx)\big| \le  CK^{-1+\xi} 
\label{sy}
\ee
with
$$
 \cO(\bx) : = O\big(  (N\varrho_m^G)(x_m-x_{m+1}), \ldots , (N\varrho_m^G)(x_m-x_{m+n})\big)
$$
for any $\wt L -K \le m \le \wt L +K-n$
(note that the observable is not rescaled).
\end{lemma}

{\it Proof.}
Define the Gaussian log-gas  
$$
   \mu_{\wt\by}^s \sim e^{-N\beta\cH_{\wt\by}^s}
$$
with $\cH^s_{\wt \by}$ defined exactly as in  \eqref{Vyext} but $V_\by(x)$ is replaced with 
$$
   V_{\wt\by}^s (x) = V_s(x) - \frac{2}{N}\sum_{j\not\in \wt I}\log |x-\wt y_j|, \qquad V_s(x) = \frac{1}{2} s^2x^2, \qquad 
     \wt I : = \llbracket \wt L-K, \wt L +K\rrbracket . 
$$
Then  by scaling
\be\label{sca}
    \E^{\mu_{s\wt\by}}  \cO(\bx) =  \E^{\mu_{\wt\by}^s}  \cO(\bx/s)
  = \E^{\mu_{\wt\by}^s}  \cO(\bx)  +   O\big( \| O'\|_\infty  |s-1|),
\ee
where in the last step we used that the observable $O$ is a differentiable function  with compact support. The error term 
is negligible by \eqref{JJ} and \eqref{K}.

In order to control $\big[\E^{\mu_{\wt\by}^s} -\E^{\mu_{\wt\by}} \big] \cO(\bx)$, it is
sufficient to bound  the relative entropy $S(\mu_{\wt\by}^s|\mu_{\wt\by})$.
However, for any $\by\in R_{L,K}$ we have
\be\label{Hyconv}
  \cH_\by'' \ge \min_{x\in J_\by}\frac{1}{N}\sum_{j\not \in I} \frac{1}{|x-y_j|^2} 
 \ge \frac{cN}{K}
\ee
with a positive constant. Applying this for $\wt \by$,
we see that  $\mu_{\wt\by}$ satisfies the  logarithmic Sobolev inequality (LSI)
$$
    S(\mu_{\wt\by}^s|\mu_{\wt\by}) \le \frac{CK}{N} D (\mu_{\wt\by}^s|\mu_{\wt\by}),
$$
where 
$$
   S(\mu|\om): = \int \Big(\frac{\rd \mu}{\rd\om} \log \frac{\rd \mu}{\rd\om}\Big) \rd\om,
  \qquad D(\mu|\om) : = \frac{1}{2  N}\int \Big|\nabla \sqrt{\frac{\rd\mu}{\rd\om}}\Big|^2
 \rd\om
$$
is the relative entropy and the relative Dirichlet form of two probability measures.
Therefore
$$
    S(\mu_{\wt\by}^s|\mu_{\wt\by})
 \le \frac{CK}{N^2}\E^{\mu_{\wt\by}} \sum_{i\in \wt I} |N V_s'(x_i)- NV'(x_i)|^2
  = CK(s^2-1)^2 \E^{\mu_{\wt\by}} \sum_{i\in \wt I} x_i^2
  \le CK^4N^{-2}(s-1)^2.
$$
In the last step we used \eqref{wtL} which, by rigidity for
the Gaussian log-gas, guarantees that $|x_i|\le CK/N$ 
with very high probability for any $i\in \wt I$.
Together with \eqref{sca} and \eqref{JJ} we  obtain \eqref{sy}
and this proves Lemma~\ref{lm:res}. \qed

\bigskip

Summarizing, we can from now on assume that \eqref{match} holds and that
$\by,\wt\by$ satisfy $|J_\by| = |J_{\wt\by}|$.  By a straightforward shift we can
also assume that  $J_\by =J_{\wt\by}$   so that  the condition \eqref{J=J} of
Theorem~\ref{thm:local} is satisfied.  The condition \eqref{Jlen}
has  already been proved to hold in Lemma~\ref{lm:goody}. 
Condition \eqref{Ex} follows from the definition of the sets $\cR_{L,K,\mu}^*$
 and  $\cR_{\wt L,K,\wt\mu}^*$, 
see Lemma~\ref{lm:goody1}.
Thus all conditions of Theorem~\ref{thm:local} 
are verified. Finally, we remark that the multiplicative
factors $\varrho^V_{L+p}$ and $\varrho^G_{\wt L +p}$ in \eqref{univ1}
coincide by  \eqref{match} and \eqref{rhov}.
 Then Theorem~\ref{thm:local}
(with an observable $O$ rescaled by the
common factor $\varrho^V_{L+p}=\varrho^G_{\wt L +p}$) implies Proposition~\ref{prop:local}. \qed

\section{Gap universality for Wigner matrices:  proof of  
Theorem~\ref{thm:wigner} }\label{sec:wigner}

In our recent results on the universality of Wigner matrices \cite{ESY4, EYYBand, EYY2},
 we established the
 universality   for  Gaussian  divisible matrices 
by  establishing the local ergodicity of the Dyson Brownian motion 
 (DBM).  By local ergodicity     we meant  
 an effective estimate on the time to
equilibrium for  local average of observables depending on the gap. In fact, we gave
an almost optimal estimate on this time.
Then we used the Green function comparison theorem to connect  Gaussian
divisible matrices  to general Wigner matrices. 
The   local ergodicity   of DBM  
was done by studying the flow of the global Dirichlet form. 
The  estimate on the global Dirichlet form   in all these works 
was sufficiently  strong so 
that it  implied the ``ergodicity for locally averaged observables"
without having to go through  the local 
equilibrium measures. 
 In the earlier
 work \cite{ErdRamSchYau},  however,  we used an approach common in the hydrodynamical limits by studying the properties 
of local equilibrium measures.  
Since by Theorem \ref{thm:local}
we now know the local equilibrium measures very well, 
we will now combine the virtue of both methods to prove  Theorem~\ref{thm:wigner}. 
To explain the new method we will be using, we first  recall 
the standard approach to the universality  from  \cite{ESY4, EYYBand, EYY2}
that consists  of the following three steps:

\begin{itemize}
\item[i)] rigidity estimates on the precise location of the eigenvalues.

\item[ii)] Dirichlet form  estimates  and local ergodicity of  DBM. 

\item[iii)]  Green function comparison theorem  to remove the small
Gaussian convolution.
\end{itemize}

In order to prove the single gap universality,  we will need to apply a similar strategy
for the local equilibrium measure $\mu_\by$. However, apart from establishing rigidity for $\mu_\by$, 
we will  need to 
 strengthen   Step ii). The idea is to use 
Dirichlet form estimates 
 as in the previous approach, but we then use  these  estimates 
to show that  the ``local structure"  after the evolution of the  DBM for a short time
is characterized by  the local equilibrium  $\mu_\by$ in a strong sense,  i.e. without
averaging.  Since Theorem~\ref{thm:local} 
provides a single gap universality for the  local equilibrium  $\mu_\by$, this proves the single gap universality 
 after a short time DBM evolution 
 and thus obtain the strong form of the Step ii) without averaging the observables.
Notice that the key input 
here is Theorem~\ref{thm:local} which contains an effective
estimate on the time to equilibrium for each single gap. We will
call this property the {\it strong local ergodicity of DBM.}
In particular, our result shows that the local averaging taken in
our previous works is not essential.

\bigskip

We now  recall the rigidity estimate which asserts that 
the eigenvalues 
$\lambda_1, \lambda_2, \ldots, \lambda_N$
of a generalized Wigner matrix follow the Wigner semicircle law
$\varrho_G(x)$ \eqref{scdef} in 
a very strong local sense. 
More precisely, Theorem 2.2. of \cite{EYYrigi} states that  the 
eigenvalues are near  their classical locations, $\{\gamma_j\}_{j=1}^N$,
\eqref{defgamma}, in the sense that 
\be\label{rigidity}
\P \Bigg\{  \exists j\; : \; |\lambda_j-\gamma_j| 
\ge (\log N)^{\zeta}  \Big [ \min \big ( \, j ,  N-j+1 \,  \big) \Big  ]^{-1/3}   N^{-2/3} \Bigg\}
 \le  C\exp{\big[-c(\log N)^{\phi \zeta} \big]}
\ee
for
any exponent $\zeta$ satisfying 
$$
    A_0 \log\log N \le\zeta \le \frac{\log (10N)}{10\log\log N}
$$
where the positive constants $C, \phi, A_0$, depend only on $C_{inf}, C_{sup}, \theta_1, \theta_2 \nc $,
see \eqref{subexp}, \eqref{1.3C}.
In particular, for any fixed
$\alpha>0$ and  $\nu>0$,  there are constants
 $C_0,c_1, c_2>0$  such that for any $N\geq 1$ and $k\in\llbracket \alpha N,(1-\alpha) N\rrbracket$,
we have  
\begin{equation}\label{bulkrig1}
\P\left(|\lambda_k-\gamma_k|> N^{-1+\nu}\right)\leq  C_0\exp{\big(-c_1N^{c_2}\big)} 
\end{equation}
and  \eqref{rigidity} also implies 
\be\label{l2dist}
\E \sum_{k=1}^N (\lambda_k- \gamma_k)^2 \le N^{-1+2 \nu }
\ee
for any $\nu>0$. 
 The constants $C_0, c_1, c_2$ may be different from 
the ones in \eqref{bulkrig} but they play a similar role so
we keep their notation. 
With a slight abuse of notation, we introduce the
 set $\cR_{L,K} = \cR_{L,K}(\xi,\al)$ from \eqref{yrig}
in the generalized Wigner setup as well, just $\gamma_k$
denote the classical localitions with respect to the semicircle
law, see  \eqref{defgamma}. In particular \eqref{rigidity}
implies that for any $\xi,\al>0$
\be\label{muR1}
   \P\big( \cR_{L,K}(\xi,\al)\big)\ge 1-   C_0\exp{\big(-c_1N^{c_2}\big)}
\ee
holds with some positive constants $C_0, c_1, c_2$, analogously to 
\eqref{muR}. We remark that the rigidity bound \eqref{rigidity}
for the generalized Wigner matrices is optimal (up to logarithmic factors)
throughout the spectrum
and it gives a stronger control than 
 the estimate used in the intermediate regime 
in the second line of the definition \eqref{yrig}. 
For the forthcoming argument the weaker estimates are sufficient, so
for notational simplicity we will not modify the definition of $\cR$.

\medskip

The Dyson Brownian motion (DBM) describes the evolution of the eigenvalues
 of a  flow of Wigner matrices, $H=H_t$, 
 if each matrix element $h_{ij}$ evolves according to independent
 (up to symmetry restriction) 
Brownian motions.  The  dynamics of the matrix elements are given 
by an Ornstein-Uhlenbeck (OU) process  which leaves the standard Gaussian
distribution invariant. 
 In the Hermitian case, the OU process for the rescaled matrix elements 
 $v_{ij}: = N^{1/2}h_{ij}$ is given by the 
stochastic differential equation 
\be
  \rd v_{ij}= \rd \beta_{ij} - \frac{1}{2} v_{ij}\rd t, \qquad
i,j=1,2,\ldots N,
\label{zij}
\ee
where $ \beta_{ij}$,  $i <  j$, are independent complex Brownian
motions with variance one and $ \beta_{ii}$ are real 
Brownian motions of the same variance.  The real symmetric
case is analogous, just $\beta_{ij}$ are real Brownian motions. 

 Denote the distribution of 
the eigenvalues $\bla=(\lambda_1, \lambda_2,\ldots, \lambda_N)$
 of $H_t$ at  time $t$
by $f_t (\bla)\mu (\rd \bla)$
where the Gaussian measure $\mu$ is given by \eqref{H}.
The density $f_t=f_{t,N}$ satisfies the forward equation 
\be\label{dy}
\partial_{t} f_t =  \cL f_t,
\ee
where
\be
\cL=\cL_N:=   \sum_{i=1}^N \frac{1}{2N}\partial_{i}^{2}  +\sum_{i=1}^N
\Bigg(- \frac{\beta}{4} \lambda_{i} +  \frac{\beta}{2N}\sum_{j\ne i}
\frac{1}{\lambda_i - \lambda_j}\Bigg) \partial_{i}, \quad 
\partial_i=\frac{\partial}{\partial\lambda_i},
\label{L}
\ee
with $\beta=1$ for the real symmetric case and $\beta=2$ 
in the complex hermitian case.
The initial data $f_0$ given by the original generalized Wigner matrix.

 Now we define a useful technical tool that was first introduced in \cite{ESY4}.
For any $\tau>0$
denote by $W=W^\tau$  an  auxiliary  potential   defined by 
\be
    W^\tau(\bla): =    \sum_{j=1}^N
W_j^\tau (\lambda_j)   , \qquad W_j^\tau (\lambda) := \frac{1}{2 \tau } (\lambda_j -\gamma_j)^2,
\label{defW}
\ee
i.e. it is a quadratic confinement on scale $\sqrt \tau $ for each eigenvalue
near its classical location, where the parameter $\tau>0$ will be chosen later.

\begin{definition} \label{def:locallyConstrained}
We define the probability measure $\rd\mu^{\tau}:= Z_{\tau}^{-1} e^{- N \beta  \cH^\tau} $, where 
the total Hamiltonian is given by 
\begin{equation}\label{eqn:omega}
  \cH^\tau: = \cH +W^\tau.
\end{equation}
 Here   $\cH$ is the Gaussian Hamiltonian given by \eqref{H}
and  $Z_\tau=Z_{\mu^{\tau}}$ is the partition function. 
The measure
$\mu^{\tau}$ will be referred to as the relaxation measure 
with  relaxation time $\tau$. 
\end{definition}

Denote by $Q$ the  following quantity 
\be
 Q:= \sup_{0\le t\le  1} 
 \frac{1}{N}
 \int \sum_{j=1}^N(\lambda_j-\gamma_j)^2
 f_t( \bla )\mu(\rd \bla). 
\label{assume}
\ee
Since  $H_t$ is a generalized Wigner matrix for all $t$,  \eqref{l2dist}
implies that 
\be
  Q\le N^{-2+2\nu}
\label{Qbound}
\ee
for any $\nu>0$ if $N\ge N_0(\nu)$ is large enough.

Recall the definition of the Dirichlet form w.r.t. a probability measure $\om$
$$
 D^\om(\sqrt{g}):= \sum_{i=1}^ND_i^\om(\sqrt{g}), \qquad
 D_i^\om(\sqrt{g}):= \frac{1}{2  N  } \int  |\partial_i \sqrt{ g}|^2 \rd\om 
 =   \frac{1}{  8   N}\int  |\partial_i \log g |^2 g\rd\om, 
$$
and the definition of the relative entropy of two probability measures $g\om$ and $\om$
$$
  S(g\om|\om) := \int g\log g \rd\om.
$$
Now we recall Theorem 2.5 from \cite{EYBull}:

\begin{theorem}\label{thm1}
For any  $\tau >0 $ and consider the local relaxation 
measure $ \mu^{\tau}$. 
Set $\psi:= \frac {  d \mu^{\tau}} { d \mu } $ and
let  $g_t: = f_t/\psi$.
Suppose there is a constant $m$ such that 
\be\label{entA}
S(f_{\tau} \mu^\tau | \mu^\tau )\le CN^m.
\ee
Then for any $ t \ge \tau N^{\e'}$ 
the entropy and the Dirichlet form satisfy the estimates:
\be\label{1.3}
S(g_t \mu^\tau | \mu^\tau) \le
 C   N^2    Q \tau^{-1}, \qquad
D^{\mu^\tau} (\sqrt{g_t})
\le CN^2  Q \tau^{-2},
\ee
 where the constants depend on  $\e'$  and $m$.
\end{theorem}

\begin{corollary} Fix  $\fa > 0$ and let  $\tau \ge  N^{- \fa}$.
Under the assumptions of Theorem \ref{thm1},  for any $ t \ge \tau N^{\e'}$
the entropy and the Dirichlet form satisfy the estimates:
\be\label{1.31}
D^\mu (\sqrt{f_t})
\le CN^2  Q \tau^{-2}.
\ee
Furthermore, if the initial data of the DBM, $f_0$, is given by a generalized Wigner ensemble, then 
\be\label{1.32}
D^\mu (\sqrt{f_t})
\le C N^{ 2  \fa + 2  \nu   }
\ee
for any $\nu>0$. 
\end{corollary}

\begin{proof}
Since $g_t = f_t/\psi$, we have
\begin{align*}
 D^\mu (\sqrt{f_t}) & = \sum_{i=1}^N  \frac{1}{ 8    N} \int |\partial_i   \log g_t + \partial_i \log \psi |^2  f_t  \rd\mu  \\
& \le  \frac{1}{ 4   N} \sum_{i=1}^N  \int |\partial_i   \log g_t  |^2  f_t  \rd\mu
+    \frac{1}{  4   N} \sum_{i=1}^N   \int | \partial_i \log \psi |^2  f_t  \rd\mu \\
& \le   2   D^{\mu^\tau} (\sqrt{g_t}) + 2 N^2Q\tau^{-2}.
\end{align*}
Thus \eqref{1.31}  follows from Theorem \ref{thm1}.  Finally,
\eqref{1.32} follows from \eqref{1.31} and \eqref{Qbound}.
\end{proof}

\bigskip

Define $f_\by$ to be the conditional density of $f \mu$ given $\by$ w.r.t. $\mu_\by$,
i.e. it is defined by the relation $f_\by \mu_\by = (f\mu)_\by$. 
For any $\by\in \cR_{L,K}$ we have the convexity bound  \eqref{Hyconv}.
Thus we have the logarithmic Sobolev inequality 
\be\label{lsi}
S(f_{\by}  \mu_\by | \mu_\by)  \le  C \frac K N
 \sum_{i\in I} D_i^{\mu_\by} ( \sqrt {f_{\by}} )
\ee
and the bound  
\be\label{lsi2}
\int \rd \mu_\by | f_\by - 1| \le  \sqrt { S(f_{\by}  \mu_\by | \mu_\by)} 
  \le C  \sqrt {  \frac K N    
 \sum_{i\in I} D_i^{\mu_\by} ( \sqrt {f_{\by}} )}  .
\ee
To control the Dirichlet forms $D_i$ for most external
configurations $\by$, we need the following Lemma.

\begin{lemma}\label{lm:rig1} Fix $\fa>0$,  $\nu>0$,  and $\tau \ge N^{-\fa}$. Suppose  the initial data $f_0$ 
of the DBM is given by a generalized Wigner ensemble.
Then,  with some small $\e'>0$,  
for any $ t \ge \tau N^{\e'}$   there exists a set $\cG_{L, K}\subset\cR_{L,K}$ of good boundary conditions $\by$
with
\be\label{PG}
   \P^{f_t \mu} (\cG_{L, K})\ge 1-CN^{-\e'},
\ee
such that for any $\by\in \cG_{L, K}$ we have 
\be\label{29}
\sum_{i\in I} D_i^{\mu_\by}( \sqrt { f_{t, \by}} ) \le 
 C N^{3\e' +2 \fa+2\nu}, \quad f_{t, \by} =  (f_t)_\by,
 \quad I=I_{L,K},
\ee
and for any bounded observable  $O$ 
\be\label{300}
  \big| [\E^{f_{t, \by} \mu_\by  } - \E^{\mu_\by} ] O(\bx)  \big|
 \le C K^{1/2} N^{2\e' +\fa+ \nu - 1/2}.
\ee 
Furthermore, 
 for any $k\in I$  we also have 
\be\label{30}
  |\E^{f_{t, \by} \mu_\by  } x_k - \gamma_k| \le CN^{-1+\nu}.
\ee  
\end{lemma}

\begin{proof} In this proof, we omit the subscript $t$, i.e. we use $f=f_t$.
By definition of the conditional measure and by  \eqref{1.32}, we have 
$$
  \E^{f \mu}   \sum_{i\in I} D_i^{\mu_\by}( \sqrt { f_\by}) =  \sum_{i\in I} D_i^{\mu}( \sqrt f ) 
\le D^{\mu}( \sqrt f )  \le   C N^{2\fa+2\nu}.
$$
By Markov inequality, \eqref{29}
holds for all $\by$ in a set $\cG^1_{L,K}$ with $\P^{f\mu}(\cG^1_{L, K})\ge 1- CN^{-3\e'}$.
Without loss of generality, by \eqref{muR1} we can assume that $\cG^1_{L,K}\subset \cR_{L,K}$.
The estimate \eqref{300} now follows from \eqref{29} and \eqref{lsi2}. 

Similarly, the rigidity bound \eqref{bulkrig1}
with respect to $f \mu$ can  be  translated to the measure  $f_\by\mu_\by$,  i.e.
there exists a set $\cG^2_{L, K}$, with 
$$  
  \P^{f\mu}(\cG^2_{L, K}) \ge   1-C_0\exp{\big( -\frac{1}{2}c_1N^{c_2}\big)},
$$
such that for any $\by\in \cG^2_{L, K} $ and 
for any $k\in I$, we have
$$
   \P^{f_\by\mu_\by} \Big(  |x_k - \gamma_k| \ge N^{-1+\nu}\Big) \le  \exp{\big(-\frac{1}{2}c_1N^{c_2}\big)}.
$$
In particular, we can conclude \eqref{30} 
 for any $\by\in \cG^2_{L,K}$. 
Setting $\cG_{L,K} := \cG^1_{L,K}\cap  \cG^2_{L,K}$
we proved the lemma.
\end{proof}

\medskip

\begin{lemma}\label{ec}  Fix $\fa>0$, $\nu>0$,  and $\tau \ge N^{-\fa}$.
 Suppose  the initial data $f_0$ of the DBM is given by a generalized Wigner ensemble.
Then, with some small $\e'>0$,  for any $ t \ge \tau N^{\e'}$, $k\in I$ and   $\by\in \cG_{L, K}$, we have 
\be\label{31}
 | \E^{\mu_\by} x_k - \E^{  f_{t, \by} \mu_\by } x_k|\le 
KN^{-3/2+\nu+ \fa+2\e'} .
\ee
In particular, if the parameters chosen such that
$$
KN^{-3/2+\nu+ \fa+2\e'} \le N^{-1}K^\xi, \qquad  \mbox{and} \quad N^{-1+\nu}\le N^{-1}K^\xi 
$$
 with some small $\xi>0$, 
then 
\be\label{Ex1}
 \big|\E^{\mu_\by}x_k - \al_k\big|\le CN^{-1}K^\xi,  \qquad k\in I,
\ee
where $\alpha_k$ is defined in \eqref{aldef}. In other words,  the analogue of\eqref{Ex}
 is satisfied.  
\end{lemma}

Notice that if we apply \eqref{300} 
with the special choice $O(\bx) = x_k$ then the error estimate would be much worse than  \eqref{31}. 
We wish to emphasize that \eqref{Ex1} is not an obvious fact although we know that it holds 
for $\by$ with high probability w.r.t. the equilibrium measure $\mu$.  The key point of \eqref{Ex1}
is that it holds for any $\by\in \cG_{L, K}$ and thus with "high probability" w.r.t  $f_t \mu$!

\medskip  

\begin{proof} Once again, we omit the subscript $t$. 
The estimate \eqref{Ex1}  is a simple consequence of \eqref{31}, \eqref{30}
and \eqref{alga}.
To prove \eqref{31},  we run the reversible dynamics
\be\label{dk}
   \partial_s q_s = \cL_\by q_s
\ee
starting from initial data $q_0= f_\by$,  where the generator $ \cL_\by$ is the unique 
reversible generator with the Dirichlet form $D^{\mu_\by}$, i.e., 
$$
-\int f \cL_\by \,  g  \,  \rd \mu_\by = \sum_{i \in I} \frac {1}{ 2 N } \int   \nabla_i f \cdot \nabla_i g \, \rd \mu_\by. 
$$
\nc
Recall  that from the convexity bound \eqref{Hyconv}, $\tau_K = K/N$ is the time 
to equilibrium of this dynamics.  After 
 differentiation and integration we get,  
$$
 \Big| \E^{\mu_\by} x_k - \E^{  f_\by \mu_\by }  x_k \Big|= 
\Big|\int_0^{K^{\e'} \tau_K } \rd u \frac{1}{2  N} \int 
  (\partial_k  {q_u} ) \rd\mu_\by \Big|+ O(\exp{(-cK^{\e'})}).
$$
{F}rom the Schwarz inequality  with a free parameter $R$, we can bound the last line by 
$$
 \frac{1}{N} \int_0^{K^{\e'} \tau_K } \rd u  \int \Big(  R (\partial_k \sqrt{q_u} )^2 +R^{-1}\Big) 
\rd\mu_\by+  O(\exp{(-cK^{\e'})}).
$$
Dropping the trivial subexponential error term and
using that the time integral of the Dirichlet form 
is bounded by the initial entropy, we can bound the last line by 
$$
   RS(f_{\by} \mu_\by | \mu_\by) + \frac{K^{\e'} \tau_K   }{NR}.
$$
Using the logarithmic Sobolev inequality for $\mu_\by$ and  optimizing 
the parameter $R$, we can bound the last term by   
\begin{align}\label{g1}
  \Big| \E^{\mu_\by} x_k - \E^{  f_\by \mu_\by }  x_k \Big|  & \le \tau_K  R    
 \sum_{i\in I} D_i^{\mu_\by} (\sqrt{ f_{\by}})  + \frac{K^{\e'} \tau_K }{NR}+ O(\exp{(-cK^{\e'})}) \nonumber  \\
&  \le  \frac{ K^{\e'} \tau_K }{N^{1/2}}\Big(  \sum_{i\in I} D_i^{\mu_\by} (\sqrt{ f_{\by}})  \Big)^{1/2}+ O(\exp{(-cK^{\e'})}).
\end{align}  
 Combining this bound with \eqref{29}, we 
obtain \eqref{31}.
\end{proof} 

\medskip

 We now prove the following comparison 
for the local statistics of $\mu$ and $f_t\mu$, where $\mu$ is
the Gaussian $\beta$-ensemble, \eqref{mubeta}, with quadratic $V$,
and $f_t$ is the solution of \eqref{dy} with initial data $f_0$ given by
the original generalized Wigner matrix.

\begin{lemma}\label{lm:COMP} Fix  $n>0$,  $\fa>0$ and $\tau\ge N^{-\fa}$.  Then 
for sufficient small $\fa$ there exist positive $\e$ and $\e'$ such that
for any $t\ge \tau N^{\e'}$, for any $n$   and for any $n$-particle observable $O$
we have
\be\label{301}
   \Big| \big[\E^{f_t \mu}  -\E^{\mu}\big] O\big( N(x_j-x_{j+1}), N(x_{j}-x_{j+2}),
 \ldots , N(x_j - x_{j+n})\big)\Big|
    \le CN^{-\e} \| O'\|_\infty , 
\ee
for any $j\in \llbracket \al N, (1-\al)N\rrbracket$ and
for any sufficiently large $N$. 
\end{lemma}

\begin{proof}  We will apply Lemma~\ref{lm:rig1}, and we 
choose $L = j$.  Since  $K\le N^{1/4}$,  \nc the right hand side of  \eqref{300} 
is smaller than $N^{-\e}$.  Then we have 
\be\label{302}
   \Big| \big[\E^{f_{t, \by} \mu_\by}  -\E^{\mu_\by }\big] O\big( N(x_j-x_{j+1}),
 N(x_{j}-x_{j+2}), \ldots , N(x_j - x_{j+n})\big)\Big|
    \le CN^{-\e}, 
\ee
for all $\by\in \cG_{L, K}$ with the probability of  $\cG_{L, K}$ satisfying  \eqref{PG}. 
 Choose any   $\tilde \by \in \cR^\ast $ defined in Lemma \ref{lm:goody1}.
 We now apply Theorem \ref{thm:local} 
with both $\mu_\by$ and $\mu_{\tilde \by}$
 given by  local Gaussian $\beta$-ensemble.
Thus   the  estimate
 \eqref{Ex}  is 
guaranteed by \eqref{Exmu} and   \eqref{Ex1}.   Since $\by, \wt\by \in \cR= \cR_{L,K}(\xi^2\delta/2, \al)$
and Lemma~\ref{lm:goody} guarantees \eqref{Jlen}, 
 we  can apply Theorem \ref{thm:local}  so that 
\be\label{3031}
   \Big| \big[\E^{\mu_\by}  -\E^{ \mu_{\tilde \by} }\big] O\big( N(x_j-x_{j+1}), N(x_{j}-x_{j+2}), 
\ldots , N(x_j - x_{j+n})\big)\Big|
    \le CN^{-\e} \|O'\|_\infty, 
\ee
for all $\by\in \cG_{L, K}$ and $\tilde \by \in \cR^\ast$.
 Since  $\P^\mu ( \cR^\ast) \ge 1 - N^{-\e}$, see \eqref{muR*}, 
we have thus proved that  
\be\label{303}
   \Big| \big[\E^{\mu_\by}  -\E^{  \mu }\big] O\big( N(x_j-x_{j+1}), N(x_{j}-x_{j+2}), 
\ldots , N(x_j - x_{j+n})\big)\Big|
    \le CN^{-\e}\|O'\|_\infty , 
\ee
for all $\by\in \cG_{L, K}$. 
{F}rom  \eqref{302}, \eqref{303} and the probability estimate \eqref{PG} for $\cG_{L, K}$, 
 with possibly reducing $\e$ so that $\e\le \e'$,  we obtain that   
\be\label{3022}
   \Big| \big[\E^{f_{t} \mu}  -\E^{ \mu }\big] O\big( N(x_j-x_{j+1}), N(x_{j}-x_{j+2}), \ldots , N(x_j - x_{j+n})\big)\Big| 
    \le CN^{-\e} \|O' \|_\infty. 
\ee
This proves Lemma~\ref{lm:COMP}. 
\end{proof}

\newcommand{\f}[1]{\boldsymbol{\mathrm{#1}}}

\medskip

Recall that $H_t$ is the generalized Wigner matrix  whose    matrix elements evolve by independent OU processes. 
Thus  in Lemma~\ref{lm:COMP} 
we have proved that the local statistics of  $H_t$,  for $t\ge N^{-2\fa +\e'}$,
 is the same as the corresponding Gaussian one for 
any initial  generalized matrix $H_0$.  
Finally, we need to approximate the generalized Wigner ensembles by Gaussian divisible ones.
 The idea  of approximation
first appeared in \cite{EPRSY} via a  ``reverse heat flow'' argument and   was also  used in \cite{TV}  via a four moment 
theorem. We will follow the Green function comparison theorem of \cite {EYYBand, EYY2} and in particular, the 
result in \cite{KY} since these results were formulated and proved  for the generalized  Wigner matrices.

Theorem 1.10 from \cite{KY} implies  that if the first four  moments of
two generalized Wigner ensembles
 $H^{\f v}$ and $H^{\f w}$  are the same  then 
\be\label{12}
 \lim_{N\to \infty} \big [ \E^{\f v} - \E^{{\f w}}   \big ] 
 O\big( N(x_j-x_{j+1}), N(x_{j}-x_{j+2}), \ldots , N(x_j - x_{j+n})\big)   \;=\; 0,
\ee
provided that one of the ensembles, say  $H^{\f w}$,  satisfies the following   level repulsion estimate:
For any $\kappa > 0$, there is an $\alpha_0 > 
0$ such that for any $\alpha$ satisfying $0 < \alpha \le \alpha_0$ there exists a  $\nu > 0$ \nc such 
that
\be\label{lrb}
\P^{\f w} \left (   \cN (E - N^{-1 - \alpha} , E + N^{-1 - \alpha})  \ge 2 \right  ) \;\leq\; N^{-\alpha-\nu}
\ee
for all $E \in [-2 + \kappa, 2 - \kappa]$,  where $\cN(a,b)$ denotes the
number of eigenvalues in the interval $(a,b)$.  Although  this theorem was stated 
with the assumption that all four moments
of the matrix elements  
of the two ensembles match exactly, it in fact only requires  that the first three moments match exactly and 
the differences of the fourth moments are  less than $N^{-c''}$   for some  small $c'' > 0$.  The relaxation 
of the fourth moment assumption was carried out in details in \cite{EYYBand, EKYY2} and we will not repeat it here.

We now apply \eqref{12}  with the choice  $H^{\f v}$ being  the   generalized Wigner ensemble 
for which we wish to prove the universality
and $H^{\f w} = H_t$  with  $t = N^{-c'}$ for some small $c'$. The  necessary estimate on the 
level repulsion \eqref{lrb} 
follows from the gap universality  and the rigidity estimate  for 
$H_t$. More precisely, for any energy $E$ in the bulk , choose  the index $k$ such that $ |\gamma_{k} - E |\le C /N$. 
Then from the rigidity estimate \eqref{rigidity}, we have for any $c > 0$ that 
\begin{align*}
 \P^{\f w} \Big(   \cN (E - N^{-1 - \alpha} , &\,  E + N^{-1 - \alpha})  \ge 2 \Big)    \\
  \;\leq\; &
\sum_{j:  |j-k| \le N^{c_0}} \P^{\f w} \left (    \lambda_{j+1} - \lambda_j \le  N^{-1 - \alpha} \right  )
  + e^{-N^{c}}  \\
 \;\leq\; &   \sum_{j:  |j-k| \le N^c} \Big [   \P^{\mu} \left (    \lambda_{j+1} - \lambda_j \le  N^{-1 - \alpha} \right  )
 +  CN^{-\e} \Big ] + e^{-N^{c}}   \\
 \;\leq\; &
C N^{\xi} N^{- (\beta + 1) \alpha+ c}  + CN^{\al-\e}  .
\end{align*}
 Here in the first inequality we used the rigidity \eqref{rigidity}   and in the second inequality 
we used \eqref{3022} with an observable $O$ that is 
a smoothed version of the characteristic function on scale $N^{-\al}$, i.e. $\|O'\|_\infty \le CN^\al$.
In the last step  we used the level repulsion bound for GOE/GUE for  $\beta = 1$ or $2$, respectively. 
The level repulsion bound for GOE/GUE is well-known; 
  it also 
 follows from  part ii) of Theorem~\ref{lr2} and the fact  that
\eqref{l20} holds for all $\by\in \cR_{L.K}$, i.e.  with  a very high probability (see \eqref{muR}).
Finally we  choose $\alpha_0 \le  \e/4$. Then  for any $\alpha< \alpha_0$, there exist 
 small exponents $\nu, c, \xi$
 such that $\nu + c +\xi < \alpha$. 
This proves \eqref{lrb} for the ensemble
$ H_t$.

Following  \cite{EYY2}, we  construct an auxiliary Wigner matrix $ H_0$ such that  the first three moments of 
$ H_t$
 and   the {\it original} matrix $H^{\f v}$ 
are identical while the differences of the fourth moments are  less than $N^{-c''}$   for some  small $c'' > 0$ 
depending on $c'$   (see Lemma 3.4 of \cite{EYY2}).
The gap statistics of $H^\bv$ and $H^{\f w} = H_t$ coincide by  \eqref{12} and the gap statistics
of $H_t$ coincides with those of GUE/GOE by Lemma~\ref{lm:COMP}.
This completes the proof of \eqref{EEO} showing
that the local gap statistics with the same gap-label $j$  is identical for
the generalized Wigner matrix and the Gaussian case.
 The proof of \eqref{EEO1} 
follows now  directly from Theorem~\ref{thm:beta} that, in particular, compares the local
gap statistics for different gap labels ($k$ and $m$) in the Gaussian case.  
This completes the proof
Theorem~\ref{thm:wigner}. \qed

\section{Rigidity and level repulsion  of  local measures}\label{sec:profmu}

\subsection{Rigidity of  $\mu_\by$:   proof of Theorem~\ref{thm:omrig} }
\label{sec:rig}

We will prove Theorem~\ref{thm:omrig} using a method similar to the proof of Theorem 3.1 in \cite{BEY}. 
Theorem 3.1 of \cite{BEY} was proved by a quite complicated argument involving induction on scales and 
the loop equation.  The loop equation, however, requires analyticity of the potential
and it cannot be applied to prove Theorem~\ref{thm:omrig} 
for a local measure whose potential $V_\by$ is not analytic. 
 We note, however, that  in \cite{BEY}  the loop equation 
was used only to estimate  the {\it expected locations}  of the particles. 
 Now this estimate is given  as a condition  by  \eqref{Exone}
and thus we can adapt the proof in \cite{BEY} to the current setting.
For later application, however, we will need a stronger form of
the rigidity bound, namely we will establish that the tail of
the gap distribution has a Gaussian decay. This stronger statement
requires some modifications to the argument from \cite{BEY}
which therefore we partially repeat here. 
 We now introduce the notations needed to prove 
Theorem~\ref{thm:omrig}. 

Let $\theta$ be a  continuously differentiable  nonnegative function with $\theta=0$ on $[-1,1]$ 
and $\theta''\geq 1$ for $|x|>1$.
We can take for example $\theta(x)=(x-1)^2 \mathds{1}_{x>1}+(x+1)^2 \mathds{1}_{x<-1}$
 in the following.

For any $m\in\llbracket\alpha N,(1-\alpha) N\rrbracket$ and any integer
 $1\leq  M\leq \alpha N$, we denote $I^{(m,M)}=\llbracket m-M,m+M\rrbracket$ and 
$ \cM=|I^{(m,M)}|=2M+1$.   Let $\eta: =\xi/3$. 
For any $k,M$ with $|k-L|\le K-M$,
 define
\be\label{def:phi}
\phi^{(k,M)}(\bx):=
\sum_{ i<j, \; i,j \in I^{(k,M)} }
\theta\left(\frac{ N(x_i-x_j)}{\cM K^{2\eta}}\right).
\ee 
Let  
$$ 
   \om^{(k,M)}_\by = Z_{\by, \phi}\mu_\by e^{- \phi^{(k,M)}},
$$
where $Z_{\by,\phi}$ is a normalization constant. 
Choose an increasing sequence of integers, $ M_1 < M_2 < \ldots < M_A$
such that $M_1= K^{\xi}$, $M_A=  CK^{1-2\eta}$ with a large constant $C$,
 and $M_\ga/M_{\ga-1}\sim K^{\eta}$  (meaning that $cK^\eta\le M_\ga/M_{\ga-1}\le C K^{\eta}$). 
We can choose the sequence such that $A\le C\eta^{-1}$.
We set
$\om_\ga : = \om^{(k, M_\ga)}_\by$ and we study the rigidity properties of the
measures $\om_A, \om_{A-1}, \ldots ,\om_1$ in this order.
Note that  $\mu_\by = \om_A$  since 
$\by \in \cR_{L, K}=\cR_{L, K}(\xi\delta/2,\al) $ guarantees that  $|x_i-x_j|\le |J_\by|\le CK/N$,
see \eqref{Jlength}, thus for $M=M_A= CK^{1-2\eta}$
the argument of $\theta$ in \eqref{def:phi} is smaller than 1,
so $\phi\equiv 0$ in this case.
We also introduce the notation
$$
   x_k^{[M]}: = \frac{1}{2M+1}\sum_{j=k-M}^{k+M} x_j.
$$

\begin{definition} 
We say that $\mu_\by$ has exponential rigidity on scale $\ell$  if there are constants
 $C, c$, such that  the following bound holds
\be\label{rigg}
  \P^{\mu_\by}( |x_k-\al_k|\ge \ell + uK^{\xi}N^{-1})\le Ce^{-cu^2} ,\qquad u>0,
\ee
for any $k\in I$. 
\end{definition}

First we prove that $\mu_\by$ has exponential rigidity on scale $M_AN^{-1}$.
Starting from $\ga=A$, 
by the Herbst bound and the logarithmic
Sobolev inequality for $\mu_\by$ with LSI constant of order $K/N$ \eqref{lsi}, we have
for any $k \in \llbracket L-K+M_A, L+K-M_A\rrbracket$ that
\be 
\P^{\mu_\by}\Big(  \Big |x_k^{[M_A]} -
\E^{\mu_\by} x_k^{[M_A]} \Big|\ge \frac{b}{\sqrt{M_A}} \Big)
 \le e^{-c(N/K)Nb^2}, \qquad b\ge 0,
\ee
i.e.
\be\label{A1}
 \P^{\mu_\by}\Big(  \Big |x_k^{[M_A]} -
 \E^{\mu_\by} x_k^{[M_A]} \Big|\ge   \frac {u K^{\eta}} N 
  \Big)
  \le Ce^{-c u^2 }.
\ee
Using the estimate  \eqref{Ex1}
we have that 
$$
\Big | \E^{\mu_\by} x_k^{[M_A]} -  \al_k^{[M_A]} \Big|
\le  
  CN^{-1}K^\xi.  
$$

Thus  we obtain
\be\label{Aav}
 \P^{\mu_\by}\Big(  \Big |x_k^{[M_A]} - \al_k^{[M_A]} \Big|\ge CN^{-1}K^\xi +
 \frac {u K^{\eta}} N 
  \Big)
  \le Ce^{-c u^2 }.
\ee
Since $x_{k-M}^{[M]}\le x_k \le x_{k+M}^{[M]}$ and the $\al_k$'s are regular
with spacing of order $1/N$, we  get
$$
   x_k-\al_k\le x_{k+M}^{[M]} - \al_{k-M}^{[M]}
  \le x_{k+M}^{[M]} - \al_{k+M}^{[M]} + CMN^{-1}
$$
and we also have a similar lower bound. Thus
\be\label{Ak}
 \P^{\mu_\by}\Big(  \Big |x_k - \al_k \Big|\ge CM_A N^{-1} + \frac {u K^{\eta}} N  \Big)
  \le Ce^{-c u^2 }
\ee
for any $k \in \llbracket L-K+2 M_A, L+K- 2M_A\rrbracket$, where we used that
$M_A\ge K^\xi$.
If $k \in \llbracket L-K,  L-K+2 M_A \rrbracket$, then
$$
    x_k -\al_k \le  x_{L-K+2 M_A} - \al_{L-K+2M_A}  + CM_AN^{-1}
$$
and
$$
   x_k-\al_k \ge y_{L-K-1} - \al_k \ge - CM_AN^{-1}. 
$$
Thus we have the estimate 
$$
   |x_k -\al_k| \le  |x_{L-K+2 M_A} - \al_{L-K+2M_A}|  + CM_AN^{-1}.
$$
Since \eqref{Ak} holds for the difference $x_{L-K+2 M_A}- \al_{L-K+2M_A}$, 
  we have  that it holds for $x_{k}-\al_k$
as well (with at most an adjustment of $C$) for any $k \in \llbracket L-K,  L-K+2M_A \rrbracket$.
Similar argument holds for $k \in \llbracket  L+K-2M_A, L+K \rrbracket$.  
Thus we proved \eqref{Ak} for all $k \in \llbracket L-K, L+K\rrbracket$,
i.e. we showed exponential rigidity on scale $M_AN^{-1}$.

\bigskip

Now we use an induction on scales and we show that if

\begin{itemize}

\item[(i)] for any $ k\in \llbracket L-K+M_\ga, L+K-M_\ga\rrbracket$ we have
\be\label{jscale1}
 \P^{\mu_\by}\Big( |x_k^{[M_\ga]}-\al_k^{[M_\ga]}|\ge 
  uK^{\xi}N^{-1}\Big)\le Ce^{-cu^2}, \qquad u\ge 0;
\ee

\item[(ii)]  exponential rigidity holds on some scale $M_\ga N^{-1}$,
\be\label{jscale}
 \P^{\mu_\by}\Big( |x_k-\al_k|\ge CM_\ga N^{-1} + uK^{\xi}N^{-1}\Big)\le Ce^{-cu^2}, \qquad k\in I, 
\quad u\ge 0;
\ee

\item[(iii)] we have the entropy bound
\be\label{entt}
 S(\mu_\by|\om_{\ga}) \le C e^{-cM_{\ga}^2K^{-5\eta}},
\ee
\end{itemize}
then (i)--(iii) also hold with $\ga$ replaced by $\ga-1$  as long as $M_{\ga-1}\ge K^\xi$. 
The iteration can be started from $\gamma=A$, since
\eqref{jscale1} and \eqref{jscale} were proven in \eqref{Aav}
and in \eqref{Ak} (even with a better bound), and \eqref{entt}
is trivial for $\gamma=A$ since $\om_A=\mu_\by$.

We first notice that on any scale $M_\ga$, the bound \eqref{jscale1} implies
\eqref{jscale} by the same argument as we concluded \eqref{Ak} for any $k\in I$
from \eqref{Aav}. So we can focus on proving \eqref{jscale1} and \eqref{entt}
on the  scale $M_{\ga-1}$.

To prove \eqref{entt} on scale $M_{\ga-1}$, 
 notice that \eqref{jscale}, with the choice $u= M_\ga K^{-\xi}$, implies
\be
   \P^{\mu_\by}( |x_k-\al_k|\ge CM_\ga N^{-1} )\le Ce^{-cM_\ga^2K^{-2\xi}}, \qquad k\in I.
\label{sser}
\ee
Since
$$
\theta\left(\frac{ N(x_i-x_j)}{\cM_{\ga-1} K^{2\eta}}\right)= 0
$$
unless $|x_i-x_j|\ge CM_{\ga-1} N^{-1}K^{2\eta} \ge CM_\ga N^{-1}K^\eta$, 
we have that the scale $ CM_\ga N^{-1}$ is by a factor $K^{\eta}$ smaller than the scale
of $x_i-x_j$  built into the definition of $\phi^{(k, M_{\ga-1})}$, see \eqref{def:phi}. 
But for $i,j\in I^{(k, M_{\ga-1})}$ we have
$|x_i-x_j|\le |x_i-\al_i|+|x_j-\al_j| + CM_{\ga-1}N^{-1}$. Thus
$ \phi^{(k, M_{\ga-1})}=0$ unless we are on the event described in \eqref{sser}
at least for one $k$.
Moreover, $|\nabla  \phi^{(k, M_{\ga-1})}(\bx)|\le N^C$ for any configuration $\bx$ in $J$.
Thus, following the argument in Lemma 3.15 of \cite{BEY}, via the logarithmic Sobolev
inequality for $\mu_\by$, we get
\be\label{wert}
   S(\mu_\by|\om_{\ga-1}) \le  CKN^{-1} \E^{\mu_\by} |\nabla  \phi^{(k, M_{\ga-1})}|^2
 \le CN^C e^{-cM_\ga^2K^{-2\xi}}\le C e^{-cM_{\ga-1}^2K^{-5\eta}}.
\ee
 Here we used
that the prefactor $N^C$ can be absorbed in 
the exponent 
by using that  $M_\ga^2K^{-2\xi} -  M_{\ga-1}^2K^{-5\eta}  \ge 
K^{2\xi - 5 \eta} = K^\eta \ge N^{\eta\delta}$. Here we have  used $\xi =3 \eta$ and $M_{\ga-1}\ge K^\xi$.
 We will not need it here, but we note that the
same bound on the opposite relative entropy,
$$
   S(\om_{\ga-1}|\mu_\by) 
 \le C e^{-cM_{\ga-1}^2K^{-5\eta}},
$$
is also correct. 
Thus \eqref{entt} for $\ga-1$ is proved.

\medskip

Now we focus on proving \eqref{jscale1} on the scale $M_{\ga-1}$.
Set $1\le M'\le M\le K$ and fix an index $k\in I$ such that $|k-L|\le K-M$.
We state the following 
 slightly generalized version of  Lemma 3.14 of \cite{BEY}

\begin{lemma}\label{lem:concGapsOmega} 
For any integers $1\leq M'\leq M\leq K$,  $k\in\llbracket L-K+M,  L +K-M \rrbracket$
and   $k'\in\llbracket  k-M+M',  k+M-M' \rrbracket$, 
we have
$$
\P^{\omega^{(k,M)}}\left(\left|\lambda_{k'}^{[M']}-\lambda_k^{[M]}
-\E^{\omega^{(k,M)}}\left(\lambda_{k'}^{[M']}- \lambda_k^{[M]}\right)\right|>\frac{u K^{2\eta} }{N}
\sqrt{\frac{M}{M'}}\right)
\leq C e^{-c u^2}.
$$
\end{lemma}

Compared with  Lemma 3.14 of \cite{BEY}, we first note that $N^\e$ 
in Lemma 3.14 \cite{BEY} is changed to $K^{2\eta}$ due to that  $\phi^{(k,M)}(\bx)$ in \eqref{def:phi}
is defined with a $K^{2\eta}$  factor instead of $N^\e$. Furthermore, 
here we allowed the center at the scale  $M'$ 
to be different from $k$. The only condition is that the 
interval $\llbracket  k'-M',  k'+M' \rrbracket 
\subset \llbracket  k-M,  k+M \rrbracket$. The proof of this lemma is
 identical to that of Lemma 3.14 of \cite{BEY}.

In particular, for any  $\ga=2,3, \ldots A$
and with  $M'=M_{\ga-1}$ and $M=M_\ga\le K^{  \eta }M_{\ga-1}$
and with any  choice of $k_\ga\in \llbracket L-K +M_\ga, L+K-M_\ga\rrbracket$,
 $k_{\ga-1}\in \llbracket L-K +M_{\ga-1}, L+K-M_{\ga-1} \rrbracket$,
so that $\llbracket  k_{\ga-1}-M_{\ga-1},  k_{\ga-1}+M_{\ga-1} \rrbracket 
\subset \llbracket  k_\ga-M_\ga,  k_\ga+M_\ga \rrbracket$, we get
\be\label{tel}
\P^{\omega_\ga}\left(\left|x_{k_{\ga-1}}^{[M_{\ga-1}]}-x_{k_\ga}^{[M_\ga]}
-\E^{\omega_\ga  }\left(x_{k_{\ga-1}}^{[M_{\ga-1}]}- x_{k_\ga}^{[M_\ga]}\right)\right|>
\frac{u K^{5\eta/2}}{N}\right)\leq  C e^{-c u^2}.
\ee

The entropy bound \eqref{entt} and the boundedness of $x_k$ imply that
$$
  \big|\E^{\omega_\ga  } x_k - \E^{\mu_\by}x_k\big| \le C \sqrt{ S(\mu_\by|\om_{\ga})}
   \le  C e^{-cM_{\ga}^2K^{-5\eta}};
$$
 where $M_{\ga}^2K^{-5\eta} \ge K^{2\xi-5\eta} \ge K^\eta$ ($\eta = \xi/3$). 
We  can combine it with \eqref{Exone} to have
$$
 \big| \E^{\omega_\ga  } x_k - \al_k \big| \le CK^\xi/N.
$$
The measure $\om_\ga$ in \eqref{tel} can also be changed to $\mu_\by$ at the expense of 
an entropy term  $S(\mu_\by|\om_\gamma)$.  Using \eqref{entt}, we thus have
\be\label{te1}
\P^{\mu_\by}\left(\left|x_{k_{\ga-1}}^{[M_{\ga-1}]}- x_{k_\ga}^{[M_\ga]}
-\left(\al_{k_{\ga-1}}^{[M_{\ga-1}]}- \al_{k_\ga}^{[M_\ga]}\right)\right|\ge CK^\xi N^{-1}+
\frac{u K^{5\eta/2}}{N}\right)\leq  C e^{-c u^2} + C e^{-cM_{\ga}^2K^{-5\eta}}.
\ee
Combining it with \eqref{jscale1} and recalling $\xi =3\eta$, we get
\be\label{te2}
\P^{\mu_\by}\left(\left|x_{k_{\ga-1}}^{[M_{\ga-1}]}- \al_{k_{\ga-1}}^{[M_{\ga-1}]}\right|\ge
 CK^\xi N^{-1}+
\frac{u K^{\xi}}{N}\right)\leq  C e^{-c u^2} + C e^{-cM_{\ga}^2K^{-5\eta}}.
\ee
This gives \eqref{jscale1} on scale $M_{\ga-1}$ if $u\le c M_\ga K^{-5\eta/2}$
with a small constant $c$.
Suppose now that $u\ge cM_\ga K^{-5\eta/2}$,  which, in particular, means
that $u\ge cK^{-\eta/2}$.  Then,   by \eqref{jscale}, we have 
\begin{align*}
   \P^{\mu_\by}\Big(\left|x_{k_{\ga-1}}^{[M_{\ga-1}]}-  \al_{k_{\ga-1}}^{[M_{\ga-1}]}\right| & \ge
 CK^\xi N^{-1}+ \frac{u K^{\xi}}{N}\Big) \\
& \le   \P^{\mu_\by}\left(\left|x_{k_{\ga-1}}^{[M_{\ga-1}]}- \al_{k_{\ga-1}}^{[M_{\ga-1}]}\right|\ge
 C M_\ga N^{-1} + (1- CK^{-\eta/2})u\frac{ K^{\xi}}{N}\right) \\
 & \le   \sum_{ k \in I} \P^{\mu_\by}\Big( |x_k-\al_k|\ge CM_\ga N^{-1}
 + (1- CK^{-\eta/2})u \frac{K^{\xi}}{N}\Big) \\
& \le 
 C K  e^{-c(1- CK^{-\eta/2})^2u^2} \le Ce^{-c'u^2}.
\end{align*}
This proves \eqref{jscale1} for $\gamma-1$.
Note that the constants slightly deterioriate at each iteration step,
but the number of iterations is finite (of order $1/\eta =  3/\xi$),
so eventually the constants $C, c$ in \eqref{rig} may depend on $\xi$.
In fact, since the deterioriation is minor, one can  also prove
\eqref{rig} with $\xi$-independent constants,
but for simplicity of the presentation we did not follow the change of these
constants at each step.

After completing the iteration, from \eqref{jscale} for $\gamma=1$, $M_1=K^\xi$, we have
$$
 \P^{\mu_\by}\Big( |x_k-\al_k|\ge CK^\xi N^{-1} + uK^{\xi}N^{-1}\Big)\le Ce^{-cu^2}, \qquad k\in I;
$$
This concludes \eqref{rig} for $u\ge 1$. Finally,
\eqref{rig} is trivial for $u\le 1$  if the constant $C$ is sufficiently large. 
This completes the
proof of Theorem~\ref{thm:omrig}. 
\qed

\subsection{Level repulsion estimates
of  $\mu_\by$: proof of Theorem~\ref{lr2}}\label{sec:lr}

We now prove the level repulsion  estimate, Theorem~\ref{lr2},  
 for   the local log-gas $\mu_\by$ 
with  good boundary conditions $\by$. There are two key ideas in the following argument. 
We first recall the weak level repulsion estimate (4.11) in \cite{BEY}, which 
 in the current notation  asserts
$$
   \P^{\mu_\by} ( x_{L-K}- y_{L-K-1} \le s/N) \le CNs
$$
for any $s>0$, and similar estimates may be deduced for internal gaps. 
Compared with \eqref{k521}, this estimate does not contain any $\beta$ exponent,
moreover, in order to obtain \eqref{k52}, the $N$ factor has to be reduced to $K^\xi$
(neglecting the irrelevant $\log N$ factor). 
Our first idea is to run this proof for a 
local measure  with only $K^\xi$ particles
to reduce the  $N$   factor to $K^{  \xi}$.  The second idea involves introducing
 some auxiliary measures to catch some of the 
$\beta$ related factors. We first introduce these two auxiliary measures which 
are  slightly modified  versions of the local equilibrium measures:   
\be
 \mu_0 := \mu_{\by,0} = 
 Z_0   (x_{L-K} - y_{L-K-1})^{-\beta} \mu_\by; \quad  \mu_1 := \mu_{\by,1}=Z_1 W^{-\beta}\mu_\by, 
 \ee
 \be
\; W =   (x_{L-K} - y_{L-K-1})  (x_{L-K+1} - y_{L-K-1}),  
\ee
where  $Z_0, Z_1 $  are chosen for normalization.  
In other words, we drop the term $(x_{L-K} - y_{L-K-1})^\beta$ from  
the measure $ \mu_\by$ in $\mu_0$ and we drop $W^{\beta} $
in $\mu_1$.   To estimate the upper gap, $y_{L+K+1}- x_{L+K}$,  similar  results will be needed when we drop
the term $(y_{L+K+1}-x_{L+K})^\beta$ and the analogous version of $W$, but we
will not state them explicitly.
We first prove the  following results which are weaker than  Theorem~\ref{lr2}.

\begin{lemma} \label{lr} Let $L$ and $K$ satisfy \eqref{K} and
consider the
local equilibrium measure $\mu_\by$ defined in \eqref{muyext}.

i) Let $\xi, \al$ be any
fixed positive constants and
let $\by \in \cR_{L, K}(\xi\delta/2,\al) $. 
Then for any $s>0$ we have 
\be\label{k5}
\P^{ \mu_\by} [  x_{L-K} - y_{L-K-1} \le s/N   ] \le
  C \left ( Ks \log N \right ) ^{\beta + 1},
\ee
and
\be\label{l2}
\P^{ \mu_\by} [  x_{L-K+1} - y_{L-K-1} \le s/N   ] \le
  C \left ( Ks \log N \right ) ^{ 2 \beta + 1}.
\ee

ii)  Let $\by$ be arbitrary with the only condition that
 $|y_i| \le C$ for all $i$. Then for any $s>0$  
 we have the weaker estimate  
\begin{align}\label{k55}
\P^{ \mu_\by}  [ x_{L-K} -y_{L-K-1} \le s/N ] & \le     \left (  \frac{Cs  K }{|J_\by|}   \right )^{\beta + 1},
  \\
\label{l27}
\P^{ \mu_{\by,j}} [ x_{L-K+1} -y_{L-K-1} \le s/N ] & \le    \left ( \frac{Cs K }{|J_\by|} \right )^{2 \beta + 1},   \qquad j=0,1.
\end{align}
\end{lemma}
To prove Lemma \ref{lr}, we  first prove    estimates  even  
weaker than  \eqref{k5}--\eqref{l27}  for $\mu_\by$ and  $\mu_{\by,j}$.

\begin{lemma}\label{43} Let $L$ and $K$ satisfy \eqref{K}.

i) Let $\xi, \al$ be any
fixed positive constants and
let $\by \in \cR_{L.K}= \cR_{L, K}(\xi\delta/2,\al) $, 
then we have for any $s>0$
\begin{align}
\label{k3}
\P^{ \mu_\by}  ( x_{L-K} -y_{L-K-1} \le s/N )& \le   C K  s \log N , 
\\
\label{k34}
\P^{ \mu_{\by,j}}  ( x_{L-K} -y_{L-K-1} \le s/N )& \le   C K  s \log N  ,  \qquad j=0,1.
\end{align} 

ii) Let $\by$ be arbitrary with the only condition that
 $|y_i| \le C$ for all $i$. Then for any $s>0$ 
 we have the weaker estimate  
\begin{align}\label{k3-1}
\P^{ \mu_\by}  ( x_{L-K} -y_{L-K-1} \le s/N )& \le      \frac{Cs K }{|J_\by|},   
\\
\label{k34-1}
\P^{ \mu_{\by,j}}  ( x_{L-K} -y_{L-K-1} \le s/N )& \le     \frac{Cs K }{|J_\by|} ,   \qquad j=0,1.
\end{align}
\end{lemma}

\begin{proof} We will prove  \eqref{k3},
the same proof with only change of notations
 works for \eqref{k34}  case as well.
 We will comment on this at the end of the proof. 

For notational simplicity, we first  shift the coordinates by $S$  such that in the new coordinates  
$\bar y=0$, i.e. $y_{L-K-1}=-y_{L+K+1}$ and $J$ is symmetric to the origin.
With the notation   $a:=-y_{L-K-1}$ and $I=\llbracket L-K, L+K \rrbracket$,
 we first estimate the following quantity, for any  $0\le \varphi\le c$ (with a small constant) 
\begin{align}
Z_\varphi :=  & \int\ldots\int_{-a+ a \varphi }^{a- a \varphi}   \rd \bx
 \prod_{i,j\in I\atop i < j} (x_i-x_j)^\beta
e^{- N\frac{\beta}{2} \sum_j V_\by (S+ x_j)  } \nonumber \\
 & = (1-\varphi)^{ K+\beta K(K-1)/2} \int\ldots\int_{-a }^{a}   \rd \bw
\prod_{i < j} (w_i-w_j)^\beta e^{- N \frac{\beta}{2}\sum_j V_\by (S+ (1-\varphi) w_j)},
\nonumber
\end{align}
where we set 
\be\label{change}
w_j:=(1-\varphi)^{-1}x_{L+j}, \qquad \rd \bx =  \prod_{|j|\le K} \rd x_{L+j}\qquad
\rd\bw =   \prod_{|j|\le K} \rd w_j.
\ee
By definition,
\be\label{ch21}
e^{- N \frac{\beta}{2}  V_\by (S+ (1-\varphi) w_j) }
   = e^{- N \frac{\beta}{2}  V (S+ (1-\varphi) w_j)}
  \prod_{k \le L-K-1} ( (1-\varphi) w_j - y_k)^\beta
 \prod_{k \ge L+K+1} ( y_k-(1-\varphi) w_j)^\beta.
\ee
For the smooth potential $V$, we have 
\be
\Big |  V (S+ (1-\varphi) w_j)) -  V( S+ w_j ) \Big | 
\le  C |\varphi w_j| 
\le \frac { C K \varphi } { N }
\ee
with a constant depending on $V$,
where we have used $|w_j|\le a \le C K / N$ which follows from $|J_\by|\le C K /N$
due to $ \by \in  \cR_{L, K} $,  see \eqref{Jlength}.

Using $(1-\varphi) w_j - y_k \ge (1-\varphi) (w_j - y_k)$ for $L-2K \le k \le  L-K-1$
and the  identity
$$
 (1-\varphi) w_j - y_k =  ( w_j - y_k)\Big [ 1  -  \frac { \varphi w_j}{w_j - y_k}   \Big ]
$$ 
 for any $k$,
we have 
\be\label{66}
\prod_{k \le L-K-1} ( (1-\varphi) w_j - y_k)^\beta \ge (1-\varphi)^{\beta K}
 \prod_{ k \le  L-K-1} (  w_j - y_k)^\beta    \prod_{n  <  L-2K}  
\Big [ 1  -  \frac { \varphi w_j}{w_j - y_n}   \Big ] ^\beta, 
\ee
and a similar estimate holds for $k\ge L+K+1$. After multiplying these
estimates for all $j=1,2, \ldots, K$, 
we thus have the bound 
\be
\frac{Z_\varphi}{Z_0}  \ge \Bigg[ e^{- C\beta K\varphi } (1-\varphi)^{ \beta K}
\min_{|w|\le a} \Bigg( \prod_{k  <  L-2K} \Big [ 1  -  \frac { \varphi w}{w - y_k}   \Big ] ^\beta 
\prod_{k  >  L+2 K} \Big [ 1  -  \frac { \varphi w}{ y_k - w}   \Big ] ^\beta \Bigg) \Bigg]^K.
\ee 
Recall that $\by \in  \cR_{L, K} $, i.e. we have the  rigidity bound for $\by$
with accuracy  $N^{-1}K^\xi \ll  K/N \sim a$, see \eqref{yrig}, i.e.
$y_k$'s are regularly spaced on scale $a$ or larger.
Combining this with 
 $|w| \le a\le C K /N$,   we have 
\be
\sum_{k\le L-2K} \frac { \varphi w}{w - y_k} \le   C \varphi K \log N  .  
\ee
Hence 
\be
 \prod_{k  <  L-2K} \Big [ 1  -  \frac { \varphi w}{w - y_k}   \Big ] ^\beta
 \ge   1 - C \varphi K \log N ,
\ee
and similar bounds hold for the $k \ge L+2K$ factors.
Thus for  any  $\varphi\le c $  
we get 
$$
   \frac{Z_\varphi}{Z_0} \ge 1- C\big( \beta K^2 + K^2\log N\big)\varphi 
\ge 1- CK^2\varphi  (\log N).
$$

Now we choose $\varphi := s/(aN)$ and recall $a \sim K /N$. 
 Therefore the $\mu_\by$-probability of $ x_{L+1}-y_L\ge a \varphi = s/N$ 
can be estimated by
$$
\P^{\mu_\by}  ( x_{L-K}- y_{L-K-1}  \ge  s / N ) \ge\frac{Z_\varphi}{Z_0}
\ge 1-  C   K s    (\log N).
$$
for all $ sK \log N $ sufficiently small. If $sK \log N $ is large,
then \eqref{k3} is automatically satisfied.
  This proves \eqref{k3}. 
  
   In order to prove \eqref{k3-1}, we now 
 drop the assumption $\by \in  \cR_{L, K} $ and replace it with $|y_i|\le C$. 
Instead of \eqref{66}, we now have 
\be\label{666}
\prod_{k \le L-K-1} ( (1-\varphi) w_j - y_k)^\beta \ge (1-\varphi)^{\beta N}
 \prod_{ k \le  L-K-1} (  w_j - y_k)^\beta,
\ee 
and a similar estimate holds for $k\ge L+K+1$.
We thus have the bound 
\be
\P^{\mu_\by}  ( x_{L-K}- y_{L-K-1}  \ge  s / N ) \ge \frac{Z_\varphi}{Z_0}
  \ge  \Bigg[ e^{- C\beta K\varphi } (1-\varphi)^{ \beta N}\Bigg]^K
\ge 1 - C \varphi     N  K. 
\ee 
With the choice   $\varphi := s/(|J_\by|N)$ this proves  \eqref{k3-1}.

The proof of \eqref{k34}  and \eqref{k34-1}  for $\mu_{\by,0}$ 
is very similar, just 
the $k=L-K-1$ factor is missing from \eqref{ch21}
in case of  $j=-K$.  For $\mu_{\by,1}$, two factors
are missing.  These modifications do \nc not alter the basic estimates. 
This concludes the proof of Lemma \ref{43}. 
\end{proof}

\bigskip 
\noindent 
{\bf Proof of Lemma \ref{lr}.}
Recalling the definition of $ \mu_0$ and setting $X:=  x_{L-K} - y_{L-K-1}$
for brevity, we have
\be\label{k4}
\P^{ \mu_\by} [ X \le s/N   ] =  
 \frac { \E^{ \mu_0}  [  1 ( X \le s/N  )  X^\beta ] }
{ \E^{ \mu_0 }  [  X^\beta ]}.
\ee
{F}rom \eqref{k34}
 we have 
$$
\E^{ \mu_0}  [  {\bf 1}
 ( X \le s/N  )  X^\beta ] \le C  (s/N)^\beta K s  \log N 
$$
 and with the choice $s=cK^{-1}  (\log N)^{-1}$ in \eqref{k34} we also have  
$$
\P^{ \mu_0}\left  (  X \ge    \frac{c }{N K \log N }  \right )  \ge 1/2 
$$
with some positive constant $c$.
This implies that  
$$
\E^{ \mu_0 }  [  X^\beta ] \ge \frac 1 2  \left ( \frac c {N K \log N }\right )^\beta.
$$
We have thus proved that 
\be
\P^{ \mu_\by} [  X \le s/N   ] \le   C  (s/N)^\beta K s  \log N 
  \left ( {N K  \log N}  \right )^\beta 
= C \left ( { Ks \log N}  \right ) ^{\beta + 1},
\ee
 i.e. we obtained \eqref{k5}. 

 For the proof of \eqref{l2}, we similarly use 
\be\label{k41}
\P^{ \mu_\by} [ x_{L-K+1} - y_{L-K-1} \le s/N   ] =  
 \frac { \E^{ \mu_1}  [  1 ( x_{L-K+1} - y_{L-K-1} \le s/N  )  W^\beta ] }
{ \E^{ \mu_1 }  [  W^\beta ]}.
\ee
{F}rom \eqref{k34}  
 we have 
$$
\E^{ \mu_1}  [  {\bf 1}
 ( x_{L-K+1} - y_{L-K-1} \le s/N  )  W^\beta ] 
\le   (s/N)^{2 \beta} \P^{\mu_1} [  x_{L-K} - y_{L-K-1} \le s/N ] 
\le C  (s/N)^{2 \beta} K s \log N. 
$$
By the same inequality  and with the choice $s=cK^{-1} (\log N)^{-1}$,  we  have
$$
\P^{ \mu_1}\left  (  W \ge    \frac{c }{(N K  \log N )^2 }  \right )  \ge 1/2 
$$
with some positive constant $c$.
This implies that  
$$
\E^{ \mu_1 }  [  W^\beta ] \ge \frac 1 2  \left ( \frac c { (N K \log N)^2 }\right )^\beta.
$$
We have thus proved that 
\be\label{l211}
\P^{ \mu_\by} [  x_{L-K+1} - y_{L-K-1} \le s/N   ] \le   C  (s/N)^{2\beta} K s  \log N
  \left ( {(N K \log N)^2 }  \right )^\beta 
= C \left ( Ks \log N  \right ) ^{2\beta + 1},
\ee
 which proves \eqref{l2}.   Finally, \eqref{k55} and \eqref{l27} can be proved using \eqref{k3-1} and \eqref{k34-1}.
 This  completes the proof of Lemma \ref{lr}. \qed 

\bigskip

{\it Proof of Theorem~\ref{lr2}.} For a given $i$,
define the index set
$$
   \wt I: =  \llbracket   
 \max (i -K^{\xi}, L-K-1 ) ,  \min (i + K^{\xi}, L+K+1) \rrbracket
$$
to be the indices in a $K^{\xi}$ neighborhood of $i$.
We further condition the measure $\mu_\by$ on the points
$$
    z_j:= x_j\; \qquad  j\in \wt I^c: = I_{L,K}\setminus \wt I
$$
and we let $\mu_{\by,\bz}$ denote the conditional measure
on the remaining $x$ variables $\{ x_j\; : \; j\in \wt I\}$.
Setting $L'=i$, $K' = K^{\xi}$, 
from the rigidity estimate  \eqref{weakrig} 
 we have $(\by, \bz) \in \cR=\cR_{L',K'}(\xi^2\delta/2,\alpha)$  with a
very high probability w.r.t. $\mu_\by$.  
 We will now apply  \eqref{k5}   to the measure $\mu_{\by,\bz}$
with a new $\delta'=\delta\xi$  and $K'=K^\xi$.  This ensures that
the condition $N^{\delta'}\le K'$ is satisfied
and by the remark after \eqref{K}, the change of $\delta$
affects only the threshold $N_0$.
We  obtain
\be\label{k51}
\P^{ \mu_{\by, \bz}} [  x_{i} - x_{i+1} \le s/N   ] \le
  C \left ( K^{\xi } s  \log N \right ) ^{\beta + 1}
\ee
with a high probability in $\bz$ w.r.t. $\mu_\by$.
 The subexponential lower bound on $s$, assumed in part ii) of Theorem~\ref{lr2}, 
 allows us to include the probability of
the complement of $\cR$ in the estimate, we thus have proved \eqref{k52}.
Similar argument but with \eqref{k5} replaced by  \eqref{l2} yields  \eqref{l20}.  

 To prove the weaker bounds  \eqref{k521}, \eqref{l21} for any $s>0$, 
we  may assume that  $L-K \le i \le L $;  $i>L$ is treated similarly. 
Since $\by\in \cR_{L,K}$, we have $|J_\by|\ge cK/N$. 
We consider two cases, either $x_i-y_{L-K-1} \le c' K /N $ 
or  $x_i-y_{L-K-1} \ge c' K /N $ with $c'<c/2$.  In the first case,  we condition
on $x_{L-K}, \ldots, x_i$ and we 
apply \eqref{k3-1} to the measure 
$\nu_1 = \mu_{\by, x_{L-K}, \ldots x_i}$. The configuration interval
of this measure  has length at least $cK/(2N)$, so we have 
\be\label{nu1}
\P^{ \nu_1}  ( x_{i+1} -x_{i} \le s/N ) \le    \frac{CKs}{cK/(2N)}   \le CNs.   
\ee
In the second case,  $x_i-y_{L-K-1} \ge c' K /N $, we condition on
$x_{i+1}, x_{i+2}, \ldots x_{L+K}$. The corresponding measure, denoted by
 $\nu_2 = \mu_{\by, x_{i+1}, \ldots x_{L+K}}$, has a configuration interval of
length at least $c'K/N$. We can now have the estimate \eqref{nu1}
for $\nu_2$. \nc
 Putting these two estimates together, we have proved \eqref{k521}. 
Finally \eqref{l21} can be proved in a similar way. This completes the proof of 
Theorem~\ref{lr2}.

\qed

\section{Proof  of Theorem \ref{thm:local}} \label{sec:pflocal}

\subsection{Comparison  of the local statistics 
of two local measures}\label{sec:comp}

In this section, we start to compare gap distributions  of two local log-gases
 on the same configuration interval but with different external potential and 
boundary conditions. 
We will express the differences of gap distributions between two measures 
in terms of  random walks in  time dependent 
random environments.
{F}rom now on, we use microscopic  coordinates  and we relabel the indices so that the coordinates 
of  $x_j$ are  $j \in I=\{-K, \ldots, 0, 1, \ldots K\}$, i.e. we set $L=\wt L =0$ in the earlier notation.
 This will have the effect that the labelling of the external points $\by$ will not
run from 1 to $N$, but from some $L_-<0$ to $L_+>0$ with $L_+-L_- = N$. The important input is 
that the index set $I$ of the internal points is macroscopically
 separated away from the edges, i.e. $|L_\pm|\ge \al N$. 

 The local equilibrium measures and their Hamiltonians 
will be denoted by the same symbols, $\mu_\by$ and $\cH_\by$, as before,
but with a slight abuse of notations we redefine them now  to the microscopic scaling.   
Hence we have  two measures  $\mu_\by = e^{ -\beta \nc H_\by} /Z_\by $ and 
$\mu_{\wt \by}= e^{  -\beta \nc\wt H_{\wt \by}} /Z_{\wt \by}$, 
 defined on the same configuration interval $J=J_\by=J_{\wt\by}$ with center
$\bar y$,  which, for simplicity, we assumed   $\bar y = 0$. 
 The local density at the center is $\varrho(0)>0$. 
The  Hamiltonian is given by 
$$
   \cH_\by (\bx): = \sum_{i\in I}  \frac{1}{2} V_{ \by} ( x_i) -
   \sum_{i,j\in I\atop i<j} \log |x_j-x_i| 
$$
\be\label{Vz} 
   V_\by (x) :=  N V(x/N) - 2\sum_{j\not\in I} \log |x-y_j|,
\ee
and $\wt H_{\wt \by}$ is defined in a similar way with $V$ in \eqref{Vz} replaced
 with another external potential  $\wt V$.
Recall also the assumption that $V'', \wt V'' \ge -C$ \eqref{lowerder}.
 We will need the rescaled version of the bounds \eqref{Jlength}, \eqref{Vby1}
and \eqref{Vbysec}, i.e.
\begin{align}\label{Jlengthresc}
    |J_\by| & =   \frac{\cK}{\varrho(0) } + O(K^\xi), 
\\
\label{Vby1resc}
   V_\by'(x) & = \varrho(0) \log \frac{d_+(x)}{d_-(x)}
   + O\Big(\frac{K^\xi}{d(x)}\Big),   \qquad x\in J,
\\
\label{Vbysecresc} 
  V_\by''(x) & \ge \frac{\inf V''}{N} + \frac{c}{d(x)},  \qquad x\in J,
\end{align}
where
\be\label{ddef}
   d(x) := \min\{ |x-y_{-K-1}|, |x-y_{K+1}|\}
\ee
is the distance to the boundary and  we redefined $d_\pm(x)$ as 
$$
 d_-(x) := d(x) + \varrho(0)K^\xi, 
\qquad d_+(x) := \max\{ |x-y_{-K-1}|, |x-y_{K+1}|\} + \varrho(0) K^\xi.
$$
 The rescaled version of Lemma~\ref{lm:goody} states that 
\eqref{Jlengthresc}, \eqref{Vby1resc} and \eqref{Vbysecresc}
hold for any $\by \in \cR_{L,K}(\xi\delta/2, \al/2)$, where
the set $\cR_{L,K}$, originally defined in \eqref{yrig}, 
is expressed in microscopic coordinates.

We also  rewrite \eqref{Ex}  in the microscopic coordinate as 
\be\label{Exm}
   | \E^{ \mu_\by} x_j -  \alpha_j| +  | \E^{\wt \mu_{\wt \by}} x_j -  \alpha_j |\le C  K^{\xi}, \quad 
\ee
where 
\be\label{aldefnew}
\alpha_j: = \frac{j}{\cK +1} |J| 
\ee 
is the rescaled version of the definition given in \eqref{aldef},
but we keep the same notation.

 The Dirichlet form is also redefined; 
 in microscopic coordinates it is now given by 
\be\label{Dirdef}
 D^{\mu_\by} (\sqrt{g})= \sum_{i\in I} D^{\mu_\by}_i (\sqrt{g})=  \frac{1}{2}
\sum_{i \in I}  \int |\partial_i \sqrt{ g}|^2 \rd\mu_\by .
\ee
Due to the rescaling,  the LSI from \eqref{lsi}   now takes the form, for $\by \in \cR_{L, K}$,
\be\label{lsim}
S(g   \mu_\by | \mu_\by)  \le  CK    D^{\mu_\by} ( \sqrt {g} ). 
\ee

Define the interpolating  measures 
\be\label{omd}
\om_{\by, {\wt \by}}^r =    Z_r e^{-\beta r (\wt V_{\wt \by} (\bx) - V_\by (\bx) )}  \mu_\by,
 \qquad  r\in[0,1], 
\ee
so that $ \om_{\by, {\wt \by}}^1 =\wt\mu_{\wt \by}$ and $\om_{\by, {\wt \by}}^0= \mu_\by$
 ($Z_r$ is a normalization constant). 
This is again a local log-gas with Hamiltonian
\be\label{Hyy}
  \cH_{\by, \wt\by}^r (\bx)= \frac{1}{2}\sum_{i\in I} V_{\by, \wt\by}^{r}(x_i) 
  - \sum_{i<j} \log |x_i-x_j|
\ee
and external potential
\begin{align*}
   V_{\by, \wt\by}^{r}(x) : & =   (1-r)  V_\by(x)+ r \wt V_{\wt\by}(x) \\
    V_\by(x): & =  NV(x/N)-  2 \nc\sum_{j\not\in I} \log(x-y_i),  \\
  \wt V_{\wt\by}(x): & =  N\wt V(x/N)-  2 \nc \sum_{j\not\in I} \log(x-\wt y_i). 
\end{align*}
The Dirichlet for $D^\om$ w.r.t. the measure $\om = \om_{\by, {\wt \by}}^r$ is defined
similarly to \eqref{Dirdef}. 

For any bounded smooth function $Q (\bx)$ with compact support we can express the difference
of the expectations w.r.t. two different measures  $\mu_\by$ and $\mu_{\wt \by}$ as 
\be\label{rcorr}
 \E^{\wt \mu_{\wt \by}} Q(\bx) -  \E^{\mu_{\by}} Q(\bx)   = \int_0^1 \frac{\rd}{\rd r} 
 \E^{\om_{\by, {\wt \by}}^r}
  Q(\bx) \rd r =  
\int_0^1    \beta  \langle h_0(\bx);  Q  (\bx )  \rangle_{\om_{\by, {\wt \by}}^r } \rd r,
\ee
\nc
where
   \be\label{h0def}
h_0 = h_0(\bx)= \sum_{i \in I}   ( V_\by(x_i) - \wt V_{\wt \by}(x_i) )
\ee
and $\langle f ; g\rangle_\om : = \E^\om fg - (\E^\om f)(\E^\om g)$ denotes the correlation. 
{F}rom now on, we will fix $r$. 
Our main result is the following   estimate on the gap correlation function.

\begin{theorem}\label{cor}   
 Consider two smooth potentials $V, \wt V$ with $V'', \wt V''\ge -C$ 
and two boundary conditions, $\by, \wt\by\in \cR_{L=0,K}(\xi^2\delta/2,\al)$,
with some sufficiently small $\xi$,  such that
$J=J_\by = J_{\wt\by}$. 
 Assume that
 \eqref{Exm} holds
  for both boundary conditions $\by,\wt\by$. 
Then, in  particular, the rescaled version of the
rigidity bound \eqref{rig} and the level repulsion  bounds~\eqref{k52}, \eqref{l20}  hold for both
 $\mu_\by$ and $\wt\mu_{\wt\by}$ by Theorem~\ref{thm:omrig} and Theorem~\ref{lr2}.

 Fix $\xi^*>0$. Then there exist  $\e> 0$ and $C>0$, 
depending on $\xi^*$, such that for
any sufficiently small $\xi$, for any 
 $0 \le r \le 1$ and  for $ |p| \le K^{1-\xi^*}$  we have 
\be\label{eq:cor}
 | \langle h_0;  O(x_{p}-x_{p+1},\ldots x_{p}-x_{p+n} ) 
 \rangle_{\om_{\by, {\wt \by}}^r }| \le   K^{C  \xi} K^{-\e} \|O'\|_\infty
\ee
 for any $n$-particle observable $O$,
 provided that  
$K\ge K_0(\xi,\xi^*, n)$  is large enough.
\end{theorem} 

 Notice that this theorem is formulated in terms of $K$ being 
the only large parameter; $N$ disappeared.
 We also remark that the restriction $ |p| \le K^{1-\xi^*}$  can be
easily relaxed to $ |p| \le K-K^{1-\xi^*}$ with an additional argument
conditioning on set  $\{ x_i \; : \; i\in I\setminus \wt I\}$ to ensure that $p$ 
is near the middle of the new index set $\wt I$. We will not need
this more general form in this paper.

 First we complete the proof of Theorem~\ref{thm:local}
assuming Theorem~\ref{cor}. 

\medskip

{\it Proof of Theorem~\ref{thm:local}.} 
 The family of measures $\om_{\by, {\wt \by}}^r$, $0\le r\le 1$,
interpolate between $\mu_\by$ and $\wt \mu_{\wt\by}$. So we can express the right hand side of \eqref{univ},
in the rescaled coordinates and with $L=\wt L=0$ as
\be\label{rinteg}
  \Big| [\E^{\mu_\by} - \E^{\wt \mu_{\wt\by}}] O(x_{p}-x_{p+1},\ldots x_{p}-x_{p+n} ) \Big|
  \le \int_0^1 \rd r  \frac{\rd}{\rd r}\E^{\om_{\by, {\wt \by}}^r } O(x_{p}-x_{p+1},\ldots x_{p}-x_{p+n} ).
\ee
Using \eqref{rcorr} and \eqref{eq:cor} we obtain that this difference is bounded by $K^{C\xi}K^{-\e}$.
Choosing  $\xi$ sufficiently small
so that $K^{C\xi}K^{-\e}\le K^{-\e/2}$, we obtain
 \eqref{univ}  (with $\e/2$ instead of $\e$). This completes the proof of 
Theorem~\ref{thm:local}.
\qed

\medskip

In the  rest of the paper
we will prove Theorem \ref{cor}.
The main difficulty 
is due to the fact that 
the correlation function  of the points, $\langle x_i;  x_j  \rangle_{\om}$,  decays only logarithmically.
 In fact, for the GUE, Gustavsson 
proved that  (Theorem 1.3 in \cite{Gus}) 
\be
\langle x_i; x_j \rangle_{GUE}  \sim 
     \log \frac {N}{  [ |i-j|+1]},
\ee
and a similar formula is expected for $\om$.   
 Therefore, it is very difficult to prove 
Theorem \ref{cor} based on this slow logarithmic decay. 
We notice that, however, the correlation function of the type 
\be
\langle g_1(x_i); g_2( x_j-x_{j+1}) \rangle_\om
\ee
decays much faster  in $|i-j|$ due to that  the second factor $g_2(x_j-x_{j+1})$ 
 depends only on the difference.  Correlations of the form 
$\langle g_1(x_i-x_{i+1}); g_2( x_j-x_{j+1}) \rangle_\om$ decay even faster.  
The  fact that observables of differences 
of particles behave much nicer was a basic observation  in our
 previous approach \cite{ESY4, EYYBand, EYY2} of universality.

The measure $\om= \om_{\by, {\wt \by}}^r $ is closely related to 
the measures $\mu_\by$ and $\mu_{\wt \by}$.  
Our first task
in Section~\ref{sec:omprop}  is to show that both the  rigidity and level repulsion
estimates hold w.r.t. the measure $ \om$.
Then  we will  
rewrite  the correlation functions in terms of a random walk
 representation in Proposition \ref{prop:repp}.  
The  decay of correlation functions will be 
translated into a regularity property
 of the corresponding parabolic equation, whose proof  will be 
the main content of Section  \ref{Caff}.  Section \ref{sec:corproof} consists
 of various cutoff estimates  to remove the singularity of  the diffusion 
coefficients in the random walk representations. 
We emphasize that these cutoffs are critical  at $\beta=1$; we do not know if 
our argument can be extended to  $\beta< 1$.

\subsection{Rigidity  and level repulsion of  the  interpolating measure 
$ \om^r_{\by, {\wt \by}}$}\label{sec:omprop}  

In this section we establish  rigidity and level repulsion
results for the interpolating measure $ \om_{\by, {\wt \by}}^r$,
similar to the ones established for $\mu_\by$ in Section~\ref{sec:profmu}
and stated in Theorems~\ref{thm:omrig} and~\ref{lr2}.

\begin{lemma}\label{lm:inter}  Let $L$ and $K$ satisfy \eqref{K} and 
$\by, \wt \by  \in \cR_{L, K}( \xi^2 \nc \delta/2,\al) $. 
With the notation $\om =  \om_{\by, {\wt \by}}^{r} $
there exist constants  $C$,  $\theta_3$,  $C_2$  and $C_3$  such that 
 the following estimates hold:

i) [Rigidity bound]   
\be\label{rigi}
   \P^{\om}\big( \big| x_i- \alpha_i\big| \ge C K^{ C_2 \xi^{2 } }\big)\le C e^{- K^{\theta_3}}, 
 \quad i \in I.
\ee

ii) [Weak form of  level repulsion]  For any $s>0$ we have 
\be\label{level}
\P^{  \om} \big(  x_{i+1} - x_{i} \le s   \big) \le
  C \left ( N s \right ) ^{\beta + 1}, \quad  i \in \llbracket L-K-1, L+K\rrbracket,  \quad s>0,
\ee
\be\label{secondlevel}
\P^{  \om} \big(  x_{i+2} - x_{i} \le s   \big) \le
  C \left ( N s \right ) ^{2\beta + 1}, \quad i \in \llbracket L-K-1, L+K-1\rrbracket , \quad s>0,
\ee

iii) [Strong form of level repulsion] With some small $\theta>0$,  for any
$  s\ge  \exp{(-K^\theta)}$  we have 
\be\label{level11}
\P^{  \om} \big(  x_{i+1} - x_{i} \le s   \big) \le
  C \left ( K^{ C_3 \nc\xi } s \right ) ^{\beta + 1}, \quad 
 i \in \llbracket L-K-1, L+K\rrbracket,  
\ee
\be\label{secondlevel11}
\P^{  \om} \big(  x_{i+2} - x_{i} \le s   \big) \le
  C \left ( K^{C_3\xi} s \right ) ^{2\beta + 1}, \quad i \in \llbracket L-K-1, L+K-1\rrbracket , 
\ee

iv) [Logarithmic Sobolev inequality] 
 \be\label{lsi22}
S(g  \om    |  \om )  \le  CK   D^{ \om } ( \sqrt {g} ).  
\ee
\end{lemma}

Note that in \eqref{rigi} we state only the weaker form of the rigidity
bound, similar to \eqref{weakrig}. It is possible to prove the strong form of rigidity with  Gaussian tail \eqref{rig} 
for $\om$, but we will not need it in this paper. 

 The level repulsion bounds  will mostly be used in the following estimates
which trivially follow from \eqref{level}--\eqref{secondlevel11}:

\begin{corollary}\label{cor:mom} Under the assumptions of Lemma~\ref{lm:inter}, for any  $p< \beta + 1$  we have
\be\label{expinv}
   \E^\om \frac{1}{|x_i-x_{i+1}|^p} \le C_p K^{C_3\xi}, \qquad   i \in \llbracket L-K-1, L+K\rrbracket, 
\ee
and for any  $p<2\beta+1 $
\be\label{expinvsec}
   \E^\om \frac{1}{|x_i-x_{i+2}|^p} \le C_p K^{C_3\xi}, \qquad  i \in \llbracket L-K-1, L+K-1\rrbracket. 
\ee
\qed
\end{corollary}

The key to translate the  rigidity   estimate of the measures $\mu_\by$ and 
$\mu_{\wt \by}$ to  the measure $\om=\om_{\by, {\wt \by}}^r$
is to show that the analogue of \eqref{Exm}
holds for $\om$.

\begin{lemma}\label{lm:91}
Let $L$ and $K$ satisfy \eqref{K} and 
$\by, \wt \by  \in \cR_{L, K}(\xi\delta/2,\al) $. Consider
the local equilibrium measure $\mu_\by$ defined in \eqref{Vyext}
  and  assume that \eqref{Ex} is satisfied. 
Let  $ \om_{\by, {\wt \by}}^r$ be the measure   defined in \eqref{omd}.  
Recall that $\al_k$ denote the equidistant
points in $J$, see \eqref{aldefnew}.
Then there exists a constant $C$, independent of  $\xi$, such that 
\be\label{Exm3}
    \E^{ \om_{\by, {\wt \by}}^r }   \left | x_j -  \alpha_j \right |
   \le C  K^{ C  \xi}.
\ee
\end{lemma}

{\it Proof of Lemma~\ref{lm:91}.}
We first prove the following estimate on the entropy.  
\begin{lemma}\label{ent}
Suppose $\mu_1$ is a probability measure and $\om = Z^{-1} e^g d \mu_1$
for some function $g$ and normalization $Z$. 
Then we can bound the entropy by 
\be\label{88}
S:= S(\om| \mu_1) = \E^\om g - \log \E^{\mu_1} e^g  \le \E^\om g - \E^{\mu_1} g.
\ee
Consider two probability  measures $\rd \mu_i = Z_i^{-1} e^{- H_i} \rd \bx $, $i=1,2$.
Denote by $g$ the function 
\be
g = r (  H_1 - H_2), \quad 0 < r < 1
\ee
and set $\om= Z^{-1} e^g d \mu_1$ as above. Then we can bound the entropy by 
\be\label{89}
\min (S(\om| \mu_1), S(\om| \mu_2)) \le\Big [ \E^{\mu_2} - \E^{\mu_1}   \Big ]  (H_1-H_2).
\ee
\end{lemma}

\begin{proof} The first inequality is a trivial consequence of the Jensen  inequality 
$$
S= \E^\om g - \log \E^{\mu_1} e^g \le \E^\om g - \E^{\mu_1} g.
$$ 
The entropy inequality yields that 
\be\label{entroin}
\E^\om g \le r \log \E^{\mu_1} e^{g/r}  + r S.
\ee
By the definition  of $g$,  we have  
\[
\log \E^{\mu_1} e^{g/r} 
= - \log  \int e^{ - g/r} \rd\mu_2  \le  \E^{\mu_2} g/r.
\]
Using this inequality and \eqref{entroin}  in \eqref{88}, we have proved 
\be
S\le  \frac r { 1-r}  \Big [ \E^{\mu_2} - \E^{\mu_1}   \Big ]  (H_1-H_2).
\ee
We can assume that $r \le 1/2 \le 1-r$ since otherwise we can switch the roles of $H_1$ and $H_2$. 
Hence \eqref{89} holds and this concludes the proof of Lemma \ref{ent}.
\end{proof}
\medskip

We now apply this lemma with  $\mu_2 = \wt \mu_{\wt \by}$ and  $\mu_1 = \mu_{ \by}$ to prove that 
\be\label{85}
\min[ S(\om_{\by, {\wt \by}}^r|\mu_\by), S(\om_{\by, {\wt \by}}^r|\wt \mu_{\wt \by}) ]  \le   K^{ C \xi}
\ee
To see this, by definition  of $g$  and the rigidity estimate \eqref{rig},  we have 
\begin{align}\label{842}
\E^{\mu_2} g- \E^{\mu_1} g  &  =  \frac{r}{2}\Big [ \E^{\mu_2} -  \E^{\mu_1} \Big ] \sum_{i \in I}
\Big [ V_\by(x_i) - \wt V_{\wt\by}(x_i)
 \Big ]  \non \\
& =   \frac{r}{2}\Big [ \E^{\mu_2} -  \E^{\mu_1} \Big ] \sum_{i \in I} \int_0^1\rd s\Big [ V_\by'(s\al_i+(1-s)x_i ) - \wt V_{\wt\by}'(s\al_i+(1-s)x_i)
 \Big ] (x_i-\al_i)  \non \\
& = \Big [ \E^{\mu_2} +  \E^{\mu_1} \Big ] O\Big(  \sum_{i \in I}
 \sup_{s\in [0,1]}\frac{K^\xi}{d(s\al_i+(1-s)x_i)}|x_i-\al_i|\Big) \le  K^{C \xi}.
\end{align}
 In the first step we used that the leading term  $V_\by(\al_i) - \wt V_{\wt\by}(\al_i)$
in the Taylor expansion is deterministic, so it vanishes after taking the
difference of two expectations. 
In the last step we used
that with a very high $\mu_1$- or $\mu_2$-probability
 $d(s\al_i+(1-s)x_i)\sim d(\al_i)$ 
are equidistant up to an additive error
$K^\xi$ if $i$ is  away from the boundary, i.e., $-K+K^{C\xi}\le i\le  K-K^{C\xi}$,
see \eqref{rig}. For indices near the boundary, say $-K\le i \le-K+K^{C\xi}$,  we used 
$d(s\al_i+(1-s)x_i)\ge c\min\{1, d(x_{-K})\}$.  Noticing that
$d(x_{-K})= x_{-K} - y_{-K-1}$,  the
level repulsion  bound \eqref{k52} 
 (complemented with the weaker bound \eqref{k5}
that is valid for all $s>0$) 
guarantees that the short distance singularity  $[d(x_{-K})]^{-1}$
 has an  $\E^{\mu_{1,2}}$ expectation that is bounded by $CK^{C\xi}$.

We now assume that \eqref{85} holds with the
 choice of $S(\om_{\by, {\wt \by}}^r|\mu_\by)$ for simplicity 
of notation. 
By the entropy inequality, we have  
\be\label{86}
\E^{\om_{\by, {\wt \by}}^r} |x_i - \alpha_i| \le 
 \log   \E^{\mu_\by}    e^{ | x_i - \alpha_i | } + K^{ C \xi} .
\ee
From the Gaussian tail  of the rigidity estimate  \eqref{rig}, we have 
\be
\log   \E^{\mu_\by}    e^{  | x_i - \alpha_i | } \le K^{C \xi}. 
\ee
Using this bound in \eqref{86} 
we have proved \eqref{Exm3} and this concludes the proof of Lemma \ref{lm:91}. 
\qed

 {\it Proof of Lemma~\ref{lm:inter}.} 
Given \eqref{Exm3}, 
 the proof of \eqref{rigi} follows the argument in the proof of 
Theorem \ref{thm:omrig},  applying it to $\xi^2$ instead of $\xi$. 
  Once the rigidity bound \eqref{rigi} is proved, 
we can follow the proof of Theorem~\ref{lr2} to obtain all four level
repulsion estimates, \eqref{level}--\eqref{secondlevel11}, 
analogously to the proofs of \eqref{k521}, \eqref{l21}, \eqref{k52} and \eqref{l20}, respectively. 
 The $\log N$ factor can be incorporated into $K^{C_3\xi}$.

\medskip 

 Finally, to prove \eqref{lsi22}, 
let $  \cL^\om$   
be the reversible generator given by the Dirichlet form 
\be\label{LK}
 -\int f \cL^\om f \rd \om_{\by, {\wt \by}}^r   =    \frac{1}{2}  \sum_{|j| \le K}  \int   (\partial_j f )^2 
\rd \om_{\by, {\wt \by}}^r.
\ee
Thus  for the Hamiltonian $\cH=\cH^r_{\by,\wt\by}$ of the measure $\om=\om_{\by, {\wt \by}}^r$ (see \eqref{Hyy}), 
 we have 
\begin{align}\label{convexc}
    \Big\langle  \bv ,     \nabla^2    \cH(\bx)\bv\Big\rangle 
&  =     \frac{1}{2} \sum_i   \Big [ (1-r) V_\by''(x_i) + r \wt V_{\wt \by}''(x_i) \Big ]    v_i^2  +
 \sum_{i<j} \frac{( v_i -  v_j)^2}{(x_i-x_j)^2}   \ge \frac{c}{K} \sum_i v_i^2, 
 \end{align}
by using \eqref{Vbysecresc}  and $d(x)\le CK$ for good boundary conditions. \nc
Thus  LSI takes the form 
\be\label{lsim2}
S(g \om    | \om )  \le CK 
 D^{\om } ( \sqrt {g} ). 
\ee
This completes the proof of Lemma~\ref{lm:inter}. \qed

\medskip

The dynamics given by the generator $\cL^\om$ 
 with respect to the   interpolating  measure 
 $\om =  \om_{\by, {\wt \by}}^{r} $ 
 can also be characterized by the following SDE  
\be
  \rd x_i = \rd B_i + \beta \Big[ - \frac{1}{2}  (V_{\by,\wt\by}^{r})'   ( x_i ) +
 \frac{1}{2}\sum_{j\ne i} \frac{1}{(x_i-x_j)}\Big] \rd t,
\label{SDE}
\ee
where $(B_{-K}, B_{-K+1}, \ldots, B_K)$ is a family of independent standard Brownian motions.
 With a slight abuse of notations,  when we talk about the process, 
we will use $\P^\om$ and $\E^\om$ to denote the probability 
and expectation w.r.t. this dynamics with initial data $\om$, i.e., in equilibrium. 
This dynamical point of view gives rise to a representation for
the correlation \eqref{eq:cor} in terms random walks in random environment.

Starting from Section~\ref{sec:corproof} 
we will focus on proving Theorem~\ref{cor}. 
The proof 
 is based on dynamical idea and it
 will be completed in Section~\ref{sec:proof81}.

\section{Local statistics of the interpolating
measures: Proof of Theorem~\ref{cor}}\label{sec:corproof}

\subsection{Outline of the proof of Theorem~\ref{cor}}\label{sec:int}

Theorem~\ref{cor} will be proved by the following main steps. 
We remind the readers that the boundary conditions $\by, \wt \by$ are in the good sets 
and we have chosen $L = 0$ for convenience.  For simplicity, we assume that $n=1$,
i.e. we consider a single gap observable $O(x_p-x_{p+1})$.

 {\it Step 1. Random walk representation.}  The starting point is
a representation formula for the correlation
 $\langle h_0, O(x_p-x_{p+1})\rangle_\om$. 
For any smooth  observables    $F(\bx)$ and  $Q(\bx)$ and   any time $T > 0$
 we have the following representation formula for the time dependent 
correlation function
(see \eqref{reppgeneral} for the precise statement):
\be\label{RW}
\E^\om  Q(\bx) \, F(\bx)    - \E^\om   Q (\bx(0))  F(\bx (T))
= \frac{1}{2}\int_0^{T}   \rd S \; \E^\om
    \sum_{b\in I} \pt_b Q(\bx(0)) \langle \nabla F(\bx(S)) , \bv^b(S, \bx(\cdot)) \rangle.
\ee
Here the  path $\bx(\cdot)$ is the solution of the reversible stochastic dynamics
with equilibrium measure  $\om$, \eqref{SDE}.
We  use the notation $\E^\om$ also for the expectation with respect to the path  measure
starting from the initial distribution $\om$
and $\langle \cdot, \cdot \rangle$ denotes the inner product in  $\R^\cK$, recalling that $|I|=2K+1=\cK$.  Furthermore,
for any $b\in I$ and for any fixed path $\bx(\cdot)$,
 the vector  $\bv^b(t)= \bv^b(t, \bx(\cdot))\in \R^\cK$ is the solution to the
equation
\be\label{veq}
   \partial_t \bv^{b} (t) =  - \cA(t) \bv^b (t), \quad t\ge 0, \qquad v_j^b(0)=\delta_{bj}.
\ee
The matrix $\cA(t)$  depends on time through the path $\bx(t)$ and it is  given by
 $$
\cA(t): = \beta \nabla^2 \cH_{\by,\wt\by}^r(\bx(t)).
$$
From \eqref{Hyy},  it is 
 of the form $\cA(t)=\wt \cA(\bx(t))= \wt \cB(\bx(t))+ \wt \cW(\bx(t)) $ with $\wt \cW(\bx(t)) \ge 0$.
The matrix elements of $\wt \cB$ is given by:
\be\label{Bdef'}
[\wt \cB ( \bx )  \bv]_j = - \sum_{k\ne j} \wt B_{jk}( \bx ) (v_k- v_j), \quad 
 \wt B _{jk}(\bx) = \frac \beta { (x_j-x_k)^2}, \qquad j\ne k.
\ee
Furthermore,  $\cA (t) \ge C K^{-1}$,  \eqref{convexc},  and  the time to equilibrium for
 the $\bx(t)$ process is of order $K$  (Corollary \ref{cor:repp}).
Applying this representation to  $O(x_{p}-x_{p+1})$ and   cutting off the time integration
 at $C_1 K \log K$ with some large constant $C_1$, we will have (see \eqref{repp}) 
\begin{align}\label{repp'}
\langle h_0; & O(x_{p}-x_{p+1}) \rangle_{\om }
 \\
 &  = \frac{1}{2}\nc\int_0^{C_1K\log K}   \rd \sigma   
    \sum_b  \E^\om \Big[ \pt_b h_0(\bx)  
  O'(x_p-x_{p+1})  \big(v_p^b (\sigma )  - v^b_{p+1} (\sigma)\big) \Big]  
 + O\Big( \| O'\|_\infty K^{-2}\Big),
 \nonumber
\end{align}
It is easy to check that $\partial_b h_0$ satisfies the estimate
with some small $\xi'$ (see \eqref{ini1})
\be\label{ini1'}
  |\partial_b h_0(\bx) | 
 \leq \frac{K^{\xi'}}{\min ( |x_b-K|, |x_b + K|) +1 }.
\ee 

\bigskip

{\it Step 2. Cutoff of bad sets.} Setting $\cT: = [0, C_1 K\log K]$,  we define 
the  ``good set" of paths (see \eqref{K1})  for which the  rigidity estimate holds
uniformly in time: 
\be\label{K1'}
\cG :
=  \Big\{
 \sup_{s \in\cT} \;\sup_{ |j| \le K } |x_j(s)-\alpha_j| \le K^\xip  \Big\}, 
 \ee
 where  $\xi'$ is s small parameter to be specified later 
and  $\alpha_j$ is the classical location given by \eqref{aldefnew}.
For any    $Z\in I$ 
 and $\si\in \cT$  
we also define the following event   that the gaps between particles
 near $Z$ are not too small in an appropriate average sense
by   
\begin{align}
\cQ_{\si,Z} := \Big\{   \sup_{s\in \cT } 
\sup_{ 1 \le M \le K} \frac{1}{1+|s-\si|} \Big| \int_s^{\si} \rd a \; \frac{1}{M}  
\sum_{i\in I \, : \, |i-Z|\le M }  
\frac  1  {  \big| x_i(a)-x_{i+1}(a) \big|^{2} }\Big|\le K^{\rho} \Big\},  \label{K2}
\end{align}
where $\rho  > 0$ is a small parameter to be specified later. 
By convention we set $x_i(a)=y_i$ whenever $|i|>K$.
 We will need that
the gaps are not too small not only near $Z$ but also near
the boundary so we
define the  new good set  
\be\label{whQdef}
\wh \cQ_{\si,Z} := \cQ_{\si,Z} \cap \cQ_{\si, - K} \cap \cQ_{\si, K}.
\ee
Finally,  we need to control the gaps not just around one time $\sigma$
but around a sequence of times that dyadically accumulate at $\si$. 
The significance of this stronger condition will only be
clear in the proof of our version of the De Giorgi-Nash-Moser bound
in Section~\ref{Caff}. 
We define
\be\label{wtQdef}
  \wt \cQ_{\si,Z} : = \bigcap_{\tau\in \Xi }\wh \cQ_{\si + \tau ,Z},
\ee
where 
\be\label{Kassnew}
  \Xi: = \big\{ - K\cdot 2^{-m}(1+2^{-k}) \; : \; 0\le k,m \le C\log K  \big\}.
\ee

 We will choose $Z$ near the center of the interval $I$ and  
 show  in \eqref{K11} and \eqref{K22} that the bad events are small in the sense that 
\be\label{K11'}
\P^\om ( \cG^c ) \le  C e^{ - K^{\theta }}
\ee
with some $\theta>0$, 
and
\be \label{K22'}
   \P^\om ( \wt \cQ^c_{\si,Z}   ) \le  C K^{ C_4 \xi-  \rho}  
\ee
for each fixed  $Z\in I$   and fixed
 $\si\in \cT$ where $\xi$ is introduced in Theorem \ref{cor}. 
 Notice that while the rigidity bound \eqref{K11'} holds with a very high probability, the control
on small gaps \eqref{K22'} is much weaker due to the power-law behavior of the level repulsion estimates.

Our goal is to insert the characteristic functions of
 the good sets into the expectation in \eqref{repp'}. More precisely, we will prove in 
\eqref{cut3} that 
\begin{align}\label{cut3'}
\big|  \langle h_0;  & O(x_{p}-x_{p+1}) \rangle_{\om }  \big| \\
& \le \frac{1}{2} \|O'\|_\infty  \int_0^{C_1 K \log K}   
   \sum_{b\in I}  \E^\om \Big[  \wt\cQ_{\si,Z}  \cG
      |\pt_b h_0(\bx)|  
  \big| v_p^b (\sigma )  - v^b_{p+1} (\sigma) \big|\Big] \rd\sigma 
 +  O\Big(\|O'\|_\infty K^{-  \rho/6  }\Big) . \non
\end{align}  
(With a slight abuse of notations
we use $\cG$ and $\wt\cQ_{\si,Z}$ also to denote the characteristic function of these sets.)
To prove this inequality, we note that the contribution of the bad set $\cG^c$ can be estimated by \eqref{K11'}. 
To bound the contribution of  the bad set $\wt\cQ_{\si,Z}^c$, 
the estimate  \eqref{K22'} alone is not strong enough due to the time integration in \eqref{cut3'}.
We will need a time-decay estimate for the solution $\bv^b(\sigma)$. 
On the good set $\cG$, 
the matrix element $B _{jk}$ satisfies 
\be
B _{jk}(s) = \frac \beta { (x_j(s)-x_k(s))^2} 
 \ge   \frac {b}    { (j-k)^2}, \quad 0 \le s \le \si, \quad j\ne k.
\ee  
with $b= \beta K^{-2\xi'}$.
With this estimate, we will show in \eqref{decay}  that,  for any $1\le p\le q\le \infty$,  the 
following decay estimate
for the solution to \eqref{veq} holds:
\be\label{decay'}
\| \bv(s) \|_q  \le  ( sb)^{-(  \frac{1}{p}  - \frac 1 q)} 
    \| \bv(0) \|_{ p }, \qquad 0<s\le\si.
\ee
This allows us to prove \eqref{cut3'}.

\bigskip

{\it Step 3. Cutoff of the contribution from near the  center.} 
From \eqref{ini1'},  $\partial_b h_0(\bx)$  decays as a power law  
when   $x_b$ moves away from the boundary of $J$, i.e., when
the index  $b$ moves away from $\pm K$.
 With the decay estimate \eqref{decay'}, it is not difficult to show that the contribution of $b$ 
in the interior,  i.e.,  the terms with $ |b| \le K^{1-c}$ for some $c> 0$
 in the sum in \eqref{cut3'},   is negligible.  

\bigskip

{\it Step 4. Finite speed of  propagation.} 
We will prove that in the good set $\cG  \cap  \wt \cQ_{\si,Z}$ the dynamics   \eqref{veq}
satisfies the finite speed of  propagation estimate 
\be\label{finite'}
|v_p^b(s)| \le  \frac {C K^{c + 1/2 } \sqrt  {s  +1 }  } {   |p- b |}.
\ee
for some small constant $c$ (see \eqref{finite}).  This estimate is not optimal, 
but it allows us to cutoff the contribution in \eqref{cut3'} for time 
$ \sigma  \le K^{1/4}$ for $b$ away from the center,  i.e., $ K \ge  |b| \ge K^{1-c}$.
In this step we use that $|p|\le K^{1-\xi^*}$ ($\xi^*$ is some small constant) 
 and the exponents are chosen such that $|p-b|\ge cK^{1-c}$. 

\bigskip

{\it Step 5. Parabolic regularity with singular coefficients}.  Finally,  
we have to estimate the r.h.s of \eqref{cut3'} in the regime $K^{1/4}\le \sigma \le C_1K\log K$
and for $|p|\le K^{1-\xi^*}$ with the choice $Z=p$. This estimate will work uniformly in $b$.  
We will show that
 for all  paths in $\cG  \cap  \wt \cQ_{\si,p}$, 
 any 
 solution to \eqref{veq} satisfies the H\"older regularity estimate  in the interior, i.e., 
 for some constants $\al, \fq> 0$, 
\be\label{HC'}
\sup_{ |j-p| + |j'-p| \le \si^{ 1- \al} } 
   | v_j (\si )  -  v_{j'}  (\si)  |  \le CK^\xi \sigma^{-1-\frac{1}{2}\fq \al}.
\ee

Notice that the regularity depends on the time $\sigma$ and that is why we need the short time cutoff in the previous step. 
This estimate \eqref{HC'}  allows us to complete  the proof 
 that $\langle h_0;   O(x_{p}-x_{p+1}) \rangle_{\om } \to 0$ as $K \to \infty$. 
The H\"older estimate will be stated as Theorem~\ref{holderg} and the entire  Section~\ref{Caff} will be devoted 
to its proof.

\subsection{Random Walk Representation}\label{sec:rw}

First we will recall a general formula for the correlation functions of the
process \eqref{SDE} through a random walk representation, see \eqref{corrrep} below. 
This  equation in a lattice  setting  was given  in  Proposition 2.2 of
 \cite{DGI} (see also Proposition 3.1 in \cite{GOS}). The random walk representation 
already appeared  in the earlier paper of Naddaf and Spencer \cite{NS},  which  was a
probabilistic formulation  of the idea
 of Helffer and Sj\"ostrand \cite{HS}.

In this section we will work in a general setup.  Let $J\subset\R$ be an interval
and $I$ an index set with cardinality $|I|=\cK$. 
Consider a convex Hamilton function $\cH(\bx)$ on $J^\cK$  and
let $\bx(s)$ be the solution to 
\be
  \rd x_i = \rd B_i + \beta \partial_i \cH(\bx)  \rd t, \qquad i\in I,
\label{SDEgen}
\ee
with initial condition $\bx(0)=\bx \in J_\by^I$, 
where $\{ B_i\; : \; i\in I\}$
 is a family of independent standard Brownian motions.
The parameter $\beta>0$ is introduced only for consistency 
with our applications. 
 Let $\E_\bx$ denote the
expectation with respect to this path measure. 
With a slight abuse of notations, we will use 
 $\P^\om$ and $\E^\om$ to denote the probability and expectation
with respect to the path measure of the solution to \eqref{SDEgen}
with initial condition $\bx$ distributed by $\om$. 
We assume that $\P^\om ( \bx(t)\in J^\cK) = 1$, i.e. the
Hamiltonian confines the process to remain in the interval $J$. \nc
The corresponding invariant measure is $\rd\om = Z_\om^{-1} e^{-\beta\cH(\bx)}\rd\bx$
 with  generator $\cL^\om = - \frac{1}{2} \Delta + \frac{\beta}{2}  \nabla\cH \cdot\nabla$ and 
Dirichlet form
$$  
     D^\om (f) : =  \frac{1}{2} \nc \int |\nabla f|^2 \rd\om = - \int f\cL^\om f \rd \om.
$$
For any fixed path
$\bx(\cdot):=\{\bx(s)\; : \; s\ge 0\}$ we
define the operator ($\cK\times \cK$ matrix)
\be\label{cA}
  \cA(s) := \wt\cA (\bx(s)),
\ee
where $\wt\cA: =\beta\cH''$ and we assume that the Hessian
matrix is positive definite, $\cH''(\bx)\ge c >0$.

\begin{proposition} \label{prop:repp} Assume that the Hessian matrix is positive definite
\be\label{hessbound}
   \inf_\bx \cH''(\bx)\ge \tau^{-1} 
\ee
with some constant $\tau>0$.
Then for any  functions 
 $F, G \in C^1(J^\cK) \cap L^2(\rd\om)$ 
and   any time $T > 0$
 we have
\begin{align}\label{reppgeneral}
\E^\om \big[ F(\bx)\, G(\bx)\big]   - \E^\om  \big[ F (\bx(0))  G(\bx (T))\big]
=  \frac{1}{2}  \int_0^{T}   \rd S\int\omega (\rd \bx)
    \sum_{a,b =1}^\cK \pt_b F(\bx)
    \E_\bx \big[ \pt_a  G (\bx(S))   v^b_a(S, \bx(\cdot))\big].
\end{align}
 Here for any $S>0$ and for any path $\{ \bx(s)\in   J^\cK \; : \; s\in [0,S]\}$, we
define  $\bv^b(t)= \bv^b(t, \bx(\cdot))$ as the solution to the equation
\be\label{veq2}
   \partial_t \bv^{b} (t) =  - \cA(t) \bv^b (t), \quad t\in [0, S], \qquad v_a^b(0)=\delta_{ba}.
 \ee
The dependence of $\bv^b$ on the path $\bx(\cdot)$ is present via the dependence
$\cA(t) =  \wt \cA( \bx (t))$. In other words, $v^b_a(t)$ is the fundamental solution
of the heat semigroup $\pt_s+ \cA(s)$.

Furthermore, for the correlation function we have
\begin{align}\label{corrrep}
\langle F; \, G \rangle_\om & =   \frac{1}{2} \int_0^{\infty}   \rd S\int\omega (\rd \bx)
    \sum_{a,b =1}^\cK \pt_b F(\bx)
    \E_\bx \big[ \pt_a  G (\bx(S))   v^b_a(S, \bx(\cdot))\big] \\
  & =   \frac{1}{2} \int_0^{A\tau\log \cK}   \rd S\int\omega (\rd \bx)
    \sum_{a,b =1}^\cK \pt_b F(\bx)
    \E_\bx \big[ \pt_a  G (\bx(S))   v^b_a(S, \bx(\cdot))\big] + O(\cK^{-cA}) \label{corrrep2} 
\end{align}
for any constant $A>0$.
\end{proposition}

{\it Proof.} This proposition in the lattice setting was already proved in  \cite{DGI, GOS, NS};  we give here a  proof in 
the continuous  setting. 
Let $G(t, \bx)$ be the solution to the equation $\pt_t G  = \cL^\om G$ with initial condition
$G(0, \bx):= G(\bx)$. 
 By integrating the time derivative,
  we have
\begin{align}\label{QQF}
\E^\om F(\bx)\, G(\bx)   - \E^\om  F (\bx(0))  G(\bx (T))  & 
= -  \int_0^T \rd S \frac \rd { \rd S}  \E^\om  \big [ \, F e^{S \cL^\om} G  \, \big ]  \\
& =  - \int_0^T \rd S \; \E^\om \big [ \,  F  \cL^\om e^{S \cL^\om} G \, \big ]    \\
& =   \frac{1}{2}  \int_0^T \rd S \;  
  \E ^\om \langle \nabla F (\bx), \nabla G(S, \bx)\rangle.
\end{align} 
where $\langle, \rangle$ denotes the scalar product in $\R^\cK$.

Taking the gradient of the equation $\pt_t G  = \cL^\om G$  and computing the commutator
$[\nabla, \cL^\om]$  yields the equation
\be\label{nablaQeq}
   \pt_t \nabla G(t, \bx) = \cL^\om [ \nabla G(t, \bx)] - \wt \cA(\bx) [\nabla G(t, \bx)]
\ee
for the $\bx$-gradient of $G$. Setting $\bu (t, \bx) : = \nabla G(t, \bx)$ for brevity, we have the    equation
\be\label{ge}
\partial_t \bu (t,  \bx)  = \cL^\om \bu (t,  \bx)  -  \wt \cA ( \bx ) \bu(t, \bx) 
\ee
with initial condition $\bu(0, \bx) = \bu_0(\bx): = \nabla G(\bx)$. 

Notice that $\wt \cA$ is a matrix and $\cL^\om$ acts on the vector
 $\bu$ as a diagonal operator in  the index space, i.e., $[ \cL^\om \bu (t,  \bx)]_i 
= \cL^\om [ \bu (t,  \bx)_i ]$. The equation \eqref{ge} can be solved by solving an
 equations \eqref{111} over the indices with coefficients that depend on the 
path generated by the operator $\cL^\om$ and then by taking expectation over the paths starting at $\bx$.  
To obtain such a representation, we start with the time-dependent
Feynman-Kac formula:
\be\label{ut}
\bu(\sigma,  \bx)  = \E^{ \bx} \Big[ \wt{\mbox{Exp}} \Big( -\int_0^\sigma  
  \wt \cA ( \bx (s)  )  \rd s \Big) \bu_0 ( \bx(\sigma) )\Big] , \qquad \sigma >0,
\ee
where 
\be\label{tildeexp}
\wt{\mbox{Exp}} \Big( -\int_0^\sigma   { \wt \cA} ( \bx (s)  )  \rd s\Big)
 := 1  -   \int_0^\sigma   \wt \cA ( \bx (s_1)  ) \rd s_1 +   \int_{0\le s_1 < s_2 \le \sigma }
   {\wt \cA} ( \bx (s_1)  ) \wt  \cA ( \bx (s_2)  ) \rd s_1 \rd s_2 
+ \ldots  
\ee
 is the time-ordered exponential. 
To prove that \eqref{ut} indeed satisfies \eqref{ge}, we notice
from the definition \eqref{tildeexp} that 
\begin{align}\label{ut1}
\bu(\sigma,  \bx)  & =   \E_{ \bx} \,  \wt{\mbox{Exp}}\Big( -\int_0^\sigma   {\wt \cA} ( \bx (s)  )
  \rd s \Big) 
 \bu_0 ( \bx(\sigma) )  \\
& = \E_\bx \bu_0 ( \bx(\sigma) )   -  \int_0^\sigma   \E_\bx {\wt \cA} ( \bx (s_1)  )
 \E_{\bx(s_1)}  \wt{\mbox{Exp}}\Big(\int_{s_1}^\sigma   {\wt \cA} ( \bx (s)  )  \rd s\Big)  
\bu_0 ( \bx(\sigma) ) \rd s_1. \nonumber
\end{align}
Using that the process is stationary in time, we have 
\begin{align*}
\bu(\sigma,  \bx)   & = \E_\bx \bu_0 ( \bx(\sigma) )  - 
 \int_0^\sigma  \E_\bx {\wt \cA} ( \bx (s_1)  ) \bu(\sigma-s_1,  \bx(s_1)) \rd s_1
\\
&
= \E_\bx \bu_0 ( \bx(\sigma) )   -  \int_0^\sigma  \E_\bx {\wt \cA} ( \bx (\sigma-s_1)  )
 \bu(s_1,  \bx(\sigma-s_1)) \rd s_1
\\
& = e^{\sigma \cL} \bu_0 ( \bx )   - 
 \int_0^\sigma  [e^{(\sigma-s_1) \cL}  {\wt \cA} ( \cdot )  \bu(s_1,  \cdot ) ] (\bx) \rd s_1.
\end{align*}
Differentiating  this equation in $\sigma$  we obtain that $\bu$ defined in \eqref{ut}
indeed satisfies \eqref{ge}.

For any fixed path $\{ \bx(s)\; : \; s>0\}$, the time-ordered exponential in \eqref{ut}
$$
   \cU(t)= \cU (t; \bx(\cdot)) :=  \wt{\mbox{Exp}}\Big( -\int_0^t   {\wt \cA} ( \bx (s)  )
  \rd s \Big) 
$$
satisfies the matrix evolution equation
$$
   \pt_t \cU(t) = - \cU(t) \cA(t), \qquad \cU(0)=I,
$$
which can be seen directly from \eqref{tildeexp}.
Let $\bv^b(t)$ be the transpose of  the $b$-th row of the matrix $\cU(t)$, then the equation 
for the column vector $\bv^b(t)$ reads
\be\label{111}
  \pt_t \bv^b(t) = - \cA(t) \bv^b(t), \qquad \bv^b_a(0) =\delta_{ab}.
\ee
Thus taking the $b$-th component of \eqref{ut} we have
$$
 u_b(\sigma, \bx)=   \partial_b G(\sigma,\bx) =  \E^{ \bx} \Big[ \cU (\sigma)\nabla G ( \bx(\sigma) )\Big]_b 
 = \sum_a \E_\bx\big[ \partial_a G(\bx(\sigma)) v_a^b(\sigma)\big],
$$
and plugging this into \eqref{QQF}, we obtain \eqref{reppgeneral}
by using that  $\E^\om[\cdot ] = \int \E_\bx[\cdot]\om(\rd \bx)$. 

Formula \eqref{corrrep2} follows directly from \eqref{reppgeneral} and from the fact that 
 $\cH''\ge \tau^{-1}$ implies a  spectral gap of
order $\tau$, in particular,
$$
 \Big|  \E^\om  \big[ F (\bx(0))  G(\bx (T))\big]- \E^\om[F] \, \E^\om[G]\Big|
  \le e^{-cT/\tau} \| F\|_{L^2(\om)} \|G\|_{L^2(\om)}.
$$
Finally,  \eqref{corrrep} directly follows from this, by taking the $T\to\infty$ limit.
\qed

\bigskip

Now we apply our general formula to 
 the gap correlation function
on the left hand side of \eqref{eq:cor}.  For brevity of the formulas, we consider
only the single gap case, $n=1$; the general case is  a straightforward extension. 
 The gap index $p\in I$, $p\ne K$, is fixed, later we will impose further conditions
on $p$ to separate it from the boundary. The index set is $I= \llbracket -K, K\rrbracket$,
the Hamiltonian 
in \eqref{SDEgen} is given by $\cH_{\by, \wt\by}^r$  and \eqref{SDEgen} takes the
form of \eqref{SDE}.    It is well known  \cite{AGZ} that due to the logarithmic
interaction in the Hamiltonian, $\beta\ge 1$ implies that
the process $\bx(t) = (x_{-K}(t), \ldots , x_K(t))$ preserves the
initial ordering, i.e., $x_{-K}(t)\le \ldots \le x_K(t)$ and
$x_i(t)\in J$ for every $i\in I$. 
The matrix $\wt\cA$ is given by
 $ \wt \cA :=  \wt \cB + \wt \cW$  where $\wt \cB$ and $\wt \cW$ are the
following $\bx$-dependent matrices acting on vectors $\bv\in \bR^\cK$:
\be\label{Bdef}
[\wt \cB ( \bx )  \bv]_j = - \sum_k \wt B_{jk}( \bx ) (v_k- v_j), \quad 
 \wt B _{jk}(\bx) = \frac \beta { (x_j-x_k)^2}
\ge 0
\ee
and 
\be\label{W}
[\wt\cW(\bx) \bv]_j : = \wt W_j(\bx)v_j
\ee
with the function
\be
\wt W_j(\bx): =  \frac{\beta}{2} \Bigg\{ \sum_{|k|\ge K+1}  \Big [   \frac {1- r }  { (x_j-y_k)^2}  
+  \frac { r}  { (x_j-\wt y_k)^2} \Big ] +\frac{1-r}{N} V''\big(\frac{x_j}{N}\big)
  +\frac{r}{N} \wt V''\big(\frac{x_j}{N}\big) \Bigg\}.
\ee
Here $r\in [0,1]$ is a fixed parameter which we will omit from the notation
of $\wt\cW$.

For any fixed path $\bx(\cdot )$,
define the following time-dependent operators (matrices)  on $\bR^\cK$ 
\be\label{61}
\cA(s) := {\wt \cA} (\bx(s) ), \quad  \cB(s) : = \wt  \cB (\bx(s) ), \quad  
\cW(s) :=  \wt \cW(\bx( s)),
\ee
 where $\cW$ is a multiplication operator with 
the $j$-th diagonal $W_j(s)= \wt W_j(x_j(s))$ depending only the
$j$-th component of the process $\bx(s)$.  Clearly $\cA(s) = \cB(s) + \cW(s)$. 
We also define 
   the associated (time dependent) quadratic forms which we denote by 
the corresponding lower case letters, in particular
\begin{align} \non
 \fb(s)[\bu,\bv]& := \sum_{i\in I} u_i [\cB(s)\bv]_i
  = \frac{1}{2}\sum_{k,j \in I} B_{jk}(s)    (u_k-u_j)(v_k-v_j) 
\\ \non
 \fw(s)[\bu,\bv]&: = \sum_{i\in I} u_i [\cW(s)\bv]_i = \sum_i u_i W_i(s)v_i 
\\\label{Adef}
   \fa(s) [\bu,\bv]
&: = \fb(s)[\bu,\bv]+ \fw(s)[\bu,\bv].
\end{align}

\newcommand{\bg}{{\bf g}}
 
With these notations we can apply  Proposition~\ref{prop:repp}
to our case and we get:
\begin{corollary}\label{cor:repp}
Let $h_0$ be given by \eqref{h0def}, let  $O=O_N:\R\to\R$ be an observable for $n=1$, see \eqref{obs},   and assume that 
$\by, \wt \by  \in \cR_{L=0, K}( \xi^2  \delta/2,\al) $, 
in particular $\cA(s)$ given  in \eqref{61} satisfies
$\cA(s)\ge \tau^{-1}$ with $\tau =CK$ by \eqref{convexc}.
Then with a large constant $C_1$ and for any  $p \in I$, $-K \le p \le  K-1 $,  we have  
\begin{align}\label{repp}
\langle h_0; & O(x_{p}-x_{p+1}) \rangle_{\om }
 \\
 &  = \frac{1}{2}\nc\int_0^{C_1K\log K}   \rd \sigma \int  
    \sum_{b\in I}  \pt_b h_0(\bx)  
\E_\bx  \Big[ O'(x_p-x_{p+1})  \big(v_p^b (\sigma )  - v^b_{p+1} (\sigma)\big) \Big]  
 \omega (\rd \bx) + O\Big( \| O'\|_\infty K^{-2}\Big),
 \nonumber
\end{align}
where $\bv^b(s)= \bv^b(s, \bx(\cdot))$ solves \eqref{veq2} with $\cA(s)$ given in \eqref{61}.
\end{corollary}

{\it Proof.} If $h_0$ were a smooth function, then \eqref{repp} directly followed from
 \eqref{corrrep2}. The general case is a simple cutoff argument using that 
$h_0\in L^2(\rd\om)$ and
\begin{align}
   \E^\om |\pt_b h_0|  & \le \E^\om\big[ |(V_\by)'(x_p)| + 
   |(V_{\wt \by})'(x_p)|\big] \\
& \le \sum_{j\not\in I} \E^\om \Big[ \frac{1}{|y_j-x_p|}
  +  \frac{1}{|\wt y_j-x_p|}\Big] +C \nonumber \\
&\le   CK^{(C_3+1)\xi}.
 \nonumber
\end{align}
Here we used  \eqref{expinv}  and  that $\by, \wt\by\in \cR_{L=0,K}= \cR_{L=0,K}(\xi^2\delta/2, \al)$
 are regular on scale $K^{\xi^2}\le K^\xi$, so 
the summation is effectively restricted to $K^\xi$ terms. \qed

\medskip

 The representation \eqref{repp} expresses the
correlation function in terms of the discrete spatial derivative
of the solution to \eqref{veq2}.
To estimate  $v_p^b (\sigma, \bx(\cdot) )  -  v_{p+1}^b  (\sigma,  \bx(\cdot))$
in \eqref{repp}, 
we will now study the H\"older continuity of the solution $\bv^b (s, \bx(\cdot) )$ to \eqref{veq2}
at time $s=\sigma$ and at the spatial point $p$. 
For any fixed $\si$ we will do it for each fixed
path $\bx(\cdot)$, with the exception of a set of ``bad'' paths that will have a  small probability. 

Notice that if all points $x_i$  were approximately regularly spaced in the interval $J$,
then the operator $\cB$ has a kernel $B_{ij}\sim (i-j)^{-2}$, i.e. it 
is essentially a discrete version of the operator $|p|=\sqrt{-\Delta}$. 
H\"older continuity will thus be the consequence of the De Giorgi-Nash-Moser bound
for the parabolic equation \eqref{veq2}. However, we need to control
the coefficients in this equation, which  depend on the random walk $\bx(\cdot )$.

For the De Giorgi-Nash-Moser theory we need both upper and lower bounds on
the kernel $B_{ij}$. The rigidity bound  \eqref{rigi} guarantees 
a lower bound on $B_{ij}$, up to a factor $K^{-C_2\xi^2}\ge K^{-\xi}$.  The level repulsion estimate
implies  certain upper bounds  on $B_{ij}$, but only in
an average sense. In the next section we define
the good set of paths that satisfy both requirements.

\subsection{Sets of good paths}

{F}rom now on we assume  the conditions of Theorem~\ref{cor}.
In particular we are given some $\xi>0$  and we assume that the
boundary conditions satisfy
$\by, \wt \by  \in \cR_{L=0,K}=\cR_{L=0, K}( \xi^2  \delta/2,\al) $
and \eqref{Exm} with this $\xi$. 
We define the following ``good sets": 
\be\label{K1}
\cG :
=  \Big\{  \sup_{0 \le s \le C_1K\log K} \;\sup_{ |j| \le K } |x_j(s)-\alpha_j| \le K^\xip  \Big\}, 
 \ee
 where  
$$
  \xip: =  (C_2+1)   \xi^{ 2 },
$$ 
with $C_2$ being the constant in  \eqref{rigi}
and $\alpha_j$ is given by \eqref{aldefnew}.
We recall the definition of the event $\wt\cQ_{\si,Z}$ 
for any  $Z\in I$ 
 and $\si\in \cT=[0, C_1K\log K]$ from \eqref{wtQdef}.

\begin{lemma}\label{lem:83}
There exists a positive constant $\theta$,  depending on $\xip= (C_2+1)   \xi^{ 2 } $, 
  such that    
\be\label{K11}
\P^\om ( \cG^c ) \le  C e^{ - K^{\theta }}.
\ee
Moreover, there is a constant  $C_4$,  depending on the constant 
$C_2$ in \eqref{rigi} and on 
 $C_3$ in 
\eqref{level11}, \eqref{secondlevel11} 
such that  for any $\xi$ and $\rho$ small enough, we have 
\be \label{K22}
   \P^\om ( \wt \cQ^c_{\si,Z}   ) \le  C K^{ C_4 \xi-  \rho}  
\ee
 for each fixed  $Z\in I$   and fixed
 $\si\in \cT$. 
\end{lemma}

\begin{proof} 
From the stochastic differential equation of the dynamics  \eqref{SDE} we have 
\begin{align}\label{stoccont}
| x_i (t) -x_i(s) |  \le &  C|t-s| +
 \int_s^t \Big[ \sum_{j \in I\atop j \not = i} \frac 1 { |x_j(a) - x_i(a) |} 
  + \sum_{j\in I^c}   \frac 1 { |y_j - x_i(a) |} \Big]
  \rd a +  |B_i(t)-B_i(s)|.
\end{align}

Using \eqref{expinv} and that $\bx(\cdot)$ is invariant under $\om$, we have the bound
\be
\E^\om \Big[ \int_s^t \sum_{j \not = i}  \frac 1 { |x_j(a) - x_i(a) |}\Big]^{3/2} \le
 CK^{3}|t-s|^{3/2} \max_{i\in I} \E^\om   \frac 1 { |x_i - x_{i+1} |^{3/2}} \le  CK^{3+C_3\xi}|t-s|^{3/2}.
\ee
This implies for any fixed $s<t\le C_1K\log K$  and for any $R>0$  that 
\be\label{PR}
\P^\om \left [  \int_s^t \sum_{j \not = i}  \frac 1 { |x_j(a) - x_i(a) |} \ge  R \right ] \le 
 CK^{3+C_3\xi}|t-s|^{3/2} R^{-3/2}.
\ee
  A similar bound holds for the second summation in \eqref{stoccont}; the summation over large $j$
can be performed by using that $\by$ is regular, $\by\in \cR_{L=0, K}$. 

Set a parameter $q\le cR$ and 
 choose a discrete set of increasing times   $\{ s_k \; : \; k\le C_1K\log K/q\}$
such that 
\be\label{times}
0=s_0<s_1\le s_2\le \ldots \le C_1K\log K, \qquad \mbox{and}\quad |s_k - s_{k+1} | \le q.
\ee
{F}rom standard large deviation bounds on the Brownian motion 
increment $B_i(t)-B_i(s)$
and from \eqref{stoccont}, we have the stochastic  continuity estimate
$$
\P^\om \left (  \sup_{ s, t\in [s_{k}, s_{k+1}], |i| \le K } 
|x_i(s)- x_i(t) | \ge R  \right ) \le K e^{- C R^2/q } + CK^4 q^{3/2} R^{-3/2}
$$
for any fixed $k$. 
Taking sup over $k$, and overestimating $C_1K\log K\le K^2$, we have 
\be\label{supxx}
\P^\om \left (  \sup_{ 0\le s, t\le C_1K\log K, |t-s|\le q , |i| \le K } 
|x_i(s)- x_i(t) | \ge R  \right ) \le K^3 q^{-1} e^{- C R^2/q } + CK^6 q^{1/2} R^{-3/2}
\ee
for any  positive $q$ and $R$ with  $q\le cR$.

 {F}rom the rigidity bound \eqref{rigi} we know that for some $\theta_3> 0$
and for any  fixed $k$ we have, 
\be
 \P^\om \nc \Big\{   |x_j(s_k)-\alpha_j| \ge C K^{C_2\xi^2} \Big\} \le C e^{-K^{\theta_3}}, \quad j\in I.
\ee
 Choosing $R = K^{\xip}/2$ and  $q = \exp{(-K^{\theta_3/2})}$, and using 
that  $C K^{C_2\xi^2}\le K^{\xi'}/2$ with the choice of $\xi'$, 
 we have  
\be\label{K110}
\P^\om ( \cG^c ) \le C e^{-K^{\theta_3}} K^3  q^{-1}  + K^3  q^{-1} e^{- C R^2/q }
 + CK^6 q^{1/2} R^{-3/2} \le   C \exp{ ( - K^{\theta_3 /3  })}, 
\ee
 for sufficiently large $K$,
and this proves \eqref{K11}   with $\theta =\theta_3/3$.

\medskip 
  We will   now  prove \eqref{K22}. The number of intersections in
 the definition of $\wt Q_{\si, Z}$ is only a $(\log K)$-power,
so it will be sufficient to prove \eqref{K22} for one set $\cQ^c$.
 We will consider only  the set $\cQ^c_{\si,Z} $ 
and only for  $Z=0$ and $\si=0$. 
The modification needed for 
the general case is only notational.
We start the proof by noting that for $s>0$  
\begin{align}\label{zetain}
  \frac{1}{1+ s'} \int_0^{s'} \rd a \frac{1}{M'} \sum_{i=-M'}^{M'} 
& \frac  1  {  |x_i(a)-x_{i+1}(a)|^{2}}  \\
\le & \;
C   \frac{1}{1+ s} \int_0^s \rd a \frac{1}{M} \sum_{i=-M}^{M} 
 \frac  1  {  |x_i(a)-x_{i+1}(a)|^{2}} \nonumber
\end{align}
holds for any $s'\in [s/2, s]$ and  $M'\in [M/2, M]$. 
 Hence it is enough to estimate the probability
\be
 \P^\om\Big\{  \frac{1}{1+ s} \int_0^{s} \rd a \frac{1}{M} \sum_{i=-M}^{M} 
 \frac  1  {  |x_i(a)-x_{i+1}(a)|^{2}}  \ge K^\rho \Big\}
\label{FIX}
\ee
for fixed dyadic points  $ (s,M)  =\{ (2^{-p_1} K^2, 2^{-p_2} K) \}$ in space-time 
for each integer $p_1,p_2 \le C \log K$. 
Since the cardinality of the set of these dyadic 
points is just $C(\log K)^2$, it suffices to estimate
\eqref{FIX}  only for a fixed $s, M$. 

The proof is different for $\beta=1$ and $\beta>1$. 
In the latter case,
 from \eqref{expinv} we see that the random variable in \eqref{FIX}
has expectation $CK^{C_3\xi}$.  Thus the probability in \eqref{FIX} is bounded by
 $CK^{C_3\xi -\varrho}$,   so 
\eqref{K22} holds in this case with $C_4$ slightly larger than  $C_3+1$ 
to accommodate the  $\log K$ factors.

In the case $\beta=1$  the random variable in \eqref{FIX} has
a logarithmically divergent expectation. To prove \eqref{K22} for $\beta=1$, 
we need to regularize the interaction
on a  very small scale  of order $K^{-C}$ with a large constant $C$.
This regularization is a minor technical detail which does not affect 
other parts of this paper.
 We now explain how it is introduced, but for simplicity 
we will not carry it in the notation  in the subsequent sections.

For any $\by,\tby \in \cR_{L,K}$ satisfying \eqref{J=J}
 and for $\epsilon>0$, we
define the extension $ \om^\epsilon:= \om_{\by,\tby}^{r, \epsilon}$  of  the measure $\om=\om_{\by, \tby}^r$
(see \eqref{omd})
 from the simplex  $J^\cK\cap \Xi^{(\cK)}$ 
 to $\R^\cK$ by replacing  the singular logarithm with a $C^2$-function. 
For
$\bx \in \R^{\cK}$ and $a:=|J|\sim K$ we set
\begin{align*}
   \cH_\epsilon (\f x) := \frac{1}{2}\sum_{i\in I} U^\epsilon(x_i) -  \sum_{i < j} \log_{a\epsilon} (x_j - x_i) \, \qquad U^\epsilon(x):= 
   U_{\by, \tby}^{r,\epsilon}(x) = (1-r)V_\by^\epsilon(x) + r \wt V^\epsilon_{\tby} (x)    
\end{align*}
\begin{align}\label{Vdel}
V_\by^\epsilon(x): = NV(x/N) -2\sum_{k< -K} \log_{a\epsilon} (x-y_k) -2\sum_{k> K} \log_{a\epsilon} (y_k-x), 
\end{align}
where we define 
\begin{align}\label{logd}
\log_\epsilon(x) := {\bf 1}(x \geq \epsilon) \log x + {\bf 1}(x < \epsilon) \Big\{ \log \epsilon + \frac{x - \epsilon}{\epsilon} -
\frac{1}{2 \epsilon^2} (x - \epsilon)^2\Big\}\,.
\end{align}
We remark that the same regularization for a different purpose was introduced in  Appendix A of \cite{EKYY2}.
It is easy to check that $\log_\epsilon(x) \in C^2(\R)$, is concave, and satisfies 
\begin{align*}
\lim_{\epsilon \to 0} \log_\epsilon(x) \;=\;
\begin{cases}
\log x &\text{if } x > 0
\\
-\infty &\text{if } x \leq 0\,.
\end{cases}
\end{align*}
Furthermore, we have the lower bound 
\begin{align}\label{secderlog}
\partial_x^2 \log_\epsilon(x) \;\ge\;
\begin{cases}
- \frac 1 {x^2} &\text{if } x > \epsilon 
\\
-\frac 1 {\epsilon^2}  &\text{if } x \leq \epsilon\,.
\end{cases}
\end{align}
We then  define
$$
    \om^\epsilon (\rd \bx) := Z_\epsilon^{-1} e^{-\beta \cH_\epsilon(\bx)}\rd\bx, \qquad \mbox{on} \quad \R^\cK, \quad
\mbox{where} \quad Z_\epsilon: = \int  e^{-\beta \cH_\epsilon(\bx)}\rd\bx.
$$
Notice that on the support of $\om^\epsilon$  the particles do not necessarily keep their natural order and they are not
confined to the interval $J$. 
 We recall that $\om_{\by,\tby}^{r=0}= \mu_\by$ and  $\om_{\by,\tby}^{r=1}= \wt\mu_{\tby}$
so these definitions also regularize the initial local measures in  Theorem~\ref{thm:local}.

In order to apply  the proof of  Theorem~\ref{thm:local} to $\om^\epsilon$, we need two facts.
First that  $\om$ and $\om^{\epsilon}$ are close in entropy sense, i.e.
\be\label{Sest}  
 S(\om_0 |\om_\epsilon) \le CK^{C}\epsilon^2
\ee
Using this entropy bound with $\epsilon = K^{-C'}$ for a sufficiently large $C'$,  we see that
the measures $\mu_\by$ and $\wt\mu_{\tby}$ can be replaced with their regularized versions  $\mu_\by^\e$, $\wt\mu_{\tby}^\e$
 both in the condition \eqref{Ex} and in the statement \eqref{univ}. We can now use the argument of  Section~\ref{sec:pflocal}
with the regularized measures.

The second fact is that rigidity and level repulsion estimates given in Lemma~\ref{lm:inter} also hold for 
the regularized measure
$\om^\e$. In fact, apart from the rigidity in the form of \eqref{rigi},
 we also need the following weaker level repulsion bound:
\be\label{gapdel}
\P^{\om^\epsilon} (x_{i+1}-x_i\le s) \le CK^{C\xi}s^2, \quad i\in [L-K-1, L+K], \qquad s\ge K^\xi\epsilon.
\ee
Using  \eqref{secderlog}, this bound easily implies
\be\label{gapexp}
\E^{\om^\epsilon} \log_\epsilon'' (x_{i+1}-x_i) \le  CK^{C\xi}|\log\epsilon|.
\ee
Thus the regularized version of the random variable in \eqref{FIX} has a finite
expectation and we obtain \eqref{K22} also for $\beta=1$.

 With these comments in mind, these two facts can be proved 
following 
   the same path 
as the corresponding results in Section~\ref{sec:profmu}. 
 The only slight complication
is that the particles are not ordered, but for $\epsilon = K^{-C'}$ the regularized potential 
strongly suppresses switching order. More precisely, 
we have 
\be\label{trunc11}
  \P^{\om^\epsilon}( x_{i+1}-x_i \le -Ma\epsilon) \le e^{-cM^2}
\ee
for any $M\ge K^3$.  This inequality follows from
the integral
$$
  \int_{-\infty}^{-Ma\epsilon}e^{\log_{a\epsilon} v}\rd v
 \le (a\epsilon)^2 \int_{-\infty}^{-M} e^{-cu^2}\rd u \le e^{-cM^2},
$$
since for $M\ge K^3$ all other integrands in the measure
$\om^\epsilon$ can be estimated trivially at the expense of a multiplicative error 
 $K^{CK^2}$ that is still negligible  when compared with the factor  $\exp (-cM^2)$.
The estimate \eqref{trunc11} allows us to restrict the analysis to
$x_{i+1}\ge x_i - K^{-C''}$ with some large $C''$. This condition
replaces the strict ordering $x_{i+1}\ge x_i$  that is  present  in  Section~\ref{sec:profmu}.
This replacement introduces irrelevant error factors that can be easily estimated. 
This completes the proof of Lemma~\ref{lem:83}.
\end{proof}
 In the rest of the paper  we will work with the regularized measure $\om^\epsilon$ but
for simplicity we will not
carry this regularization in the notation.

\nc

\subsection{Restrictions to the good paths}

\subsubsection{Restriction to the set $\cG$}

 Now we show that
the expectation \eqref{repp} can be restricted to the good set $\cG$
with a small error.  
We just estimate the complement as
\begin{align*}
\int  
    \sum_{b\in I}  |\pt_b h_0(\bx) | 
\E_\bx   \cG^c  & \Big[ |O'(x_p-x_{p+1})|\,  |v_p^b (\sigma )  - v^b_{p+1} (\sigma)| \Big]  
 \omega (\rd \bx) \\
 & \le  C \|O'\|_\infty \int \E^\om  \sum_b  |\pt_b h_0(\bx) | 
 \cG^c  \big [ |v_p^b (\sigma )|  + |v^b_{p+1} (\sigma)|\big].
\end{align*}
Since $\cA\ge 0$ as a $\cK\times\cK$ matrix,   the  equation \eqref{veq2} is contraction in $L^2$.
 Clearly $\cA$ is a contraction in $L^1$ as well, hence it is a contraction in any $L^q$,
$1\le q\le 2$, by interpolation.
By the H\"older inequality and the  $L^q$-contraction for some $1<q<2$,  we have
for each fixed $b\in I$ that 
\begin{align*}
   \E^\om  |\pt_b h_0(\bx) |    \cG^c  |v_p^b (\sigma)| & \le \big[   \E^\om    \cG^c\big]^{q/(q-1)}
   \big[\E^\om   |\pt_b h_0(\bx) |^q |v_p^b(\sigma )|^q \big]^{1/q} \\
& \le  \big[   \P^\om    \cG^c\big]^{q/(q-1)}
   \Big[\E^\om |\pt_b h_0(\bx) |^q \sum_{i\in I}  |v_i^p (0 )|^q \Big]^{1/q} \\
& \le CK^{C_3\xi}e^{- c K^{\theta_4}} \le  e^{- c K^{\theta_4}}
\end{align*}
with some $\theta_4>0$. Here  we used \eqref{K11} for the first factor.
The second factor was
 estimated by  \eqref{expinv}  (recall the definition of $h_0$ from \eqref{h0def}).
After the summation over $b$,   we get
\begin{align*}
  \E^\om  \sum_b  |\pt_b h_0(\bx) | 
  \cG^c  & \Big[ |O'(x_p-x_{p+1})|\,  |v_p^b (\sigma )  - v^b_{p+1} (\sigma)| \Big]  
\le Ce^{- c K^{\theta_4}} \| O'\|_\infty.
\end{align*}
Therefore,   under the conditions of  Corollary~\ref{cor:repp}, 
 and using the notation $\E^\om$ for the process,
we have 
\begin{align}\label{finrep}
\Big|  \langle h_0; &  O(x_{p}-x_{p+1}) \rangle_{\om }  \Big| \\
& \le 
 \frac{1}{2} \| O'\|_\infty\int_0^{C_1K\log K}     \sum_{b\in I}  \E^\om \Bigg[\cG
      |\pt_b h_0(\bx)|  
  \Big| \big(v_p^b (\sigma )  - v^b_{p+1} (\sigma)\big) \Big|\Bigg] \rd\sigma 
  + O\Big(\|O'\|_\infty K^{-2}\Big), \nonumber
\end{align}
where $\bv^b$ is the solution to   \eqref{veq2}, assuming that the constant 
$C_1$ in the upper limit of the integration is large enough.

\subsubsection{Restriction to the set $\wt \cQ$ and the decay estimates}

The complement of the set $\wt\cQ_{\si, Z}$ includes the  ``bad'' paths
for which the level repulsion estimate in an  average sense does not hold.
However, the probability of  $\wt\cQ_{\si, Z}^c$  is not very small,
it is only a small negative power of $K$, see \eqref{K22}. This estimate would not be sufficient
against the time integration of order $C_1K\log K$ in \eqref{finrep}; we will have to
use an $L^1-L^\infty$ decay property of \eqref{veq2} which we now derive.
Denote  the  $L^p$-norm  of a vector $\bu = \{ u_j \; : \; j\in I\}$ by 
\be\label{lpno}
\| \bu \|_p  =  \Big(\sum_{j\in I}  |u_j|^p \Big)^{1/p}.
\ee

\begin{proposition}\label{prop:heat}
Consider the evolution equation 
\be\label{ve}
\partial_s  \bu (s) =  - \cA(s)  \bu (s), \qquad \bu(s)\in \R^I=\R^\cK
\ee
and fix $\si>0$. 
Suppose that for
 some
constant $b$ we have 
\be\label{B}
  B_{jk}(s) \ge   \frac b   { (j-k)^2}, \quad 0 \le s \le \si, \quad j\ne k,
\ee  
and 
\be\label{W1}
W_j (s)  \ge  \frac b { d_j } ,   \qquad d_j := \big| |j|-K\big|+1,
\quad 0 \le s \le \si .
\ee
Then  for any $1\le p\le q\le \infty$  we have the decay estimate 
\be\label{decay}
\| \bu(s) \|_q  \le  ( sb)^{-(  \frac{1}{p}  - \frac 1 q)} 
    \| \bu(0) \|_{ p }, \qquad 0<s\le\si.
\ee
\end{proposition} 

\begin{proof} 
We consider only the case $b = 1$, the  
 general case follows from scaling.  We  follow  the idea of Nash and 
start from the $L^2$-identity 
\be
\partial_s  \| \bu(s) \|_2^2 =   - 2 \fa(s)[\bu(s), \bu(s)]. 
\ee
For each $s$ we can extend $\bu(s):I\to \bR^{ \cK}$ to a function
 $\wt \bu(s): $  on $\ZZ$ by defining  $\wt u_j(s) = u_j(s)$ for $|j|\le K$
and $\wt u_j(s) = 0$ for $j > |K|$.
 Dropping the time argument, we have, by  the estimates \eqref{B} and  \eqref{W1} with $b=1$,
\be
 2 \fa (\bu, \bu) \ge  \sum_{i, j \in \ZZ} \frac { (\wt u_i- \wt u_j)^2 } { (i-j)^2}  
 \ge c  \| \wt \bu \|_4^4 \, \| \wt \bu \|_2^{-2},
\ee
with some positive constant, 
where, in the second step, we  used the  Gagliardo-Nirenberg inequality 
for the discrete operator $\sqrt{-\Delta}$,
 see \eqref{s} in the Appendix~\ref{sec:GN} with $p=4, s=1$.  Thus    we have 
\be\label{sob}
\fa[\bu, \bu] \ge c  \| \bu \|_4^4 \, \| \bu \|_2^{-2},  
\ee
and  the energy inequality 
\be
\partial_s  \| \bu \|_2^2  \le -c  \| \bu \|_4^4 \, \| \bu \|_2^{-2} 
\le  -c  \| \bu \|_2^4 \, \| \bu \|_1^{-2},
\ee
 using the H\"older estimate $\| \bu\|_2\le \|\bu\|_1^{1/3}\|\bu\|_4^{2/3}$. 
Integrating this inequality  from 0 to $s$ we get 
\be\label{uss}
\| \bu(s) \|_2 \le C  s^{-1/2}  \| \bu (0) \|_1,
\ee
 and similarly we also have $\| \bu(2s)\|_2\le  C  s^{-1/2}  \| \bu (s) \|_1$.
Since the previous proof  uses only the time independent lower bounds \eqref{B}, \eqref{W1},
 we can use  duality in the time interval 
$[s, 2s]$ to have
$$
\| \bu(2s) \|_\infty  \le C  s^{-1/2}  \| \bu (s) \|_2.
$$
Together with \eqref{uss}  we have 
\[
\| \bu (2s) \|_\infty  \le C s^{-1}   \| \bu (0) \|_1.
\]
By interpolation, we have thus proved \eqref{decay}.
\end{proof}

\medskip

\newcommand{\fs}{{\frak s}}
In the good set $\cG$ (see \eqref{K1}), the bounds \eqref{B} and \eqref{W1} hold with $b = cK^{-\xip}$. Hence 
from the decay estimate \eqref{decay},   for any fixed $\si, Z$,
we can  insert the other good set $\wt\cQ_{\si, Z}$ into  the
expectation in \eqref{finrep}. 
This is obvious since 
 the contribution of its complement is bounded by   
\begin{align}\label{cut}
  & \int_{0}^{C_1 K \log K}   \rd \sigma \,\sum_b  \E^\om  \wt\cQ^c_{\si, Z} \cG  |\pt_b h_0(\bx)|  
   \big(v_p^b (\sigma )  + v^b_{p+1} (\sigma)\big) 
 \non  \\
& \le CK^\xip  \int_{0}^{C_1 K \log K} \rd \sigma\;  \sigma^{-\frac{1}{1+\xi}}    \;
  \E^\om  \Big[ \cG \big( \sum_{b\in I} |\pt_b h_0 (\bx)|^{1+\xi}\big)^{\frac{1}{1+\xi}} 
\wt\cQ^c_{\si, Z} \Big]  \non \\
& \le CK^{2\xip}  \int_{0}^{C_1 K \log K} \rd \sigma\;  \sigma^{-\frac{1}{1+\xi}}    \;
  \E^\om  \Big[\big[1+d(x_K(\si))^{-1} + d(x_{-K}(\si))^{-1}\big] 
\wt\cQ^c_{\si, Z} \Big] \non\\
& \le CK^{2\xip}  \int_{0}^{C_1 K \log K} \rd \sigma\;  \sigma^{-\frac{1}{1+\xi}}    \;
  \Big[ \E^\om  \big[1+d(x_K(\si))^{-1} + d(x_{-K}(\si))^{-1}\big]^{3/2}\Big]^{2/3} 
 \Big[ \P^\om\big( \wt\cQ^c_{\si, Z}\big)
 \Big]^{\frac{1}{3}} \non\\\
&  \le CK^{2\xip} (C_1K\log K)^\xi K^{C_3\xi} K^{(C_4\xi-\rho)/3},
\end{align}
where in the first line we used a H\"older inequality with exponents
$1+\xi$ and its dual, then 
we used the decay estimate \eqref{decay} with $q=\infty$, $p=1+\xi$
in the second line. The purpose of taking a H\"older inequality
with a power slightly larger than one was to avoid the logarithmic singularity in
the $\rd\sigma$ integration 
at $\sigma\sim 0$.   In the third line
we  split the sum into two parts
and used the bound
\be\label{ini1}
  |\partial_b h_0(\bx) | \le   |(V_\by)'(x_j) - (\wt V_{\wt \by})'(x_j) |
 \leq \frac{K^{\xip}}{d(x_b)},
\ee 
that follows from
 \eqref{Vby1resc} (with $\xi$ replaced by $\xi^2$
since $\by,\wt\by \in \cR_{L,K}(\xi^2\delta/2, \al/2)$).
Recall that  $d(x)$  is the  distance to the boundary, see \eqref{ddef}. 
 For indices away from the boundary, $|b|\le K-CK^{\xi'}$, we have
 $|d(x_b)|\ge K^{-\xi'}\min\{ |b-K|, |b+K|\}$  on the set $\cG$
that guarantees the finiteness of the sum. For indices near the boundary
we just estimated every term with  the worst one, i.e.  with the term $b=\pm K$.  We used a H\"older inequality in the fourth line
of \eqref{cut} and computed the  expectation by using   
\eqref{expinv} in the last line.
Hence we have proved the following proposition:

\begin{proposition}
Suppose that 
\be\label{rhobound}
\rho \ge  12\xip + 6(C_4+C_3+1)\xi 
\ee
holds with $C_3$ and $C_4$ defined in \eqref{level11} and in \eqref{K22}, respectively. 
Then for 
 any   fixed $Z, p\in I$ with $p\ne K$, 
  we have 
\begin{align}\label{cut3}
\Big|  \langle h_0;  & O(x_{p}-x_{p+1}) \rangle_{\om }  \Big| \\
& \le \frac{1}{2} \|O'\|_\infty  \int_0^{C_1 K \log K}   
   \sum_{b\in I}  \E^\om \Big[  \wt\cQ_{\si,Z}  \cG
      |\pt_b h_0(\bx)|  
  \big| v_p^b (\sigma )  - v^b_{p+1} (\sigma) \big|\Big] \rd\sigma 
 +  O\Big(\|O'\|_\infty K^{-  \rho/6  }\Big) . \non
\end{align}
\end{proposition}

\subsection{Short time cutoff and finite speed of propagation }

The H\"older continuity of the parabolic equation \eqref{veq2} emerges
only after a certain time, thus for the small $\sigma$ regime in
the integral \eqref{cut3} we need a different argument.
Since we are interested in the H\"older continuity around
the middle of the interval $I$  (note that $|p|\le K^{1-\xi^*} $ in
Theorem~\ref{cor}), and
the initial condition $\pt_b h_0$ is small if $b$ is in this region,
a finite speed of propagation estimate for  \eqref{veq2} will guarantee
that $v_p^b(\si)$ is small if $\si$ is not too large.

{F}rom now on, we fix $\sigma\le C_1K\log K$,  $|Z|\le K/2$  and 
 a path $\bx(\cdot)$, and assume that $\bx(\cdot) \in \cG\cap \wt\cQ_{\si, Z}$. 
In particular,  thanks to the definition of $\cG$ and 
the regularity of the  locations $\al_j$, the time dependent coefficients $B_{ij}(s)$ and $W_i(s)$ 
of the equation \eqref{veq2} satisfy \eqref{B} and \eqref{W1}
with $b= K^{-\xip}$.  

We split the summation in \eqref{cut3}. Fix a positive
constant $\theta_5>0$.
 The contribution of the indices $|b|\le  K^{1 - \theta_5 }$
to \eqref{cut3} is bounded by 
\begin{align} \label{inest}
 \int_0^{C_1 K \log K}  & 
  \E^\om  \wt\cQ_{\si,Z}  \cG  \sum_{|b|\le K^{1 - \theta_5 } }  |\pt_b h_0(\bx)|
  \Big[      v_p^b (\sigma )  + v^b_{p+1} (\sigma) \Big]
 \rd\sigma \\
& \le  C\int_0^{C_1 K \log K}  \E^\om \Bigg[ \wt\cQ_{\si,Z}  \cG 
 \Big[ \sum_{|b|\le K^{1 - \theta_5 } }  |\pt_b h_0(\bx)|\Big]
  \times \max_{|b|\le  K^{1 - \theta_5 }}
    \big|v_p^b (\sigma )\big| \Bigg]
 \rd\sigma \non\\ 
& \le  C K^{\xip- \theta_5}\int_0^{C_1 K \log K}  
 \E^\om \Bigg[ \wt\cQ_{\si,Z}  \cG  \max_{|b|\le  K^{1 - \theta_5 }}
    \big|v_p^b (\sigma )\big| \Bigg]
 \rd\sigma \non \\
& \le K^{\xip- \theta_5}  \int_0^{C_1 K \log K} \sigma^{-1}\;  \rd \sigma 
 \le K^{2\xip- \theta_5}, \non
\end{align}
where we neglected the $v_{p+1}^b$ term for simplicity since it can be estimated exactly
in the same way. From the second to the third line we used that
$$
  |\pt_b h_0(\bx)|\le \frac{K^{\xip}}{\min \{ |b-K|, |b+K|\}+1} \le CK^{\xip-1}, \qquad 
|b|\le  K^{1 - \theta_5 },
$$
holds on the set $\cG$ from \eqref{ini1} and from the rigidity bound provided by $\cG$. Arriving at the last line
of \eqref{inest}
we used the  $L^1\to L^\infty$ decay estimate \eqref{decay}
and we recall  that the singularity $\sigma\sim 0$ can be cutoff exactly as in \eqref{cut},
i.e. by considering a power slightly larger than 1 in the first line.
Note that the set $ \wt\cQ_{\si,Z}$ played no role in this argument.

Together with   \eqref{cut3}  and with the choice 
\be\label{t5}
\theta_5  >  \rho 
\ee
 and recalling $\rho\ge 4\xip$ from \eqref{rhobound},  
 we have 
\begin{align}\label{cut2}
\Big|  \langle h_0;  &  O(x_{p}-x_{p+1}) \rangle_{\om }  \Big|  \\
 &\le  \frac{1}{2} \| O'\|_\infty  \int_0^{C_1 K \log K} 
  \sum_{|b|> K^{1 - \theta_5 } } \E^\om \Big[ \wt\cQ_{\si, Z}  \cG  |\pt_b h_0(\bx)| 
  \big | v^b_p (\sigma )  -  v^b_{p+1} (\sigma) \big | \Big]
 \rd \sigma 
 + O\Big( \|O'\|_\infty K^{- \rho/6  }\Big). \non
\end{align}

The following lemma provides a finite speed   of
 propagation  estimate for the equation \eqref{veq2}
which will be used to control the short time regime in \eqref{cut2}.
This estimate is not optimal, but it is sufficient for our purpose. 
 The proof will be given in the next section.

\begin{lemma}\label{lem-finite}  [Finite Speed of Propagation Estimate] 
 Fix $b\in I$ 
 and $\sigma\le C_1K\log K$. Consider  $\bv^b(s)$,
 the solution to \eqref{veq2}
and  assume
 that the coefficients of $\cA$  satisfy
\be\label{g3}
   W_i(s) \ge \frac{K^{-\xip}}{d_i},   
    \qquad  B_{ij}(s)\ge \frac{K^{-\xip}}{|i-j|^2}, \qquad 0\le s\le\si,
\ee 
where $d_i: = \min\{ |i+K|, |i-K|\}+1$. 
Assume that  the bound 
\be\label{Kass0}
   \sup_{0 \le s \le \si}\sup_{0 \le M \le K} \frac{1}{ 1+ s} \int_0^s \frac{1}{M}
 \sum_{i\in I\, : \, |i-Z| \le M}\sum_{j\in I\, : \, |j-Z| \le M} 
 B_{ij}(s)
 \rd s \le CK^{\rho_1},
\ee
 is satisfied for some fixed  $Z$, $|Z|\le K/2 $.
Then for any $s>0$ we have the  estimate
\be\label{finite}
|v_p^b(s)| \le  \frac {C K^{  \rho_1  +  2 \xip + 1/2 } \sqrt  {s  +1 }  } {   |p- b |}.
\ee
\end{lemma}

\subsection{Proof of the Finite Speed of Propagation Estimate, Lemma \ref{lem-finite}} 

 Let $1\ll \ell \ll K $  be a parameter to be specified later.
 Split the time dependent operator $\cA=\cA(s)$ defined in \eqref{61} into a short range and
a long range part, $\cA= \cS + \cR$, with 
\be
(\cS \bu)_j :=  -  \sum_{k\; : \; |j-k| \le \ell } B_{jk} (u_k-u_j)  + W_j u_j  
\ee
and
\be
(\cR \bu)_j :=  - \sum_{ k\; : \; |j-k| > \ell } B_{jk} (u_k-u_j).
\ee
 Note that $\cS$ and $\cR$ are time dependent. 
Denote by $ U_\cS ( s_1,   s_2)$  
 the  semigroup associated with $\cS$ from time $s_1$ to time $s_2$,  i.e.
$$
    \partial_{s_2} U_\cS ( s_1,   s_2) = -\cS(s_2)  U_\cS ( s_1,   s_2)
$$
for any $s_1\le s_2$, and $U_\cS(s_1,s_1)=I$; 
 the notation $U_\cA(s_1,s_2)$ is analogous.  Then  by the Duhamel formula
 \be
\bv(s) = U_\cS (0, s) \bv_0 +  \int_0^s  U_\cA ( s', s) \cR(s')  U_\cS (0,  s')  \bv_0 \rd s'.
\ee
Notice that for $\ell \gg K^\xip$ and for $\bx(\cdot)$ in the good set $\cG$ (see \eqref{K1}),
 we have   
\be
\|\cR  \bu\|_1 =   \sum_{|j| \le K} \Big| \sum_{k:  |j-k | \ge \ell }   
\frac 1 { (x_j-x_k)^2} u_k\Big|   \le C \ell ^{-1} \| \bu\|_1,
\ee
or more generally, 
\be
\|\cR \bu\|_p  \le C\ell ^{-1} \| \bu\|_p,  \quad 1 \le p \le \infty.
\ee
Recall the decay estimate \eqref{decay}  for the semigroup $U_\cA$
that is applicable by \eqref{g3}. 
Hence we have,  for $s\ge 2$,
\begin{align*}
\int_0^s  \Lnorm  \infty {   U_\cA ( s',s ) \cR(s')  U_\cS (0,  s')  \bv_0 }  \rd s' 
\\ \le 
K^\xip \int_0^s  (s-s')^{-1}  & \Lnorm  1 {  \cR(s')  U_\cS (0,  s')  \bv_0 }  \rd s' \le  
K^\xip \ell^{-1}  (\log s)    \| \bv_0\|_1 ,
\end{align*}
 where we used that $U_\cS$ is a contraction on  $L^1$.  
 The non-integrable short time singularity for $s'$ very close to $s$, $|s-s'|\le K^{-C}$, can
be removed by  using the $L^p\to L^\infty$ bound \eqref{decay}
with some $p>1$,  invoking a similar argument in \eqref{cut}. 
 In this short time cutoff argument we used that  $U_\cS$ is  an $L^p$ contraction for 
 any $1\le p \le 2$ by  interpolation, and that
 the rate of the $L^p \to L^\infty$  decay of $U_\cA$ are given in \eqref{decay}.
\be\label{21}
\| \bv(s) - U_\cS (0, s) \bv_0 \|_\infty  \le  \ell^{-1} (\log s)     K^\xip 
 \le C\ell^{-1} (\log K)K^\xip,
\ee
where we have used  that $\bx(\cdot)$ is in  the good set $\cG$
 and that $s\le C_1K\log K$.

\newcommand{\br}{{\bf r}}

We now prove a cutoff estimate for  the short range dynamics.
Let $\br(s): = U_\cS (0,  s)  \bv_0$ 
 and define
\be
f(s)  =    \sum_j \phi_{  j}  r_j^2 (s), \quad \phi_j = e^{ | j - b |/\theta}
\ee
with some parameter  $\theta\ge \ell$  to be specified later.  Recall that $b$ is
the location of the initial condition, $\bv_0 = \delta_b$. 
In particular, $f(0)=1$. 

Differentiating $f$  and using $W_j \ge 0$, we have 
\begin{align*}
f'(s)  =   \partial_s  \sum_j  \phi_j   r_j^2 (s)  & \le  
 2  \sum_j  \phi_j   \sum_{k:  |j-k| \le \ell } r_j (s)  B_{kj}(s)(r_k-r_j)(s) 
\\ & 
= 
  \sum_{ |j-k| \le \ell } B_{kj} (s) (r_k-r_j)(s)  \left [ r_j (s)\phi_j  -  r_k (s) \phi_k    \right ]
\\ & 
= 
 \sum_{ |j-k| \le \ell } B_{kj} (s) (r_k-r_j)(s) \phi_j  \left [ r_j -  r_k   \right ] (s) 
\\ & 
\qquad +  
\sum_{ |j-k| \le \ell } B_{kj}(s) (r_k-r_j) (s) \left [ \phi_j -  \phi_k   \right ] r_k (s).
\end{align*}
 In the second term
we use Schwarz inequality and absorb the quadratic term in $r_k-r_j$ into the
first term that is negative.
Assuming  $\ell\le \theta$, we have
$\phi_k^{-2}  \left [ \phi_j -  \phi_k   \right ]^2 \le C\ell^2/\theta^2$
for $|j-k|\le\ell$. Thus
\begin{align*}
f'(s)  
\le   C \sum_{ |j-k| \le \ell } B_{kj} (s) \phi_k^{-1}  \left [ \phi_j -  \phi_k   \right ]^2 r_k^2 (s) 
\le C \theta^{-2} \ell^2  \left ( \sum_{k',  j:  |j-k'| \le \ell } B_{k' j}  (s)  \right )   \sum_k \phi_k  r_k^2 (s) .
\end{align*}
 {F}rom a Gromwall argument we have 
\be
f(s) \le  \exp \left [ C\theta^{-2} \ell^2  \int_0^s  \sum_{k,  j:  |j-k| \le \ell } B_{k j}  (s')  \rd s'   \right ] f(0).
\ee 
{F}rom the assumption \eqref{Kass0} with $M = K$  and
any  $Z$,   we can bound the integration in the exponent by
\be
\int_0^{s}   \sum_{k,  j:  |j-k| \le K } B_{k j}  (s')  \rd s'   \le  K^{ 1 +   \rho_1  } 
 (s+1). 
\ee 
Thus we have 
\be\label{8423}
 \sum_j e^{   | j - b | /\theta}  r_j^2 (s) = f(s)
 \le \exp{ \big[ \theta^{-2} \ell^2 K^{1 + \rho_1  }  (s+1) \big]} f(0)  \le C, 
\ee
provided that we choose  
\be\label{e1}
\theta  =  \ell K^{(\rho_1 + 1)/2 } \sqrt  {s+1}.
\ee
 In particular, this shows the following exponential finite speed of
propagation estimate for the short range dynamics
\be\label{shortexp}
r_j(s)  \le C\exp\Big( - \frac{|j-b|}{ \ell K^{(\rho_1 + 1)/2 } \sqrt  {s+1}} \Big).
\ee   
Now we choose 
$$
 \ell =   | p - b |  K^{  -\xip - (\rho_1 + 1)/2 } (s+1)^{-1/2}
$$
so that  $e^{   | p - b | /\theta} \ge \exp{( K^\xip)}$.
 Using this choice in \eqref{8423} and   \eqref{21} to estimate  $v_p^b(s)$, 
 we have thus proved that 
\be
|v_p^b(s)| \le      \ell^{-1}   (\log K)
  K^\xip  + C  e^{-   K^{-\xip } } \le \frac { K^{  2 \xip + (\rho_1+ 1)/2 }
 \sqrt  {s+1}  } {   |p- b |}.
\ee
This concludes the proof of  Lemma \ref{lem-finite}.

\subsection{Completing the proof of Theorem  \ref{cor}}\label{sec:proof81}

In this section we complete the
 proof of Theorem \ref{cor}  assuming 
a discrete version of the De Giorgi-Nash-Moser
H\"older regularity estimate for the solution \eqref{veq}
  (Theorem~\ref{holderg} below).

 Notice that on the  set $\cG\cap \wt\cQ_{\si, Z}$  the conditions of Lemma~\ref{lem-finite} are 
satisfied,  especially  \eqref{Kass0} with the choice
\be\label{rho1def}
\rho_1:= \rho+\xi'
\ee
follows from the definition \eqref{K2} 
since  for the summands with $|i-j|\ge K^{\xip}$ in \eqref{Kassnew3}
we can use $B_{ij} \le C|\al_i-\al_j|^{-2} \le
C|i-j|^{-2}$.  
 Thus we can  use \eqref{finite}  to estimate
the short time integration regime in \eqref{cut2}. 
 Setting 
\be\label{theta5}
\theta_5: = \min\Big\{ \frac{\xi^*}{2}, \frac{1}{100} \Big\},
\ee 
we obtain, for any  $|Z|\le 2K^{1-\xi^*}$ and 
 $|p|\le K^{1-\xi^*}$
\begin{align}\label{cut8}
 \int_0^{K^{1/4}}   \sum_{|b|>K^{1-\theta_5}}
\; & \E^\om  \Big[ \wt   \cQ_{\si, Z}   \cG  |\pt_b h_0 (\bx)| \big | v^b_p (\sigma )  
-  v^b_{p+1} (\sigma) \big | \Big]
 \rd \sigma  \\
 &  \le C
\int_0^{K^{1/4}} \; \E^\om \Big[ \wt   \cQ_{\si, Z}    \cG   
 \sum_{ |b| >  K^{1 - \theta_5 }} |\partial_b h_0( \bx)| v_p^b(\si)\Big]\rd\si \nonumber \\
&  \le  CK^{2\xip+\rho_1+1/2 + \frac{1}{4} +\frac{1}{8} - (1 - \theta_5)}  \E^\om \Big[ \wt   \cQ_{\si, Z}    \cG
 \sum_{ |b| >  K^{1 - \theta_5 }} |\partial_b h_0( \bx)| \Big] \nonumber \\
&\le  C K^{4\xip+\rho_1 +\theta_5 - \frac{1}{8}}   \E^\om \; \Big[ \wt   \cQ_{\si, Z}  \cG
 \Big( \frac{1}{d(x_K)} + \frac{1}{d(x_{-K})}  \Big) \Big] \nonumber \\
&\le  C K^{4\xip+\rho_1 + C_3\xi+ \theta_5 - \frac{1}{8} } \le K^{-\frac{1}{10}} \nonumber
\end{align}
provided that 
\be\label{allsmall}
  4\xip+\rho_1 + C_3\xi  \le \frac{1}{100}.
\ee
In the third line above we used \eqref{finite}
together with  $|p-b|\ge \frac{1}{2} K^{1 - \theta_5 }$. This latter bound
follows from $|b|>  K^{1 - \theta_5 }$ and  $|p|\le K^{1-\xi^*}$
and  from the choice $\theta_5 < \xi^*$. 
 In the fourth line we used \eqref{ini1} and that
on the set $\cG$ we have 
$$
 \sum_j \frac{1}{d(x_j)} \le (\log K) K^{\xi'}\Big[ \frac{1}{d(x_K)} + \frac{1}{d(x_{-K})}   \Big].
$$
Moreover, in  the last step we used  \eqref{expinv}.   This completes
the estimate for the small $\sigma$ regime. Notice that the set $\wt   \cQ_{\si, Z}$
did not play a role in this argument.

\medskip

After the short time cutoff \eqref{cut8}, we finally have to control the 
regime of large time and large  $b$-indices, i.e.
$$
\int_{K^{1/4}}^{C_1 K \log K} 
  \sum_{|b|> K^{1 - \theta_5 } } \E^\om \Big[ \wt\cQ_{\si, Z}  \cG  |\pt_b h_0(\bx)| 
  \big | v^b_p (\sigma )  -  v^b_{p+1} (\sigma) \big | \Big]
 \rd \sigma 
$$
from \eqref{cut2}. 
We will exploit the H\"older
regularity of the solution $\bv^b$ to \eqref{veq2}. 
We will assume that  the coefficients of $\cA$ in \eqref{veq2} satisfy
a certain regularity condition.  
\begin{definition}\label{def:strreg}
The equation 
\be\label{veq7}
   \pt_t \bv(t) = -\cA(t)\bv(t), \qquad \cA(t) = \cB(t) + \cW(t), \qquad t\in \cT
\ee
 is called {\bf regular} at the
space-time point $(Z, \si)\in I\times \cT$ with exponent $\rho$, if
\be\label{Kassnew3}
   \sup_{s\in \cT}\sup_{1 \le M \le K} 
\frac{1}{ 1+ |s-\si| }\Big|  \int_s^\si \frac{1}{M}
 \sum_{i\in I\, : \, |i-Z| \le M}\sum_{j\in I\, : \, |j-Z| \le M} 
 B_{ij}(u)
 \rd u \Big| \le K^{\rho}.
\ee
Furthermore, the equation is called {\bf strongly regular} at the
space-time point $(Z, \si)\in I\times \cT$ with exponent $\rho$ if it is 
regular at all points $\{ Z\}\times \{ \Xi + \si\}$, where we recall the definition of $\Xi$ 
from \eqref{Kassnew}:
$$
  \Xi = \big\{ - K\cdot 2^{-m}(1+2^{-k}) \; : \; 0\le k,m \le C\log K  \big\}.
$$
\end{definition}

Fix a $Z\in I$, $|Z|\le K/2$ and  a $\si\in \cT$.
 Recall that  on the set $\cG\cap \cQ_{\si,Z}$
the regularity  at $(p, \si)$ with exponent $\rho_1$ from \eqref{rho1def}
follows from \eqref{K2}.
Analogously, on the event $\cG\cap \wt\cQ_{\si,Z}$, the
strong regularity at  $(Z, \si)$ with a slightly increased exponent $\rho_1$ holds.

We formulate the partial H\"older regularity result for the equation \eqref{veq7}.
We collect the following facts on the coefficients $B_{ij}(s)$ and $W_i(s)$ 
that  follow from $\bx(\cdot )\in \cG$.
\be\label{g3new7}
     B_{ij}(s) \ge \frac{K^{-\xip}}{|i-j|^2}, \quad W_i(s)\ge \frac{K^{-\xip}}{d_i} \qquad
   \mbox{for any} \; s\in \cT, \;\; i,j \in I.
\ee
\be\label{g4}
W_i(s)   \le  \frac{K^{\xip}}{d_i}, \; \; \text {for any } \quad s\in \cT, \quad d_i \ge K^{C \xip}.
\ee  
and
\be\label{far1}
    \frac{1}{ C (i-j)^2} \le  B_{ij} (s) \le \frac{C}{(i-j)^2}
  \quad  \mbox{for any} \; s\in \cT, \;\;\ |i-j|\ge \wh CK^{\xip}.
\ee

\begin{theorem} 
 \label{holderg} There exists a universal constant $\fq>0$ with the following properties.
 Let  $\bv(t)= \bv^b(t)$ be a solution to \eqref{veq7}
for any choice of $b\in I$,
with initial condition  $v^b_j(0)=\delta_{jb}$.
Let $Z\in I$, $|Z|\le K/2$ and
$\si \in [K^{c_3}, C_1K\log K]$ be fixed, where $c_3>0$ is an arbitrary positive constant.
There exist positive constants $\xi_0$, $\rho_0$ (depending only on $c_3$) such that
if the coefficients of $\cA$
satisfy \eqref{g3new7}, \eqref{g4},
 \eqref{far1} with some $\xip\le \xi_0$ 
 and the equation is strongly  regular at the point $(Z,\si)$ with an exponent $\rho_1\le \rho_0$
then for any $\al\in [0,1/3]$ we have
\be\label{HC}
\sup_{ |j-Z| + |j'-Z| \le \si_1^{ 1- \al} } 
   | v_j (\si )  -  v_{j'}  (\si)  |  \le C K^{\xip} \sigma^{-1-  \frac{1}{2}\fq \al}, \qquad \si_1:= \min\{\si, K^{1-c_3}\},
\ee
where $\bv= \bv^b$ for any choice of $b$.
The constant $C$ in \eqref{HC}
depends only on  $c_3$.
\end{theorem}
 Theorem~\ref{holderg} follows directly from the slightly more general
Theorem~\ref{thm:holdl1} presented in  Section~\ref{Caff} and it
will be proved there.

\nc

 Armed with Theorem~\ref{holderg},  we now complete the proof of Theorem~\ref{cor}.
As we already remarked, the conditions of  Theorem~\ref{holderg} are satisfied
on the set $\wt   \cQ_{\si, Z} \cap  \cG$ with some $\rho_0, \xi_0$ small universal constants.
 For any $|p|\le K^{1-\xi^*}$ fixed, we choose $Z = p$ (in fact, we could
 choose any  $Z$ with   $|Z-p|\le C$).
Using   \eqref{ini1}, we have, 
 for the large time integration regime in \eqref{cut2}, 
\begin{align}\label{cut9}
 \int_{K^{1/4}}^{C_1K\log K}  \sum_{|b|>K^{1-\theta_5}}
\; & \E^\om  \Big[ \wt   \cQ_{\si, p}   \cG  |\pt_b h_0 (\bx)| \big | v^b_p (\sigma )  
-  v^b_{p+1} (\sigma) \big | \Big]
 \rd \sigma \\
& \le  CK^\xip \int_{K^{1/4}}^{C_1K\log K}  
\; \E^\om  \Bigg[ \wt   \cQ_{\si, p} \cG  \sum_{ |b| >  K^{1 - \theta_5 }} \frac{1}{d(x_b)} 
\big| v^b_p(\si) - v^b_{p+1}(\si)  \big |\Bigg]  \rd \sigma 
\nonumber \\
& \le  CK^{2\xip} \int_{K^{1/4}}^{C_1K\log K}  \;\si^{-1 - \frac{1}{6}\fq }  
\; \E^\om \Bigg[ \wt   \cQ_{\si, p}  \cG  \sum_{ |b| >  K^{1 - \theta_5 }}
 \frac{1}{d(x_b)}  \Bigg]\rd\sigma
\nonumber\\
& \le  CK^{ 3 \xip+\rho_1+ C_3\xi  -\frac{1}{24} \fq }.
\nonumber
\end{align} 
In the third line we used Theorem~\ref{holderg} with $c_3=1/4$ and $\al=1/3$. 
In the last line of \eqref{cut9} we used a similar argument as in the last step of \eqref{cut8}.

 Finally, from \eqref{cut2},  \eqref{cut8} and \eqref{cut9} and  $\rho_1=\rho+\xip$ we have 
\begin{align}\label{cut5}
\left |  \langle h_0;  O(x_{p}-x_{p+1}) \rangle_{\om }  \right
 | \le   C\|O'\|_\infty\Big( K^{4\xip+\rho+ C_3\xi -\frac{1}{24}\fq} 
 +O(K^{-\frac{1}{10}})+ O(K^{-\rho/6}) \Big).
\end{align}

 For a given $\xi^*>0$, recall that we defined
$\theta_5: = \min\{ \frac{1}{2}\xi^*, \frac{1}{100} \}$ and we now choose 
\be\label{defrho}
   \rho : = \min\Big\{ \frac{\fq}{100}, \frac{\theta_5}{2}\Big\}= 
 \min\Big\{ \frac{\fq}{100}, \frac{\xi^*}{4},
  \frac{1}{200} \Big\},
\ee
in particular \eqref{t5} is satisfied. Since $\fq>0$ is a universal constant,
 it is then clear that for any sufficiently small $\xi$
all conditions  in  \eqref{allsmall} and \eqref{rhobound}
on the  exponents $\xi$, $\xip= (C_2+1)\xi^2$ and $\rho_1=\rho+\xip$ 
can be simultaneously satisfied.
Thus  we can make the error term  in \eqref{cut5} 
smaller than  $K^{C\xi}K^{-\rho/6}$. 
With the choice of $\e=\rho/6$, where $\rho$ is from 
\eqref{defrho}, we proved Theorem \ref{cor}.  
\qed

\medskip

Although the choices of parameters seem to be complicated, 
the underlying mechanism is that there is  a universal positive 
 exponent  $\fq$ in \eqref{HC}.
This exponent provides an extra smallness factor in addition to
the natural size of $v_j(\sigma)$, which is $\sigma^{-1}$ from the
$L^1\to L^\infty$ decay. As \eqref{HC} indicates, this gain
comes from a H\"older regularity on the relevant scale.  
The parameters $\xi, \xi'$ and $\xi^\ast$ can be chosen 
arbitrarily small (without affecting the value of $\fq$). 
 These parameters govern the cutoff levels  in 
the regularization of the coefficients of $\cA$.
 There are other minor considerations due to an additional
cutoff for small time where we have to use a
finite speed estimate.  But the arguments for this part are of
 technical nature and most estimates are not optimized. We just worked out 
estimates sufficient for the purpose of proving Theorem \ref{cor}.
 The choices of exponents related to  the various cutoffs  do not have 
intrinsic meanings. 

As a guide to the reader, 
  our choice of parameters, roughly speaking, are given by the following rule:
 We first fix a small parameter  $\xi^*$. 
Then we choose the cutoff parameter $\theta_5$ to be slightly smaller than $\xi^*$,
\eqref{theta5}.   The exponent $\rho$ in \eqref{K2} has a lower bound 
by $\xi$ and $\xi'$ in \eqref{rhobound}. On the other hand, $\rho$ will affect the cutoff bound and so we have the condition $\rho < \theta_5$ 
(i.e., \eqref{t5}). So we choose $\rho \lesssim \xi^*$ and make $\xi, \xi'$ very small so that the lower bound requirement on $\rho$ is satisfied. 
Finally, if the parameter $\xi^* \le \fq/100$, 
we can use the gain from the H\"older continuity to compensate all the errors which 
depend only on $\xi, \xi', \xi^*$.

\section{A discrete  De Giorgi-Nash-Moser estimate}\label{Caff}

In this section we prove Theorem~\ref{holderg}, which is a H\"older regularity estimate
 for the parabolic evolution equation
\be
   \pt_s \bu(s) = -\cA(s) \bu(s).
\label{ve2}
\ee
where   $\cA(s) = \cB(s) + \cW(s)$ are symmetric matrices defined by 
\be
 [\cB(s)\bu]_j
  = - \sum_{k \not = j \in I} B_{jk}(s)    (u_k-u_j), \quad 
[\cW(s)\bu]_i = W_i(s) u_i
\ee
and $B_{ij}(s)\ge 0$.
Here  $I = \{ -K, -K+1, \ldots,  K \}  $ and  $\bu \in \bC^I$.
We will study this equation in a time interval $\cT\subset \bR$ of length $|\cT|=\si$
and we will assume that $\si\in [K^{c_3}, CK\log K]$. 
The reader can safely think of $\si= CK\log K$.  In the applications
we set $\cT=[0,\sigma]$, but we give some definitions more generally.
The reason is that traditionally in the regularity theory for
parabolic equations one sets the initial condition $\bu(-\sigma)$ at some 
negative time $-\sigma<0$ and one is interested in the regularity
of the solution $\bu(s)$ around $s=0$.
In this case $\cT$ starts at $-\si$, so in this section 
$\cT=[-\si, 0]$.
 This convention 
is widely used for parabolic equations and in particular in \cite{C}. Later on in our application, 
we will need to make an obvious  shift in time.

 We will now state a general
H\"older continuity result, Theorem \ref{thm:hold},   concerning 
the  deterministic equation \eqref{ve2} over the  finite set $I$
and on the time interval $\cT=[-\si,0]$. 
Theorem \ref{thm:hold} will be a local H\"older continuity result around an interior point $Z\in I$
separated away from the boundary. 
We recall the definition of strong regularity from Definition~\ref{def:strreg}.
The following conditions on $\cA$ will be needed that are characterized by two
parameters $\xi, \rho>0$.

\begin{itemize}
\item[${\bf (C1)}_\rho$] 
The equation \eqref{ve2} is strongly  regular with exponent $\rho$  at
the space-time point $(Z,0)$.
\item[${\bf (C2)}_\xi$] Denote by  $d_i=d_i^I: = \min\{ |  i + K+ 1|, |1+ K-i|\}$
the distance of $i$ to the boundary of $I$.
 For  some large constants $C, \wh C\ge 10$ , the following conditions are satisfied: 
\be\label{g3new}
B_{ij}(s)\ge \frac{K^{-\xi}}{|i-j|^2}, \qquad\mbox{for any} \;\;
  s\in \cT, \;\;  d_i \ge \frac{K}{C}, \;\; d_j\ge \frac{K}{C},
\ee  
\be\label{g5new}
W_i(s)   \le  \frac{K^{\xi}}{d_i}, \; \; \text { if } \quad d_i \ge K^{C \xi}, \quad  s\in \cT, 
\ee
\be\label{far1new}
    \frac{ {\bf 1}(\min\{ d_i,d_j\}\ge \frac{K}{C}) }{ C (i-j)^2} \le  B_{ij} (s) \le \frac{C}{(i-j)^2}, \quad 
\text{if $|i-j|\ge \wh CK^{\xi}$ 
 and  $s\in \cT$.  } 
\ee
\end{itemize}

\begin{theorem}   [Parabolic  regularity with singular coefficients] 
 \label{thm:hold}  
There exists a universal  constant $\fq>0$ such that the following holds.
 Consider the equation \eqref{ve2} on the time interval $\cT=[-\si,0]$
with some  $\si \in [K^{c_3}, K^{1-c_3}]$, where $c_3>0$ is a positive constant.  
Fix  $|Z| \le  K/2$ and $\al\in[0, 1/3]$.
 Suppose that  ${\bf (C1)}_\rho$ and ${\bf (C2)}_\xi$
hold with some exponents  $\xi, \rho$ small enough depending on $c_3$.
 Then  for the solution $\bu$  to \eqref{ve2} we have
\be\label{HC1}
\sup_{|j-Z|+|j'-Z| \le \si^{ 1- \al} } 
   | u_j (0)  -  u_{j'}  (0)  |  \le C \si^{-\fq \al}  \| \bu(-\si )\|_\infty  .
\ee
The constant $C$ in \eqref{HC1}
depends only on  $c_3$ and is uniform in $K$.
The result holds for any $K\ge K_0$, where $K_0$ depends on  $c_3$.
\end{theorem}

We remark that the upper bound $\si\le K^{1-c_3}$ is not an important condition, it is imposed only
for convenience to state \eqref{HC1} with a single scaling parameter. More generally,
for any $\si\ge K^{c_3}$ it holds that
\be\label{HC1uj}
\sup_{|j-Z|+|j'-Z| \le \si_1^{ 1- \al} } 
   | u_j (0)  -  u_{j'}  (0)  |  \le C \si_1^{-\fq \al}  \| \bu(-\si )\|_\infty  .
\ee
where $\si_1:= \min \{ \si, K^{1-c_3}\}$. If $\si \ge K^{1-c_3}$, then \eqref{HC1uj}
immediately follows by noticing that $\|u(-\si_1)\|_\infty \le\|u(-\si)\|_\infty$
and apply \eqref{HC1} with $\sigma_1=K^{1-c_3}$ instead of $\si$.

To understand why Theorem~\ref{thm:hold} is a H\"older regularity result, we rescale the
solution so that it runs up a time of order one. I.e, for a given $\si\ll 1$ we define
 the rescaled solution
$$
     U(T,X): = u_{[\si X]+Z}(T\si) , \qquad \si\gg 1,
$$
(where $[\; \cdot\;]$ denoted the integer part). Then the bound \eqref{HC1} says that
$$
    \sup_{|X|+|Y|\le\e} |U(0,X)-U(0,Y)|\le C \e^\fq \| U(-1, \cdot)\|_\infty, \qquad \e \in [\si^{-1/3},1].
$$
Thus, in the macroscopic coordinates $(T,X)$ the H\"older regularity for $U$ holds around $(0,0)$ from 
order one scales down to order $\sigma^{-1/3}$ scales. Note that H\"older regularity
holds only at one space-time point, since the strong regularity condition ${\bf (C1)}_\rho$  was
centered around a given space-time point $(Z,0)$ in miscroscopic coordinates.

Notice that by imposing the regularity condition
we only require  the time integration of  the singularity of $B_{ij}$  is  bounded. 
Thus we substantially weaken the standard assumption in parabolic regularity theory
on the supremum bound on  the ellipticty.

\medskip

Theorem~\ref{thm:hold} is a H\"older regularity result with $L^\infty$ initial data.
Combining it with the decay estimate Proposition~\ref{prop:heat}, we get  
a H\"older regularity result with $L^1$ initial data. However, for the application of 
the decay estimate, we need to strengthen the condition \eqref{g3new} to
\be\label{g3newglob}
    B_{ij}(s) \ge \frac{K^{-\xi}}{|i-j|^2}, \quad W_i(s)\ge \frac{K^{-\xi}}{d_i} \qquad
   \mbox{for any} \; s\in \cT, \;\; i,j \in I.
\ee
Let  ${\bf (C2)}^*_\xi$ be the condition idential to ${\bf (C2)}_\xi$
except that \eqref{g3new} is replaced with \eqref{g3newglob}.

\begin{theorem}  
 \label{thm:holdl1}  
There exists a universal  constant $\fq>0$ such that the following holds. 
 Consider the equation \eqref{ve2} on the time interval $\cT=[-\tau-\si,0]$
with some $\tau>0$ and
  $\si \in [K^{c_3}, K^{1-c_3}]$, where $c_3>0$ is a positive constant.  
Fix  $|Z| \le  K/2$ and $\al\in[0,1/3]$.
 Suppose that  ${\bf (C1)}_\rho$ and ${\bf (C2)}_\xi^*$
hold with some small exponents  $\xi, \rho$ depending on $c_3$.
 Then  for the solution $\bu$  to \eqref{ve2} we have
\be\label{HC2}
\sup_{|j-Z|+|j'-Z| \le \si^{ 1- \al} } 
   | u_j (0)  -  u_{j'}  (0)  |  \le C K^\xi \si^{-\fq \al} \tau^{-1}  \| \bu(-\tau-\si )\|_1 .
\ee
The constant $C$ in \eqref{HC1} 
depends only on  $c_3$ and is uniform in $K$.
The result holds for any $K\ge K_0$, where $K_0$ depends on $c_3$.
\end{theorem}

{\it Proof.} We can apply Proposition~\ref{prop:heat} with $b=K^{-\xi}$,
$p=1$, $q=\infty$ on the time interval $[-\tau-\si, -\si]$.
 Then \eqref{decay} asserts that
$$   
    \| \bu(-\si )\|_\infty \le K^\xi \tau^{-1}  \| \bu(-\tau-\si )\|_1 
$$
and \eqref{HC2} follows from \eqref{HC1}. \qed

\medskip

{\it Proof of Theorem~\ref{holderg}.} To avoid confusion between the roles of $\si$,
in this proof we denote  the $\si$ in the statement of Theorem~\ref{holderg} by $\si'$.
 We will apply  Theorem~\ref{thm:holdl1} and we choose  $\si$ and $\tau$
such that $\si'=\si+\tau$. We also shift the time by $\si'$ so   that the initial
time is zero and the final time  $\si+\tau =\si'$.
The conditions ${\bf (C1)}_\rho$ and ${\bf (C2)}_\xi^*$
 directly follow from
 \eqref{g3new7}, \eqref{g4}, \eqref{far1} and from strong regularity at $(Z,\si')$ 
but $\varrho_1$ and $\xi'$ are replaced by $\varrho$ and $\xi$ for simplicity of notations.
Given $\si' \in [K^{c_3}, C_1K\log K]$, we consider two
cases. If $\si'\le K^{1-c_3}$, we apply Theorem~\ref{thm:holdl1} with $\si=\tau= \si'/2$.
Then $\| \bu(-\tau-\si)\|_1$ becomes $\|\bv^b\|_1=1$ on the right hand side of \eqref{HC2}
and  \eqref{HC} follows. If  $\si'\ge K^{1-c_3}$, then we  
apply Theorem~\ref{thm:holdl1} with 
$\si = \frac{1}{2}  K^{1-c_3}$ and  $\tau : = \si' -\si$.
In this case $\tau$ is comparable with $\si'$ and $\si' \le \si^{3/2}$
and  \eqref{HC}  again follows. \qed

\medskip

 The rest of this section is devoted to the proof of Theorem \ref{thm:hold}. 
 Our strategy  follows the approach of  \cite{C}; the multiscale iteration scheme and 
the key  cutoff functions   (\ref{psi}, \ref{psit}) are also  the same as in \cite{C}. 
The main new feature of our argument is  the derivation 
of the local energy   estimate, 
 Lemma~\ref{lm:energy}, for parabolic equation with singular coefficients satisfying 
 ${\bf (C1)}_\rho$ and ${\bf (C2)}_\xi$.
The proof of Lemma~\ref{lm:energy}  will  proceed  in two steps.  
We first use the condition  ${\bf (C1)}_\rho$
and the argument of the energy estimate
in \cite{C}  to provide a bound in $L^\infty_t (L^2(\ZZ)) $ on the solution to
 \eqref{ve2} (part (i) of Lemma~\ref{lm:energy}). Along this proof
we also prove an energy dissipation estimate which can be translated 
into the statement that the energy is small for most of the times.
Using a new Sobolev type inequality (Proposition~\ref{prop:newGNdiscr})
designed to deal with weak ellipticity 
we can improve the bound in $L^\infty_t (L^2(\ZZ)) $ to an $L^\infty$ estimate in
space for most of the times to obtain part (ii) of  Lemma~\ref{lm:energy}. 
Finally, we run the argument again to improve   the  $L^\infty_t (L^2(\ZZ)) $  estimate 
for short times (part (iii) of  Lemma~\ref{lm:energy}) that is needed to close the iteration scheme.
Besides this proof,  the derivation of  the second De Giorgi estimate (Lemma~\ref{lm:2nd})
is also adjusted to the weaker condition ${\bf (C1)}_\rho$.

 We warn the reader that the notations of various constants in this
section  will follow \cite{C} as much as possible for the sake of easy comparison with the paper 
\cite{C}.  The conventions of these constants  will differ  from the ones in the previous sections, and,
 in particular, we will restate all conditions.

\nc

\subsection{H\"older regularity}

 Define 
 for any set $S$ and any real function $f$ the oscillation  $\mbox{Osc}_S f := \sup_S f -\inf_S f$.

\begin{theorem}\label{thm:caff}   
There exists a universal positive constant $\fq$ with the
following property.
For any two thresholds $1<\vartheta_1<\vartheta_0$
there exist two positive constants $\xi, \rho$, depending
only on $\vartheta_1$ and 
$\vartheta_0$ such that the following hold:

Set  $ \cM: = 2^{-\tau_0} K$ where $\tau_0\in \bN$ is chosen such that
$\vartheta:=\log K/\log \cM \in [\vartheta_1, \vartheta_0]$.
Suppose that \eqref{ve2} satisfies  ${\bf (C1)}_\rho$  and ${\bf (C2)}_\xi$ 
with some  $Z\in \llbracket - K/2, K/2\rrbracket$.
Suppose $\bu$ is a solution to \eqref{ve2} in the time interval $\cT=[-3 \cM, 0]$.
 Assume that  $u$ satisfies 
\be\label{supbound}
   \sup_{t\in [-3 \cM, 0]}\max_i |u_i(t)|\le \ell
\ee
for some $\ell$.  Then 
for any $\al\in[0, 1/3] $ 
 there is a set  $\cG\subset [-\cM^{1-\al},0]$ such that
 \be
 \mbox{Osc}_{Q^{(\alpha)*}} (u) \le  4\ell  \cM^{-  \fq \alpha }, \qquad 
Q^{(\alpha)*}: =   \cG \times
\llbracket Z- 3\cM^{1-\alpha}, Z+3\cM^{1-\alpha} \rrbracket,
\label{holder}
\ee 
with
$$
   |[-\cM^{1-\al}, 0]\setminus \cG|\le \cM^{1/4},
$$
 i.e. the oscillation of the solution on scale $\cM^{1-\alpha}$ 
(and away from the edges of the configuration space)
 is smaller than $ 4\ell \cM^{- \fq \alpha}$
for most of the times.
Moreover, 
\be
 \mbox{Osc}_{\bar Q^{(\alpha)}} (u) \le  C \ell  \cM^{-  \fq \alpha }, \qquad 
\bar Q^{(\alpha)}: =  [- \cM^{1/2}, 0]\times
\llbracket Z- 3\cM^{1-\alpha}, Z+3\cM^{1-\alpha} \rrbracket,
\label{sureholder}
\ee 
i.e. the oscillation is controlled for all times near 0.

These results hold for any $K\ge K_0$ sufficiently large, where the threshold $K_0$ 
as well as the constant $C$ in \eqref{sureholder} depend 
only on the parameters $\vartheta_0$, $\vartheta_1$.
\end{theorem}

We remark that the  constant $\fq$ plays the role of
the H\"older exponent and it depends only on  
 $\e_0$ from Lemma \ref{lm:energy}. This will be explained after
Lemma~\ref{lm:dircalc} below.

\medskip

{\it Proof of Theorem~\ref{thm:hold}.}
With  Theorem~\ref{thm:caff},
we now complete the proof of Theorem~\ref{thm:hold}.
Given $\si\in [K^{c_3}, K^{1-c_3}]$, define $\cM: =2^{-\tau_0} K$
with some $\tau_0\in \bN$ such that $\si/6\le \cM \le \si/3$.
  Choosing
$$
   \vartheta_1: = 1+\frac{1}{2}c_3,\quad  \vartheta_0:= \frac{2}{c_3},
$$
clearly $\vartheta =\log K/\log\cM \in [\vartheta_1, \vartheta_0]$. Then 
\eqref{HC1} follows from \eqref{sureholder} at time $t=0$
using that $\si^{1-\al}\le 3\cM^{1-\al}$. 
\qed

\medskip

The proof of Theorem~\ref{thm:caff} will be a multiscale argument. 
On each scale $n=0,1,2\ldots, n_{max}$ we define
a space-time scale $M_n := \nu^n \cM$ and a size-scale $\ell_n: = \zeta^n \ell$
with some $\nu, \zeta<1$ scaling parameters to be chosen later.
The initial scales are
$M_0 =\cM$ and $\ell_0=\ell$.
For notational convenience we assume
that $\nu$ is of the form $\nu = 2^{-j_0}$ for some $j_0>0$ integer.
We assume that $\nu\le \zeta^{10}/10$,
and eventually $\zeta$ will be very close 1, while $\nu$ will be very close to 0.
The corresponding space-time box on scale $n$ is given by
$$
    Q_n := [-M_n , 0]\times [Z-M_n, Z+M_n].
$$
We will sometimes use an enlarged box
$$
   \wh Q_n : = [-3M_n , 0]\times [Z-\wh M_n, Z+ \wh M_n], \qquad 
\wh M_n: = 
L M_n
$$
with some  large parameter $L$
that will always be chosen such that 
$\nu\le \frac{1}{2L}$ and
thus $\wh Q_n\subset Q_{n-1}$.
We stress that the scaling parameters $\nu,\zeta, L$
 will be absolute constants, independent of any
parameters in the setup of Theorem~\ref{thm:caff}.

The smallest scale is given by the relation $M_{n_{max}}\sim \cM^{1-\al}$, i.e. 
$n_{max}= \al\frac{\log \cM}{|\log\nu|}$. In particular, since $\al\le 1/3$, all scales arise in 
the proofs will be between $\cM^{2/3}$ and $\cM$:
\be\label{between}
   \cM^{2/3} \le M_n \le \cM=M_0, \qquad \forall n=0,1,\ldots n_{max}.
\ee 
The following statement is the main technical result that will immediately imply
Theorem~\ref{thm:caff}. In the application we will need only the second part
of this technical theorem, but its formulation is taylored to its proof that will
be an iterative argument from larger to smaller scales.  

There will be several exponents in this theorem, but the really important one
is $\chi$, see explanation around \eqref{c2def} later. The exponents $\xi$ and $\rho$
can be chosen arbitrarily small and the reader can safely neglect them at first reading.

\begin{theorem}[Staircase estimate]\label{thm:iteration} 
There exist positive parameters $\nu, \zeta, L$, satisfying
$$
   \nu< \min\{ \zeta^{10}/10, 1/(2L)\}
$$
with the following property.

For any two  thresholds $1<\vartheta_1<\vartheta_0$
there exist  three positive constants $\chi$, $\xi$, and $\rho$ depending 
only on $\vartheta_1$ and 
$\vartheta_0$ (given explicitly in \eqref{chichoice} and \eqref{xiro}
later) such that under the setup and conditions of
Theorem~\ref{thm:caff},  for any scale $n=0,1,2, \ldots n_{max}$
there  exists a descreasing sequence of  sets $\cG_n\subset  [-3M_n, 0]$
 of ``good'' times, $\cG_n\subset \cG_{n-1}\subset \ldots$, with
\be\label{cGmeas}
    \big| \cG_n^c\big|\le C\sum_{m=0}^{n-1} M_m^{1/4}, \qquad \cG^c_n : = [-3M_n, 0]\setminus \cG_n, 
\ee
such that
we have the following two estimates:
\begin{itemize}
\item[i)]  [Staircase estimate]
Define the constant $\bar u_n$ by 
\be\label{35}
   \sup_{Q_n^*} | u - \bar u_n| =\frac{1}{2}\mbox{Osc}_{Q_n^*} (u), 
\ee
where 
$$
Q_n^*: =  \big([-M_n , 0]\cap \cG_n\big)\times [Z-M_n, Z+M_n]
$$
 and 
for any $m<n$ define
\be\label{Smn}
  S_{m,n}: =  \sum_{j=m}^{n-1} |\bar u_j-\bar u_{j+1}|.
\ee
Then
\be\label{coll}
   |u_i(t)-\bar u_n|\le \Psi_i^{(n)}(t)  
\quad \forall t\in [-3M_n, 0], \quad \forall
i   \qquad \qquad  {\bf (ST)}_n
\ee
where $\Psi^{(n)}$ is a function on $[-3M_n, 0]\times I$ defined by 
$$
    \Psi_i^{(n)}(t): = \Lambda_i^{(n)} \cdot {\bf 1}(t\in  \cG_n) + 
     \Phi_i^{(n)}(t) \cdot {\bf 1}(t\in \cG_n^c)
$$
with 
$$
   \Lambda_i^{(n)}:=  {\bf 1}( \wh M_0 
\le |i-Z| ) \cdot \ell_0 + \sum_{m=0}^{n-1} {\bf 1}( \wh M_{m+1} 
\le |i-Z|\le \wh M_{m})\cdot
 \Big[\ell_m + S_{m,n}\Big]
$$
$$
  + {\bf 1}(|i-Z|\le \wh M_{n})\cdot
 \ell_n
$$
and
\begin{align}\label{defPhi} 
   \Phi_i^{(n)}(t):= &\; C_\Phi\cdot {\bf 1}( \wh M_0 
\le |i-Z| ) \cdot \ell_0 \\  \nonumber & 
+ C_\Phi\sum_{m=0}^{n-1} {\bf 1}( \wh M_{m+1} 
\le |i-Z|\le \wh M_{m})\cdot \Big[ 
 \ell_m \Big(1+\sqrt{\frac{|t|+  \cM^{1/2}  }{M_m}} M_m^{\chi/2} \Big) + S_{m,n}\Big] \\
\nonumber &  + C_\Phi\cdot {\bf 1}(|i-Z|\le \wh M_{n})\cdot
 \ell_n \Big(1+\sqrt{\frac{|t|+  \cM^{1/2} }{M_n}} M_n^{\chi/2} \Big)
\end{align} 
with some fixed constant $C_\Phi$.
The subscript $\Phi$ in $C_\Phi$ indicates that this specific constant controls the
functions $\Phi^{(n)}$.

\item[ii)]  [Oscillation estimate]  For the good times we have 
\be\label{oscreduce}
   \frac{1}{2} \mbox{Osc}_{Q_{n+1}^*} (u) \le \zeta\ell_n =\ell_{n+1},  \qquad \qquad  {\bf (OSC)}_n
\ee
 i.e. \eqref{oscreduce} asserts that in
the smaller box  $Q_{n+1}^*\subset Q_{n}^*$ the oscillation is reduced from $\ell_n$ to $\ell_{n+1}$.
\end{itemize}
All statements hold for any $K\ge K_0$ sufficiently large, where 
the threshold $K_0$ as well as the constant $C_\Phi$ depend
only on the universal constants  $\nu$, $\zeta$, $L$ and on  the parameters
$\vartheta_0$, $\vartheta_1$,
 $\xi$, $\rho$.
\end{theorem}

Here the time-independent profile $\Lambda^{(n)}$ is the ``good'' staircase
function, representing the control for most of the times (``good times'').
The function  $i\to \Lambda^{(n)}_i$  is a stepfunction that monotonically increases in $|i-Z|$
at a rate approximately
$$
   \Lambda^{(n)}_i \sim \ell_n \Big(\frac{|i-Z|}{M_n}\Big)^\fq, \qquad  |i-Z|\gg M_n,
$$
where
\be\label{c2def}
 \fq=\frac{|\log \zeta|}{|\log \nu|}
\ee
is a small positive exponent. Note that this exponent is the same as the
final H\"older exponent in Theorems~\ref{thm:caff} and \ref{thm:hold}.

For the ``bad times'' (the complement of good times), a larger control described by $\Phi^{(n)}(t)$
holds. This weaker control is time dependent and deteriorates with larger $|t|$.
 The exponent $\chi$ in the definition of $\Phi$, see \eqref{defPhi}, will be essentially equal to $\fq$
(modulo some upper cutoff, see \eqref{chichoice} later). 
The factor $M_n^{\chi/2}$ on scale $n$ expresses how much the estimate deteriorates
for  ``bad times''  compared with the estimate at ``good times''.

The bound \eqref{coll} for good times $t\in\cG_n$ with the control function $\Lambda^{(n)}$ directly
follows from \eqref{oscreduce} and \eqref{supbound}.
The new information in \eqref{coll} is the weaker
estimate expressed by $\Phi^{(n)}$ that holds for all times.
Note that $\Lambda^{(n)}_i\le  \Phi_i^{(n)}(t)$, i.e.
the bound 
\be\label{alltime}
|u_i(t)-\bar u_n|\le \Phi^{(n)}_i(t), \qquad \forall t\in [-M_n,0], \;
 \forall i
\ee follows from \eqref{coll}. 
We also remark that
\eqref{oscreduce} implies $|\bar u_n - \bar u_{n+1}|\le \ell_n$
and thus  
\be\label{ubarbound}
 S_{m,n}=\sum_{j=m}^{n-1} |\bar u_j-\bar u_{j+1}| \le \sum_{j=m}^{n-1}\ell_j \le \frac{\ell_m}{1-\zeta},
\ee
gives an estimate for the effect $S_{m,n}$ of the shifts
 in the definition of $\Lambda^{(n)}$ and $\Phi^{(n)}$.
Moreover, the uniform bound \eqref{supbound} shows that for any $n$
\be\label{ubartriv}
|\bar u_n|\le \ell_0=\ell.
\ee

\medskip

{\it Proof of Theorem~\ref{thm:caff}.}  Without loss of generality
we can assume that  $\cM^{-\al}\le \nu^2$, otherwise
$\cM^{-\fq\al} \ge \zeta^2\ge 2/3$, so \eqref{holder} immediately follows
from \eqref{supbound}. For $\cM^{-\al}\le \nu^2$,
the estimate \eqref{holder} directly
follows from \eqref{oscreduce}, by
choosing $n\ge 1$
 such that $M_{n+2}\le 3\cM^{1-\alpha}\le M_{n+1}$, i.e.  $\nu^{n+2}\le 
3\cM^{-\alpha}\le \nu^{n+1}$.
 Then $\ell_{n+1} = \ell \zeta^{n+1} \le  2\cM^{-\fq\al}$
with $\fq$ defined in \eqref{c2def}. The set $\cG$ in Theorem~\ref{thm:caff}
will be just $\cG_{n+1}\cap [-  \cM^{1-\alpha},0]$.
The proof of \eqref{sureholder} follows from \eqref{coll}
noting that for $|t|\le \cM^{1/2}\le M_n^{3/4}$ (see \eqref{between}) the terms
$$
   \sqrt{\frac{|t|+  \cM^{1/2} }{M_m}} M_m^{\chi/2} \ll 1, \qquad m=0,1,2\ldots n,
$$
are all negligible and simply we have
$$
   \Phi_i^{(n)}(t) \le C_\Phi\Lambda_i^{(n)}, \qquad  |t|\le \cM^{1/2}.
$$
Thus  \eqref{sureholder} follows exactly as  \eqref{holder}.
This completes the proof. \qed
\nc
\bigskip

For the rest of the section we will prove Theorem~\ref{thm:iteration}.
We will iteratively check the main estimates, {\bf $(ST)_n$}  and {\bf $(OSC)_n$}
from scale to scale. For $n=0$, the bound {\bf $(ST)_0$} is given by \eqref{supbound}.
In Section~\ref{sec:firststep} we  prove for any $n$ that {\bf $(ST)_n$} implies {\bf $(OSC)_n$}.
In Section~\ref{sec:secondstep} we prove that  {\bf $(ST)_n$} and
  {\bf $(OSC)_n$} imply {\bf $(ST)_{n+1}$}. From these two statements
it will follow that {\bf $(ST)_n$}  and {\bf $(OSC)_n$} hold for any $n$.
Sections~\ref{sec:firstdg} and \ref{sec:seconddg} contain the proof of two independent results
(Lemma~\ref{lm:energy} and \ref{lm:2nd})
formulated on a fixed scale, that are used in  Section~\ref{sec:firststep}.
These are the generalizations of the first and second De Giorgi lemmas in \cite{C},
adjusted to our situation where no supremum bound is available on the
coefficients $B_{ij}(s)$, we have control only in a certain average sense.

\subsection{Proof of  {\bf $(ST)_n$}  $\Longrightarrow$ {\bf $(OSC)_n$}}\label{sec:firststep}

For any real number $a$, we will use the notation 
 $a_+ = \max (a, 0)\ge 0$ and $a_- = \min (a, 0)\le 0$, in particular $a=a_++a_-$. 
Fix a large  integer number $M$ 
and a center $Z\in I$
with $d_Z\ge K/2$ (recall that $d_i$  was defined above \eqref{g3new};
it is the distance of $i$ to the boundary).
For any $\ell>0$ and $\lambda\in (0,1/10)$ define the functions 
\be\label{psi}
  \psi_i =\psi^{(M,Z,\ell)}_i:
 = \ell\Big( \Big|\frac{i-Z}{M}\Big|^{1/2}-1\Big)_+,
\ee
\be\label{psit}
  \wt \psi_i = \wt \psi^{(M,Z,\ell,\lambda)}_i : = 
\ell\Big[ \Big(   \Big|\frac{i-Z}{M}\Big| - \lambda^{-4}\Big)^{1/4}  - 1\Big]_+.
\ee
Notice that $\psi_i=0$ if $|i-Z|\le M$ and  $\wt\psi_i =0$ if $|i-Z|\le M\lambda^{-4}$.
Here $\ell$ will  play the role of the typical size of $u-\psi$.
One could scale out $\ell$ completely, but we keep it in.
We also define the scaled versions of these functions for  any $n\ge 0$:
$$
   \psi_i^{(n)} : = \psi^{(M_n,Z,\ell_n)}_i, \qquad \wt\psi_i^{(n)}: = \wt\psi_i^{(M_n, Z, \ell_n, \lambda)}.
$$

\begin{proposition}\label{prop:oscc}
Suppose that for some $n\ge 0$ we know  {\bf $(ST)_j$} for any $j=0,1,\ldots, n$. Then
 {\bf $(OSC)_n$} holds. Furthermore, we have
\be\label{l2prop}
    \sum_i \big( u_i(t)- \bar u_n -\ell_n- \psi_i^{(n)}\big)_+^2
  \le C 
\Big(\frac{ |t|+  \cM^{1/2} }{M_n}\Big)M_n^\chi \ell_n^2, \qquad t\in [-M_n, 0].
\ee 
\end{proposition}

{\it Proof of Proposition~\ref{prop:oscc}.} 
With a small constant $\lambda\in (10L^{-1/4} ,1)$ and a large integer $k_0$,
to be specified later, define the rescaled
and shifted functions
\be\label{vdeff}
   v^{(n,k)}_i(t) := \ell_n + \lambda^{-2k}\Big( [u_i(t) - \bar u_n] -\ell_n\Big), \qquad k =0,1,2, \ldots k_0.
\ee
In particular, from 
 {\bf $(ST)_n$} we have
\be\label{400}
   v^{(n,k)}_i(t) \le \ell_n + \lambda^{-2k}\Big( \Psi_i^{(n)}(t) -\ell_n\Big), \qquad t\in [-3M_n, 0].
\ee

We will show that  with an appropriate choice $k=k(n)$, 
$v= v^{(n,k+1)}$ satisfies a better upper bound than \eqref{400}
which then translates into a decrease in the oscillation of $u$
on scale $n$. The improved upper bound on $v$
will follow from applying two basic lemmas
 from parabolic regularity theory, 
traditionally called the first and the second De Giorgi lemmas.
The second De Giorgi lemma asserts that going from a larger to a smaller space-time
regime, the maximum of $v_i(t)$ decreases in an average sense.
The first De Giorgi lemma enhances this statement to a supremum bound for $v_i(t)$
that is strictly below the maximum of $v_i(t)$ on a larger space-time regime. 
This is equivalent to the reduction of the oscillation of $v$.

In the next section we first state these two basic lemmas, then we continue
the proof of Proposition~\ref{prop:oscc}. The proofs of the De Giorgi lemmas
are deferred to Sections~\ref{sec:firstdg} and \ref{sec:seconddg}.

\subsubsection{Statement of the generalized De Giorgi lemmas}

Both results will be formulated on a fixed space-time scale $M$ and with a fixed size-scale $\ell$.
 We fix a center $Z\in I$ with
$|Z|\le K/2$.  Recall the definition of $\psi=\psi^{(M,Z,\ell)}$ from \eqref{psi}.
The first De Giorgi lemma is a local dissipation estimate:

\begin{lemma}\label{lm:energy} There exists a small positive 
universal constant $\e_0$ with the following properties.
Consider the parabolic equation \eqref{ve2} on the time interval $\cT= [-\si , 0]$
with some $\si\in [K^{c_3}, K^{1-c_3}]$ and
let $\bu$ be a solution. Define $\bv : = \bu - \bar u$ with some constant shift $\bar u\in\bR$.
Fix small positive constants $\kappa, \xi, \rho,\chi$
and a large constant $\vartheta_0$ such that
\be\label{kappacond1}
 10 \vartheta_0(\xi + \rho)\le \kappa\le \frac{1}{1000},
\quad  \kappa+10 \vartheta_0(\xi + \rho)\le \chi\le \frac{1}{1000}
\ee
holds. 
Let $M$ be defined by $M: = K^{1/\vartheta}$ with some
 $\vartheta \in [ 1+2\kappa, \vartheta_0]$. 
 We assume that the matrix elements of  $\cA=\cB+\cW$
satisfy \eqref{g3new}, \eqref{g5new}, \eqref{far1new}
with exponent $\xi$  and that \eqref{ve2}
is regular with exponent $\rho$ at the space-time points
$(Z,t)$, $t\in \Xi_0$, where
\be\label{Tau0}
   \Xi_0: = \{ -M\cdot 2^{-m} (1+ 2^{-k}) \; : \; 0\le m,  k \le  C\log M\} .
\ee 
Assume 
\begin{align}\label{barvbound}
|\bar u| & \le C\ell K^{1-\xi}M^{-1},
\\
\label{defell1/2}
  \Big[
 \frac{1}{M^2}\int_{-2M}^0 \rd t \sum_i  (v_i(t)-\psi_i)_+^2
 \Big]^{1/2}  & \le\e_0\ell,  \qquad \psi_i =\psi_i^{(M,Z,\ell)}, 
\\
\label{numbercontrol}
 \sup_{t\in [-2M, 0]} \sup
\{ |i-Z| \; : \; v_i(t) > \psi_i \} & \le M^{1+\kappa},
\\ \label{vlinfty}
  \sup_{t\in [-2M, 0]} \max\big\{ v_i (t) \; : \; |i-Z|  \le M^{1+\kappa}\big\}& \le C\ell M^{\chi/2},
\end{align}
and there exists a set $\cG^* \subset [-2M,0]$ with $| [-2M, 0]\setminus \cG^*|\le CM^{1/4}$ such that
\be\label{vlinftygood}
  \sup_{t\in [-2M, 0]\cap \cG^*} \max\big\{ v_i (t) \; : \; |i-Z|\le M^{1+\kappa}\big\}\le C\ell M^{\chi/10}.
\ee
Then, for any sufficiently large  $K\ge K_0(\vartheta_0)$, we have
the following statements:

\begin{itemize}
\item[i)] We have
\be
   \sup_{t\in [-M,0]} \sum_i \Big(v_i(t)-\psi_i -\frac{\ell}{3}\Big)_+^2 \le CM^\chi \ell^2.
\label{1step}
\ee

\item[ii)] There exists a set $\cG\subset [-M, 0]$ of ``good'' times such
that
\be
   \sup_{t\in \cG}  v_i(t)\le \frac{\ell}{2} + \psi_i, \quad \forall i, 
 \qquad \mbox{and} \quad \big| [-M, 0]\setminus
 \cG \big| \le C M^{1/4}.
\label{goodtimes}
\ee

\item[iii)] For any $\wt M$ with $ M^{2\chi} \ll \wt M \le\frac{1}{2} M$ we have
\be
   \sup_{t\in [-\wt M,0]} \sum_i \Big(v_i(t)-\psi_i -\frac{2\ell}{5}\Big)_+^2 
\le C \Big(\frac{\wt M}{M}\Big)M^\chi \ell^2.
\label{11step}
\ee
\end{itemize}
These results hold for any $K\ge K_0$, where the threshold $K_0$ and
the constants in \eqref{1step}--\eqref{11step}
 may depend on $\chi,\kappa, \xi, \rho, \vartheta_0$
 and on the constants $C$ and $\wh C$ in \eqref{g3new}, \eqref{g5new}, \eqref{far1new}.
\end{lemma} 

For the orientation of the reader we mention how the various  exponents will
be chosen in the application. The important exponents are $\kappa$ and $\chi$;
they will be related by $\kappa =3\chi/4$, see \eqref{kappach} later
(actually, the really important relation is that $\kappa<\chi$).
The exponents $\xi, \rho$ will be chosen much smaller; the reader may neglect them
at first reading.

Notice that  \eqref{1step} is off from the optimal bound by a factor $M^\chi$.
However, \eqref{goodtimes} shows that for most of the times,
this factor is not present, while \eqref{11step} shows that
this factor is reduced if the time interval is shorter.  We remark that precise
coefficients of $\ell$
in the additive shifts appearing in \eqref{1step}--\eqref{11step}
are not important; instead of  $\frac{1}{2} > \frac{2}{5} > \frac{1}{3}$ 
essentially any three numbers between 0 and 1 with  the same ordering could have been chosen.


\medskip

The second De Giorgi lemma is a local descrease of oscillation on a single scale.
As before, we are given three parameters, $M, Z, \ell$.
Define a new function $F$ by
\be\label{Fdef}
   F_i = F_i^{(M,Z,\ell)}: = \ell \cdot \max \Big\{ -1, \min\Big(0, \Big|\frac{i-Z}{ M}\Big|^2 
- 81  \Big)\Big\}
\ee
for any $M, Z, \ell$.
Notice that $  - \ell \le F \le 0$, furthermore $F_i = 0$ if $|i-Z| \ge   9  M$
 and $F_i=-\ell$  if
$|i-Z| \le 8 M$.  We also introduce  a new parameter $\lambda\in (0, 1/10)$.  
 Recalling the definition of $\wt\psi$ from \eqref{psit}, 
we also define three cutoffs, all depending on all four parameters, $M,Z,\ell,\lambda$
\begin{align*}
   \varphi^{(0)}_i & := \ell + \wt\psi_i + F_i
\\
   \varphi^{(1)}_i & := \ell + \wt\psi_i +  \lambda F_i
\\
   \varphi^{(2)}_i & := \ell + \wt\psi_i + \lambda^2 F_i.
\end{align*}
Notice  that 
\be\label{alleq}
   \varphi^{(0)}_i \le  \varphi^{(1)}_i \le  \varphi^{(2)}_i \le \ell + \wt\psi_i, 
\ee
and  when  $ |i-Z|\ge 9M$ all inequalities become equalities. 
 Notice that $\varphi^{(0)}_i = 0$ if $ |i - Z | \le  8M$.

\begin{lemma}\label{lm:2nd} 

Consider the parabolic equation \eqref{ve2} on the time interval $\cT= [-\si , 0]$
with some $\si\in [K^{c_3}, K^{1-c_3}]$ and
let $\bu$ be a solution. Define $\bv : = \bu - \bar u$ with some constant shift $\bar u\in\bR$.
Fix small positive constants $\kappa_1, \kappa_2$,  $\xi, \rho$ 
and a large constant $\vartheta_0$ such that 
\be\label{kappacond}
\kappa_1 + \kappa_2 +10 \vartheta_0(\xi + \rho) \le  \frac{1}{1000}, 
\ee
Let $M$ be defined by $M: = K^{1/\vartheta}$ with some
  $\vartheta \in [ 1+2\kappa_1, \vartheta_0]$.  
 We assume that the matrix elements of  $\cA=\cB+\cW$
satisfy \eqref{g3new}, \eqref{g5new}, \eqref{far1new}
with exponent $\xi$  and that \eqref{ve2}
is regular with exponent $\rho$ at the space-time points
$(Z,t)$, $t\in \Xi_0$, where $\Xi_0$ was given in \eqref{Tau0}.

For any $\delta>0$ and $\mu>0$  there exist  $\gamma>0$
and $\lambda\in (0,1/8)$ such that  whenever 
\be\label{barvbound1}
|\bar u| \le C\lambda \ell K^{1-\xi}M^{-1},
\ee
and the  shifted solution $\bv(t)=\bu(t)-\bar u$ 
 satisfies the following five properties;
\begin{align}\label{1c}
  \exists \cG \subset [-3M, 0],  \quad \big| [-3M, 0] \setminus \cG\big|\le CM^{1/4}, 
\quad \mbox{s.t.} \quad    v_i(t) & \le  \ell   + \wt\psi_i, \qquad t\in \cG, \quad \forall i,
\\ \label{11c}
    \sup_{t\in [-3M, 0]} \max\Big\{ |i-Z| \; : \;  v_i(t)  > \ell+ \wt\psi_i
 \Big\} &\le M^{1+\kappa_1},
\\ \label{23c}
    \sup_{t\in [-3M, 0]} \sup \big\{  v_i(t) \; : \;|i-Z|\le  M^{1+\kappa_1} \big\} & \le \ell M^{\kappa_2},
\\ \label{2c}
  \frac{1}{M^2} \int_{-3M}^{-2M}  {\bf 1}(t\in \cG)\cdot 
 \#\Big\{ |i-Z|\le M \; : \; v_i(t)<\varphi^{(0)}_i\Big\} \rd  t & \ge \mu, 
\\ \label{opt1}
   \frac{1}{M^2} \int_{-2M}^{0} {\bf 1}(t\in \cG)\cdot  \#\Big\{ i \; : \; v_i(t)>\varphi^{(2)}_i\Big\} \rd t  
 & \ge  \delta,  
\end{align}
then 
\be\label{opt2}
  \frac{1}{M^2} \int_{-3M}^{0}  {\bf 1}(t\in \cG)\cdot 
\#\Big\{ i\; : \;  \varphi_i^{(0)}< v_i(t)<\varphi^{(2)}_i\Big\}\rd t \ge \gamma .
\ee
This conclusion holds for any $K\ge K_0$ where the threshold $K_0$ depends on 
all parameters $\vartheta_0$, $\kappa_1$, $\kappa_2$, $\xi$, $\rho$, $\delta$, $\mu$ and the 
constants in the conditions \eqref{g3new}, \eqref{g5new}, \eqref{far1new}.

We remark that the choices of $\gamma$ and $\lambda$ are explicit, one may choose
\be\label{gamlamchoice}
  \gamma: = c\delta^3, \quad \lambda := c\delta^6\mu
\ee
 with a small absolute constant $c$. 
\end{lemma}
This lemma asserts that whenever the a substantial part of the
function $v$ increases from $\varphi^{(0)}$ to $\varphi^{(2)}$ 
in time of order $M$, then there is a time interval of order
 $M$ so that a substantial part of $v$ lies between  $\varphi^{(0)}$ and $\varphi^{(2)}$.


\subsubsection{Verifying  the assumptions of  Lemma~\ref{lm:2nd}}

We will apply
Lemma~\ref{lm:2nd}  to the function $v= v^{(n,k)}$ given in \eqref{vdeff}
with the choice $M=M_n$, $\ell=\ell_n$. The following lemma collects
the necessary information on $v= v^{(n,k)}$ to verify 
 the assumptions in  Lemma~\ref{lm:2nd}. The complicated relations 
among the parameters, listed in \eqref{guarant} and \eqref{zetabound} below,
can be simultaneously satisfied; their appropriate choice will be
given in Section~\ref{sec:parchoice}.

\begin{lemma}\label{lm:dircalc}  
Assume that  {\bf $(ST)_n$} holds, see \eqref{coll}. 
Suppose that in addition to the previous relations $\nu<\min\{ \zeta^{10}/10, 1/(2L)\}$
and $\lambda\ge 10L^{-1/4}$ among the parameters,  the following further 
 relations also hold:
\be\label{guarant}
10\le (1-\zeta)\lambda^{2k_0} \zeta L^{1/4}, \quad 
 \chi +10\vartheta_0(\xi+\rho)\le \frac{1}{1000}, \quad
 100\vartheta_0(\xi+\rho)\le
\chi\le \frac{|\log \zeta|}{|\log\nu|}, 
\ee
\be\label{zetabound}  
\vartheta \in [1+2\chi, \vartheta_0], \qquad
 1 - \frac{1}{2}\lambda^{2(k_0+1)} \le \zeta < 1 .
\ee  
Then for any $ v^{(n,k)}_i(t)$ with $k\le k_0$,  
defined in \eqref{vdeff} and satisfying \eqref{400}, we have the
following three bounds:
\begin{align}\label{uppb}
 \sup_{t\in \cG_n}\sup_{k\le k_0}  v^{(n,k)}_i(t) & \le \ell_n + \wt\psi_i^{(n)},  \; 
\\
 \label{11ccheck}
   \sup_{k\le k_0} \sup_{t\in [-3M_n, 0]} \max\Big\{ |i-Z| \; : \;  v_i^{(n,k)}(t)
 > \ell_n+\wt\psi_i^{(n)}  \Big\} & \le M_n^{1+3\chi/4},  
\\
 \label{23ccheck}
   \sup_{k\le k_0} \sup_{t\in [-3M_n, 0]} \sup \big\{  v_i^{(n,k)}(t) \; : \;|i-Z|\le  
M_n^{1+3\chi/4} \big\} & \le C\ell_n M_n^{\chi/2}.
\end{align}
For the shift in \eqref{vdeff} we have the bound
\be \label{shiftcheck}
 \big| \ell_n -\lambda^{-2k}(\bar u_n +\ell_n)\big| \le C\lambda\ell_nK^{1-\xi}
 M_n^{-1}.
\ee
The constants $C$ may depend on all parameters in \eqref{guarant}, \eqref{zetabound}.
\end{lemma}
We remark that the factor $3/4$ in the exponent in \eqref{11ccheck} can be improved to $2/3+\e'$ 
for any $\e'>0$, but what is really important for the proof is that it is 
{\it strictly smaller than 1}, since this will translate into the crucial $\kappa <\chi$
condition in \eqref{vlinfty}.

\medskip

{\it Proof of Lemma~\ref{lm:dircalc}.}  
 All four estimates follow by 
direct calculations from the definition of $\Psi^{(n)}(t)$ and from
the relations \eqref{guarant}, \eqref{zetabound} among the parameters. Based upon \eqref{400}, 
the estimate \eqref{uppb} amounts to checking
\be\label{tochh}
   \Lambda_i^{(n)} \le \ell_n + \lambda^{2k_0} 
\ell_n\Big[ \Big(   \Big|\frac{i-Z}{M_n}\Big| - 
\lambda^{-4}\Big)^{1/4}  - 1\Big]_+.
\ee
For  $|i-Z|\le \wh M_n$ we immediately have $ \Lambda_i^{(n)} =\ell_n$
and thus \eqref{tochh} holds.
For $\wh M_{m+1}\le |i-Z|\le \wh M_m$ (with some $m\le n-1$)
we can use \eqref{ubarbound}, 
to have that $\Lambda_i^{(n)}\le 2(1-\zeta)^{-1} \ell_m$.
The right hand side of \eqref{tochh} is larger than
$$
    \ell_n + \lambda^{2k_0} 
\ell_n\Big[ \Big(   \Big|\frac{\wh M_{m+1}}{M_n}\Big| - 
\lambda^{-4}\Big)^{1/4}  - 1\Big]_+.
$$ 
which is larger than $\ell_n(1 + \frac{1}{2} \lambda^{2k_0} L^{1/4} \nu^{(m-n)/4})$.
Now \eqref{tochh}  follows from the first inequality in \eqref{guarant} 
and from $\nu\le \zeta^{10}/10$.

For the proof of \eqref{11ccheck}, starting from \eqref{400}, it is sufficient to check that
\be\label{tochh21}
  \Phi_i^{(n)}(t) \le \ell_n +  \lambda^{2k_0} 
\ell_n\Big[ \Big(   \Big|\frac{i-Z}{M_n}\Big| - 
\lambda^{-4}\Big)^{1/4}  - 1\Big]_+
\ee
for any $|i-Z|\ge \frac{1}{2}M_n^{1+3\chi/4}$ and $t\in [-3M_n,0]$.
 On the left hand side
we can use the largest time $|t|=3M_n\ge \cM^{1/2}$.
Considering the regime
$\wh M_m \le |i-Z|\le \wh M_{m-1}$ with $M_m = M_n^{1+\beta}$ for some $0<\beta< \frac{1}{2}$, we see that 
$$
  \mbox{l.h.s. of \eqref{tochh21}}\le  2\ell_m \Big(\frac{M_n}{M_m}\Big)^{1/2}M_m^{\chi/2} , \qquad 
    \mbox{r.h.s. of \eqref{tochh21}} \ge \frac{1}{2} \lambda^{2k_0} \ell_n \Big(\frac{M_m}{M_n}\Big)^{1/4},
$$
Using  $\chi\le \frac{|\log \zeta|}{|\log\nu|}$ from \eqref{guarant}, we have
$$
   \frac{\ell_m}{\ell_n} \le \Big(\frac{M_m}{M_n}\Big)^\chi
$$
therefore \eqref{tochh21} holds if
\be\label{iff}
  \Big(\frac{M_n}{M_m}\Big)^{\frac{1}{2}-\chi}M_m^{\chi/2}
 \le \frac{1}{4}\lambda^{2k_0}   \Big(\frac{M_m}{M_n}\Big)^{1/4}.
\ee
Recalling that $M_m = M_n^{1+\beta}$, we see that for small $\chi$ \eqref{iff} is satisfied
if $\beta > \frac{2\chi}{3-6\chi}$ (and $M_n$ is sufficiently large
depending on all constants $\lambda, \nu, L, k_0, \nu, \zeta$).
This is guaranteed if $\beta \ge 3\chi/4$ since we assumed
 $\chi\le 1/1000$. This proves \eqref{11ccheck}.

For the proof of \eqref{23ccheck} we notice that 
\be\label{phim}
   \max\Big\{ \Phi_i^{(n)}(t)\;: \;  |i-Z|\le M_n^{1+3\chi/4} \Big\} \le C_\Phi\big(\ell_m + M_n^{\chi/2}\ell_n\big)
  \le C  M_n^{\chi/2}\ell_n
\ee
for any $t\in [-3M_n,0]$, where $m<n$ is defined by $\wh M_{m+1} \le M_n^{1+3\chi/4} \le \wh M_m$.
The first inequality in \eqref{phim} follows from \eqref{defPhi}; the second one is a consequence of
\be\label{ellM}
  \frac{\ell_m}{\ell_n} = \Big( \frac{M_m}{M_n}\Big)^{\frac{|\log \zeta|}{|\log \nu|}} \le
   \Big( \frac{M_m}{M_n}\Big)^{\frac{1}{10}}
   \le M_n^{\chi/10}
\ee
by $|\log \zeta|\le\frac{1}{10}|\log \nu|$. Then  \eqref{23ccheck} 
 directly follows from \eqref{400} and \eqref{phim}.

Finally, \eqref{shiftcheck}   follows from the facts that
$|\bar u_n|\le \ell=\ell_0$, $K^{1-\xi}\ge \cM=M_0$ (using $\vartheta\ge 1+2\xi$)
and that $\ell_0/\ell_n \le M_0/M_n$.
This completes the proof of Lemma~\ref{lm:dircalc}. \qed

\subsubsection{Completing the proof of Proposition~\ref{prop:oscc}}

We now continue the proof of Proposition~\ref{prop:oscc}. Set 
$$
    F^{(n)}_i: = F_i^{(M_n, Z, \ell_n)},
$$
where $F$ is given in \eqref{Fdef}
and  we define further cutoff functions:
$$
   \varphi^{(0),(n)}_i: = \ell_n + \wt\psi_i^{(n)} + F_i^{(n)}
$$
$$
   \varphi^{(1),(n)}_i: = \ell_n + \wt\psi_i^{(n)} +  \lambda F_i^{(n)}
$$
$$
   \varphi^{(2), (n)}_i: = \ell_n + \wt\psi_i^{(n)} + \lambda^2 F_i^{(n)}.
$$
Throughout this section $n$ is fixed, so we will often omit this
from the notation. In particular $\ell=\ell_n$, $M=M_n$, $\bar u=\bar u^{(n)}$,
 $v^{(k)}= v^{(n,k)}$, $F=F^{(n)}$, $\wt\psi= \wt\psi^{(n)}$, 
$\varphi^{(a)}_i=\varphi^{(a),(n)}_i$ for $a=0,1,2$, $\cG=\cG_n$
etc. At the end of the proof we will add back the superscripts.

From the definitions of these cutoff functions, we have
\be\label{alleq1}
   \varphi^{(0)}_i \le  \varphi^{(1)}_i \le  \varphi^{(2)}_i \le \ell + \wt\psi_i, 
\ee
and  when  $ |i-Z|\ge 9M$ all inequalities become equalities. 
 Notice that $\varphi^{(0)}_i = 0$ if $ |i - Z | \le 8 M$.

Choose a small constant $\mu\in (0, 1/10)$, say 
\be\label{muchoice}
\mu:=\frac{1}{100}.
\ee
Without loss of generality, we can assume
\be
\frac{1}{M^2} \int_{-3M}^{-2M}
 \#\Big\{ i\; : \; |i-Z|\le M, \; u_i(t)-\bar u<\varphi^{(0)}_i\Big\} \rd t \ge \mu
\label{441}
\ee
(otherwise we can take $-u$).

Notice that for any $|i-Z|\le M$ and $t\in \cG$ the sequence $v^{(k)}_i(t)$
is decreasing in $k$,  in particular $v^{(k)}_i(t) \le \ell$. This follows
from  \eqref{400} and that $\Psi_i^{(n)}(t)\le \ell_n$ in this regime.  
From \eqref{441} therefore we have
\be
\frac{1}{M^2} \int_{-3M}^{-2M} {\bf 1}(t\in \cG)\cdot 
 \#\Big\{ i\; : \;  |i-Z|\le M, \; v_i^{(k)}(t)<\varphi^{(0)}_i\Big\}  \rd t \ge \mu,
\label{442}
\ee
since the set of $i$ indices in \eqref{442}  is increasing in $k$ for any  $t\in \cG$
and $v^{(0)}= u-\bar u$.

Assuming that the parameters satisfy \eqref{guarant} and \eqref{zetabound},
we can now use
the conclusions \eqref{uppb}--\eqref{shiftcheck} in Lemma~\ref{lm:dircalc}.  These bounds together with  \eqref{442} 
allow us to apply   Lemma~\ref{lm:2nd} to $v^{(k)}=v^{(n,k)}$ 
with the choice 
\be\label{deltachoice}
\kappa_1 := \frac{3}{4}\chi, \quad \kappa_2 := \frac{1}{2}\chi, \quad \delta:=\frac{\e_0^2}{100}, 
\ee
 where $\e_0>0$ is
a universal constant which was determined 
 in Lemma~\ref{lm:energy}.
 Notice that with these choices \eqref{kappacond} follows from \eqref{guarant} and
 $\vartheta\in [1+2\kappa_1, \vartheta_0]$
follows from $ \vartheta\in [1+2\chi, \vartheta_0]$.
 Thus the application of Lemma~\ref{lm:2nd}   yields that 
there exist a $\lambda$  (introduced explicitly in the
construction of the cutoffs $\varphi^{(a)}$ and used also in \eqref{psit} and \eqref{400})  and a 
 $\gamma>0$ (see \eqref{gamlamchoice} for their explicit values) such that if
\be 
\frac{1}{M^2} \int_{-2M}^{0} {\bf 1}(t\in \cG)\cdot 
 \#\Big\{ i\; : \;  \; v_i^{(k)}(t)>\varphi^{(2)}_i\Big\} \rd t > \delta
\label{443}
\ee
then 
\be
\frac{1}{M^2} \int_{-3M}^{0} {\bf 1}(t\in \cG)\cdot 
 \#\Big\{ i\; : \;  \; \varphi^{(0)}_i <v_i^{(k)}(t)< \varphi^{(2)}_i\Big\} \rd t \ge\gamma.
\label{444}
\ee
Therefore
\begin{align}
\label{91}
   \frac{1}{M^2} \int_{-3M}^{0} {\bf 1}(t\in \cG)\cdot  &
 \#\Big\{ i\; : \;  \; v_i^{(k)}(t)> \varphi^{(2)}_i\Big\} \rd t 
\\  \nonumber
&  \le  \frac{1}{M^2} \int_{-3M}^{0} {\bf 1}(t\in \cG)\cdot 
 \#\Big\{ i\; : \;  \; v_i^{(k)}(t)> \varphi^{(0)}_i\Big\} \rd t -\gamma .
\end{align} 
Notice that, by \eqref{uppb} and $F_i = 0$ if $|i-Z| \ge 9M$, 
for any $k\le k_0$ the inequality $v_i^{(k)}(t)> \varphi^{(0)}_i$
(for $t\in \cG$) can hold
only if $ |i-Z| \le 9 M$.  Assuming $|i-Z| \le 9M$, $t\in \cG$ and 
$v_i^{(k)}(t)> \varphi^{(0)}_i$, we have 
\be
\frac 1 { \lambda^2}  (v_i^{(k-1)}(t) - \ell) + \ell  = v_i^{(k)}(t)> \varphi^{(0)}_i.
\ee
Since $|i-Z| \le 9 M \le  \lambda^{-4} M$, we have,  together with \eqref{alleq1}
and that $\wt\psi_i=0$ in this regime,  that 
\be
v_i^{(k-1)}(t) \ge   \lambda^2 ( \wt \psi_i + F_i) +  \ell  \ge \varphi^{(2)}_i.
\ee
Therefore, we can bound the last integral in \eqref{91} by  
\begin{align} 
\frac{1}{M^2} \int_{-3M}^{0} & {\bf 1}(t\in \cG)\cdot 
 \#\Big\{ i\; : \;  \; v_i^{(k)}(t)> \varphi^{(0)}_i\Big\} \rd t   \nonumber \\
&  \le  \frac{1}{M^2} \int_{-3M}^{0} {\bf 1}(t\in \cG)\cdot 
 \#\Big\{ i\; : \;  |i-Z| \le 9 M,   \; v_i^{(k-1)}(t)> \varphi^{(2)}_i\Big\} \rd t.
 \label{93}
\end{align} 
We have thus proved that 
\begin{align} 
   \frac{1}{M^2} \int_{-3M}^{0} & {\bf 1}(t\in \cG)\cdot 
 \#\Big\{ i\; : \;  \; v_i^{(k)}(t)> \varphi^{(2)}_i\Big\} \rd t 
\nonumber \\
 & \le  \frac{1}{M^2} \int_{-3M}^{0} {\bf 1}(t\in \cG)\cdot 
 \#\Big\{ i\; : \;  |i-Z| \le 9 M,   \; v_i^{(k-1)}(t)> \varphi^{(2)}_i\Big\} \rd t- \gamma.
 \label{94}
\end{align} 
Iterating this estimate  $k$ times, we  get
\begin{align*}
   \frac{1}{M^2} \int_{-3M}^{0} {\bf 1}(t\in \cG) & \cdot 
 \#\Big\{ i\; : \;  \; v_i^{(k)}(t)> \varphi^{(2)}_i\Big\} \rd t 
\\ & \le  \frac{1}{M^2} \int_{-3M}^{0} {\bf 1}(t\in \cG)\cdot 
 \#\Big\{ i\; : \; |i-Z| \le 9 M,   \; v_i^{(0)}(t)> \varphi^{(2)}_i\Big\} \rd t- k\gamma,
\end{align*}
which becomes negative if   $k\gamma \ge 100$. 
Setting 
\be\label{k0choice}
k_0:= \frac{100}{\gamma},
\ee
thus
 there is a  $k< k_0$   such that \eqref{443} is violated,  i.e., 
\be
\frac{1}{M^2} \int_{-2M}^{0} {\bf 1}(t\in \cG)\cdot 
 \#\Big\{ i\; : \;  \; v_i^{(k)}(t)>\varphi^{(2)}_i\Big\} \rd t \le \delta.
\label{445}
\ee
From now on let $k=k(n)$ denote the smallest index so that \eqref{445} holds
(recall that the underlying
$n$ dependence was omitted from the notation in most of this section). 
Furthermore, since 
$\varphi^{(0)}_i=0$  for $|i-Z|\le 8M$, we have 
\begin{align}\label{lm1cond}
 \frac{1}{M^2} &\int_{-2M}^{0} {\bf 1}(t\in \cG)\cdot
 \#\Big\{ i\; : \;  |i-Z|\le 8M, \; v_i^{(k+1)}(t)>0\Big\} \rd t \\ \nonumber
&   =   \frac{1}{M^2} \int_{-2M}^{0} {\bf 1}(t\in \cG)\cdot
 \#\Big\{ i\; : \;  |i-Z|\le 8M, \; v_i^{(k+1)}(t)>\varphi^{(0)}_i\Big\} \rd t 
\\ \nonumber
&  \le  \frac{1}{M^2} \int_{-2M}^{0} {\bf 1}(t\in \cG)\cdot
 \#\Big\{ i\; : \;   \; v_i^{(k)}(t)>\varphi^{(2)}_i\Big\} \rd t \le \delta = \frac{\e_0^2}{100},
\end{align}
where we have used \eqref{93} in the last inequality.

Armed with \eqref{lm1cond}, our goal is to apply Lemma~\ref{lm:energy}
 with $M=M_n$ to  $v=v^{(n,k(n)+1)}$  with the value $k=k(n)$ determined after \eqref{445}. Clearly $v$ 
 is of the form
\be\label{vu}
    v= \lambda^{-2k-2}u + \big[ \ell_n -\lambda^{-2k-2}(\bar u_n + \ell_n)\big],
\ee
i.e. it is a solution to \eqref{ve2} (namely $ \lambda^{-2k-2}u$) shifted by
$ \big[ \lambda_n -\lambda^{-2k-2}(\bar u_n + \ell_n)\big]$. The value $\kappa$ in
 Lemma~\ref{lm:energy}  will be set to
\be\label{kappach}
  \kappa :=\frac{3}{4}\chi
\ee
and the set $\cG^*$ in  Lemma~\ref{lm:energy} will be chosen as $\cG^*:=\cG_{n}$
(for $n=0$ we set $\cG^* = [-3M_0, 0]$, i.e. at the zeroth step of the iteration
every time is ``good'', see \eqref{supbound}).
The choice $\kappa=3\chi/4$ together 
with the constraints on $\chi$ in \eqref{guarant} guarantee that the 
relations in \eqref{kappacond1} hold.
We need to check
 five conditions \eqref{barvbound}, \eqref{defell1/2},  \eqref{numbercontrol}, \eqref{vlinfty}
and \eqref{vlinftygood}. The sixth condition, the regularity at $(Z, t)$ for $t\in \Xi_0$
follows automatically from ${\bf (C1)_\rho}$ 
 since $M=M_n =\cM\nu^n = 2^{-\tau_0} \nu^n K$ with an integer $\tau_0$ and $\nu$ itself is a negative
power of $2$, thus $\Xi_0\subset \Xi$, see \eqref{Kassnew}. 
The first condition \eqref{barvbound} for the shift in \eqref{vu} was verified in \eqref{shiftcheck}.

For the second condition  \eqref{defell1/2}, 
with the notation $\cG^c: = [-3M, 0]\setminus \cG$ we write 
\begin{align}
\frac{1}{M^2}\int_{-2M}^0  \sum_i & (v_i^{(k+1)}(t)-\psi_i)_+^2\rd t \label{splitg}
\\ & \le \frac{1}{M^2}\int_{-2M}^0  {\bf 1}(t\in \cG)\cdot  \sum_i  (v_i^{(k+1)}(t)-\psi_i)_+^2\rd t
  + \frac{|\cG^c|}{M^2} \sup_{t\in [-2M, 0]}  \sum_i  (v_i^{(k+1)}(t)-\psi_i)_+^2.
\nonumber
\end{align}
In the first term we use that
$$
   v_i^{(n,k+1)}(t) \le \ell_n+\wt \psi_i^{(n)}, \qquad t\in \cG_n
$$
from \eqref{uppb} (we reintroduced the superscript $n$). Since $\ell_n+\wt \psi_i^{(n)} \le \psi_i^{(n)}$ 
if $|i-Z|\ge 8M$, we see that the summation in the first term on the right hand side of \eqref{splitg}
is restricted to $|i-Z|\le 8M$,
and for these $i$'s we have $v_i^{(n,k+1)}(t)\le \ell_n$ since $\wt\psi_i^{(n)}=0$.  
We can therefore apply   \eqref{lm1cond} and we get
\be\label{59}
\frac{1}{M_n^2}\int_{-2M_n}^0  \sum_i  (v_i^{(n,k+1)}(t)-\psi^{(n)}_i)_+^2\rd t
  \le 4 \delta\ell_n^2+
 \frac{|\cG_n^c|}{M_n^2} \sup_{t\in [-2M_n, 0]}  \sum_i  (v_i^{(n,k+1)}(t)-\psi^{(n)}_i)_+^2.
\ee
To estimate the second term,
we use \eqref{400} and $\Psi^{(n)}\le \Phi^{(n)}$
to note that 
\be\label{58}
  \psi_i^{(n)} \le v_i^{(n,k+1)}(t) \quad \Longrightarrow \quad  \psi_i^{(n)} \le 
\ell_n + \lambda^{-2k-2}\Big( \Phi_i^{(n)}(t) -\ell_n\Big), \qquad t\in [-3M_n,0].
\ee
Suppose first that $|i-Z|\ge M_n^{1+3\chi/4}$. 
In this case \eqref{tochh21} holds, thus \eqref{58} would imply
$$
   \psi_n^{(n)} = \ell_n \Big( \Big|\frac{i-Z}{M_n}\Big|^{1/2}-1\Big)_+ \le  \lambda^{-2k_0} \ell_n + 
\ell_n\Big[ \Big(   \Big|\frac{i-Z}{M_n}\Big| - 
\lambda^{-4}\Big)^{1/4}  - 1\Big]_+,
$$
but this is impossible for $|i-Z|\ge M_n^{1+3\chi/4}$
if $M_n$ is large enough. In particular, 
this verifies \eqref{numbercontrol}. 
We therefore conclude that the summation in the second term in the right
hand side of \eqref{59} is restricted to $|i-Z|\le M^{1+3\chi/4}$.
For these values we have 
\be\label{vit}
   v_i^{(n,k+1)}(t) \le \ell_n + \lambda^{-2k}\Big( \Phi_i^{(n)}(t) -\ell_n\Big) \le C\lambda^{-2k} \ell_n M_n^{\chi/2}
\ee
(the first inequality is from \eqref{400}, the second is from \eqref{phim}).
This verifies \eqref{vlinfty}, recalling the choice of $\kappa=3\chi/4$.

Inserting these information into \eqref{59}, we have
$$
   4 \delta\ell_n^2+ \frac{|\cG^c_n|}{M^2_n} \sup_{t\in [-2M_n, 0]}  \sum_i  (v_i^{(n, k+1)}(t)-\psi^{(n)}_i)_+^2 
\le  4 \delta\ell_n^2+ C\lambda^{-2k} M_n^{-\frac{1}{2} + 2\chi}\ell_n^2 \le \e_0^2 \ell^2_n,
$$
where we used $|\cG_n^c|\le CM_0^{1/4} \le M_n^{1/2}$ from \eqref{cGmeas} and \eqref{between}.
 In the last step we used the choice $\delta= \e_0^2/100$.
This verifies \eqref{defell1/2}.

Finally,  we verify \eqref{vlinftygood} with the previously mentioned choice $\cG^*:=\cG_{n}$.
Let $i$ such that $|i-Z|\le M_n^{1+3\chi/4}$ and $t\in \cG_n$. Then  from \eqref{400} we have
$$
    v^{(n,k+1)}_i(t) =\lambda^{-2k-2} \Lambda^{(n)}_i +   \ell_n (1-\lambda^{-2k-2})
  \le C\ell_m\le CM_n^{\chi/10}\ell_n
$$
where $m$ is chosen such that $\wh M_{m+1}\le  M_n^{1+3\chi/4}\le\wh M_{m}$ and in the
last step we used \eqref{ellM}.

Thus we can apply  Lemma~\ref{lm:energy} to $v=v^{(n,k+1)}$ and from \eqref{goodtimes} we get the 
existence of a set of times, denoted by $\cG_n'\subset [-M_n, 0] $, such that
$$
      \sup_{ t\in  \cG_{n}'} v^{(n,k+1)}_i(t) \le \frac{\ell_n}{2}, \qquad \forall \; |i-Z|\le M_n
$$
and
$$
     \big| [ -M_{n}, 0]\setminus \cG_n']\big| \le CM_{n}^{1/4}.
$$
Defining $\cG_{n+1}: = \cG_n \cap \cG_n'\cap[-3M_{n+1}, 0]$ and using  $\wh M_{n+1}\le M_n$,
we obtain that
\be\label{vestfin}
     \sup_{t\in \cG_{n+1}}
   v^{(n,k+1)}_i(t) \le \frac{\ell_n}{2}, \qquad \; |i-Z|\le \wh M_{n+1}
\ee
and
\be\label{gn+1}
     \big| \cG_{n+1}^c\big| \le CM_n^{1/4}+ |\cG_n^c|\le C\sum_{m=0}^n M_m^{1/4}
\ee
where we used the measure of $\cG_n^c$ from \eqref{cGmeas}.

Recalling the definition \eqref{vdeff}, from \eqref{vestfin} we have
$$
   u_i(t) - \bar u_n \le \ell_n \big(1 - \frac{1}{2}\lambda^{2(k+1)}\big) 
 \le \ell_n \big(1 - \frac{1}{2}\lambda^{2(k_0+1)}\big) \le \ell_n \zeta =\ell_{n+1}
  \qquad \mbox{$i\in \wh Q_{n+1}$ and $t\in \cG_{n+1}$},
$$
where we recall that $k\le k_0$ and \eqref{zetabound}.
Repeating the argument for $-u$ instead of $u$, we obtain a similar lower bound
on $ u_i(t) - \bar u_n$. Since $Q_{n+1}^*\subset \cG_{n+1}\times [Z-\wh M_{n+1}, Z+\wh M_{n+1}]$,
 and this proves \eqref{oscreduce} for $n$, i.e. {\bf $(OSC)_n$}. 

The application of  Lemma~\ref{lm:energy} also yields (see \eqref{11step}) 
 that
  for  $t  \in [-M_n, 0]$ we have
\be
    \sum_i \Big(v_i^{(n,k+1)}(t)-(\frac{2}{5}\ell_n+\psi_i^{(n)})\Big)_+^2 \le C 
\Big(\frac{ |t|+  \cM^{1/2} }{M_n}\Big)M_n^\chi \ell_n^2,
\ee 
which implies, by \eqref{vdeff} and an elementary algebra, 
the second statement in Proposition~\ref{prop:oscc} (the constant
$C$ in \eqref{l2prop} includes a factor $\lambda^{-2k}\le \lambda^{-2k_0}$).

\subsubsection{Summary of the choice of the parameters}\label{sec:parchoice}

Finally we present a possible choice of the parameters that were used
in  the proof of Proposition~\ref{prop:oscc}. Especially, we need to satisfy 
the complicated relations \eqref{guarant}, \eqref{zetabound}.

 Lemma~\ref{lm:energy}
gives an absolute constant $\e_0$. Then we choose $\delta = \e_0^2/100$,  $\gamma = c\delta^3$,
$\lambda = c\delta^6\mu$ (with a small constant $c$),
 $k_0=100/\gamma$
and $\mu=1/100$. These choices can be found in \eqref{deltachoice}, \eqref{gamlamchoice},
\eqref{k0choice} and \eqref{muchoice}, respectively.

Having $\lambda, k_0$ determined, we define
$$
   \zeta: =1 - \frac{1}{2}\lambda^{2(k_0+1)}, \quad L := \lambda^{-16(k_0+1)}, 
\quad \nu=:2\lambda^{16(k_0+1)}.
$$
 If needed, reduce $\lambda$ so that
$|\log \zeta|/|\log \nu|\le 1/10$.
Note that  five numbers, $\lambda, k_0$, $\zeta, \nu, L$ 
are absolute positive constants (meaning that
they do not depend on any input parameters in Theorem~\ref{thm:caff}).
In particular, they determine
the absolute constant $\fq$ \eqref{c2def}, which is the final H\"older exponent.

Next we set
\be\label{chichoice}
  \chi:= \min\Big\{  \frac{|\log \zeta|}{|\log\nu|}, \frac{\vartheta_1-1}{2},
\frac{1}{2000} \Big\}
\ee
and then choose  the exponents $\xi, \rho$ as
\be\label{xiro}
  \xi := \rho:=  \frac{\chi}{200\vartheta_0}.
\ee
Finally, $\cM= M_0$ (or, equivalently $K_0$) has to be sufficiently large
depending on all these exponents.

It is easy to check that this choice of the parameters satisfies
all the relations  that were used in the proof  of Proposition~\ref{prop:oscc}.
This completes the proof of  Proposition~\ref{prop:oscc}.
 \qed

\subsection{Proof of  {\bf $(ST)_n$} +  {\bf $(OSC)_n$} $\Longrightarrow$ {\bf $(ST)_{n+1}$}}
\label{sec:secondstep}

\begin{proposition}\label{prop:stst}
Suppose that for some $n$ integer   {\bf $(ST)_n$}  and  {\bf $(OSC)_n$} hold. Then 
 {\bf $(ST)_{n+1}$} also holds. 
\end{proposition}

{\it Proof.} 
For $t\in \cG_{n+1} \subset \cG_n$ we have
$$
   |u_i(t) -\bar u_{n+1}| \le \ell_{n+1}, \quad |i-Z|\le M_n
$$
by {\bf $(OSC)_n$}. Since $\wh M_{n+1}\le M_n$ (as $\nu\le\frac{1}{2L}$), 
we immediate get  $|u_i(t) -\bar u_{n+1}|\le \Lambda_i^{(n+1)}$ for
$ |i-Z|\le \wh M_{n+1}$. For $\wh M_{n+1} \le |i-Z|\le \wh M_1$ we just use
$$
    |u_i(t) -\bar u_{n+1}| \le  |u_i(t) -\bar u_{n}| + |\bar u_{n+1}- \bar u_n| 
  \le \Lambda_i^{(n)} +  |\bar u_{n+1}- \bar u_n|  \le  \Lambda_i^{(n+1)}
$$
where the last estimate is from the definition of $\Lambda$.
For $|i-Z|\ge \wh M_1$ we have the trivial bound $\ell_0$. 

Now we need to check the case $t\in [-3M_{n+1}, 0]\setminus \cG_{n+1}$. For  $\wh M_{n+1}\le |i-Z|\le \wh M_{1}$, 
from {\bf $(ST)_n$}
we have
$$
   |u_i(t) -\bar u_{n+1}| \le |u_i(t) -\bar u_{n}| + |\bar u_{n+1}- \bar u_n| 
  \le  \Phi^{(n)}_i (t) +|\bar u_{n+1}- \bar u_n| \le  \Phi^{(n+1)}_i (t),
$$
where the last  inequality is just from the definition of $\Phi$.
Finally, if $|i-Z|\le \wh M_{n+1} (\le M_n)$, we use  \eqref{l2prop}
$$
   |u_i(t) -\bar u_{n+1}|\le  |u_i(t)-\bar u_n| +| \bar u_n -\bar u_{n+1}|
  \le 2\ell_n + \ell_n  \sqrt{C\frac{|t|+  \cM^{1/2} }{M_{n}}} M_{n}^{\chi/2} 
$$
since in this regime $\psi_i^{(n)}=0$. The constant $C$ is from \eqref{l2prop}. The right hand side is bounded by
$$
 C_\Phi
 \ell_{n+1} \Big(1+\sqrt{\frac{|t|+  \cM^{1/2} }{M_{n+1}}} M_{n+1}^{\chi/2} \Big),
$$
by using that $\ell_n/\ell_{n+1}=\zeta^{-1} \le \nu^{-1/10} = (M_n/M_{n+1})^{1/10}$
and choosing $C_\Phi$ large enough. This completes the proof of Proposition~\ref{prop:stst}. \qed

\subsection{Proof of  Lemma~\ref{lm:energy} (first De Giorgi lemma)}\label{sec:firstdg}

 Assume for notational simplicity that $Z=0$ and we set
$$
   \psi_i^\ell: = \psi_i+\ell.
$$
Using that $\bv$ solves the equation
\be\label{ve2uj}
  \pt_s v_i(s) = -\big[\cA(s)\bv(s)\big]_i - W_i(s)\bar u,
\ee
by direct computation
we have
\be\label{heat}
  \pt_t  \frac{1}{2}
 \sum_i [v_i-\psi^\ell_i]_+^2 = - \sum_{ij} (v_i-\psi^\ell_i)_+  B_{ij} (v_i-v_j)
 - \sum_i (v_i-\psi^\ell_i)_+ W_i (v_i+\bar u). 
\ee
 Recall that $B_{ij}$ depends on time, but we will omit this from the notation.
Since $W_i  \ge 0$, the last term can be bounded by 
\begin{align*}
-\sum_i (v_i-\psi^\ell_i)_+ W_i (v_i+\bar u) &  \le -\sum_i (v_i-\psi^\ell_i)_+ W_i (v_i-\psi^\ell_i)_+
  -\bar u\sum_i (v_i-\psi^\ell_i)_+ W_i \\
&  \le -\fw[ (v-\psi^\ell)_+, (v-\psi^\ell)_+] + |\bar u|\sum_i (v_i-\psi^\ell_i)_+ W_i .
\end{align*}
 In the first term on the right hand side of \eqref{heat} we can symmetrize and then 
add and subtract  $\psi^\ell$ to $v$ we get
\begin{align*}
  - \sum_{ij} (v_i-\psi^\ell_i)_+  B_{ij} (v_i-v_j) = & - \fb[ ( v-\psi^\ell)_+, v] \\
  = & - \fb[( v-\psi^\ell)_+, (v-\psi^\ell)_+] - \fb[( v-\psi^\ell)_+, (v-\psi^\ell)_-] - 
 \fb[ (v-\psi^\ell)_+, \psi^\ell].
\end{align*}
Since $B_{ij}\ge 0$ and $  [  a_+ - b_+ ] [  a_- - b_- ] \ge 0$ for any real numbers $a, b$,
for the cross-term we have $ \fb[( v-\psi^\ell)_+, (v-\psi^\ell)_-]\ge 0$. 
Thus the last equation is bounded by 
\be\label{der10}
  \le  - \fb[( v-\psi^\ell)_+, (v-\psi^\ell)_+]  - \fb[ (v-\psi^\ell)_+, \psi^\ell].
\ee
Using the definition of $\fa$ \eqref{Adef}, we have thus proved that 
\be\label{enee}
  \pt_t \frac{1}{2}\sum_i [v_i-\psi^\ell_i]_+^2 \le - \fa [( v-\psi^\ell)_+, ( v-\psi^\ell)_+]- 
\fb[ (v-\psi^\ell)_+, \psi^\ell]  + |\bar u|\sum_i (v_i-\psi^\ell_i)_+ W_i .
\ee

Decompose  the first error term into 
$$
   \fb[(v-\psi^\ell)_+, \psi^\ell] = \Omega_1 +\Omega_2 +\Omega_3, 
$$
$$
  \Omega_1:=\frac{1}{2}\sum_{|i-j|\ge M} B_{ij}[\psi^\ell_i-\psi^\ell_j]
 \big( (v_i-\psi^\ell_i)_+ - (v_j-\psi^\ell_j)_+\big)\cdot {\bf 1}(  \max\{ d_i^I, d_j^I\}\ge K/3)
$$
and $\Omega_2$ and $\Omega_3$ are defined in the same way except that the summation 
is restricted to $\wh C K^\xi\le |i-j|\le M$ for $\Omega_2$ and $|i-j|\le \wh C K^\xi$ for $\Omega_3$,
 where $\wh C$ is the constant from   \eqref{far1new}. 
Notice that we inserted the characteristic function  ${\bf 1}(  \max\{ d_i^I, d_j^I\}\ge K/3)$
for free,  since \eqref{numbercontrol} together with $|Z|\le K/2$ and $M^{1+\kappa}\ll K$
(from $\vartheta\ge 1+2\kappa$) guarantees that
$(v_i-\psi^\ell_i)_+=0$ unless $d_i^I\ge K/3$. Thus
 the summation over $i,j$ can be restricted to index pairs, where
at least one of them is far away from the boundary. 
Recall from \eqref{far1new} that in the regime $ |i-j|\ge M$ 
  we have  $   B_{ij} \le C|i-j|^{-2}$ since $M\ge \wh C K^\xi$.  Moreover, 
we have 
\be\label{psibb}
  |\psi^\ell_i-\psi^\ell_j|
 \le \ell  M^{-1/2}|i-j|^{1/2}. 
\ee 
Altogether we have 
$$
  |\Omega_1|\le \ell M^{-1/2}\sum_{|i-j|\ge M} \frac{1}{|i-j|^{3/2}}
  \big[ (v_i-\psi^\ell_i)_+ + (v_j-\psi^\ell_j)_+\big]\le \frac{\ell}{M}\sum_i (v_i-\psi^\ell_i)_+.
$$

For $\Omega_2$,  by symmetry of $B_{ij}$, we can rewrite it as  
\begin{align*}
   -\Omega_2: & =-
 \sum_{\wh CK^\xi\le |i-j|\le M, \psi^\ell_i \le \psi^\ell_j  }  B_{ij} [\psi^\ell_i-\psi^\ell_j]
 \big( (v_i-\psi^\ell_i)_+ - (v_j-\psi^\ell_j)_+\big)  \\
 & \le  - \sum_{\wh C K^\xi\le |i-j|\le M, \psi^\ell_i \le \psi^\ell_j  }  B_{ij} [ \psi^\ell_i-\psi^\ell_j ] \;
 \big [ (v_i-\psi^\ell_i)_+ - (v_j-\psi^\ell_j)_+\big ] \cdot 
{\bf 1}( v_i-\psi^\ell_i>0) \\
  &  \le   \frac{1}{4} \sum_{\wh C K^\xi\le |i-j|\le M} B_{ij}
 \big[ (v_i-\psi^\ell_i)_+ - (v_j-\psi^\ell_j)_+\big]^2 \\
&  \quad   +  4 \sum_{\wh C K^\xi \le |i-j|\le M} B_{ij} |\psi^\ell_i-\psi^\ell_j|^2
\cdot {\bf 1}( v_i-\psi^\ell_i>0).
\end{align*}
\nc 
The first term is bounded by $\frac{1}{2}\fb[ (v-\psi^\ell)_+, (v-\psi^\ell)_+]$ and can be absorbed in
the first term on the r.h.s. of \eqref{der10}. 
By the simple estimate $|\psi_i^\ell -\psi_j^\ell | \le C\ell |i-j|/M$ 
and \eqref{far1new}, the second term is bounded by
\be\label{2}
  4\sum_{\wh C K^\xi\le |i-j|\le M} B_{ij} |\psi^\ell_i-\psi^\ell_j|^2
 \cdot {\bf 1}( v_i-\psi^\ell_i>0) \le C \ell^2M^{-1} \sum_i {\bf 1}( v_i-\psi^\ell_i>0),
\ee
where we again used that the summation over $i$ is restricted to $d_i^I\ge K/3$.
Thus 
\begin{align*}
  -\Omega_2\le & \frac{1}{2}\fb[ (v-\psi^\ell)_+, (v-\psi^\ell)_+]
 + C \ell^2M^{-1} \sum_i {\bf 1}( v_i-\psi^\ell_i>0) \\
\le &   \frac{1}{2}\fa[ (v-\psi^\ell)_+, (v-\psi^\ell)_+]
 + C \ell^2M^{-1} \sum_i {\bf 1}( v_i-\psi^\ell_i>0)
\end{align*}
 using that $\fb \le \fa$.

A similar estimate is performed for $\Omega_3$, but in the corresponding last term we use
that 
$$
 |\psi^\ell_i-\psi^\ell_j|\le C K^{\xi} (\ell/M) 
$$
for $|i-j|\le \wh C K^\xi$. 
Thus we have 
\begin{align}
   -\Omega_3 & \le 
  \sum_{ |i-j|\le \wh C K^\xi} B_{ij}   |\psi^\ell_i-\psi^\ell_j|^2
 \cdot {\bf 1}( v_i-\psi^\ell_i>0) \nonumber \\
 & \le C K^{2\xi}(\ell/M)^2 \sum_{ |i-j|\le \wh C K^\xi}   {\bf 1}( v_i-\psi^\ell_i>0)  B_{ij}  \nonumber \\
& 
  \le C  \frac{K^{3\xi}\ell^2}{M^2} \sum_{i}   {\bf 1}( v_i-\psi^\ell_i>0) [ B_{i, i+1}+ B_{i, i-1}]
\label{tosep}
\end{align}
Here we just overestimated sums by $\wh C K^\xi$.
The conclusion of the energy estimate is 
\begin{align}
 \pt_t \frac{1}{2}\sum_i [v_i-\psi^\ell_i]_+^2 
  \le &  -  \frac{1}{2}\fa [( v-\psi^\ell)_+, (v-\psi^\ell)_+]  + |\bar u| \sum_i (v_i-\psi^\ell_i)_+ W_i \nonumber \\
& 
+ \frac{C\ell}{M}\sum_i (v_i-\psi^\ell_i)_+ 
+ \frac{C\ell^2}{M}\sum_i {\bf 1}( v_i-\psi^\ell_i>0)  + \Omega_4 ,
\label{ener}
\end{align}
\be\label{o4}
\Omega_4: = 
\frac{CK^{3\xi}\ell^2}{M^2} \sum_{i}   {\bf 1}( v_i-\psi^\ell_i>0) [ B_{i, i+1}+ B_{i, i-1}] .
\ee
Due to \eqref{numbercontrol}, we can assume that the summations in \eqref{ener} over $i$ are 
restricted to $ |i | \le M^{1 + \kappa}$.  In this regime we have $d_i\ge cK$  thanks to $M^{1+\kappa}\le K/2$,
therefore  $W_i\le CK^{-1+\xi}$ by \eqref{g5new}. Using the
bound \eqref{barvbound}, we see that the error term $ |\bar u| \sum_i (v_i-\psi^\ell_i)_+ W_i$ can
be absorbed into the first error term in line \eqref{ener}.

Let $T_k := -M(1+ 2^{-k})$, $\ell_k:= \frac{\ell}{3}(1-2^{-k}) \nearrow \frac{ \ell}{3}$
 where $k=1,2,\ldots C\log M$.
We claim that
\be\label{Om4}
\int_{\tau}^t  \, \Omega_4 \rd s \le\frac{ C K^{3\xi}\ell^2}{  M^{1-\kappa}}   \int_{\tau}^t \rd s 
\frac{1}{M^{1+\kappa}} \sum_{ |i| \le M^{1 + \kappa} } 
  [ B_{i, i+1}+ B_{i, i-1}] (s) \le  C[ (t - \tau) +1]  K^{3\xi + \rho}\ell^2 M^{ \kappa-1}
\ee
for any integer $k\le C\log M$ and for any pairs
$(t,\tau) \in [T_k, 0]\times  [T_{k-1}, T_{k-1}+ 2^{-k-1}M]$.
The estimate \eqref{Om4} holds because 
$$
  \int^t_\tau \Big[ \ldots \Big] \rd s \le \int^t_{T_{k-1}}  \Big[ \ldots \Big] \rd s
 \le 8 \big| t -\tau \big|+1,
$$
where we used that the point $(T_{k-1}, Z=0)$ is regular, see \eqref{Tau0}.

Define
\be\label{defUk}
   U_k = \sup_{t\in [T_k, 0]} \frac{1}{M\ell_k^2}\sum_i (v_i-\psi^{\ell_k}_i)_+^2 (t) 
  + \frac{1}{M\ell_k^2}\int_{T_k}^0 \fa [( v-\psi^{\ell_k})_+, (v-\psi^{\ell_k})_+](s)\rd s.
\ee
Integrating \eqref{ener} from $\tau$ to $t$ with $\tau \in  [T_{k-1}, T_{k-1}+ 2^{-k-1}M]=  [T_{k-1}, T_{k}- 2^{-k-1}M] $
and $t\in [T_k, 0]$,
 we have from \eqref{Om4}
\begin{align}
   \sum_i [v_i-\psi^{\ell_k}_i]_+^2(t) & + \int_\tau^t 
\fa[( v-\psi^{\ell_k})_+, (v-\psi^{\ell_k})_+](s) \rd s   \nonumber \\
 \le  &  \sum_i [v_i-\psi^{\ell_k}_i]_+^2(\tau) 
+ C \int_\tau^t 
\Bigg [\frac{\ell_k}{M} \sum_i  (v_i-\psi^{\ell_k}_i)_+ (s)
+  \frac{\ell^2_k}{M}\sum_i {\bf 1}( v_i-\psi^{\ell_k}_i>0)(s)\Bigg]\rd s \nonumber \\
& +  C  [(t - \tau) +1 ]  K^{3\xi + \rho}  \ell^2 M^{ \kappa-1}  .
\label{e5}
\end{align}
Taking the average over $\tau\in [T_{k-1}, T_{k-1}+ 2^{-k-1}M]$
and using that in this regime $  2^{-k-1}M \le t- \tau \le M $,  we have
\begin{align*}
&  \sum_i [v_i-\psi^{\ell_k}_i]_+^2(t)  + \int_{T_{k}}^t 
\fa [( v-\psi^{\ell_k})_+, (v-\psi^{\ell_k})_+](s) \rd s    \\
 \le &  \; C\frac{2^{k+1}}{M}\int_{T_{k-1}}^{T_k- 2^{-k-1}M} \sum_i [v_i-\psi^{\ell_k}_i]_+^2(s) \rd  s \\
& + C \int_{T_{k-1}}^t 
\Bigg [\frac{\ell_k}{M}\sum_i (v_i-\psi^{\ell_k}_i)_+ (s)
+  \frac{\ell^2_k}{M}\sum_i {\bf 1}( v_i-\psi^{\ell_k}_i>0)(s)\Bigg]\rd s  
+ C    K^{3\xi + \rho}  \ell^2 M^{ \kappa}.
\end{align*}  
Dividing through by $M\ell_k^2$ and taking the  supremum
over $t\in [T_k, 0]$, for $k \ge 1$ we have
\begin{align}\label{averag}
    U_k \le  & C\frac{2^{k+1}}{M^2}  \int_{T_{k-1}}^0 \sum_i \Big[ \frac{1}{\ell_k^2}[v_i-\psi^{\ell_k}_i]_+^2 + \frac{1}{\ell_k }  (v_i-\psi^{\ell_k}_i)_+ 
  + {\bf 1}( v_i-\psi^{\ell_k}_i>0) \Big](s)\rd s
\nonumber \\
&  + M^{ \kappa }   \frac { C K^{3\xi + \rho}} M. 
\end{align}

The first three integrands have the same scaling dimensions as $v^2/\ell^2$.
One key idea is to estimate them 
in terms of the $L^4$-norm of $v$ and then using the Sobolev inequality. 
It is elementary to check  these three integrands can be bounded by the
$L^4$-norm of  $(v-\psi^{\ell_k})_+$,  by using  that
if $v_i\ge \psi^{\ell_k}_i$, then $v_i-\psi^{\ell_{k-1}}_i\ge \ell_k-\ell_{k-1} = 2^{-k}\frac{\ell}{3}
\ge 2^{-(k+2)}\ell$: 
\begin{align}\label{norms}
    \sum_i  (v_i-\psi^{\ell_k}_i)_+ \le  &  \sum_i (v_i-\psi^{\ell_k}_i)_+ 
\cdot {\bf 1}( v_i-\psi^{\ell_{k-1}}_i>2^{-(k+2)}\ell)
\\ \non  \le & (2^{k+1})^3\ell^{-3}_k \sum_i  (v_i-\psi^{\ell_{k-1}}_i)_+^4, 
\\ \non 
  \sum_i {\bf 1}( v_i-\psi^{\ell_k}_i>0) \le &   (2^{k+2})^4\ell_k^{-4} \sum_i  (v_i-\psi^{\ell_{k-1}}_i)_+^4, 
\\  \non
\sum_i   [v_i-\psi^{\ell_k}_i]_+^2  \le &   (2^{k+2})^2\ell_k^{-2} \sum_i  (v_i-\psi^{\ell_{k-1}}_i)_+^4.
\end{align} 
We now use the local version of
 Proposition~\ref{prop:newGNdiscr} from Appendix~\ref{sec:GN}; we first verify its conditions.
Set  
\be\label{cI}
\cI:= \llbracket - 2K/3, 2K/3\rrbracket, \qquad 
\wh\cI:= \llbracket - 3K/4, 3K/4\rrbracket.
\ee
 Clearly  
 $f_i:=  (v_i-\psi^{\ell_{k-1}}_i)_+$ is supported in $\cI$; this follows from  $|Z|\le K/2$,
\eqref{numbercontrol} and that $M^{1+\kappa}\le M^{(\vartheta+1)/2}\ll M^\vartheta=K$.
By the lower bounds  on $B_{ij}(s)$ in \eqref{g3new} and  \eqref{far1new}
 (with $C\ge 4$ in \eqref{far1new} to guarantee that the lower bound holds for any $i,j\in \wh\cI$) 
the conditions \eqref{displ1loc}, \eqref{displ2loc} hold with the choice $b=K^{-\xi}$, $a= \wh C^{-1} K^{-\xi}$
and $r=C$, where $C$ and $\wh C$ are constants from \eqref{far1new}.  

From \eqref{GNdiscrloc} we then have 
\begin{align}\nonumber
  \sum_i  (v_i-\psi^{\ell_{k-1}}_i)_+^4 \le \; & C \Big[ \sum_i  (v_i-\psi^{\ell_{k-1}}_i)_+^2 \Big]
 \Big[  \fa[ (v-\psi^{\ell_{k-1}})_+,  (v-\psi^{\ell_{k-1}})_+] + 
\frac{1}{K} \sum_i  (v_i-\psi^{\ell_{k-1}}_i)_+^2 \Big]
\\ & + CK^{4\xi}  \max_i (v_i-\psi^{\ell_{k-1}}_i)_+^4
\label{GNappl}
\end{align}
(we omitted the time variable $s\in \cT$).
The last term can be estimated by using \eqref{numbercontrol} and \eqref{vlinfty} as
$$
   \max_i (v_i(t)-\psi^{\ell_{k-1}}_i)_+ \le   \max\big\{ v_i(t) \; : \; 
|i-Z|\le M^{1+\kappa}\big\} \le  C\ell M^{\chi/2}
$$
for any $t\in [-2M, 0]$. For $t\in \cG^*$ we have the stronger bound from
\eqref{vlinftygood}
$$
   \max_i (v_i(t)-\psi^{\ell_{k-1}}_i)_+ \le   \max\big\{ v_i(t) \; : \; 
|i-Z|\le M^{1+\kappa}\big\} \le  C\ell M^{\chi/8}, \qquad t\in \cG^*.
$$
Inserting these estimates, \eqref{norms} and \eqref{GNappl}  into \eqref{averag}, 
splitting the time integration into $\cG^*$ and its complement, we have proved that for $k \ge 2$  
\begin{align}
  U_k
  \le &  C (2^{k+2})^5   \frac{1}{M^2  \ell_k^{4} }
\int_{T_{k-1}}^0 \rd s
 \Big[ \sum_i  (v_i-\psi^{\ell_{k-1}}_i)_+^2(s)\Big]
\\
 & \qquad \qquad\qquad \qquad \times \Big[ 
  \fa[ (v-\psi^{\ell_{k-1}})_+,  (v-\psi^{\ell_{k-1}})_+](s) + 
 \frac{1}{K} \sum_i  (v_i-\psi^{\ell_{k-1}}_i)_+^2(s)\Big]  \nonumber \\
+  &  \frac{1}{M}\Big[  CM^{ \kappa }   K^{3\xi + \rho} +
  32^k M^{\chi/2} K^{4\xi}   +  32^k M^{2\chi-1} K^{4\xi}  \big| [-T_{k-1}, 0]\setminus \cG^*\big| \Big]
 \nonumber \\
 \le & \; 32^k \big[  C_1  U_{k-1}^2 +  M^{-1 +  \chi}K^{-\rho}\big] , 
\label{nl}
\end{align}
 recalling that $\big| [-T_{k-1}, 0]\setminus \cG^*\big| \le CM^{1/4}$, 
$K=M^\vartheta\le M^{\vartheta_0}$ and $\chi\ge \kappa + 10(\xi + \rho) \vartheta_0$.  
We also used that $|T_k| \le K$.

For $k=1$,  we estimate the integrands in \eqref{averag} by
$L^2$-norms.   We have the following general estimates for any $\ell' <\ell''$ 
\begin{align}\label{l2est}
  \sum_i  (v_i-\psi^{\ell''}_i)_+ \le & \sum_i (v_i-\psi^{\ell'}_i)_+ \cdot {\bf 1}( v_i-\psi^{\ell'}_i> \ell''-\ell' )
  \le  \frac{1}{\ell''-\ell'} \sum_i  (v_i-\psi^{\ell'}_i)_+^2\\ \nonumber
  \sum_i {\bf 1}( v_i-\psi^{\ell''}_i>0) \le &  \frac{1}{(\ell''-\ell')^2} \sum_i  (v_i-\psi^{\ell'}_i)_+^2.
\end{align}
We use \eqref{l2est} with $\ell'' =\ell_1$ and $\ell'=0$ in 
 \eqref{averag}, this implies  that  
$$
  U_1 \le  \frac{C}{\ell_1^2M^2}\int_{-2M}^0\sum_i \rd s  (v_i-\psi_i)_+^2(s) + CM^{ - 1 + \chi}  . 
$$
Without loss of generality, we assume that $C_1 \ge 2$,
where $C_1$ is the constant in \eqref{nl}.
Now choose the universal constant $\e_0$ in \eqref{defell1/2} so small and $M$ big enough so  that 
this last inequality implies
\be\label{Uk1}
  U_1\le \frac{1}{ 32^6 C_1}.
\ee
Choose $k_*$ such that $32^{k_*+2}C_1 = K^{\rho}$, i.e. $k^*$ is of 
order  $\rho \log K\ge\rho \log M$.
Then from \eqref{nl} for any $k\le k_*$ we have the recursive inequality
$$
   B_k \le B_{k-1}^2 + M^{-1+\chi}, \quad \mbox{with} \quad B_k := 32^{k+2} C_1 U_k.
$$
By a simple induction, this recursion implies 
$$    
B_{k+1} \le (2B_1)^{2^k-1} + 2 M^{-1+\chi}.
$$
Together with the initial estimate \eqref{Uk1} we obtain
 that $B_{k+1} \le 4M^{-1+\chi}$ for any integer $k$ with $100\log \log M\le k\le k_*$, in particular we can
apply it to $k'=100\log \log M $ and obtain
$U_{k'} \le C   M^{-1+\chi}$. 
Notice that $U_k$ is decreasing in $k$ as it 
 can be seen from the monotonicity in the definition
of $U_k$ \eqref{defUk} and from the fact that $T_k$ and $\ell_k$ increase.  Thus
\be
\label{Ukb}
   U_{k} \le C   M^{-1+\chi}
\ee
for any $k\ge 100\log \log M$.
Taking $k\to\infty$, we  find from the $L^2$-norm term in $U_k$ 
that \eqref{1step}  in Lemma~\ref{lm:energy} holds.

For the proof of  \eqref{goodtimes}, we notice that the estimate \eqref{Ukb} 
together with the monotonicity
also implies that
$$
 \frac{1}{M\ell^2}\int_{-M}^0 \fa [( v-\psi^{\ell/3})_+, (v-\psi^{\ell/3})_+](s)\rd s \le CM^{-1 + \chi}
$$
from the dissipation term in the definition of $U_k$. 

Set
$$
   \cG: =\Big\{ t\in [-M, 0]\; : \;  \fa [( v-\psi^{\ell/3})_+, (v-\psi^{\ell/3})_+](t) \le M^{\chi-1/4}\ell^2 \Big\}
$$
then clearly 
$$
   \big| [-M, 0]\setminus \cG\big|\le CM^{1/4}.
$$

We now use a Sobolev inequality  \eqref{s} from  Appendix~\ref{sec:GN}, with the choice of
$p=4$, $s=1$ and $f_i:=  (v_i-\psi^{\ell/3})_+$. We recall the definitions of $\cI$ and $\wh\cI$
from \eqref{cI}
and that  $f_i=  (v_i-\psi^{\ell/3})_+$ is supported in $\cI$ by \eqref{numbercontrol}.
Thus 
$$
  \sum_i f_i^4 \le C \sum_i f_i^2 \Big[ \sum_{i\ne j\in \wh \cI} \frac{|f_i-f_j|^2}{|i-j|^2}
  + 2\sum_{i\in \cI} |f_i|^2 \sum_{j\not\in \wh\cI}  \frac{1}{|i-j|^2} \Big]
 \le CK^{2\xi} \| f\|^2 \fa [f,f] +\frac{C}{K}\|f\|_2^4,
$$
where we used the lower bound on $B_{ij}$ in
\eqref{g3new}. Thus
\be\label{1sob}
 \sum_i  (v_i-\psi^{\ell/3})_+^4 \le C K^{2\xi} \sum_i  (v_i-\psi^{\ell/3}_i)_+^2
  \fa[ (v-\psi^{\ell/3})_+,  (v-\psi^{\ell/3})_+] + \frac{C}{K} \Big[\sum_i  (v_i-\psi^{\ell/3}_i)_+^2\Big]^2.
\ee
This 
 implies that for any $t\in \cG$ and any $i$
\begin{align}\label{sobappl}
    (v_i(t) - \psi^{\ell/3}_i)_+ & \le \| (v(t)-\psi^{\ell/3})_+\|_4 
\\ \nonumber 
& \le CK^{\xi/2}\Big( \sum_i  (v_i(t)
 - \psi^{\ell/3}_i)_+^2\Big)^{1/4} \big( M^{\chi-1/4}\ell^2\big)^{1/4} +  \frac{C}{K^{1/4}}
  \Big[\sum_i  (v_i(t)-\psi^{\ell/3}_i)_+^2\Big]^{1/2}
\\ \nonumber
&
  \le C M^{-1/20} \ell,
\end{align}
where we used \eqref{1step} in the last step and the fact that $\chi\ge 10\xi\vartheta_0$
together with \eqref{kappacond1}. 
  This proves  \eqref{goodtimes}.

For the proof of \eqref{11step}, we first notice that it is sufficient
to consider the case when $\wt M$ is of the form
$\wt M= 2^{-m} M$, $m=1,2\ldots C\log M$. 
We now repeat the proof of \eqref{1step}  but with $\ell_k, k \ge 1$,  replaced by 
\be\label{lt}
\wh \ell_k= {\frac{2 \ell}{5}}(1-2^{-k-2}) 
\ee
in the definition of $\psi^{\ell_k}$ and working in the time interval of scale $\wt M$.

Set $\wh T_k:= -\wt M(1+ 2^{-k})$. Define
$$
  \wh U_k = \sup_{t\in [\wh T_k, 0]} \frac{1}{M\wh\ell_k^2}\sum_i (v_i-\psi^{\wh \ell_k}_i)_+^2 (t) 
  + \frac{1}{M\wh \ell_k^2}\int_{\wh T_k}^0 \fa [( v-\psi^{\wh\ell_k})_+, (v-\psi^{\wh \ell_k})_+](s)\rd s.
$$
The previous proof is unchanged up to \eqref{Om4}, the integral of
$$
  \wh \Omega_4(s) : = \frac{CK^{3\xi}\wh\ell^2}{M^2} \sum_{i} 
  {\bf 1}( v_i(s)-\psi^{\wh \ell}_i>0) [ B_{i, i+1}(s)+ B_{i, i-1}(s)]
$$
is still estimated by (cf. \eqref{Om4})
$$
  \int_\tau^t\wh\Om_4(s)  \rd s 
  \le C[(t-\tau) +1]K^{3\xi+\rho}\wh\ell^2 M^{-1+\kappa}\le C[(t-\tau)+1]\wh\ell^2 K^{3\xi+\rho} M^{-1+\kappa}
$$
for $\tau \in [\wh T_{k-1}, \wh T_{k-1}+ 2^{-k-1}\wt M]=  [\wh T_{k-1}, \wh
T_{k}- 2^{-k-1}\wt M] $ and $t\in [\wh T_k, 0]$. Here we used \eqref{kappacond1}.

Similarly to \eqref{e5}, we integrate \eqref{ener} (with $\wh\ell$ replacing $\ell$)
from $\tau$ to $t$ 
\begin{align}
   \sum_i [v_i-\psi^{\wh\ell_k}_i]_+^2(t) & + \int_\tau^t 
\fa[( v-\psi^{\wh\ell_k})_+, (v-\psi^{\wh\ell_k})_+](s) \rd s   \nonumber \\
 \le  &  \sum_i [v_i-\psi^{\wh\ell_k}_i]_+^2(\tau) 
+ C \int_\tau^t 
\Bigg [\frac{\ell_k}{M} \sum_i  (v_i-\psi^{\wh\ell_k}_i)_+ (s)
+  \frac{\wh\ell^2_k}{M}\sum_i {\bf 1}( v_i-\psi^{\wh\ell_k}_i>0)(s)\Bigg]\rd s \nonumber \\
& +  C  [(t - \tau) +1]   \ell^2 K^{3\xi+\rho} M^{-1+\kappa}.
\label{e51}
\end{align}
Taking the average over $\tau\in [\wh T_{k-1}, \wh T_{k-1}+ 2^{-k-1}\wt M]=  [\wh T_{k-1}, \wh
T_{k}- 2^{-k-1}\wt M] $
and using that in this regime $  2^{-k-1}\wt M \le t-\tau \le \wt M$,  we have
\begin{align*}
&  \sum_i [v_i-\psi^{\wh\ell_k}_i]_+^2(t)  + \int_{\wh T_{k}}^t 
\fa [( v-\psi^{\wh\ell_k})_+, (v-\psi^{\wh\ell_k})_+](s) \rd s    \\
 \le &  \; C\frac{2^{k+1}}{\wt M}\int_{\wh T_{k-1}}^{\wh T_k- 2^{-k-1}\wt M} 
\sum_i [v_i-\psi^{\wh\ell_k}_i]_+^2(s) \rd  s \\
& + C \int_{\wh T_{k-1}}^t 
\Bigg [\frac{\wh \ell_k}{M}\sum_i (v_i-\psi^{\wh \ell_k}_i)_+ (s)
+  \frac{\wh \ell^2_k}{M}\sum_i {\bf 1}( v_i-\psi^{\wh \ell_k}_i>0)(s)\Bigg]\rd s  
+ C   \ell^2 \wt M K^{3\xi+\rho} M^{-1+\kappa}.
\end{align*}  
Dividing through by $M\wh \ell_k^2$ and taking supremum
over $t\in [\wh T_k, 0]$, for $k \ge 1$ and using $\wt M\le M$, we have,
as in \eqref{averag},
\begin{align}\label{averag1}
    \wh U_k \le  & C\frac{2^{k+1}}{M\wt M}  \int_{\wh T_{k-1}}^0 \sum_i 
\Big[ \frac{1}{\wh \ell_k^2}[v_i-\psi^{\wh \ell_k}_i]_+^2 + \frac{1}{\wh \ell_k } 
 (v_i-\psi^{\wh \ell_k}_i)_+ 
  + {\bf 1}( v_i-\psi^{\wh \ell_k}_i>0) \Big](s)\rd s
\nonumber \\
&  +C   \wt M K^{3\xi+\rho} M^{-2+ \kappa} .
\end{align}

Using the bounds \eqref{norms} and Proposition~\ref{prop:newGNdiscr} as
in \eqref{GNappl}--\eqref{nl}, 
instead of \eqref{nl} we get
\begin{align}
  \wh U_k
  \le &  \frac{ (2^{k+2})^5}{M \wt M  \wh \ell_k^{4} }
\int_{\wh T_{k-1}}^0 \rd s
 \Big[ \sum_i  (v_i-\psi^{\wh \ell_{k-1}}_i)_+^2(s)\Big]
  \fa[ (v-\psi^{\wh \ell_{k-1}})_+,  (v-\psi^{\wh \ell_{k-1}})_+](s) \nonumber \\
 &  
 +C  \wt M M^{-2+ \chi} K^{-\rho} 
 \nonumber \\
 \le &  32^k \Big[ C_1 \frac{M}{\wt M} \wh U_{k-1}^2 + C \frac{\wt M}{M}M^{-1 + \chi} K^{-\rho}\Big] , \quad k \ge 2. 
\label{nl3}
\end{align} 
Similarly to the proof of \eqref{Ukb}, this new recurrence inequality  has the solution
\be
  \wh U_{k} \le  C \wt MM^{-2 + \chi} 
\label{newU}
\ee
for any sufficiently large $k$, as long as the recursion can be started, i.e.
if we knew
\be\label{u1}
  \wh U_1 \ll \frac{\wt M}{M}.
\ee
For $k=1$ the estimate \eqref{averag1}  together with \eqref{l2est}
(with $\wh\ell_1$ replacing $\ell_1$) becomes
\begin{align*}
    \wh U_1 \le & \frac{C}{M\wt M}  \int_{-2\wt M }^0 \sum_i 
\Big[ \frac{1}{\wh \ell_1^2}[v_i-\psi^{\wh \ell_1}_i]_+^2 + \frac{1}{\wh \ell_1 } 
 (v_i-\psi^{\wh \ell_1}_i)_+ 
  + {\bf 1}( v_i-\psi^{\wh \ell_1}_i>0) \Big](s)\rd s
  +C   \wt M M^{-2+ \chi} \\ 
\le & \frac{C}{M\wt M}  \int_{-2\wt M }^0 \sum_i 
 \frac{1}{\ell^2}[v_i(s)-\psi^{\ell/3}_i]_+^2 
\rd s
  +C   \wt M M^{-2+ \chi} \\
\le & \frac{CM^\chi}{M}  +C   \wt M M^{-2+ \chi} .
\end{align*}
In the second step we used \eqref{l2est} with $\ell'' = \wh\ell_1$ and $\ell'=\ell/3$
noting that $\wh\ell_1 = \frac{7}{20}\ell > \frac{1}{3}\ell$. 
In the last step we used \eqref{1step} and $2\wt M\le M$. 
Thus \eqref{u1} is satisfied if $\wt M\gg M^\chi$.

Finally, taking $k\to\infty$ in \eqref{newU} implies \eqref{11step}.
This completes the proof of Lemma~\ref{lm:energy}.

\subsection{Proof of Lemma~\ref{lm:2nd} (second De Giorgi lemma)}\label{sec:seconddg}

 Set $Z=0$ for simplicity. Since the statement is stronger if $\mu$ and $\delta$ are reduced,
we can assume that they are small positive numbers, e.g. 
we can assume $\mu,\delta<1/8$. 
We are looking for a sufficiently small $\lambda$ so that there will be a positive $\gamma$
with the stated properties.
 The key ingredient of the proof is an energy inequality \eqref{goodterm} including a new dissipation term 
which was dropped in the proof of Lemma~\ref{lm:energy}.
 Most of this section  closely follows the argument in \cite{C}; the
main change 
  is that we need split time integrations into ``good'' and ``bad'' times.
The argument \cite{C} applies to the good times. The bad times
have a small measure, so their contribution is negligible.

\subsubsection{Dissipation with the good term}

Let $-3M\le T_1< T_2 <0$.  For any $t\in [-3M, 0]$, define
\be\label{thetade}
\theta_i(t): = {\bf 1}( |i| \le 9 M) \cdot {\bf 1}(t\in \cG)
 + {\bf 1}( |i| \le M^{1+\kappa_1}) \cdot {\bf 1}(t\not\in \cG),
\ee
\nc
We use the calculation \eqref{der10}--\eqref{enee} (with cutoff $\varphi^{(1)}$ instead of $\psi^\ell$)
 but we keep the ``good'' $\fb[( v-\varphi^{(1)})_+, (v-\varphi^{(1)})_-]\ge 0$  term
that was estimated trivially in \eqref{der10} and we drop the (positive) potential term in $\fa$. We have
\begin{align}\label{diss2}
   \frac{1}{2} \sum_i [v_i(t)& -\varphi^{(1)}_i]_+^2\Bigg|_{t=T_1}^{T_2}
 + \int_{T_1}^{T_2}  \fb  [( v(t)-\varphi^{(1)})_+, (v(t)-\varphi^{(1)})_+]\rd t
\\ \nonumber
   \le &   -\int_{T_1}^{T_2}  \fb [( v(t)-\varphi^{(1)})_+, (v(t)-\varphi^{(1)})_-]\rd t
 -\int_{T_1}^{T_2}  \fb [ (v(t)-\varphi^{(1)})_+ \theta,\varphi^{(1)} ]\rd t 
\\ \nonumber
 & + |\bar u|\int_{T_1}^{T_2} \sum_i  (v_i(t)-\varphi_i^{(1)})_+ W_i \theta_i\rd t  .
\end{align}
Notice that we inserted the characteristic function  $\theta_i(t)$
using the fact that  \eqref{1c} and \eqref{alleq}
imply $v_i(t) \le \varphi^{(1)}_i$ for $|i|\ge 9M$, $t\in \cG$, and
$v_i(t) \le \varphi^{(1)}_i=\wt\psi_i$ for $|i|\ge M^{1+\kappa_1}$ and for all $t\in [-3M,0]$ \nc
 i.e. $v_i-\varphi^{(1)}= (v_i-\varphi^{(1)})\theta_i$ for any time. 
Moreover, $v_i(t) - \varphi_i^{(1)}\le \lambda\ell$ for $t\in \cG$ and $|i|\le 9M$.

The last error term in \eqref{diss2} is estimated trivially; in the regime $|i|\le M^{1+\kappa_1}$ we have
 $W_i\le CK^{-1+\xi}$
 and then from \eqref{23c}, $|\cG|\le CM^{1/4}$ 
and \eqref{barvbound1} we have
\be\label{potterm}
 |\bar u|\int_{T_1}^{T_2} \sum_i  (v_i(t)-\varphi_i^{(1)})_+ W_i \theta_i\rd t 
\le \lambda^2 \ell^2  (T_2-T_1) +
   C\lambda\ell^2  M^{\kappa_1+\kappa_2}   |\cG^c| \le C\lambda^2 \ell^2 \big[ (T_2-T_1) + 
\lambda^{-1}M^{1/2}\big]
\ee
after splitting the integration regime into ``good''times $\cG$ and ``bad'' times $\cG^c:= [-3M, 0]\setminus \cG$.
We also used \eqref{kappacond}.

The other error term in \eqref{diss2}
 will be estimated by a Schwarz inequality,
here we use the identity
$$
  \fb (f\theta, g) = \sum_{ij} ( f_i\theta_i - f_j\theta_j) B_{ij} (g_i-g_j)
  = \sum_{ij}( f_i\theta_i - f_j\theta_j) (\theta_i + \theta_j - \theta_i\theta_j)  B_{ij}   (g_i-g_j)
$$
for any functions $f$ and $g$, so
$$
 | \fb (f\theta, g)| \le \frac{1}{2} \sum_{ij}( f_i\theta_i - f_j\theta_j)^2  B_{ij}  + 
2\sum_{ij} \theta_i B_{ij}  (g_i-g_j)^2
$$
i.e.
$$
  \big|  \fb [ (v(t)-\varphi^{(1)})_+ \theta,\varphi^{(1)} ]\big|\le
  \frac{1}{2}  \fb [ (v(t)-\varphi^{(1)})_+\theta, (v(t)-\varphi^{(1)})_+\theta]+ 2 
   \sum_{ij} \theta_i B_{ij} (\varphi^{(1)}_i- \varphi^{(1)}_j)^2.
$$
The first term will be absorbed in the  quadratic term in the
left of \eqref{diss2}. By definition of $\varphi^{(1)}$, for the second term we have to control
\be\label{C1}
 \int_{T_1}^{T_2}\Bigg[
 \lambda^2 \sum_{i,j} (F_i-F_j)^2 B_{ij} + \sum_{i,j} (\wt\psi_i-\wt\psi_j)^2 B_{ij} \theta_i\Bigg] (t) \rd t.
\ee
Since $|F_i-F_j|\le C\ell M^{-1} |i-j|$ and $F_i-F_j$ is  supported on $|i|, |j|\le 9 M$,
by  splitting the summation to the regime $|i-j|\le K^\xi$  and its complement, 
we can bound the first term by 
$$  
  \int_{T_1}^{T_2}
 \lambda^2 \sum_{i,j} (F_i-F_j)^2 B_{ij}(t)  \rd t\le 
\lambda^2 \ell^2 M^{-2} \int_{T_1}^{T_2} \sum_{|i|, |j|\le 9M} |i-j|^2 B_{ij} (t)
$$
$$
  \le \lambda^2 \ell^2 M^{-2} K^{3\xi} \int_{T_1}^{T_2} \sum_{|i| \le 9M}  B_{i, i+1}(t)\rd t
  +C \lambda^2 \ell^2 M^{-2} \int_{T_1}^{T_2} \sum_{|i|,|j|\le 9M\atop |i-j|\ge K^\xi} \frac{|i-j|^2}{|i-j|^2}
$$
$$
  \le \lambda^2 \ell^2 M^{-2} K^{3\xi} \int_{-3M}^{0} \sum_{|i| \le 9M}  B_{i, i+1}(t)\rd t
  +C \lambda^2 \ell^2  (T_2-T_1),
$$
where we have used   $B_{i, j} \le B_{i, i+1}$ in the first regime and 
the upper bound in \eqref{far1new}
in the other regime. 
  By the regularity at $(Z,0)=(0,0)$  we can bound the last line by 
$$
   C\lambda^2 \ell^2  K^{3\xi + \rho }
 + C \lambda^2 \ell^2 (T_2-T_1) \le C \lambda^2 \ell^2  \big[ (T_2-T_1) 
 +   M^{1/2} \big] 
$$
(we also used \eqref{kappacond} and $K\le M^{\vartheta_0}$).

For  the second term in \eqref{C1}  and for $t\in \cG$ 
 we use that $\wt\psi_i\theta_i(t)=0$ and the supports of
$\theta_i$ and $\wt \psi_j$ are separated by a distance of order $M\gg K^\xi$. Thus  we
can use the upper bound in \eqref{far1new} to estimate the kernel:
$$
 \int_{T_1}^{T_2} {\bf 1}(t\in \cG) \sum_{i,j} (\wt\psi_i-\wt\psi_j)^2 B_{ij} (t) \theta_i(t)\rd t
 \le C\int_{T_1}^{T_2} \sum_{|i|\le 9M}\sum_{|j|\ge M\lambda^{-4}}\frac{\wt\psi_j^2}{|i-j|^2} \rd t
$$
\be\label{onG}
 \le CM(T_2-T_1) \sum_{|j|\ge M\lambda^{-4}}\frac{\wt\psi_j^2}{|j|^2}  \le C\ell^2 \lambda^2(T_2-T_1), 
\ee
where we have used  $\wt \psi_j \sim \ell (j/M)^{1/4}$ for large $j$. 
For times $t\not\in \cG$, we use
$$
    (\wt\psi_i-\wt\psi_j)^2 \le \frac{C\ell^2}{M^{1/2}} \frac{(i-j)^2}{ |i|^{3/2}+|j|^{3/2}}
$$
to get
\begin{align}\label{notonG}
 \int_{T_1}^{T_2} & {\bf 1}(t\not\in \cG) \sum_{i,j} (\wt\psi_i-\wt\psi_j)^2 B_{ij} (t) \theta_i(t)\rd t 
\\ \nonumber
 \le & \; \int_{T_1}^{T_2}{\bf 1}(t\not\in \cG)  \frac{C\ell^2}{M^{1/2}}
 \sum_{|i|\le M^{1+\kappa_1}}\sum_{|j|\ge M\lambda^{-4}} \frac{1}{ |i|^{3/2}+|j|^{3/2}}  
 \rd t
\\ \nonumber
&  + \int_{T_1}^{T_2}{\bf 1}(t\not\in \cG)  \frac{C\ell^2}{M^{1/2}}
 \sum_{|i|\le M^{1+\kappa_1}}\sum_{|j|\ge M\lambda^{-4}} B_{ij}(t)  \frac{ |i-j|^2\cdot {\bf 1}(|i-j|\le K^\xi)}{ |i|^{3/2}+|j|^{3/2}}  
 \rd t
 \\ \nonumber
 \le & \;  CM^{1+\kappa_1}|\cG^c|  \frac{\ell^2}{M^{1/2}} \lambda^2
 M^{-1/2} +  CK^{2\xi} \frac{\ell^2}{M^{1/2}} \int_{-3M}^{0}
 \sum_{|i|\le M^{1+\kappa_1}}\sum_{|j|\ge M\lambda^{-4}} B_{ij}(t)  \frac{ {\bf 1}(|i-j|\le K^\xi)}{ |i|^{3/2}+|j|^{3/2}}  
 \rd t
\\ \nonumber
 \le &
  C\lambda^2\ell^2  M^{\kappa_1+1/4} +  CK^{3\xi} \frac{\ell^2}{M^{1/2}} \frac{1}{ (M\lambda^{-4})^{3/2}} \int_{-3M}^{0}
 \sum_{|i|\le M^{1+\kappa_1}} B_{i,i+1}(t)  \rd t \\ \nonumber
\le &   C\lambda^2\ell^2  M^{1/2}. 
\end{align}
Here we first separated the summations over $i,j$ into $|i-j| \ge K^\xi$ and its complement.  Then
in the first regime we used the upper bound in \eqref{far1new}
and that the measure of the bad time is small, i.e.,    \eqref{1c},  
 to estimate the time integral;
in the second regime we used regularity at $(Z,0)$ and the fact that $K^\xi\ll M^{1/10}$
by \eqref{kappacond}.  
Inserting the error estimates \eqref{potterm}, \eqref{onG} and \eqref{notonG} into \eqref{diss2}, we have
\be\label{diss3}
   \frac{1}{2} \sum_i [v_i(t)-\varphi^{(1)}_i]_+^2\Bigg|_{t=T_1}^{T_2}
 + \frac{1}{2}\int_{T_1}^{T_2}  \fb [( v(t)-\varphi^{(1)})_+, (v(t)-\varphi^{(1)})_+]\rd t
\ee 
$$
  \le    -\int_{T_1}^{T_2}  \fb [( v(t)-\varphi^{(1)})_+, (v(t)-\varphi^{(1)})_{-}]\rd t
  + C\ell^2 \lambda^2 \big[ (T_2-T_1)+ M^{1/2}\big]. 
$$
Define
$$
  H(t) = \sum_i ( v_i(t)-\varphi^{(1)}_i)_+^2. 
$$
We  have
\be\label{goodterm}
   H(T_2)+   \int_{T_1}^{T_2}  \fb [( v(t)-\varphi^{(1)})_+, (v(t)-\varphi^{(1)})_{-}]\rd t
 \le H(T_1) +  C\ell^2 \lambda^2  \big[ (T_2-T_1)+ M^{1/2}\big]
\ee
for    any $ -3M\le T_1< T_2<0$. Notice that $\fb (f_+, f_-) \ge 0$ for any function $f$. 
Since $| v_i(t)-\varphi^{(1)}_i|\le \lambda\ell \theta_i$ for all $t\in \cG$, we also have 
\be\label{Hupp}
   H(t)\le C\lambda^2 \ell^2M, \qquad t\in \cG.
\ee

\subsubsection{Time slices when the good term helps}\label{sec:slice}

Let $\Sigma\subset \cG$ be the set of times that $v(T)$ is substantially below $\varphi^{(0)}$, i.e., 
$$
  \Sigma : =\Bigg\{ T\in (-3M, - 2M)\cap\cG\; : \; \#\Big\{ |i|\le M \; : \;  v_i(T)\le \varphi^{(0)}_i\Big\} \ge \frac{1}{4}
\mu M \Bigg\} .
$$
We have from \eqref{1c} and \eqref{2c} that
\be\label{sigmabound}
   |\Sigma|\ge  \frac{1}{4}  M\mu -CM^{1/4} \ge  \frac{1}{5}  M\mu.
\ee
By  \eqref{goodterm} (applied to $T_1=\min\Sigma$, $T_2=-2M$) and \eqref{Hupp}
(applied to $t=T_1$),  we have 
\begin{align}
  C\lambda^2\ell^2 M & \ge  \int_{\Sigma}  \fb [( v(t)-\varphi^{(1)})_+, (v(t)-\varphi^{(1)})_-]\rd t
\nonumber\\
 & \ge -  \int_{\Sigma}  \sum_{ij} ( v_i(t)-\varphi^{(1)}_i)_+ B_{ij}(t)  (v_j(t)-\varphi^{(1)}_j)_{-} \rd t
\\
\label{c2}
 & \ge - cM^{-2}  \int_{\Sigma}  \sum_{ij} ( v_i(t)-\varphi^{(1)}_i)_+  (v_j(t)-\varphi^{(1)}_j)_{-} \rd t,
\end{align}
where we have used that for 
$v_i(t)- \varphi^{(1)}_i$ is supported on $|i|\le 9M$ (for $t\in\cG$) and 
\be\label{Klower}
B_{ij}(t) \ge \bar cM^{-2},\qquad |i|, |j|\le 9M,
\ee
 with some positive constant $\bar c$ (this follows from the lower bound in
\eqref{far1new},  where $|i|\le 9M$  and $ M  \le K/10$ guarantee that
$d_i\ge K/C$ holds, and $K^\xi\ll M$ guarantees that \eqref{far1new} can
be used for the extreme points $i=-9M$, $j=9M$ and finally we used monotonicity 
$B_{ij}\ge B_{-9M, 9M}$ for any $|i|, |j|\le 9M$). 
For  $t\in \Sigma$  the number of $j$'s with $|j|\le  M$ such
that $ v_j(t)\le \varphi^{(0)}_j$ is at least $\frac{1}{5}\mu M$;   for such $j$'s we
have 
$$
 -(v_j(t)-\varphi^{(1)}_j)_{-} \ge\varphi^{(1)}_j- \varphi^{(0)}_j \ge (1-\lambda)\ell\ge\frac{\ell}{2}.
$$
Thus we can bound  \eqref{c2}  by 
$$
  \ge c\ell M^{-1} \frac{\mu}{10}  \int_{\Sigma}  \sum_{i} ( v_i(t)-\varphi^{(1)}_i)_+  \rd t
  \ge cM^{-1} \frac{\mu}{10\lambda}  \int_{\Sigma}  \sum_{i} ( v_i(t)-\varphi^{(1)}_i)_+^2  \rd t,
$$
where  we have used that $( v_i(t)-\varphi^{(1)}_i)_+ \le\lambda\ell$ for $t\in \cG$.

Altogether we have proved
$$
   \int_{\Sigma}  \sum_{i} ( v_i(t)-\varphi^{(1)}_i)_+^2  \rd t \le C\lambda^3\mu^{-1}\ell^2 M^2 
 \le \lambda^{3-\frac{1}{8}}\ell^2 M^2
$$
if $\lambda$ is sufficiently small (depending on $\mu$).
Thus there exists a subset $\Theta \subset \Sigma$ such that
$$
   |\Theta |\le \lambda^{1/8}M,
$$
and we have
$$
 \sum_{i} ( v_i(t)-\varphi^{(1)}_i)_+^2  
 \le \lambda^{3-\frac{1}{4}}\ell^2 M, \qquad \forall t\in \Sigma\setminus \Theta.
$$
Choosing $\lambda$ small  and recalling \eqref{sigmabound} we see that 
\be\label{56}
 \sum_{i} ( v_i(t)-\varphi^{(1)}_i)_+^2 
 \le \lambda^{3-\frac{1}{4}}\ell^2 M
\ee
holds on a set of times $t$'s in $\Sigma\subset [-3M, -2M]\cap \cG$ of measure at least  $M\mu /8$.
In particular this set of times is non-empty.

\subsubsection{Finding the intermediate set}

Since \eqref{opt1} is  satisfied, there is a $T_0\in (-2M, 0)\cap \cG$ such that
\be\label{57}
    \#\Big\{ i \; : \; (v_i(T_0)-\varphi^{(2)}_i)_+>0 \Big\} \ge \frac{1}{2}M\delta - CM^{1/4},
\ee
and choose a $T_1\in \Sigma$ (then $T_1< T_0$) such that
\be\label{5811}
 H(T_1)=\sum_{i} ( v_i(T_1)-\varphi^{(1)}_i)_+^2 
 \le \lambda^{3-\frac{1}{4}}\ell^2 M
\ee
(such $T_1$ exists by the conclusion of the previous section, \eqref{56}).

We also have
$$
  H(T_0)=  \sum_{i} ( v_i(T_0)-\varphi^{(1)}_i)_+^2  \ge\sum_{i} ( \varphi^{(2)}_i(T_0)-\varphi^{(1)}_i)_+^2 
   \cdot {\bf 1}\big(( v_i(T_0)-\varphi^{(2)}_i)_+>0\big)
$$
\be\label{61uj}
  \ge \sum_{i} \ell^2( \lambda-\lambda^2)^2 F_i^2
  \cdot {\bf 1}\big(( v_i(T_0)-\varphi^{(2)}_i)_+>0\big) \ge C_F \frac{\lambda^2}{4}\ell^2 \delta^3M
\ee
with some positive constant $C_F$. This follows from \eqref{57}; notice first that
the set in \eqref{57} must lie in $[-9M, 9M]$ (see \eqref{alleq} and \eqref{1c}), and
 even if  the whole set \eqref{57} is near the ``corner'' (i.e. close to $i\sim \pm 9M$),
still the sum of these $F_i$'s is of order $\delta^3 M$ since $F_i$ is linear near
the endpoints $i=\pm 9M$.

Choose now $\lambda$ small enough  (depending on the fixed $\delta$) s.t.
$$
   \lambda^{3-\frac{1}{4}}\ell^2 M \le \frac{1}{16}  C_F \lambda^2\ell^2 \delta^3M.
$$
Since $H(T)$ is continuous and it goes from 
a small value $H(T_1)\le \frac{1}{16}  C_F \lambda^2\ell^2 \delta^3M $ to 
a large value $H(T_0)\ge \frac{1}{4} C_F \lambda^2\ell^2 \delta^3M $,
the set of intermediate times
$$
   D: = \Big\{ t\in (T_1, T_0)\; : \;  \frac{1}{16}  C_F \lambda^2\ell^2 \delta^3M < H(t)
   < \frac{1}{4} C_F \lambda^2\ell^2 \delta^3M \Big\}
$$
is non-empty.

\begin{lemma}\label{lemma:D} The set  $D$  contains an interval of size at least $c\delta^3M$
with some positive constant $c>0$. Moreover, for any $t\in D\cap \cG\nc$, we have
\be\label{extreme}
   \#\big\{ i \; : \; \varphi^{(2)}_i \le v_i(t) \big\}\le \frac{1}{2}\delta M.
\ee
\end{lemma}
\begin{proof}
By continuity,
there is
an intermediate time $T' \in D\subset [T_1, T_0]$ such that 
$H( T')= \frac{1}{8}  C_F \lambda^2\ell^2 \delta^3M $.  We can assume
that $T'$ is the largest such time, i.e. 
\be\label{Hlower}
H(t)> \frac{1}{8}  C_F \lambda^2\ell^2 \delta^3M \qquad\mbox{for any}\quad
 t\in  [T', T_0]\cap D.
\ee

Let $T'' = T' + c\delta^3M$ with a small $c>0$.
We claim that  $[T', T'']\subset D$.
For any $t\in [T', T'']$ we can use  \eqref{goodterm}:
\be\label{Hup}
   H(t) \le H(T') +C\ell^2\lambda^2 
\big[ (t-T') + M^{1/2}\big]\nc 
   \le \frac{1}{8}  C_F \lambda^2\ell^2 \delta^3M+  C c \ell^2\lambda^2  \delta^3M 
   < \frac{1}{4} C_F \lambda^2\ell^2 \delta^3M
\ee
if $c$ is sufficiently small.
This means that as $t$ runs through $[T', T'']$, $H(t)$ has not reached
$\frac{1}{4} C_F \lambda^2\ell^2 \delta^3M$, in particular $[T', T'']\subset (T_1, T_0)$
since $H(T_0)$ is already above this threshold. Combining then \eqref{Hup}
with \eqref{Hlower}, we get $[T', T'']\subset D$. This  proves the first statement
of the lemma.

For the second statement, we argue by contradiction. Suppose we have
$ \#\big\{ i \; : \; \varphi^{(2)}_i \le v_i(\tau) \big\}> \frac{1}{2}\delta M$
for some $\tau\in D \cap \cG\nc$. Going through the estimate \eqref{61uj} but  $T_0$ replaced with $\tau$, we would get
$H(\tau)\ge  C_F \frac{\lambda^2}{4}\ell^2 \delta^3M$, but this contradicts to $\tau\in D$.
This completes the proof of the lemma.
\qed

Define the exceptional set $\cF\subset D \cap\cG\nc$ of times where $v$ is  below
$\varphi^{(0)}$, i.e.
$$
   \cF: = \Big\{ t\in D \cap \cG\nc\; : \; \#\big\{ |j|\le 8M \; : \; v_j(t)-\varphi^{(0)}_j\le 0\big\}
 \ge \mu M  \Big\}.
$$
This set is very small, since from \eqref{Hupp} (applied to $t_{max}:=\sup\cF\in \bar \cG$) we have
\begin{align*}
  C\lambda^2 \ell^2M & \ge  - \int_{-3M}^{t_{max}} \sum_{ij}(v_i(t)-\varphi^{(1)}_i)_+  B_{ij}(t)
  (v_j(t)-\varphi^{(1)}_j)_{-} \rd t
\\ &
  \ge -  \int_\cF \sum_{|i|, |j|\le 9M}(v_i(t)-\varphi^{(1)}_i)_+ B_{ij}(t)
  (v_j(t)-\varphi^{(1)}_j)_{-} \rd t
\\ &
  \ge -  \bar cM^{-2}\int_\cF \sum_{|i|, |j|\le 9M}(v_i(t)-\varphi^{(1)}_i)_+ 
  (v_j(t)-\varphi^{(1)}_j)_{-} \rd t
\\ &
   \ge \frac{\bar c}{2M}\ell\mu\int_\cF \sum_{|i|\le 9M}(v_i(t)-\varphi^{(1)}_i)_+  \rd t,
\end{align*}
where we restricted the time integration to $\cF$ in the first step,
then we used \eqref{Klower} in the second step. In the third step
we used that whenever $v_j(t)-\varphi^{(0)}_j\le 0$ (see  the definition of $\cF$),
then $ - (v_j(t)-\varphi^{(1)}_j)_{-}\ge \ell(1-\lambda)\ge \frac{\ell}{2}$.

 By \eqref{1c}, $(v_i(t)-\varphi^{(1)}_i)_+ \le\ell\lambda $ and $(v_i(t)-\varphi^{(1)}_i)_+ = 0$  if $|i|\ge 9M$
and $t\in \cG$. \nc Hence we can continue
the above estimate
$$
   C\lambda^2 \ell^2M \ge \frac{\bar c \mu}{2M\lambda}
\int_\cF \sum_{i}(v_i(t)-\varphi^{(1)}_i)_+^2  \rd t  
=\frac{\bar c \mu}{2M\lambda}
\int_\cF H(t)  \rd t \ge \frac{\bar c C_F}{32}\lambda\ell^2 \delta^3\mu |\cF|.
$$
Here we used that $\cF\subset D$ and that in $D$ we have a lower bound on $H(t)$.
The conclusion is that
$$
   |\cF|\le \frac{C\lambda}{\delta^3\mu}M
$$
with some fixed constant $C>0$.  Using that $|D|\ge c\delta^3M$ from Lemma \ref{lemma:D}
and the smallness of $|\cG^c|$, we thus have
$$
  |\cF|\le \frac{|D\cap \cG|}{2}, \qquad |D\cap \cG|\ge \frac{1}{2}c\delta^3M
$$
 if $\lambda$ is sufficiently small, like
\be\label{lambdasmall}
 \lambda\le c \delta^6\mu.
\ee

This means that $|D\setminus\cF|\ge \frac{c}{2}\delta^3M$. Now we claim that
for $t\in (D \cap \cG \nc)\setminus \cF$ we have
\be\label{Adef1}
  A(t): = \# \Big\{ i \; : \; \varphi^{(0)}_i< v_i(t)<\varphi^{(2)}_i\Big\}\ge\frac{ M}{2}.
\ee
This is because 
 $t\not\in\cF$ guarantees that the lower bound $ \varphi^{(0)}_i\le v_i(t)$
is violated not more than $\mu M \le M/4$ times among the indices  $|i|\le 8M$.
By \eqref{extreme}, the upper bound $v_i(t)\le \varphi^{(2)}_i$ is violated
not more than $\frac{1}{2}\delta M\le M/4$ times.

Finally, integrating \eqref{Adef1} gives
$$
  \int_{-3M}^{0}
 \#\Big\{ i\; : \;  \varphi_i^{(0)}< v_i(t)<\varphi^{(2)}_i\Big\} \rd t
 = \int_{-3M}^{0} A(t)\rd t \ge \frac{M}{2}|(D\cap \cG)\nc\setminus\cF| \ge c\delta^3M^2
$$
with some small $c>0$,
which implies \eqref{opt2} with 
\be\label{gammachoice}
\gamma:= c\delta^3.
\ee
 This proves Lemma~\ref{lm:2nd}.
\end{proof}

\appendix

\section{Proof of Lemma~\ref{lm:goody}}\label{sec:goody}

First we show that
on the set $ \cR_{L, K}$, the
length of the interval $J=J_\by=(y_{L-K-1}, y_{L+K+1})$ satisfies \eqref{Jlength}.
We first write
\be\label{length}
    |J| = |y_{L+K+1}-y_{L-K-1}| = |\gamma_{L+K+1}-\gamma_{L-K-1}| +  O(N^{-1+\xi\delta/2}).
\ee
Then we use the Taylor expansion
$$
   \varrho(x) = \varrho(\bar y) + O(x-\bar y)
$$
around the midpoint $\bar y$ of $J$. Here we used that
 $\varrho\in C^1$ away from the edge. Thus
 from \eqref{defgammagen}
\be\label{cck}
  \cK+1 = N\int_{\gamma_{L-K-1}}^{\gamma_{L+K+1}}\varrho = N\int_{y_{L-K-1}}^{y_{L+K+1}} \varrho + 
 O(N^{\xi\delta/2})
  = N|J|\varrho(\bar y) + O(N|J|^2) +  O(N^{\xi\delta/2}), 
\ee
since the contribution of the second order term in the Taylor expansion
is of order $N|J|^2$. Expressing $|J|$ from this equation and using \eqref{K},
we arrive at  \eqref{Jlength}.

\medskip

Now we prove \eqref{Vby1}.
We set
$$
   U(x): = V(x)-\frac{2}{N} \sum_{j\; : \; |j-L|\ge K+K^\xi}
 \log|x-\gamma_j|.
$$
The potential $U$ is similar to $V_\by$, but the interactions
with the external points near the edges of $J$ ($y_j$'s with $|j-L|< K+K^\xi$)
have been removed and the external points $y_j$ away from the edges
have been replaced by their classical value $\gamma_j$.
In proving \eqref{Vby1}, we  will first compare $V_\by$ with an auxiliary potential $U$
and then we compute $U'$.

First we estimate the difference $V_\by'(x)-U'(x)$.
We  fix $x\in J$, and for definiteness, we assume
that $d(x)=x-y_{L-K-1}$, i.e. $x$ is closer to the lower endpoint of $J$;
the other case is analogous. We get   (explanations will be given after the equation) 
\begin{align}
  |V_\by'(x) -  U'(x)| \le & 
 \frac{1}{N} \sum_{K< |j-L|< K+K^\xi}
   \frac{1}{|x-y_j|} +
 \frac{1}{N} \sum_{ |j-L|\ge K+K^\xi}
\frac{|y_j-\gamma_j|}{|x-y_j||x-\gamma_j|} \non \\
  \le &   \frac{CK^\xi}{Nd(x)} + 
\frac{N^{-1+\delta\xi/2}}{d(x)} \frac{1}{N}
\Big[ \sum_{j=\al N/2}^{L-K-K^\xi} + \sum_{j=L+K+K^\xi}^{N(1-\al/2)} \Big]  \frac{1}{|x-\gamma_j|}\non \\
& + \frac{CN^{-4/15+\e}}{N}\Big[ \sum_{j=N^{3/5+\e}}^{\al N/2}1 + \sum_{N(1-\al/2)}^{N-N^{3/5+\e}} 1\Big]  
 +\frac{C}{N} \Big[ \sum_{j=1}^{N^{3/5+\e}} 1 + \sum_{j=N-N^{3/5+\e}}^N 1\Big] \non \\
 \le &  \frac{CK^\xi}{Nd(x)}  + \frac{CN^{-1+\delta\xi/2}\log N}{d(x)} + CN^{-4/15+\e} \non \\
  \le & \frac{CK^{\xi}}{Nd(x)}. \label{VU}
\end{align}
Here for the first bulk sum, $j\in \llbracket N\al/2, L-K-K^\xi\rrbracket$, we used
 $|y_j-\gamma_j|\le N^{-1+\xi\delta/2}$ from
 the definition of $ \cR_{L, K}$ and the fact that for $j\le L-K-K^{\xi}$ we have
\begin{align*}
    x-\gamma_j \ge & \; y_{L-K-1}- \gamma_j\\ 
 \ge & \; 
\gamma_{L-K-1}-\gamma_j -  |y_{L-K-1}-\gamma_{L-K-1}| \\
  \ge & \; cN^{-1}(L-K-1-j) - CN^{-1+\xi\delta/2} \\ \ge &\;  c'N^{-1} (L-K-1-j)
\end{align*}
with some positive constants $c, c'$. This estimate allows one to 
sum up $|x-\gamma_j|^{-1}$ at the expense of a $\log N$ factor.
Similar estimate holds  for $j\ge L+K + K^{\xi}$.
In the intermediate sum, $j\in \llbracket N^{3/5+\e}, N\al/2\rrbracket$, we
used $|y_j-\gamma_j|\le CN^{-4/15+\e}$ and that $|x-y_j|$ and $|x-\gamma_j|$ are bounded
from below by a positive constant since 
$$
  x - y_j\ge y_{L-K-1}-y_j\ge y_{\al N} -y_j \ge \gamma_{N\al} - \gamma_{N\al/2} + O(N^{-1+\xi\delta/2})
 \ge c
$$
and similarly for $x-\gamma_j$. Finally, very near the edge, e.g. for $j\le N^{3/5+\e}$,
we just estimated $|y_j-\gamma_j|$ by a constant. This explains \eqref{VU}.

Now we estimate $U'(x)$. 
We use the fact that the equilibrium measure $\varrho=\varrho_V$ satisfies
the identity
$$
  \frac{1}{2}V'(x) = \int \frac{\varrho(y)}{x-y}\rd y
$$
from the Euler-Lagrange equation of \eqref{varprin}, see \cite{APS, BPS}. Thus
$$
   \frac{1}{2}|U'(x)| \le  |\Omega_1| +|\Omega_2|+|\Omega_3|
$$
with
\begin{align*}
    \Omega_1:= & \int_{\gamma_{L-K-K^\xi}}^{\gamma_{L+K+K^\xi}} \frac{\varrho(y)}{x-y}\rd y,  \\
  \Omega_2:= &\int_A^{\gamma_{L-K-K^\xi}} \frac{\varrho(y)}{x-y}\rd y-
 \frac{1}{N} \sum_{j=1}^{L-K-K^\xi} \frac{1}{x-\gamma_j}, \\
  \Omega_3:= &\int_{\gamma_{L+K+K^\xi}}^B \frac{\varrho(y)}{x-y}\rd y-
 \frac{1}{N}\sum_{j= L+K+K^\xi}^{N} \frac{1}{x-\gamma_j}, 
\end{align*}
 where $[A, B]$ is the support of the density $\rho$.

To estimate  $\Omega_1$, we use Taylor expansion 
$$
   \varrho(y) = \varrho(x) +  O\big(|x-y|\big).
$$
For definiteness we again assume that $d(x) = x-y_{L-K-1}$, and use
that on $\cR_{L,K}$ we have 
$$
\gamma_{L-K-K^\xi}\le y_{L-K-1} \le x 
 \le y_{L+K+1}\le \gamma_{L+K+K^\xi}.
$$
We thus obtain
\begin{align}
  \Omega_1 = & \int_{ \gamma_{L-K-K^\xi} }^{\gamma_{L+K+K^\xi}} 
\frac{\varrho(x) + O\big(|x-y|\big)}{x-y}\rd y  \nonumber \\
   = & \varrho(x)\log \frac{  \gamma_{L+K+K^\xi}-x}{x-\gamma_{L-K-K^\xi}}
   + O(K/N) \nonumber  \\
  = & \varrho(\bar y) \log \frac{d_+(x)}{d_-(x)} + O(KN^{-1+\xi}). \label{41}
\end{align}
In the first step above we computed the leading term of the integral, while
the other term was estimated trivially using 
that the integration length is $\gamma_{L+K+K^\xi}- \gamma_{L-K-K^\xi}  = O(K/N)$.
In the second step we used that $\varrho\in C^1$ away the edge, i.e.
$ \varrho(x) = \varrho(\bar y) + O(K/N)$. To estimate the logarithm, we used
\begin{align*}
 \gamma_{L+K+K^\xi}-x =& (\gamma_{L+K+K^\xi} - \gamma_{L+K+1}) + ( \gamma_{L+K+1} - y_{L+K+1})
+(y_{L+K+1} - x) \\
= & \varrho(\bar y)N^{-1}K^\xi + O(N^{-1+\xi\delta/2}) +  (y_{L+K+1} - x) \\
= & d_+(x)   + O(N^{-1+\xi\delta/2})
\end{align*}
and the similar relation
$$
  x-\gamma_{L-K-K^\xi} =  d_-(x)   + O(N^{-1+\xi\delta/2}).
$$
 Notice that the error term  in \eqref{41} is smaller than the target estimate $K^{\xi}/(Nd(x))$ since
$d(x)\le K/N\ll K^{-1+\xi}N^{-\xi}$.

Now we estimate the $\Omega_2$ term;   $\Omega_3$ can be treated analogously.
We can write (with the convention $\gamma_0=A$)
\begin{align*}
   |\Omega_2 |= & \Big|\sum_{j=1}^{L-K-K^\xi} \int_{\gamma_{j-1}}^{\gamma_{j}}
\varrho(y) \Big[ \frac{1}{x-y} - \frac{1}{x-\gamma_j}\Big]\rd y\Big|\\
  \le &C\sum_{j=1}^{L-K-K^\xi} (\gamma_j-\gamma_{j-1})
 \int_{\gamma_{j-1}}^{\gamma_{j}} \frac{\varrho(y)}{ (x-y)^2}\rd y \\
\le & CN^{-1} \int_{A+\kappa}^{\gamma_{L-K-K^\xi}} \frac{\rd y}{(x-y)^2}
 + CN^{-2/3} \int_A^{A+\kappa}  \frac{\rd y}{(x-y)^2} \\
 \le & \frac{CN^{-1}}{d(x)}.
\end{align*} 
In the first step we used that
$$
   \int_{\gamma_{j-1}}^{\gamma_{j}}
\varrho(y) = \frac{1}{N}
$$
from \eqref{defgammagen}. In the second step we used that $\gamma_j-\gamma_{j-1}= O_\kappa(N^{-1})$ in
the bulk, i.e. for $\gamma_j\ge A+\kappa$, 
 and $\max_j(\gamma_j-\gamma_{j-1})= O(N^{-2/3})$ (the order $N^{-2/3}$  comes from the fact that the density 
 $\rho$  vanishes as a square root at the endpoints).  
 The parameter $\kappa=\kappa(\al)$ is chosen
such that $A+2\kappa \le y_{L-K-1}$ which can be achieved since $L\ge \al N$
and $y_{L-K-1}$ is close to $\gamma_{L-K-1}$.  
In the very last step we absorbed the $N^{-2/3}$ error
term into $(N d(x))^{-1}\ge K^{-1}\gg N^{-2/3}$.

\medskip

Finally we prove \eqref{Vbysec}. 
Since $|y_j-\gamma_j|\le K^{\xi}/N$, it follows that $|x-y_j|\sim |x-\gamma_j|$ for $|x-\gamma_j|\ge K^\xi/N$. 
Thus we have 
$$
   V_\by''(x) = V''(x) + \frac{2}{N}\sum_{j\not\in I} \frac{1}{(x-y_j)^2}
  \ge \inf V'' + \frac{c}{N}\sum_{j\not\in I} \frac{1}{(x-\gamma_j)^2}
  \ge  \inf V'' + \frac{c}{d(x)},
$$
with some positive constant $c$ (depending only on $\al$).  In estimating the summation, we used
that  the sequence $\gamma_k$ is  regularly  spaced with gaps  of  order $1/N$.
This completes the proof of Lemma~\ref{lm:goody}. \qed

\section{Discrete Gagliardo-Nirenberg inequalities}\label{sec:GN}

Recall the  integral formula for  quadratic form of  the operator $ (-\Delta)^{s/2}$ in $\R$
for any $s\in (0,2)$:
\be
\int_\R \phi(x) \, ( (-\Delta)^{s/2} \, \phi) (x) \rd x  =   
C(s)\int_\R \int_\R \frac { |\phi(x) - \phi(y)|^2} { |x-y|^{1+s}} \rd x \rd y,
\ee 
where $C(s)$ is an explicit positive constant, $C(1)= (2\pi)^{-1}$ and $\phi \in H^{s/4}(\bR)$.
We have   the following Gagliardo-Nirenberg type inequality in the critical case 
 (see (1.4) of \cite{OO} with the choice of parameters $n=1, p=4$) 
\be\label{p}
 \| \phi \|_4^4 \le C \| \phi \|_2^2   \int_\R \phi(x) \, ( \sqrt{-\Delta} \, \phi) (x) \rd x, \qquad
\phi :\R\to \R.
\ee 
We first give a slight generalization of this inequality:

\begin{proposition}  Let $p\in (2, \infty)$ and $s\in (1-\frac{2}{p},2)$. Then we have
\be\label{oo2}
  \| \phi\|_p \le C_{p,s} \|\phi\|_2^{1-\frac{p-2}{sp}}
 \Big[ \int_\R \phi(x) \, ( (-\Delta)^{s/2} \, \phi) (x)
 \rd x\Big]^{\frac{p-2}{2sp}}
\ee
with some constant $C_{p,s}$ with $\|\cdot \|_p =\|\cdot \|_{L^p(\bR)}$.
\end{proposition}

{\it Proof.} We follow the proof of Theorem 2 in \cite{OO}. Setting $q = p/(p-1)$
and using Hausdorff-Young and H\"older  inequalities for any $\lambda>0$ and $\al>1-\frac{q}{2}$
\begin{align}\label{HY}
    \| \phi\|_p & \le C_{p}\|\wh \phi\|_q \le C_p \big\|\wh \phi(\xi)(\lambda+ |\xi|)^{\al/q}\big\|_2
\big\| (\lambda+ |\xi|)^{-\al/q}\|_{2q/(2-q)}
\\ \nonumber
  & \le C_{p,\al}\big( \lambda^{\al/q} \|\phi\|_2 + \langle \phi, (-\Delta)^{\al/q} \phi)^{1/2}\rangle \lambda^{\frac{1-\al}{q} -\frac{1}{2}}
\\ \nonumber
&\le C_{p,\al} \|\phi\|_2^{1- \frac{2-q}{2\al}}   \langle\phi, (-\Delta)^{\al/q} \phi)^{\frac{2-q}{4\al}},
\end{align}
where in the last step we chose $\lambda =  (\phi, |p|^{2\al/q} \rangle^{q/2\al} \|\phi\|^{-q/\al}$.
We used $\langle \cdot , \cdot \rangle$ to denote the inner product in $L^2(\bR)$.  
Setting $s= 2\al/q$, we obtain \eqref{oo2}. \qed

Now we derive the discrete version of this inequality.

\begin{proposition}\label{prop:GNd}  Let $p\in (2, \infty)$ and $s\in (1-\frac{2}{p},2)$. 
Then there exists a positive constant $C_{p,s}$ such that 
\be\label{s}
\| f \|_p  \le C_{p,s} \| f \|_2^{1-\frac{p-2}{sp}} 
\Bigg[\sum_{i \not = j  \in \ZZ}   \frac { |f_i - f_j|^2} { |i-j|^{1+s}}  \Bigg]^{\frac{p-2}{2sp}}
\ee 
holds for any function $f: \ZZ \to \R$, where $\|f\|_p = \|f \|_{L^p(\ZZ)} =\big( \sum_i |f_i|^p\big)^{1/p}$.
\end{proposition}

\begin{proof} 
Given $f:\ZZ\to \R$,  let $\phi:\R\to\R$ be its linear interpolation, i.e. 
$\phi (i):=f_i$ for $i\in \ZZ$ and 
\be\label{explic}
  \phi (x) = f_i + (f_{i+1}-f_i) (x-i) = f_{i+1} -  (f_{i+1}-f_i)(i+1-x), \qquad x\in [i, i+1].
\ee
It is easy to see that
\be\label{normf}
C_p^{-1}  \| \phi \|_{L^p(\R)} \le \| f \|_{L^p(\ZZ)} \le  C_p \| \phi \|_{L^p(\R)} \; , \qquad 2\le p \le \infty, 
\ee
with some constant $C_p$. We claim that
\be\label{claim}
 \int_\R \int_\R \frac { | \phi(x) - \phi(y)|^2} { |x-y|^{1+s}} \rd x \rd y \le C_s 
 \sum_{i \not = j  \in \ZZ}   \frac { |f_i - f_j|^2} { |i-j|^{1+s}}
 \ee
with some  constant $C_s$, then \eqref{normf} and \eqref{claim} will yield \eqref{s}
from \eqref{p}.

To prove \eqref{claim}, we can write
\be\label{df}
 \int_\R \int_\R \frac { | \phi(x) - \phi(y)|^2} { |x-y|^{1+s}} \rd x \rd y 
 = \sum_{i,j} \int_i^{i+1}\int_j^{j+1} \frac { | \phi(x) - \phi(y)|^2} { |x-y|^{1+s}} \rd x \rd y.
\ee
Using the explicit formula \eqref{explic},
we first compute the $i=j$ terms in \eqref{df}:
\begin{align}\label{1ex}
   \sum_i \int_i^{i+1}\int_i^{i+1} & \frac { | \phi(x) - \phi(y)|^2} { |x-y|^{1+s}} \rd x \rd y
 \\ & =  \sum_i |f_i- f_{i+1}|^2  \int_i^{i+1}\int_i^{i+1} \frac{\rd x \rd y}{|x-y|^s} = C_s 
 \sum_i \frac{|f_i- f_{i+1}|^2}{|i- (i+1)|^{1+s}}
\nonumber
\end{align}
with some explicit $C_s$.
Next we compute the terms $|i-j|=1$ in \eqref{df}. We assume
$j=i-1$, the terms $j=i+1$ are analogous;
\begin{align}\label{2ex}
   \sum_i \int_i^{i+1}& \int_{i-1}^{i} \frac { | \phi(x) - \phi(y)|^2} { |x-y|^{1+s}} \rd x \rd y  \\
 &  \le  \sum_i (f_{i+1}-f_i)^2 \int_i^{i+1}\int_{i-1}^{i} \frac{(x-i)^2}{(x-y)^{1+s}}  \rd x \rd y
   + \sum_i (f_i-f_{i-1})^2  \int_i^{i+1}\int_{i-1}^{i} \frac{(i-y)^2}{(x-y)^{1+s}}  \rd x \rd y,
\nonumber
\end{align}
where we used $\phi(x) = f_i + (f_{i+1}-f_i) (x-i)$ and $\phi(y) = f_i- (f_i-f_{i-1})(i-y)$.
The above integrals are finite constants $C_s$, so we get
$$
   \sum_i \int_i^{i+1}\int_{i-1}^{i} \frac { | \phi(x) - \phi(y)|^2} { |x-y|^2} \rd x \rd y
  \le C_s \sum_i \frac{ (f_{i+1}-f_i)^2}{(i+1-i)^{1+s}} +  \frac{(f_i-f_{i-1})^2 }{(i-(i-1))^{1+s}}.
$$
Finally, for the terms $|i-j|\ge 2$, we can just replace $(x-y)^{1+s}$ by $(i-j)^{1+s}$ 
in the right hand side of \eqref{df} and use simple Schwarz inequalities to get
$$
 \int_i^{i+1}\int_j^{j+1} \frac { | \phi(x) - \phi(y)|^2} { |x-y|^{1+s}} \rd x \rd y 
\le C_s\frac{ |f_i- f_j|^2 + |f_{i+1} - f_i|^2 + |f_{j+1}-f_j|^2}{|i-j|^{1+s}}.
$$
After summing up we get
\be\label{3ex}
  \sum_{|i-j|\ge 2}  \int_i^{i+1}\int_j^{j+1} \frac { | \phi(x) - \phi(y)|^2} { |x-y|^{1+s}} \rd x \rd y 
  \le C_s  \sum_{|i-j|\ge 2} \frac{ |f_i- f_j|^2}{|i-j|^{1+s}} + C_s
\sum_i \frac{|f_{i+1} - f_i|^2 }{((i+1)-i)^{1+s}}.
\ee
The estimates \eqref{1ex}, \eqref{2ex} and \eqref{3ex} together yield
\eqref{claim}. 
\end{proof}

\bigskip

With two fixed parameters $a, b>0$,
define the function
\be\label{mdef}
    m(\xi): =|\xi| \cdot {\bf 1}(|\xi|\le a) + b|\xi|, \qquad \xi \in \bR.
\ee
We will consider the operator $T=m(\sqrt{-\Delta})$ defined by $m$ being its
Fourier multiplier, i.e.
$$
   \wh{T \phi}(\xi) = m(\xi) \wh \phi(\xi).
$$
\begin{proposition}\label{prop:newGN}
We have
\be\label{newGN}
   \| \phi\|_4^4 \le C\| \phi\|_2^2 \langle \phi, m(\sqrt{-\Delta})\phi\rangle +
 \frac{C}{ab^3}  \|\phi\|_\infty^4.
\ee
\end{proposition}

{\it Proof.}  Let $\chi\in C_0^\infty(\bR)$ be a symmetric cutoff function
such that $0\le \chi\le 1$, $\chi(\xi) =1$ for $|\xi|\le 1/2$ and $\chi(\xi)=0$ for $|\xi|\ge 1$.
Set $\chi_a(\xi)= \chi(\xi/a)$. Split $\phi=\phi_1+ \phi_2$ into low and high Fourier modes, the
decomposition is defined via their Fourier transforms,
$$
   \phi = \phi_1 + \phi_2, \qquad \wh \phi_1(\xi) : = \wh \phi (\xi)\chi_a(\xi), \quad 
  \wh \phi_2 (\xi): = \wh \phi(\xi) (1-\chi_a(\xi)).
$$
First we estimate the contribution from $\phi_1$. With the choice of $p=4$, $s=1$ in  \eqref{oo2}
we have
$$
   \| \phi_1\|_4 \le C \|\phi_1\|_2^{1/2} \Big[\int_\bR |\wh\phi_1(\xi)|^2 |\xi| \rd \xi\Big]^{1/4} 
   \le C \|\phi_1\|_2^{1/2}\langle \phi_1, m(\sqrt{-\Delta})\phi_1\rangle^{1/4} 
  \le  C \|\phi\|_2^{1/2}\langle \phi, m(\sqrt{-\Delta})\phi\rangle^{1/4} ,
$$ 
where we used that on the support of $\wh \phi_1$ we have $|\xi|\le m(\xi)$
and in the last step we used $|\wh\phi_1|\le |\wh \phi|$ pointwise.

For the contribution of $\phi_2$, with any $\delta>0$ we have
$$
    \| \phi_2\|_4 \le \|\phi_2\|_\infty^{1/4} \|\phi_2\|_3^{3/4} \le  \delta^{-4}  \|\phi_2\|_\infty
   + \delta^{4/3}  \|\phi_2\|_3.
$$
In the first term we use the Littlewood-Paley inequality
$$
    \|\phi_2\|_\infty \le C  \|\phi\|_\infty,
$$
where $C$ depends only on the choice of $\chi$ 
but is independent of $a$. In the second term we use \eqref{oo2} with $s=2/3$, $p=3$:
$$ 
\|\phi_2\|_3  \le C \|\phi_2\|^{1/2} \Big[ \int |\wh\phi_2(\xi)|^2 |\xi|^{2/3} \rd \xi\Big]^{1/4} 
 \le Cb^{-1/4}a^{-1/12}\|\phi_2\|^{1/2}\langle \phi_2, m(\sqrt{-\Delta})\phi_2\rangle^{1/4},
$$
where in the second step we used  $|\xi|^{2/3} \le 2b^{-1}a^{-1/3} m(\xi)$ for all $|\xi|\ge a/2$,
i.e. on the support of $\wh \phi_2$. Using $|\wh\phi_2|\le |\wh \phi|$, we thus
have
$$
    \| \phi_2\|_4 \le  C\delta^{-4}   \|\phi\|_\infty + C\delta^{4/3} b^{-1/4}a^{-1/12}
 \|\phi\|^{1/2}\langle \phi, m(\sqrt{-\Delta})\phi\rangle^{1/4}.
$$
Choosing $\delta = b^{3/16}a^{1/16}$, we obtain \eqref{newGN}. \qed

\medskip

Finally, we derive a discrete version and a localized discrete version of Proposition~\ref{prop:newGN}.
\begin{proposition}\label{prop:newGNdiscr} Let $\{ B_{ij}\; : \;  i\ne j \in \bZ\}$
 be a bi-infinite matrix of  nonnegative numbers with  $B_{ij}=B_{ji}$.
\begin{itemize}
\item[(i)] [Global version] Assume that for some positive constants $a, b, r$ with $b\le r\le 1$, we have
\be\label{displ1}
   B_{ij} \ge \frac{b}{|i-j|^2} , \quad \forall i\ne j\in\bZ
\ee
and
\be\label{displ2}
    B_{ij} \ge \frac{r}{|i-j|^2}, \quad \mbox{$\forall i,j\in\bZ$ with $|i-j|\ge a^{-1}$}. 
\ee
Then for any function $f:\bZ\to \bR$ we have
\be\label{GNdiscr}
  \| f\|_4^4 \le   \frac{C}{r} \| f\|_2^2 \sum_{i\ne j} B_{ij}|f_i-f_j|^2
+\frac{C}{ab^3}  \|f\|_\infty^4.
\ee
\item[(ii)]  [Local version] Let $\cI=\llbracket Z- L, Z+L\rrbracket \subset \bZ$ be a subinterval of length $|\cI|= 2L+1$
around $Z\in \bZ$ 
and let $\wh \cI: = \llbracket Z- (1+\tau)L, Z+(1+\tau)L\rrbracket \subset \bZ$ be a slightly larger interval,
where $\tau>0$.
 Assume that for some positive constants $a, b, r$ with $b\le r\le 1$, we have
\be\label{displ1loc}
   B_{ij} \ge \frac{b}{|i-j|^2} , \quad \forall i,j\in\wh \cI
\ee
and
\be\label{displ2loc}
   B_{ij} \ge \frac{r}{|i-j|^2}, \quad \mbox{$\forall i,j\in\wh \cI$ with $|i-j|\ge a^{-1}$}. 
\ee
Then for any function $f:\bZ\to \bR$ with $\mbox{supp} (f) \subset \cI$ we have
\be\label{GNdiscrloc}
  \| f\|_4^4 \le   C \| f\|_2^2 \Big[ \frac{1}{r} \sum_{i\ne j\in \cI} B_{ij}|f_i-f_j|^2 + \frac{1}{L\tau} \| f\|_2^2
  \Big]
+\frac{C}{ab^3}  \|f\|_\infty^4.
\ee
\end{itemize} 
\end{proposition}

{\it Proof.} Following the proof of Proposition~\ref{prop:GNd}, for any $f:\bZ\to \bR$
we define its continuous extension $\phi$ by \eqref{explic}. Then the combination of
\eqref{normf} and \eqref{newGN} yields
$$
     \| f\|_4^4 \le  C\| f\|_2^2 \langle \phi, m(\sqrt{-\Delta})\phi\rangle+
 \frac{C}{ab^3}  \|f\|_\infty^4,
$$
where $m$ is given in \eqref{mdef} and $a,b$ will be determined later.
We compute
\be\label{ssp}
    \langle \phi, m(\sqrt{-\Delta})\phi\rangle \le b \langle \phi, \sqrt{-\Delta}\phi\rangle +  
  \langle\phi, \sqrt{-\Delta} \; \chi_{2a}^2(\sqrt{-\Delta})\phi\rangle,
\ee
where we used that ${\bf 1}(|\xi|\le a) \le \chi_{2a}^2(\xi)$ by the definition of $\chi$ 
at the beginning of the proof of Proposition~\ref{prop:newGN}.
The first term is bounded by
\be\label{ff}
   b \langle \phi, \sqrt{-\Delta}\phi\rangle = b\int_\bR\int_\bR  \frac{|\phi(x)-\phi(y)|^2}{|x-y|^2} \rd x\rd y 
\le Cb\sum_{i<j} \frac{|f_i-f_j|^2}{|i-j|^2}
  \le \sum_{i<j} B_{ij} |f_i-f_j|^2
\ee
using \eqref{claim} in the first estimate and  \eqref{displ1} in the second one.
In the second term in \eqref{ssp} we use the trivial arithmetic inequality
$$
  |\xi| \chi_{2a}^2(\xi) \le Q(\xi) \quad \mbox{with} \quad
 Q(\xi) : = 100 a \big( 1- e^{-|\xi|/a}\big).
$$
Thus
$$ 
  \langle\phi, \sqrt{-\Delta} \; \chi_{2a}^2(\sqrt{-\Delta})\phi\rangle \le \int_\bR |\wh\phi(\xi)|^2 Q(\xi) \rd \xi
  = 50 \int_\bR\int_\bR \frac{|\phi(x)-\phi(y)|^2 }{(x-y)^2 + a^{-2}}\rd x 
\rd y.
$$
Mimicking the argument leading to \eqref{claim}, we can continue this estimate
\be\label{gg}
  \langle\phi, \sqrt{-\Delta} \; \chi_{2a}^2(\sqrt{-\Delta})\phi\rangle  \le C\sum_{i\ne j\in \bZ} 
  \frac{ |f_i-f_j|^2}{|i-j|^2 + a^{-2}} \le  \frac{C}{r}\sum_{i\ne j\in \bZ} 
  B_{ij} |f_i-f_j|^2, 
\ee
where we used \eqref{displ2} in the last step. This completes the proof of \eqref{GNdiscr}.

The proof of  \eqref{GNdiscrloc} is very similar, just in the
 very last estimates of \eqref{ff} and \eqref{gg}
we use that $f$ is supported in $\cI$. E.g. in \eqref{ff} we have
$$
  b\sum_{i<j} \frac{|f_i-f_j|^2}{|i-j|^2} = b\sum_{i<j\in \wh \cI}  \frac{|f_i-f_j|^2}{|i-j|^2} + 2b\sum_{i\in \cI} |f_i|^2
 \sum_{j\not \in \wh \cI}  \frac{1}{|i-j|^2} \le \sum_{i<j} B_{ij}|f_i-f_j|^2 + \frac{2}{L\tau} \| f\|_2^2,
$$
and the estimate in \eqref{gg} is analogous.  
\qed


\begin{thebibliography}{99}

\bibitem{AGZ}  Anderson, G., Guionnet, A., Zeitouni, O.:  {\it An Introduction
to Random Matrices.} Studies in advanced mathematics, {\bf 118}, Cambridge
University Press, 2009.


\bibitem{APS} Albeverio, S., Pastur, L.,  Shcherbina, M.:
 On the $1/n$ expansion for some unitary invariant ensembles of random matrices,
 {\it Commun. Math. Phys.}  {\bf 224}, 271--305 (2001).



\bibitem{BM} Bach, V.: Moller, J.-S.: Correlation at low temperature. I. Exponential
decay. {\it J. Funct. Anal.} {\bf 203} (2003), no.1. 93--148.



\bibitem{BB}  Ben Arous, G., Bourgade, P.: Extreme gaps between eigenvalues of random matrices.
Preprint. arxiv:1010.1294


\bibitem{BP} Ben Arous, G., P\'ech\'e, S.: Universality of local
eigenvalue statistics for some sample covariance matrices.
{\it Comm. Pure Appl. Math.} {\bf LVIII.} (2005), 1--42.


\bibitem{BI} Bleher, P.,  Its, A.:  Semiclassical asymptotics of
orthogonal polynomials, Riemann-Hilbert problem, and universality
 in the matrix model.  {\it Ann. of Math.} {\bf 150}, 185--266 (1999).


\bibitem{BEY}  Bourgade, P.,  Erd{\H o}s, L., Yau, H.-T.:
{\it Universality of  General $\beta$-Ensembles.} 
Preprint. arxiv:1104.2272. To appear in Duke J. Math.

\bibitem{BEY2}  Bourgade, P.,  Erd{\H o}s, L., Yau, H.-T.:
 Bulk Universality of  General $\beta$-Ensembles with Non-convex Potential.
{\it J. Math. Phys.} {\bf 53}, 095221 (2012) 


\bibitem{BEYedge} P. Bourgade, L. Erd{\H o}s,  H.-T. Yau:
{\it Edge Universality of  Beta Ensembles.} Preprint. arxiv:1306.5728

\bibitem{BPS}
Boutet de Monvel, A.,  Pastur, L.,  Shcherbina, M.:  On the statistical mechanics approach
in the Random Matrix Theory. Integrated Density of States. {\it  J. Stat. Phys.}
{\bf 79} , 585--611 (1995).





\bibitem{C} Caffarelli, L., Chan, C.H., Vasseur, A.: Regularity theory for parabolic nonlinear
integral operators, {\it J. Amer. Math. Soc.} {\bf 24}, No. 3 (2011), 849--889.


\bibitem{CD}  Cotar, C., Deuschel, J.D.: Decay of covariances, 
uniqueness of ergodic component and scaling
limit for a class of $\nabla\phi$ systems with non-convex potential.
{\it Annales Inst. H. Poincar\'e (B),  Probability and Statistics.}
{\bf 48}, no. 3, 819--853 (2012)

\bibitem{CDM} Cotar, C., Deuschel, J.D., M\"uller, S.: 
Strict convexity of the free energy for non-convex gradient
models at moderate $\beta$, {\it Comm. Math. Phys.} {\bf 286}, No. 1 (2009), 359-376.

\bibitem{De1} Deift, P.: Orthogonal polynomials and
random matrices: a Riemann-Hilbert approach.
{\it Courant Lecture Notes in Mathematics} {\bf 3},
American Mathematical Society, Providence, RI, 1999.


\bibitem{DG} Deift, P., Gioev, D.: Universality in random matrix theory for
 orthogonal and symplectic ensembles. Int. Math. Res. Pap. IMRP 2007, no. 2, Art. ID rpm004, 116 pp


\bibitem{DG1} Deift, P., Gioev, D.:  Random Matrix Theory: Invariant
Ensembles and Universality. {\it Courant Lecture Notes in Mathematics} {\bf 18},
American Mathematical Society, Providence, RI, 2009.


\bibitem{DKMVZ1} Deift, P., Kriecherbauer, T., McLaughlin, K.T-R,
 Venakides, S., Zhou, X.: Uniform asymptotics for polynomials
orthogonal with respect to varying exponential weights and applications
 to universality questions in random matrix theory. {\it  Comm. Pure Appl. Math}.  {\bf 52}, 1335--1425 (1999).

\bibitem{DKMVZ2} Deift, P., Kriecherbauer, T., McLaughlin, K.T-R,
 Venakides, S., Zhou, X.: Strong asymptotics of orthogonal polynomials with respect to exponential weights. {\it
Comm.  Pure Appl. Math.} {\bf 52}, 1491--1552 (1999).

\bibitem{D} Deuschel, J.D.:
The Random Walk Representation for Interacting Diffusion Processes, in Interacting
Stochastic Systems, Springer, 378-391, (2005).


\bibitem{DGI} Deuschel, J.-D., Giacomin, G., Ioffe, D.: Large deviations and 
concentration properties for $\nabla\varphi$ interface models.
{\it Probab. Th. Relat. Fields.} {\bf 117} (2000), 49--111.

\bibitem{DE} Dumitriu, I., Edelman, A.: Matrix models for beta ensembles.
{\it J. Math. Phys.} {\bf 43} (2002),  5830--5847.




\bibitem{Dy} Dyson, F.J.: A Brownian-motion model for the eigenvalues
of a random matrix. {\it J. Math. Phys.} {\bf 3}, 1191--1198 (1962).



\bibitem{EKYY2} Erd{\H o}s, L., Knowles, A.,  Yau, H.-T., Yin, J.:
 Spectral Statistics of Erd{\H o}s-R\'enyi Graphs II:
 Eigenvalue Spacing and the Extreme Eigenvalues.
{\it Comm. Math. Phys.} {\bf 314} no. 3. 587--640 (2012)






\bibitem{EPRSY}  Erd{\H o}s, L., P\'ech\'e, S., 
 Ramirez, J., Schlein, B., Yau, H.-T.:
 Bulk universality 
for Wigner matrices. {\it Comm. Pure Appl. Math.}
 {\bf 63}, No. 7,  895--925 (2010)



\bibitem{ERSTVY} Erd\H{o}s, L.,  Ram\'irez, J.,  Schlein,  B., Tao, T., Vu, V. and Yau, H.-T.,
Bulk universality for Wigner hermitian matrices with subexponential decay.
 Math. Res. Lett. {\bf 17} (2010), no. 4, 667--674.


\bibitem{ESY4} Erd{\H o}s, L.,  Schlein, B.,  Yau, H.-T.:
 Universality of Random Matrices and Local Relaxation Flow.
Invent. Math. {\bf 185} (2011), no.1, 75--119.


\bibitem{ESYY} Erd{\H o}s, L.,  Schlein, B.,  Yau, H.-T.,  Yin, J.:
 The local relaxation flow approach to universality of the local
statistics for random matrices.
 {\it Annales Inst. H. Poincar\'e (B),  Probability and Statistics.}
{\bf 48}, no. 1, 1--46 (2012)

\bibitem{ErdRamSchYau}
Erd{\H o}s, L.,  Ram\'irez, J.-A., Schlein, B.,  Yau, H.-T.:
 Universality of sine-kernel for Wigner matrices with a small Gaussian perturbation,
{\it Electronic Journal of Probability,} {\bf 15} (2010), Paper no. 18, pages 526--603.



\bibitem{EYYBand}  Erd{\H o}s, L.,  Yau, H.-T.,  Yin, J.:
Bulk universality for generalized Wigner matrices.
{\it Prob. Theor. Rel. Fields},  {\bf 154}, no. 1-2., 341--407 (2012)  


\bibitem{EYYBern} Erd{\H o}s, L.,  Yau, H.-T.,  Yin, J.:
{\it Universality for generalized Wigner matrices with Bernoulli
distribution.}   J. of Combinatorics, {\bf 1} (2011), no. 2, 15--85



\bibitem{EYYrigi} Erd{\H o}s, L.,  Yau, H.-T., Yin, J.,
 Rigidity of Eigenvalues of Generalized Wigner Matrices.
{\it Adv. Math.}   {\bf 229}, no. 3, 1435--1515 (2012)





\bibitem{EYY2}  Erd{\H o}s, L.,  Yau, H.-T., Yin, J.: 
Universality for generalized Wigner matrices with Bernoulli
distribution.  {\it  J. of Combinatorics}, {\bf 1}, no. 2, 15--85 (2011)




\bibitem{EYBull} Erd{\H o}s, L.,  Yau, H.-T.,
 Universality of local spectral statistics of random matrices.
{\it Bull. Amer. Math. Soc.} {\bf 49}, no.3, 377--414 (2012)


\bibitem{EY} Erd{\H o}s, L.,  Yau, H.-T.,
A comment on the Wigner-Dyson-Mehta bulk universality conjecture for Wigner matrices.
{\it Electron. J. Probab.} {\bf 17}, no 28. 1--5 (2012)
    
 
\bibitem{FK} Felsinger, M.; Kassmann, M.: Local regularity
for parabolic nonlocal operators. {\it Comm. Partial Differential Equations} {\bf 38}, no. 9, 1539--1573 (2013)

\bibitem{FIK} Fokas, A. S.; Its, A. R.; Kitaev, A. V.:
 The isomonodromy approach to matrix models in $2$D quantum gravity.
{\it  Comm. Math. Phys.}  {\bf 147}  (1992),  no. 2, 395--430.


\bibitem{For}  Forrester, P. J.  {\it Log-gases and random matrices.}
London Mathematical Society Monographs Series, 34. Princeton University Press, Princeton, NJ,  2010. 



\bibitem{GOS} Giacomin, G., Olla, S., Spohn, H.: Equilibrium
fluctuations for $\nabla\varphi$ interface model. 
{\it Ann. Probab.} {\bf 29} (2001), no.3., 1138--1172


\bibitem{Gus}
Gustavsson, J. :  Gaussian fluctuations of eigenvalues in the GUE. 
{\it Ann. Inst. H. Poincar\'e Probab. Statist.} 
 41 (2005), no. 2, 151–-178


\bibitem{He} Helffer, B.: Semiclassical Analysis, Witten Laplacians and Statistical
Mechanics, {\it Series on Partial Differential Equations and Applications.} Vol. 1.
World Scientific, 2002.

\bibitem{HS} Helffer, B., Sj\"ostrand, J.: On the correlation for Kac-like models in the convex case.
{\it J. Statis. Phys.} {\bf 74} (1994), no.1-2, 349--409.



\bibitem{J} Johansson, K.: Universality of the local spacing
distribution in certain ensembles of Hermitian Wigner matrices.
{\it Comm. Math. Phys.} {\bf 215} (2001), no.3. 683--705.


\bibitem{KY}     Knowles, A. and Yin, J: 
    Eigenvector Distribution of Wigner Matrices. {\it 
 Probab. Theory Related Fields} {\bf 155}, no. 3-4, 543--582 (2013)

    
\bibitem{KS}  Kriecherbauer, T.,  Shcherbina, M.:
 Fluctuations of eigenvalues of matrix models and their applications. Preprint {\tt arXiv:1003.6121}




\bibitem{M} Mehta, M.L.: Random Matrices. Academic Press, New York, 1991.

\bibitem{NS} Naddaf, A., Spencer, T.: On homogenization and scaling limit of some
gradient perturbations of a massless free field, {\it Commun. Math. Phys.} {\bf 183} (1997),
no.1., 55--84.

\bibitem{OO}
Ogawa, T.; Ozawa, T.: Trudinger type inequalities and uniqueness of weak solutions for the nonlinear Schrödinger mixed problem. 
{\it J. Math. Anal. Appl.} {\bf  155} (1991), no. 2, 531–540.


\bibitem{PS:97} Pastur, L., Shcherbina, M.: Universality of the local
eigenvalue statistics for a class of unitary invariant random
matrix ensembles. {\it J. Stat. Phys.} \textbf{86}, 109--147,
(1997)

\bibitem{PS} Pastur, L., Shcherbina M.:
Bulk universality and related properties of Hermitian matrix models.
{\it J. Stat. Phys.} {\bf 130}, no. 2., 205--250  (2008).


\bibitem{Sch}  Shcherbina, M.:
Orthogonal and symplectic matrix models: universality and other properties.
{\it Comm. Math. Phys.} {\bf 307}, no. 3, 761–-790 (2011).

\bibitem{TV} Tao, T. and Vu, V.: Random matrices: Universality of 
the local eigenvalue statistics.  \emph{ Acta Math.}, 
\textbf{ 206}, no. 1, 127--204 (2011).

\bibitem{Taogap} Tao, T.: The asymptotic distribution of a single eigenvalue gap of a Wigner matrix.
{\it Probab. Theory Related Fields} {\bf 157}, no. 1-2, 81--106 (2013)


\bibitem{TW} Tracy, C., Widom, H.: Level spacing distributions and the Airy kernel.
{\it Commun. Math. Phys.}  {\bf 159}, 151--174, 1994.

\bibitem{TW2}   Tracy, C., Widom, H.: On orthogonal and symplectic matrix ensembles,
{\it Comm. Math. Phys.} {\bf 177} (1996), no. 3, 727--754.



\bibitem{VV} Valk\'o, B.; Vir\'ag, B.:
  Continuum limits of random matrices and the Brownian carousel. {\it Invent. Math.}
{\bf  177} (2009), no. 3, 463--508.

\bibitem{VV2} Valk\'o, B.; Vir\'ag, B.:
Large gaps between random eigenvalues. {\it Ann. Probab.}
{\bf 38}, No. 3 (2010), 1263--1279.



\bibitem{Wid} Widom H.:  On the relation between orthogonal, symplectic and
 unitary matrix ensembles. {\it J. Statist. Phys.} {\bf 94} (1999), no. 3-4, 347--363.



\bibitem{W} Wigner, E.: Characteristic vectors of bordered matrices 
with infinite dimensions. {\it Ann. of Math.} {\bf 62} (1955), 548-564.

\end{thebibliography}
\end{document}